\documentclass[12pt]{article}
\usepackage{myextension}
\usepackage{tocloft}
\title{Balanced Steinhaus triangles}
\author{Jonathan CHAPPELON\footnote{E-mail address: \href{mailto:jonathan.chappelon@umontpellier.fr}{jonathan.chappelon@umontpellier.fr}}}
\affil{IMAG, Université de Montpellier, CNRS, Montpellier, France}
\date{August 7, 2025}

\begin{document}
\maketitle
\begin{abstract}
A Steinhaus triangle modulo $m$ is a finite down-pointing triangle of elements in the finite cyclic group $\mathbb{Z}/m\mathbb{Z}$ satisfying the same local rule as the standard Pascal triangle modulo $m$. A Steinhaus triangle modulo $m$ is said to be balanced if it contains all the elements of $\mathbb{Z}/m\mathbb{Z}$ with the same multiplicity. In this paper, the existence of infinitely many balanced Steinhaus triangles modulo $m$, for any positive integer $m$, is shown. This is achieved by considering periodic triangles generated from interlaced arithmetic progressions. This positively answers a weak version of a problem, due to John C. Molluzzo in 1978, that has remained unsolved to date for the even values of $m\geqslant 12$.
\end{abstract}
\\[2ex]
\noindent\fbox
  {%
    \begin{minipage}{\dimexpr\linewidth-2\fboxsep-2\fboxrule}%
      \vspace{-15pt}\tableofcontents
    \end{minipage}%
  }%

\section{Introduction}

All along this paper, $m$ and $n$ are non-negative integers and we consider the finite cyclic group $\Zn{m}$, where $\Zn{0}$ is isomorphic to $\Z$. The set of non-negative integers is denoted by $\N$. Let $t_{n}$ denote the $n$\textsuperscript{th} triangular number $t_{n}=\sum_{i=0}^{n}i=\frac{n(n+1)}{2}$ and let $\Tn{n}$ be the triangle of integers
$$
\Tn{n} = \left\{ (i,j)\in\N^2\ \middle|\ i+j<n\right\}.
$$

\begin{defn}[Steinhaus triangles modulo $m$]
A {\em Steinhaus triangle modulo $m$} of {\em size} $n$ is a down-pointing triangle $\left(a_{i,j}\right)_{(i,j)\in\Tn{n}}$ of $t_{n}$ elements in $\Zn{m}$ satisfying the same local rule as the standard Pascal triangle modulo $m$, that is,
\begin{equation}\label{eqP}
a_{i,j} = a_{i-1,j} + a_{i-1,j+1},
\end{equation}
for all $(i,j)\in\Tn{n}$ such that $i\ge1$, where the sum is the sum in $\Zn{m}$. Note that a Steinhaus triangle is completely determined by its first row $\left(a_{0,j}\right)_{j=0}^{n-1}$.
\end{defn}

\begin{figure}[htbp]
\centering{
\includegraphics{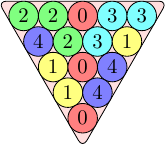}
}
\caption{A Steinhaus triangle modulo $5$ of size $5$}\label{fig*01}
\end{figure}

An example of Steinhaus triangle modulo $5$ of size $5$ is depicted in Figure~\ref{fig*01}. This kind of triangles have been introduced by Hugo Steinhaus himself, for the binary case ($m=2$), in \cite{Steinhaus:1964aa} and by John C. Molluzzo, for $m\ge3$, in \cite{Molluzzo:1978aa}, where he posed the following

\begin{prob}[Molluzzo Problem]\label{prob1}
Does there exist, for any positive integers $m$ and $n$ such that $t_n$ is divisible by $m$, a Steinhaus triangle modulo $m$ of size $n$ containing all the elements of $\Zn{m}$ with the same multiplicity?
\end{prob}

A Steinhaus triangle containing all the elements of $\Zn{m}$ with the same multiplicity is said to be {\em balanced}. For instance, the Steinhaus triangle depicted in Figure~\ref{fig*01} is balanced and it positively answers Molluzzo problem for $m=5$ and $n=5$, since each element of $\Zn{5}$ occurs three times in it.

The Molluzzo problem is still largely open. Since then, it has been positively solved for small values of $m$: for $m=2$ in \cite{Harborth:1972aa,Eliahou:2004aa,Eliahou:2005aa,Eliahou:2007aa,Chappelon2017a}, for $m\in\{3,5\}$ in \cite{Bartsch1985}, for $m\in\{3,5,7\}$ in \cite{Chappelon2008a} and for $m=4$ in \cite{Chappelon2012a}. First counter-examples appeared in \cite{Chappelon2008a}, where it is proved that there does not exist balanced Steinhaus triangles of size $5$ in $\Zn{15}$ and of size $6$ in $\Zn{21}$. Nevertheless, this problem can be positively answered for an infinite number of values of $m$. Indeed, as showed in \cite{Chappelon2008,Chappelon2008a}, there exist balanced Steinhaus triangles, for all the possible sizes, in the case where $m$ is a power of $3$. This result was obtained by studying Steinhaus triangles whose rows are arithmetic progressions. Even if the Molluzzo problem is not completely solved for the other odd values of $m$, we know that there exist infinitely many balanced Steinhaus triangles in every $\Zn{m}$ with $m$ odd \cite{Chappelon2008,Chappelon2011}. This weak version of the Molluzzo problem was posed in \cite{Chappelon2012a}.

\begin{prob}[Weak Molluzzo Problem]\label{prob2}
Does there exist, for any positive integer $m$, infinitely many balanced Steinhaus triangles modulo $m$?
\end{prob}

Problem~\ref{prob2} is thus solved for the odd numbers $m$. For the even values, the cases $m=2$ and $m=4$ come from the solutions of Problem~\ref{prob1} and a solution is known from \cite{Eliahou:aa} for $m\in\{6,8,10\}$. So far, this problem was completely open for the even numbers $m\ge12$. In this paper, Problem~\ref{prob2} is positively solved for all positive integers $m$, even if $m$ is even. For every positive integer $m$, an explicit construction of balanced Steinhaus triangles modulo $m$ of size $12\lambda m$, for all non-negative integers $\lambda$, is given. More precisely, the main result of this paper if the following


\begin{thm}\label{thm*2}
For every positive integer $m$, there exists a $12m$-tuple $A_m$ of $\Zn{m}$ such that the Steinhaus triangle whose first row is ${A_m}^\lambda$, i.e., the tuple $A_m$ repeated $\lambda$ times, is a balanced Steinhaus triangle of size $12\lambda m$ in $\Zn{m}$, for all non-negative integers $\lambda$.
\end{thm}

This paper is organized as follows. After some preliminary results in Section~2, the interlaced arithmetic progressions (IAP for short) are studied in Section~3. A characterization of periodic orbits of IAP is obtained in Section~4. This leads to the determination, in Section~5, of a particular set of integer sequences whose projection into $\Zn{2^u}$, for all non-negative integers $u$, gives solutions of Theorem~\ref{thm*2} for the powers of two. After studying arithmetic sums of binomials modulo powers of two in Section~6, it is shown, in Section~7, that the orbits of solutions given in Section~5 are periodic. In order to complete the proof, the interlaced doubly arithmetic structure of orbits is introduced in Section~8 and it is proved, in Section~9, that the orbits of solutions given in Section~5 have this particular structure. With elementary results on arithmetic triangles obtained in Section~10, we complete the proof of Theorem~\ref{thm*2} for the powers of two in Section~11. In Section~12, the case where $m$ is odd is studied. It is rediscovered that there exist integer sequences whose projection into $\Zn{m}$ gives solutions of Theorem~\ref{thm*2} for any odd number $m$. This result was already obtained in \cite{Chappelon2008,Chappelon2011}. Finally, in Section~13, using solutions of Theorem~\ref{thm*2} obtained in Section~5 for the powers of two and in Section~12 for the odd numbers, solutions of Theorem~\ref{thm*2} for even values of $m$ are given. This completes the proof of Theorem~\ref{thm*2} and positively answers Problem~\ref{prob2} (Weak Molluzzo Problem) for any positive integer $m$, even if $m$ is even.

\section{Preliminaries}

Other local rules as well as the problem of the existence of balanced structures can be considered. Generalizations in higher dimensions can be found in \cite{Chappelon2015}. The local rule considered in this paper is slightly different from that of the classical Pascal triangle. Moreover, the triangles considered here are viewed as finite subparts of doubly indexed sequences.

\begin{defn}[Derived sequences]
Let $S=\left(u_j\right)_{j\in\Z}$ be a sequence of elements in $\Zn{m}$. The {\em derived sequence} $\der{S}$ of $S$ is the sequence
$$
\der{S} = \left(-u_j-u_{j+1}\right)_{j\in\Z}.
$$
Similarly, for a finite sequence $S=\left(u_j\right)_{j=0}^{n-1}$ of length $n\ge 2$, the {\em derived sequence} $\der{S}$ of $S$ is the $(n-1)$-length sequence defined by
$$
\der{S} = \left(-u_j-u_{j+1}\right)_{j=0}^{n-2}.
$$
When $n\le 1$, the derived sequence $\der{S}$ is the empty sequence $\emptyset$.
\end{defn}

\begin{rem}
When $m\le 10$ and there is no disambiguation, the tuple or the sequence $(u_0,u_1,u_2,\ldots\ldots)$ of $\Zn{m}$ will be simply denoted by $u_0u_1u_2\cdots\cdots$ in the sequel.
\end{rem}

For instance, the derived sequence of $S=22033$ in $\Zn{5}$ is the sequence $\der{S}=1324$. This process of derivation can be repeated and the sequence of iterated derived sequences of $S$ is called the {\em orbit} of $S$.

\begin{defn}[Iterated derived sequences]
For any positive integer $i$, the {\em $i$\textsuperscript{th} derived sequence} $\ider{i}{S}$ of $S$ is recursively defined by
$$
\ider{i}{S} = \der{\ider{i-1}{S}},
$$
where $\ider{0}{S}=S$.
\end{defn}

\begin{defn}[Orbits and triangles]
For any sequence $S$ of elements in $\Zn{m}$, the {\em orbit} $\orb{S}$ of $S$ is the sequence of all its iterated derived sequences
$$
\orb{S} = \left(\ider{i}{S}\right)_{i\in\N}.
$$
When $S$ is an infinite sequence, the orbit of $S$ can also be seen as the doubly indexed sequence
$$
\orb{S}=\left(a_{i,j}\right)_{(i,j)\in\N\times\Z}
$$
recursively defined by
$$
\left(a_{0,j}\right)_{j\in\Z} = S
$$
and
\begin{equation}\label{eqLR}
a_{i,j} = -a_{i-1,j}-a_{i-1,j+1}
\end{equation}
for all $(i,j)\in\N^\ast\times\Z$. When $S$ is a finite sequence of length $n$, the orbit of $S$ can also be seen as the triangle of $\Tn{n}$ numbers
$$
\orb{S} = \left(a_{i,j}\right)_{(i,j)\in\Tn{n}}
$$
recursively defined by
$$
\left(a_{0,j}\right)_{j=0}^{n-1} = S
$$
and
$$
a_{i,j} = -a_{i-1,j}-a_{i-1,j+1},
$$
for all $(i,j)\in \Tn{n}$ such that $i\ge1$. The orbit $\orb{S}$ is also called the {\em triangle} $\ST{S}$ associated to the finite sequence $S$. The {\em size} of $\ST{S}$ is $n$, the number of elements of $S$ or the number of rows and columns of $\ST{S}$.
\end{defn}

An example of an orbit of a sequence of $\Zn{5}$ and several triangles appearing in it is depicted in Figure~\ref{fig*05}.

\begin{figure}[htbp]
\centering{
\includegraphics{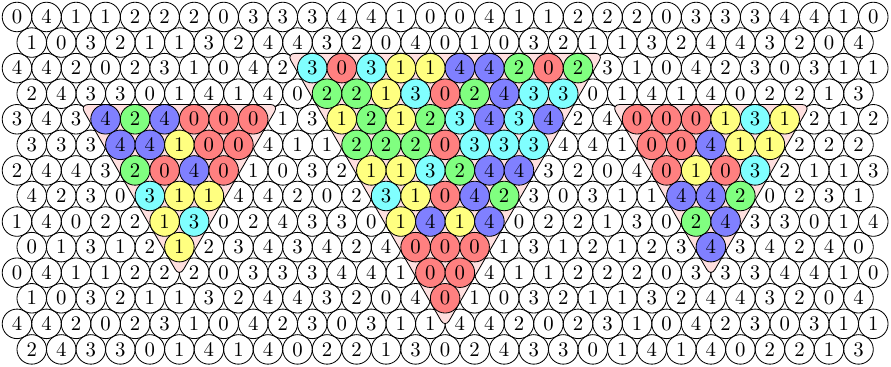}
}
\caption{Orbit of a sequence of $\Zn{5}$ and triangles}\label{fig*05}
\end{figure}

\begin{defn}[Operations on tuples, sequences or triangles]
Let $A=(a_0,a_1,\ldots,a_{n-1})$ and $B=(b_0,b_1,\ldots,b_{n-1})$ be two $n$-tuples (or finite sequences of length $n$) of $\Zn{m}$ and let $\lambda$ be an integer or in $\Zn{m}$. The $n$-tuple $A+B$ is defined by
$$
A+B = (a_0+b_0,a_1+b_1,\ldots\ldots,a_{n-1}+b_{n-1}).
$$
The $n$-tuple $\lambda A$ is defined by
$$
\lambda A = (\lambda a_0,\lambda a_1,\ldots\ldots,\lambda a_{n-1}).
$$
Similarly, for two sequences $A=(a_j)_{j\in\Z}$ and $B=(b_j)_{j\in\Z}$, the sequences $A+B$ and $\lambda A$ are defined by
$$
A+B=(a_j+b_j)_{j\in\Z} \quad\text{and}\quad \lambda A = (\lambda a_j)_{j\in\Z}.
$$
Finally, for two triangles $\ST{A}$ and $\ST{B}$ of size $n$, the triangles $\ST{A}+\ST{B}$ and $\lambda\ST{A}$ are defined by
$$
\ST{A}+\ST{B} = \ST{(A+B)}\quad\text{and}\quad\lambda\ST{A}=\ST{(\lambda A)}.
$$
\end{defn}

The set $\setT{m}{n}$ of triangles of size $n$ of $\Zn{m}$ has thus a natural structure of free $\Zn{m}$-module of rank $n$, for which the set of triangles $\{\ST{E_1},\ST{E_2},\ldots,\ST{E_n}\}$ is a basis, where $E_i$ is the {\em elementary $n$-tuple}
$$
E_i = (\delta_{i,j})_{j=1}^{n} = \underbrace{0\cdots0}_{i-1}1\underbrace{0\cdots0}_{n-i},
$$
for all $i\in\{1,2,\ldots,n\}$, with $\delta_{i,j}$ the Kronecker delta function ($1$ if $j=i$, $0$ otherwise). There is a natural isomorphism between $\setT{m}{n}$ and the module of $n$-tuples of $\Zn{m}$ and they both have $m^n$ elements.

For any non-negative integers $a$ and $b$ such that $b\le a$, the {\em binomial coefficient} $\binom{a}{b}$ is the coefficient of the monomial $X^b$ in the polynomial expansion of the binomial power $(1+X)^a$. It corresponds to the number of ways to choose $b$ elements in a set of $a$ elements. Here, we extend this notation by supposing that $\binom{a}{b}=0$, for all integers $b$ such that $b<0$ or $b>a$. For this generalization, the Pascal identity $\binom{a}{b}=\binom{a-1}{b-1}+\binom{a-1}{b}$ holds, for all positive integers $a$ and all integers $b$.

From the definition of the local rule \eqref{eqLR}, it is straightforward to express each element of the orbit $\orb{S}=\left(a_{i,j}\right)_{(i,j)\in\N\times\Z}$ as a function of the elements of the $i_0$\textsuperscript{th} row $\mathrm{R}_{i_0}=\left(a_{i_0,j}\right)_{j\in\Z}$, the $j_0$\textsuperscript{th} column $\mathrm{C}_{j_0}=\left(a_{i,j_0}\right)_{i\in\N}$ or the $k_0$\textsuperscript{th} antidiagonal $\mathrm{AD}_{k_0}=\left(a_{i,k_0-i}\right)_{i\in\N}$, when $i_0\le i$, $j_0\le j$ or $k_0\ge i+j$.

\begin{prop}
Let $S$ be a doubly infinite sequence of elements in $\Zn{m}$. In the orbit $\orb{S}=\left(a_{i,j}\right)_{(i,j)\in\N\times\Z}$, we have
\begin{equation*}
\resizebox{\textwidth}{!}{$
a_{i,j} = {(-1)}^{i-i_0}\displaystyle\sum_{k\in\Z}\binom{i-i_0}{k-j}a_{i_0,k} = {(-1)}^{j-j_0}\sum_{k\in\N}\binom{j-j_0}{k-i}a_{k,j_0} = {(-1)}^{k_0-(i+j)}\sum_{k\in\N}\binom{k_0-(i+j)}{k-i}a_{k,k_0-k},
$}
\end{equation*}
for all $(i,j)\in\N\times\Z$, where $i\ge i_0$, $j\ge j_0$ and $i+j\le k_0$.
\end{prop}

\begin{proof}
By induction on $(i,j)$ using the local rule \eqref{eqLR}.
\end{proof}

Let $r$ and $h$ be the $120$ degrees rotation and the horizontal reflection, respectively, of a triangle of $\setT{m}{n}$, that are the automorphisms of $\setT{m}{n}$ defined by
$$
r :
\begin{array}[t]{ccc}
\setT{m}{n} & \longrightarrow & \setT{m}{n} \\
\left(a_{i,j}\right)_{(i,j)\in\Tn{n}} & \longmapsto & \left(a_{j,n-i-j-1}\right)_{(i,j)\in\Tn{n}} \\
\end{array}
\ \text{and}\ \ 
h :
\begin{array}[t]{ccc}
\setT{m}{n} & \longrightarrow & \setT{m}{n} \\
\left(a_{i,j}\right)_{(i,j)\in\Tn{n}} & \longmapsto & \left(a_{i,n-i-j-1}\right)_{(i,j)\in\Tn{n}} \\
\end{array}
$$
for all non-negative integers $m$ and $n$. These automorphims verify the following identities
$$
r^3=h^2=hrhr=\mathrm{id}_{\setT{m}{n}},
$$
where $\mathrm{id}_{\setT{m}{n}}$ is the identity map on $\setT{m}{n}$. Therefore, the subgroup $\left\langle r,h\right\rangle$ of the automorphisms group of $\setT{n}{m}$ is isomorphic to the dual group $\mathrm{D}_3$. Unlike the set of Steinhaus triangles, the $\Zn{m}$--module $\setT{m}{n}$ of triangles built with the local rule \eqref{eqLR} is preserved under the action of the dihedral group $\mathrm{D}_3=\left\langle r,h\right\rangle$. For instance, for $S=03114$ in $\Zn{5}$ and for all $g\in \mathrm{D}_3$, the six triangles $g(\ST{S})$ are depicted in Figure~\ref{fig*02}.

\begin{figure}[htbp]
\centering{
\includegraphics{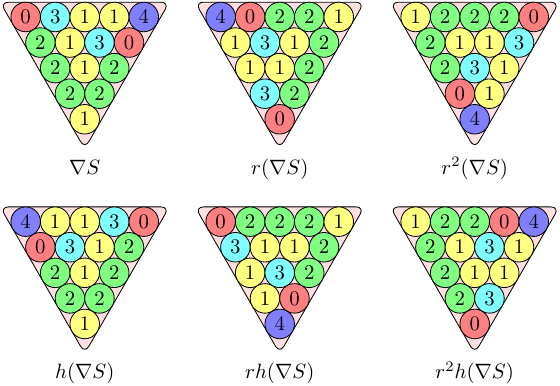}
}
\caption{Action of $\mathrm{D}_3=\left\langle r,h\right\rangle$ on $\ST{(03114)}$ in $\Zn{5}$}\label{fig*02}
\end{figure}

A triangle $\ST{S}$ is said to be {\em horizontally symmetric}, {\em rotationally symmetric} or {\em dihedrally symmetric} if $h(\ST{S})=\ST{S}$, $r(\ST{S})=\ST{S}$ or $r(\ST{S})=h(\ST{S})=\ST{S}$, respectively. Examples of such symmetric triangles appear in Figure~\ref{fig*03}. Symmetric triangles are studied in \cite{Barbe:2000aa,Brunat:2011aa,Chappelon2019} in the binary case ($m=2$) and in \cite{Barbe:2001aa} for $m\ge3$.

\begin{figure}[htbp]
\centering{
\includegraphics{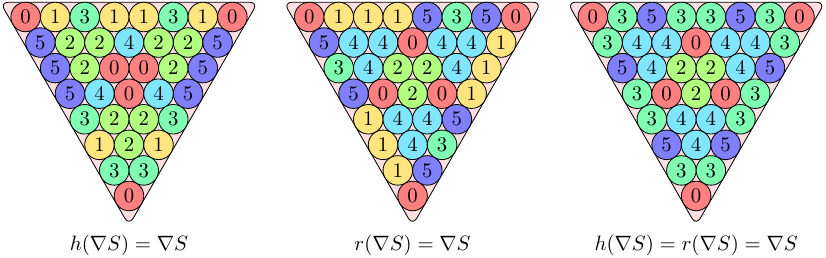}
}
\caption{Symmetric triangles of $\Zn{6}$}\label{fig*03}
\end{figure}

\begin{defn}[Multisets and multiplicity functions]
A finite {\em multiset} of elements of $\Zn{m}$ is a set of $\Zn{m}$ for which repeated elements are allowed. A finite multiset $M$ of $\Zn{m}$ corresponds to a function $\mf{M} : \Zn{m} \longrightarrow \N$, the {\em multiplicity function} associated to $M$, which assigns its multiplicity in $M$ to each element of $\Zn{m}$. The {\em cardinality} of $M$, denoted by $|M|$, is the number of elements of $M$, counted with multiplicity, that is, the non-negative integer $|M|=\sum_{x\in\Zn{m}}\mf{M}(x)$.
\end{defn}

\begin{defn}[Balanced multisets of $\Zn{m}$]
A finite multiset of $\Zn{m}$ is said to be {\em balanced} if its multiplicity function is constant, that is, if
$$
\mf{M}(x)=\mf{M}(y),\ \forall x,y\in\Zn{m}.
$$
The cardinality $|M|$ of a balanced multiset $M$ of $\Zn{m}$ is a multiple of $m$. When $|M|$ is not divisible by $m$, say $|M|=qm+r$ with $r\in\left\{1,\ldots,m-1\right\}$, the multiset $M$ is said to be {\em almost balanced} if $\mf{M}(x)\in\left\{q,q+1\right\}$, for all $x\in\Zn{m}$.
\end{defn}

The method presented in this paper permits us to obtain balanced triangles for both local rules~\eqref{eqP} and \eqref{eqLR} by considering antisymmetric sequences. In the sequel, the Steinhaus triangle whose first row is $S$ and built with local rule \eqref{eqP} is denoted by $\SteinT{S}$ and the one built with the local rule \eqref{eqLR} is simply denoted by $\ST{S}$. Note that, when $m=2$, the both local rules~\eqref{eqP} and \eqref{eqLR} correspond and $\SteinT{S}=\ST{S}$, for any finite binary sequence $S$ of $\Zn{2}$.

\begin{defn}[Antisymmetric sequences]
A finite sequence $\left(u_j\right)_{j=0}^{n-1}$ of length $n$ in $\Zn{m}$ is said to be {\em antisymmetric} if
$$
u_{n-1-j}=-u_j,
$$
for all $j\in\{0,1,\ldots,n-1\}$.
\end{defn} 

For instance, the sequence $1023405$ is antisymmetric in $\Zn{6}$. The antisymmetric structure is preserved under the derivation process. A proof can be found in \cite{Chappelon2008} for the local rule \eqref{eqP} and in Proposition~\ref{prop*1} for the local rule \eqref{eqLR}. Therefore, when $S$ is antisymmetric, all the rows of $\SteinT{S}$ and $\ST{S}$ are also antisymmetric and they both have the same multiplicity function.

\begin{prop}
For any antisymmetric sequence $S$ of $\Zn{m}$, we have $\mf{\SteinT{S}} = \mf{\ST{S}}$.
\end{prop}

\begin{proof}
Let $n$ be the length of $S$. For any $i\in\{1,\ldots,n\}$, the $i$\textsuperscript{th} row of $\SteinT{S}$ and $\ST{S}$ is denoted by $R_i\left(\SteinT{S}\right)$ and $R_i\left(\ST{S}\right)$, respectively. There exists a natural correspondence between them:
$$
R_i\left(\SteinT{S}\right) = {(-1)}^iR_i\left(\ST{S}\right),
$$
for all $i\in\{1,\ldots,n\}$. Moreover, for any antisymmetric sequence $T$, it is easy to see that $\mf{-T}=\mf{T}$. Since the rows $R_i\left(\SteinT{S}\right)$ and $R_i\left(\ST{S}\right)$ are antisymmetric, we deduce that
$$
\mf{R_i\left(\SteinT{S}\right)} = \mf{R_i\left(\ST{S}\right)},
$$
for all $i\in\{1,\ldots,n\}$. This leads to $\mf{\SteinT{S}} = \mf{\ST{S}}$.
\end{proof}

\begin{cor}\label{cor*7}
Let $S$ be an antisymmetric sequence of $\Zn{m}$. Then, the Steinhaus triangle $\SteinT{S}$ is balanced if and only if the triangle $\ST{S}$ is balanced.
\end{cor}

For instance, for the sequence $S=22033$ of Figure~\ref{fig*01}, the Steinhaus triangle $\SteinT{S}$ and the triangle $\ST{S}$ are balanced in $\Zn{5}$. The both triangles are depicted in Figure~\ref{fig*04}.

\begin{figure}[htbp]
\centering{
\includegraphics{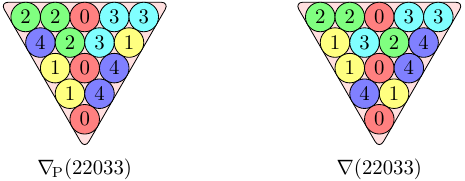}
}
\caption{The balanced triangles $\SteinT{(22033)}$ and $\ST{(22033)}$ modulo $5$}\label{fig*04}
\end{figure}

\begin{defn}[Concatenation and power of tuples]
Let $A=(a_0,a_1,\ldots,a_{r-1})$ and $B=(b_0,b_1,\ldots,b_{s-1})$ be two tuples of $r$ and $s$ elements of $\Zn{m}$, respectively. The {\em concatenation} of $A$ and $B$ is the $(r+s)$-tuple of $\Zn{m}$ defined by
$$
A\cdot B = \left(a_0,a_1,\ldots,a_{r-1},b_0,b_1,\ldots,b_{s-1}\right).
$$
For any non-negative integer $\lambda$, the $\lambda$\textsuperscript{th} power of $A$ is $(\lambda r)$-tuple of $\Zn{m}$ recursively defined by
$$
A^\lambda = A\cdot A^{\lambda-1},
$$
for all positive integers $\lambda$, where $A^0=\emptyset$ is the empty tuple. The tuple $A^\lambda$ then corresponds to the tuple $A$ repeated $\lambda$ times.
\end{defn}

The main result of this paper is the explicit construction of antisymmetric $(12m)$-tuples $A_m$ of $\Zn{m}$ such that the triangles $\ST{{A_m}^\lambda}$, of size $12\lambda m$, are balanced in $\Zn{m}$, for all non-negative integers $\lambda$, and this for every positive integer $m$. Since the sequences ${A_m}^\lambda$ are clearly antisymmetric, we know from Corollary~\ref{cor*7} that these sequences also give balanced Steinhaus triangles modulo $m$ of size $12\lambda m$, for all positive integers $\lambda$ and $m$. This leads to Theorem~\ref{thm*2}, that is a complete positive answer to Problem~\ref{prob2} (Weak Molluzzo Problem), that was previously open for all even numbers $m\ge 12$.

It is known that Molluzzo Problem and its weak version are natural problems, in the sense that the average number of every element of $\Zn{m}$ in the set of Steinhaus triangles modulo $m$ of size $n$ is exactly $\frac{1}{m}\Tn{n}$ (see \cite{Chappelon2017a} for the binary case). This is also the case for triangles built with the local rule \eqref{eqLR}.

\begin{prop}
For any non-negative integer $n$, the average number of elements $x\in\Zn{m}$ in a triangle of $\setT{m}{n}$ is exactly $\frac{1}{m}\Tn{n}$, i.e.,
$$
\frac{1}{m^n}\sum_{\nabla\in\setT{m}{n}}\mf{\nabla}(x) = \frac{1}{m}\Tn{n} = \frac{1}{m}\binom{n+1}{2},
$$
for all $x\in\Zn{m}$.
\end{prop}

\begin{proof}
First, it is clear that $\setT{m}{n}$ contains $m^n$ triangles and there are $m^n$ tuples of $n$ elements of $\Zn{m}$. We proceed by induction on $n\ge0$. The result is trivial for $n=0$ and $n=1$. Suppose now that $n\ge2$ and that the result is true for any triangle of $\setT{m}{n-1}$. Let $\ST{S}\in\setT{m}{n-1}$ be the triangle of size $n-1$ generated from the sequence $S=(a_0,a_1,\ldots,a_{n-2})$. There exist exactly $m$ sequences $S'$ of length $n$ such that $\der{S'}=S$. Then, we retrieve the triangle $\ST{S}$ as the subtriangle $\ST{S'}\setminus S'$, that is, the last $n-1$ rows of the triangle $\ST{S'}$ of size $n$. These $m$ sequences $S'$ are of the form
$$
S'
\begin{array}[t]{l}
= \displaystyle\left((-1)^{i}\left(x+\sum_{j=0}^{i-1}(-1)^ja_j\right)\right)_{i=0}^{n-1} \\ \ \\
= \displaystyle\left( x , -x-a_0 , x+a_0-a_1, \ldots\ldots, (-1)^{n-1}(x+a_0-a_1+a_2+\cdots+(-1)^{n-2}a_{n-2}) \right), \\
\end{array}
$$
with $x\in\Zn{m}$. It follows that, for all $x\in\Zn{m}$, the total number of $x$ in $\setT{m}{n}$ is the sum of $m$ times the total number of $x$ in $\setT{m}{n-1}$ and the total number of $x$ in the set of $n$-tuples of $\Zn{m}$. This leads to
$$
\sum_{\nabla\in\setT{m}{n}}\mf{\nabla}(x) = m\hspace{-15pt}\sum_{\nabla\in\setT{m}{n-1}}\mf{\nabla}(x) + m^{n-1}n = m^{n-1}\binom{n}{2} + m^{n-1}n = m^{n-1}\binom{n+1}{2}
$$
and thus
$$
\frac{1}{m^n}\sum_{\nabla\in\setT{m}{n}}\mf{\nabla}(x) = \frac{1}{m}\binom{n+1}{2},
$$
for all $x\in\Zn{m}$. This completes the proof.
\end{proof}

The Molluzzo Problem and the similar problem for the local rule \eqref{eqLR} then correspond to determination of the existence of triangles modulo $m$ with an average number of each element of $\Zn{m}$, that are balanced triangles modulo $m$.

For any non-negative integer $n$, the set $\setBT{m}{n}$ of balanced triangles of size $n$ modulo $m$ is preserved under the action of the automorphism group of $\Zn{m}$.

\begin{prop}
For any triangle $\nabla\in\setT{m}{n}$, we have
$$
\nabla\in\setBT{m}{n}\ \Longleftrightarrow\ \lambda\nabla\in\setBT{m}{n},\ \forall \lambda\in\left(\Zn{m}\right)^*,
$$
where $\left(\Zn{m}\right)^*$ is the multiplicative group of invertible elements of $\Zn{m}$.
\end{prop}

\begin{proof}
Let $\lambda$ be an invertible element of $\Zn{m}$. Since
$$
\mf{\lambda\nabla}(x) = \mf{\nabla}(\lambda^{-1}x),
$$
for all $x\in\Zn{m}$, the result follows.
\end{proof}

\begin{rem}
Since the triangles $\lambda\nabla$ are pairwise distinct whenever $n$ is positive, it is thus clear that the number $|\setBT{m}{n}|$ of balanced triangles modulo $m$ of size $n$ is always divisible by $\varphi(m)$, for all positive integers $n$ and where $\varphi(m)$ is the Euler totient's function of $m$. Therefore, if $\setBT{m}{n}\neq\emptyset$, then $|\setBT{m}{n}|\ge\varphi(m)$, for all positive integers $n$.
\end{rem}

The study of the values taken by the multiplicity function of such a triangle is generally complex. A few results exist in the literature, mainly concerning the binary case. The possible number of ones in binary (Steinhaus) triangles was explored in \cite{Harborth:1972aa,Chang:1983aa,Harborth:2005aa}. The minimum number of ones is obviously $0$ since the triangle of zeroes of size $n$ is always a binary triangle. The maximum number of ones in a binary (Steinhaus) triangle of size $n$ is $\left\lceil\frac{2}{3}\binom{n+1}{2}\right\rceil$. As shown in \cite{Harborth:1972aa,Chang:1983aa}, this maximum number of ones is obtained, for instance, for the (Steinhaus) triangle associated with the initial segment of length $n$ of the $3$-periodic sequence $(110)^k$, where $3k\ge n$, for all positive integers $n$. Determining  all possible values of the multiplicity function of a binary triangle is a very difficult problem. However, Chang managed to obtain in \cite{Chang:1983aa} the four smallest and three largest possible numbers of ones in binary (Steinhaus) triangles.

A binary Steinhaus triangle can also be considered as the upper triangular part of the adjacency matrix of a finite graph. These undirected graphs are called Steinhaus graphs in \cite{Molluzzo:1978aa}. A classical problem on Steinhaus graphs is to study those having certain graphical properties such as bipartition \cite{Dymacek:1986aa,Dymacek:1995aa,Chang:1999aa}, planarity \cite{Dymacek:2000aa} or regularity \cite{Dymacek:1979aa,Bailey:1988aa,Augier:2008aa,Chappelon2008a,Chappelon2009}. A survey on Steinhaus graphs can be found in \cite{Dymacek:1999aa}.

We end this section with the notion of projective maps and the Projection Theorem that will be useful for inductive proofs in the sequel.

\begin{defn}[Projective maps]
For any divisor $d$ of $m$, let $\pi_d$ denote the canonical projective map
$$
\pi_d : \Zn{m} \longrightarrow \Zn{d}.
$$
For any finite multiset $M$ of elements in $\Zn{m}$, its projection $\pi_d(M)$ into $\Zn{d}$ is the multiset of $\Zn{d}$ defined by
$$
\mf{\pi_d(M)}(y) = \sum_{x\in\pi_d^{-1}(\{y\})}\mf{M}(x),
$$
for all $y\in\Zn{d}$. For any finite or infinite sequence $S=\left(u_j\right)$ of elements in $\Zn{m}$, its projection $\pi_d(S)$ is the sequence
$$
\pi_d(S) = \left(\pi_d(u_j)\right)
$$
of $\Zn{d}$.
\end{defn}

\begin{thm}[Projection Theorem]\label{thm6}
Let $d$ be a divisor of $m$ and let $M$ be a finite multiset of $\Zn{m}$. Then, the multiset $M$ is balanced in $\Zn{m}$ if and only if its projection $\pi_d(M)$ is balanced in $\Zn{d}$ and $\mf{M}(x+d)=\mf{M}(x)$, for all $x\in\Zn{m}$.
\end{thm}

\begin{proof}
Since
$$
\pi_d^{-1}\left(\left\{\pi_d(x)\right\}\right) = \left\{ x+\lambda d\ \middle|\ \lambda\in\{0,1,\ldots,m/d-1\} \right\},
$$
we obtain that
$$
\mf{\pi_d(M)}\left(\pi_d(x)\right) = \sum_{\lambda=0}^{\frac{m}{d}-1}\mf{M}\left(x+\lambda d\right),
$$
for all $x\in\Zn{m}$, and the result follows.
\end{proof}

\section{Interlaced arithmetic progressions}

In this section, interlaced arithmetic progressions (IAP for short) are introduced and the orbit of such sequences is studied in details.

\begin{defn}[IAP]
Let $k$ be a positive integer and let $A=(a_0,\ldots,a_{k-1})$ and $D=(d_0,\ldots,d_{k-1})$ be two $k$-tuples of elements in $\Zn{m}$. The {\em interlaced arithmetic progression} $\IAP{A}{D}$ is the sequence $\left(u_j\right)_{j\in\Z}$ defined by
$$
u_{qk+r} = a_r+qd_r,
$$
for all integers $q$ and all $r\in\{0,1,\ldots,k-1\}$. When $k=1$, the interlaced arithmetic progression $\IAP{(a_0)}{(d_0)}$ is a simple arithmetic progression denoted by $\AP{a_0}{d_0}$.
\end{defn}

For instance, the $3$-interlaced arithmetic progression $\IAP{061}{151}$ in $\Zn{7}$ is the sequence whose first elements are
\begin{center}
\includegraphics{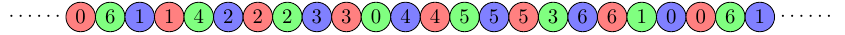}
\end{center}

\begin{nota}[Subsequences]
Let $S=\left(u_j\right)_{j\in\Z}$ be a sequence of elements in $\Zn{m}$. For any integers $i_1$ and $i_2$ such that $i_1\leq i_2$, let $S[i_1,i_2]$ denote the subsequence of $S$ for indices between $i_1$ and $i_2$, i.e., the $(i_2-i_1+1)$-tuple
$$
S[i_1,i_2] = \left(u_j\right)_{j=i_1}^{i_2} = \left( u_{i_1},u_{i_1+1},\ldots,u_{i_2}\right).
$$
For any non-negative integer $n$, let $S[n]$ denote the subsequence of the first $n$ elements of $S$, i.e., the $n$-tuple
$$
S[n] = S[0,n-1] = \left(u_j\right)_{j=0}^{n-1} = \left(u_0,u_1,\ldots,u_{n-1}\right).
$$
\end{nota}

It is known that arithmetic progressions and interlaced arithmetic progressions generate balanced Steinhaus triangles modulo $m$ odd. The Steinhaus triangles associated to arithmetic progressions are studied in \cite{Chappelon2008}, where the following result is obtained.

\begin{thm}[Chappelon \cite{Chappelon2008}]
Let $m$ be an odd number and let $a,d\in\Zn{m}$, with $d$ invertible. Then, the Steinhaus triangle $\SteinT{\AP{a}{d}[n]}$ is balanced, for all positive integers $n\equiv 0$ or $-1\bmod{\ord{2^m}{m}m}$, where $\ord{2^m}{m}$ is the multiplicative order of $2^m$ modulo $m$.
\end{thm}

By refining this result for antisymmetric sequences, a positive answer to Problem~\ref{prob1} (Molluzzo Problem) is obtained when $m=3^k$, for all positive integers $k$. This also positively answers Problem~\ref{prob2} (Weak Molluzzo Problem) for all odd numbers $m$. The Steinhaus triangles associated to interlaced arithmetic progressions are studied in \cite{Chappelon2011}, where the following result is proved.

\begin{thm}[Chappelon \cite{Chappelon2011}]
Let $U$ be the integer sequence defined by
$$
U=\IAP{(0,-1,1)}{(1,-2,1)}.
$$
For any odd number $m$, we consider its projection $S=\pi_m(U)$ in $\Zn{m}$. Then, the Steinhaus triangles
\begin{itemize}
\item
$\SteinT{S[\lambda m,2\lambda m-1]}$ of size $\lambda m$,
\item
$\SteinT{\partial S[3\lambda m-1]}$ of size $3\lambda m-1$,
\end{itemize}
are balanced modulo $m$, for all positive integers $\lambda$.
\end{thm}

The construction presented here is also based on this kind of sequences. For the convenience of the reader, we recall some basic results that could be retrieved in \cite{Chappelon2011}. First, it is straightforward to obtain that the interlaced arithmetic structure is preserved under the derivation process.

\begin{prop}\label{prop1}
Let $A=(a_0,\ldots,a_{k-1})$ and $D=(d_0,\ldots,d_{k-1})$ be two $k$-tuples of elements in $\Zn{m}$. Then,
\begin{equation*}
\resizebox{\textwidth}{!}{$
\der{\,\IAP{A}{D}} = \IAP{-(a_0+a_1,\ldots,a_{k-2}+a_{k-1},a_{k-1}+a_0+d_0)}{-(d_0+d_1,\ldots,d_{k-2}+d_{k-1},d_{k-1}+d_0)}
$}
\end{equation*}
\end{prop}

\begin{proof}
Let $S=\IAP{A}{D}=\left(u_j\right)_{j\in\Z}$ and $\der{S}=\left(v_j\right)_{j\in\Z}$. By definition of $S$, we know that
$$
u_{qk+r} = a_r+qd_r
$$
for all integers $q$ and all $r\in\{0,1,\ldots,k-1\}$. For any $r\in\{0,1,\ldots,k-2\}$, since $1\le r+1\le k-1$, we obtain that
$$
v_{qk+r} = -u_{qk+r}-u_{qk+r+1} = -(a_r+qd_r)-(a_{r+1}+qd_{r+1}) = -(a_r+a_{r+1})-q(d_r+d_{r+1}),
$$
for all integers $q$. Moreover, when $r=k-1$, we obtain that
\begin{equation*}
\resizebox{\textwidth}{!}{$
v_{qk+k-1} = -u_{qk+k-1}-u_{(q+1)k} = -(a_{k-1}+qd_{k-1})-(a_0+(q+1)d_0) = -(a_{k-1}+a_0+d_0)-q(d_{k-1}+d_0),
$}
\end{equation*}
for all integers $q$. This completes the proof.
\end{proof}

It follows that the orbit of an interlaced arithmetic progression only contains interlaced arithmetic progressions.

\begin{nota}
When $X$ is a $k$-tuple of elements of $\Zn{m}$ and $\matM$ is a square matrix of integers of size $k$, the product $X\pi_m\!\left(\matM\right)$ is simply denoted by $X \matM$.
\end{nota}

\begin{nota}[Entries of a matrix]
Let $\matM$ be a matrix of size $n_1\times n_2$ over $\Zn{m}$. The entry in the $i$th row and the $j$th column of $\matM$ is denoted by $(\matM)_{i,j}$, for all $i\in\{1,\ldots,n_1\}$ and all $j\in\{1,\ldots,n_2\}$.
\end{nota}

\begin{prop}\label{prop2}
Let $A$ and $D$ be two $k$-tuples of elements in $\Zn{m}$. Then, for every non-negative integer $i$, we have
$$
\ider{i}{\IAP{A}{D}} = {(-1)}^{i}\,\IAP{A\C{k}{i}+D\T{k}{i}}{D\C{k}{i}},
$$
where $\C{k}{i}$ is the circulant matrix of size $k$ defined by
$$
\C{k}{i} = \left(\sum_{\alpha\in\Z}\binom{i}{\alpha k+r-s}\right)_{1\le r,s\le k}
$$
and where $\T{k}{i}$ is the Toeplitz matrix of size $k$ defined by
$$
\T{k}{i} = \left(\sum_{\alpha\in\Z}\alpha\binom{i}{\alpha k+r-s}\right)_{1\le r,s\le k}.
$$
\end{prop}

\begin{lem}\label{lem*3}
Let $k$ and $i$ be non-negative integers. Then,
$$
\sum_{\alpha\in\Z}\binom{i}{\alpha k + (j + k)} = \sum_{\alpha\in\Z}\binom{i}{\alpha k+j},
$$
for all integers $j$.
\end{lem}

\begin{proof}
With the new variable $\beta=\alpha+1$, we obtain
$$
\sum_{\alpha\in\Z}\binom{i}{\alpha k + (j + k)} = \sum_{\alpha\in\Z}\binom{i}{(\alpha+1)k + j} = \sum_{\beta\in\Z}\binom{i}{\beta k + j},
$$
as announced.
\end{proof}

\begin{rem}
Using Lemma~\ref{lem*3}, we obtain that
$$
\left(\sum_{\alpha\in\Z}\binom{i}{\alpha k+r-s}\right)_{1\le r,s\le k} = \left(\sum_{\alpha\in\Z}\binom{i}{\alpha k+(r-s\bmod{k})}\right)_{1\le r,s\le k}
$$
and thus this matrix is circulant.
\end{rem}

\begin{proof}[Proof of Proposition~\ref{prop2}]
By induction on $i$. For $i=0$, the result is clear. For $i=1$, we obtain from Proposition~\ref{prop1} that
$$
\C{k}{1} = 
\begin{pmatrix}
1 & 0 & \cdots & 0 & 1 \\
1 & 1 & & & 0 \\
0 & \ddots & \ddots & & \vdots \\
\vdots & & \ddots & \ddots & 0 \\
0 & \cdots & 0 & 1 & 1 \\
\end{pmatrix}
\quad\text{and}\quad
\T{k}{1} = 
\begin{pmatrix}
0 & \cdots & \cdots & 0 & 1 \\
\vdots & & & & 0 \\
\vdots & & 0 & & \vdots \\
\vdots & & & & \vdots \\
0 & \cdots & \cdots & \cdots & 0 \\
\end{pmatrix},
$$
as announced. Suppose that the result is true for a certain positive integer $i$. Since $\ider{i+1}{}=\der{\ider{i}{}}$, then the $(i+1)$\textsuperscript{th} derived sequence of $S=\IAP{A}{D}$ is equal to
\begin{equation*}
\resizebox{\textwidth}{!}{$
\ider{i+1}{S} \begin{array}[t]{l}
= \displaystyle\der{\ider{i}{S}} = {(-1)}^{i}\,\der{\,\IAP{A\C{k}{i}+D\T{k}{i}}{D\C{k}{i}}} \\
= \displaystyle{(-1)}^{i+1}\IAP{A\C{k}{i}\C{k}{1}+D\left(\T{k}{i}\C{k}{1}+\C{k}{i}\T{k}{1}\right)}{D\C{k}{i}\C{k}{1}}. \\
\end{array}
$}
\end{equation*}
First, for $\C{k}{i+1}=\C{k}{i}\C{k}{1}$, we have
$$
\left(\C{k}{i+1}\right)_{r,s} = \sum_{u=1}^{k}\left(\C{k}{i}\right)_{r,u}\left(\C{k}{1}\right)_{u,s},
$$
for all $r,s\in\{1,\ldots,k\}$. Hence, for $s<k$, we obtain
$$
\left(\C{k}{i+1}\right)_{r,s} \begin{array}[t]{l}
 = \displaystyle\left(\C{k}{i}\right)_{r,s} + \left(\C{k}{i}\right)_{r,s+1}\\ \ \\
 = \displaystyle\sum_{\alpha\in\Z}\binom{i}{\alpha k + r-s} + \sum_{\alpha\in\Z}\binom{i}{\alpha k + r-s-1} = \sum_{\alpha\in\Z}\binom{i+1}{\alpha k + r-s}.
\end{array}
$$
For $s=k$, we obtain
$$
\left(\C{k}{i+1}\right)_{r,k} = \left(\C{k}{i}\right)_{r,1} + \left(\C{k}{i}\right)_{r,k} = \sum_{\alpha\in\Z}\binom{i}{\alpha k + r-1}+\sum_{\alpha\in\Z}\binom{i}{\alpha k + r-k}.
$$
Since
$$
\sum_{\alpha\in\Z}\binom{i}{\alpha k + r-1} = \sum_{\alpha\in\Z}\binom{i}{\alpha k + r-k-1}
$$
by Lemma~\ref{lem*3}, it follows that
$$
\left(\C{k}{i+1}\right)_{r,k} = \sum_{\alpha\in\Z}\binom{i}{\alpha k + r-k-1}+\sum_{\alpha\in\Z}\binom{i}{\alpha k + r-k} = \sum_{\alpha\in\Z}\binom{i+1}{\alpha k +r-k}.
$$
Therefore,
$$
\C{k}{i+1} = \left(\sum_{\alpha\in\Z}\binom{i+1}{\alpha k +r-s}\right)_{1\le r,s\le k},
$$
as announced.
Moreover, for $\T{k}{i+1}=\T{k}{i}\C{k}{1}+\C{k}{i}\T{k}{1}$, we have
$$
\left(\T{k}{i+1}\right)_{r,s} = \sum_{u=1}^{k}\left(\T{k}{i}\right)_{r,u}\left(\C{k}{1}\right)_{u,s} + \left(\C{k}{i}\right)_{r,u}\left(\T{k}{1}\right)_{u,s},
$$
for all $r,s\in\{1,\ldots,k\}$. Hence, for $s<k$, we obtain
$$
\left(\T{k}{i+1}\right)_{r,s} \begin{array}[t]{l}
= \displaystyle\left(\T{k}{i}\right)_{r,s} + \left(\T{k}{i}\right)_{r,s+1} \\ \ \\
= \displaystyle\sum_{\alpha\in\Z}\alpha\binom{i}{\alpha k+r-s} + \sum_{\alpha\in\Z}\alpha\binom{i}{\alpha k+r-s-1} = \sum_{\alpha\in\Z}\alpha\binom{i+1}{\alpha k+r-s}. \\
\end{array}
$$
For $s=k$, we obtain
$$
\left(\T{k}{i+1}\right)_{r,k} \begin{array}[t]{l}
= \left(\T{k}{i}\right)_{r,1} + \left(\T{k}{i}\right)_{r,k} + \left(\C{k}{i}\right)_{r,1} \\ \ \\
= \displaystyle\sum_{\alpha\in\Z}\alpha\binom{i}{\alpha k+r-1} + \sum_{\alpha\in\Z}\alpha\binom{i}{\alpha k+r-k} + \sum_{\alpha\in\Z}\binom{i}{\alpha k+r-1} \\ \ \\
= \displaystyle\sum_{\alpha\in\Z}\left(\alpha+1\right)\binom{i}{\alpha k+r-1} + \sum_{\alpha\in\Z}\alpha\binom{i}{\alpha k+r-k} \\ \ \\
= \displaystyle\sum_{\alpha\in\Z}\alpha\binom{i}{\alpha k+r-k-1} + \sum_{\alpha\in\Z}\alpha\binom{i}{\alpha k+r-k} = \sum_{\alpha\in\Z}\alpha\binom{i+1}{\alpha k+r-k}.
\end{array}
$$
Therefore,
$$
\T{k}{i+1} = \left(\sum_{\alpha\in\Z}\alpha\binom{i+1}{\alpha k+r-s}\right)_{1\le r,s\le k},
$$
as announced. This completes the proof.
\end{proof}

The following result will be useful for manipulating the matrices $\C{k}{i}$ and $\T{k}{i}$ in the sequel.

\begin{prop}\label{prop*2}
For all non-negative integers $i$ and $j$, we have
\begin{enumerate}[i)]
\item
$\C{k}{i+j} = \C{k}{i}\C{k}{j}$,
\item
$\T{k}{i+j} = \T{k}{i}\C{k}{j} + \C{k}{i}\T{k}{j}$.
\end{enumerate}
\end{prop}

\begin{proof}
Let $i$ and $j$ be two non-negative integers. Let $A$ and $D$ be two $k$-tuples of integers. First, from Proposition~\ref{prop2}, we know that
$$
\ider{i+j}{\IAP{A}{D}} = {(-1)}^{i+j}\,\IAP{A\C{k}{i+j}+D\T{k}{i+j}}{D\C{k}{i+j}}.
$$
Moreover, since $\ider{i+j}{}=\ider{j}{\ider{i}{}}$ and from Proposition~\ref{prop2} again, we obtain that
\begin{equation*}
\resizebox{\textwidth}{!}{$
\ider{i+j}{\IAP{A}{D}}
\begin{array}[t]{l}
 = \displaystyle\ider{j}{\ider{i}{\IAP{A}{D}}} = {(-1)}^{i}\ider{j}{\IAP{A\C{k}{i}+D\T{k}{i}}{D\C{k}{i}}} \\
 = \displaystyle{(-1)}^{i+j}\,\IAP{\left(A\C{k}{i}+D\T{k}{i}\right)\C{k}{j}+D\C{k}{i}\T{k}{j}}{D\C{k}{i}\C{k}{j}} \\
 = \displaystyle{(-1)}^{i+j}\,\IAP{A\C{k}{i}\C{k}{j}+D\left(\T{k}{i}\C{k}{j}+\C{k}{i}\T{k}{j}\right)}{D\C{k}{i}\C{k}{j}}. \\
\end{array}
$}
\end{equation*}
Therefore,
$$
A\C{k}{i+j}+D\T{k}{i+j} = A\C{k}{i}\C{k}{j}+D\left(\T{k}{i}\C{k}{j}+\C{k}{i}\T{k}{j}\right)
$$
and
$$
D\C{k}{i+j} = D\C{k}{i}\C{k}{j},
$$
for all $k$-tuples of integers $A$ and $D$. This leads to
$$
\C{k}{i+j} = \C{k}{i}\C{k}{j}\quad\text{and}\quad \T{k}{i+j}=\T{k}{i}\C{k}{j}+\C{k}{i}\T{k}{j}.
$$
This completes the proof.
\end{proof}

\begin{cor}\label{cor*2}
For all non-negative integers $i$ and all positive integers $\lambda$, we have
\begin{enumerate}[i)]
\item
$\C{k}{\lambda i} = \C{k}{i}^\lambda$,
\item
$\T{k}{\lambda i} = \displaystyle\sum_{l=0}^{\lambda-1}\C{k}{i}^l\T{k}{i}\C{k}{i}^{\lambda-1-l}$.
\end{enumerate}
\end{cor}

\begin{proof}
By induction on $\lambda$. The result is clear for $\lambda=0$ and $\lambda=1$. Suppose that the result is true for a certain integer $\lambda\ge1$, i.e.,
$$
\C{k}{\lambda i}=\C{k}{i}^\lambda\quad\text{and}\quad \T{k}{\lambda i} = \sum_{l=0}^{\lambda-1}\C{k}{i}^l\T{k}{i}\C{k}{i}^{\lambda-1-l}.
$$
Then, by Proposition~\ref{prop*2}, we obtain that
$$
\C{k}{(\lambda+1)i} = \C{k}{\lambda i}\C{k}{i} = \C{k}{i}^\lambda\C{k}{i} = \C{k}{i}^{\lambda+1}
$$
and
$$
\T{k}{(\lambda+1)i} 
\begin{array}[t]{l}
= \T{k}{\lambda i}\C{k}{i} + \C{k}{\lambda i}\T{k}{i} \\ \ \\
= \displaystyle\sum_{l=0}^{\lambda-1}\C{k}{i}^l\T{k}{i}\C{k}{i}^{\lambda-l} + \C{k}{i}^{\lambda}\T{k}{i} \\ \ \\
= \displaystyle\sum_{l=0}^{\lambda}\C{k}{i}^l\T{k}{i}\C{k}{i}^{\lambda-l}.
\end{array}
$$
This completes the proof.
\end{proof}

\section{Periodic orbits of IAP}

In this section, interlaced arithmetic progressions whose orbit is periodic are characterized.

\begin{defn}[Periodicity]
Let $p$ and $q$ be two positive integers. A sequence $S=\left(u_j\right)_{j\in\Z}$ of $\Zn{m}$ is said to be {\em periodic} of {\em period} $p$ (or {\em $p$-periodic}) if
$$
u_{j+p}=u_j,
$$
for all integers $j$. The $p$-tuple $P=(u_j)_{j=0}^{p-1}$ is called {\em the first period} of $S$. For any $p$-tuple $P$ of $\Zn{m}$, the $p$-periodic sequence $S$ with first period $P$ is denoted by
$$
S=P^\infty,
$$
accordingly with the notation of concatenation of tuples. Similarly, a doubly indexed sequence $S=\left(u_{i,j}\right)_{(i,j)\in\N\times\Z}$ of $\Zn{m}$ is said to be {\em periodic} of {\em period} $(p,q)$ (or {\em $(p,q)$-periodic}) if every row $\left(u_{i,j}\right)_{j\in\Z}$ is $p$-periodic, for all $i\in\N$, and every column $\left(u_{i,j}\right)_{i\in\N}$ is $q$-periodic, for all $j\in\Z$, i.e., if
$$
u_{i+q,j}=u_{i,j+p}=u_{i,j},
$$
for all $(i,j)\in\N\times\Z$. The $(q\times p)$-matrix
$$
\matr{\mathrm{P}} = \left(u_{i,j}\right)_{\substack{0\le i\le q-1\\ 0\le j\le p-1}}
$$
is called {\em the first period} of $S$. For any $(q\times p)$-matrix $\matr{\mathrm{P}}$ of $\Zn{m}$, the $(p,q)$-periodic sequence $S$ with first period $\matr{\mathrm{P}}$ is denoted by
$$
S = \matr{\mathrm{P}}^\infty.
$$
\end{defn}

\begin{rem}
When $m\ge1$, a $k$-interlaced arithmetic progression is a periodic sequence of period a divisor of $mk$. Indeed, we have
$$
\IAP{A}{D} = \left(A\cdot(A+D)\cdot(A+2D)\cdots\cdots(A+(m-1)D)\right)^\infty,
$$
for any $k$-tuples $A$ and $D$ of $\Zn{m}$. Moreover, a $p$-periodic sequence can also be seen as a $p$-interlaced arithmetic progression with null common differences. Indeed, we have
$$
P^\infty = \IAP{P}{0\cdots0},
$$
for any $p$-tuple $P$ of $\Zn{m}$.
\end{rem}

First, it is easy to see that the periodicity is preserved under the derivation process.

\begin{prop}
Let $S$ be a $p$-periodic sequence of $\Zn{m}$. Then, its derived sequence $\der{S}$ is also $p$-periodic.
\end{prop}

\begin{proof}
Let $S=\left(u_j\right)_{j\in\Z}$ and $\der{S}=\left(v_j\right)_{j\in\Z}$. Since $S$ is $p$-periodic, then we have
$$
u_{j+p}=u_j,
$$
for all $j\in\Z$. It follows that,
$$
v_{j+p} = -u_{j+p}-u_{j+p+1} = -u_j-u_{j+1} = v_j,
$$
for all $j\in\Z$. Therefore $\der{S}$ is $p$-periodic.
\end{proof}

\begin{cor}\label{cor1}
Let $S$ be a $p$-periodic sequence of $\Zn{m}$. Then, the iterated derived sequences $\ider{i}{S}$ are $p$-periodic, for all non-negative integers $i$.
\end{cor}

Therefore, the orbit of a $p$-periodic sequence only contains $p$-periodic sequences.

\begin{prop}
Let $S$ be a $k$-interlaced arithmetic progression. If $S$ is $p$-periodic, then $S$ is a $\gcd(k,p)$-interlaced arithmetic progression.
\end{prop}

\begin{proof}
Let $A=\left(a_0,\ldots,a_{k-1}\right)$ and $D=\left(d_0,\ldots,d_{k-1}\right)$ be two $k$-tuples of elements in $\Zn{m}$ such that $S=\IAP{A}{D}=\left(a_{j}\right)_{j\in\Z}$. From B\'{e}zout's identity, we know that there exist two integers $\alpha$ and $\beta$ such that
$$
\alpha p + \beta k = \gcd(k,p).
$$
First, since $S$ is $p$-periodic and $S=\IAP{A}{D}$, we have
$$
a_j = a_{j+\alpha p} = a_{j+\gcd(k,p)-\beta k} = a_{j+\gcd(k,p)}-\beta d_{(j+\gcd(k,p)\bmod{k})},
$$
for all integers $j$. Moreover,
$$
a_{j+\gcd(k,p)} = a_{j+\alpha p + \beta k} = a_{j+\beta k} = a_j + \beta d_{(j\bmod{k})},
$$
for all integers $j$. Therefore,
$$
\beta d_{(j\bmod{k})} = \beta d_{(j+\gcd(k,p)\bmod{k})},
$$
for all integers $j$. Finally, since
$$
a_{j+\gcd(k,p)} = a_j + \beta d_{(j+\gcd(k,p)\bmod{k})} = a_j + \beta d_{(j\bmod{\gcd(k,p)})},
$$
for all integers $j$, we conclude that
$$
S= \IAP{\left(a_0,\ldots,a_{\gcd(k,p)-1}\right)}{\beta\left(d_0,\ldots,d_{\gcd(k,p)-1}\right)}.
$$
This completes the proof.
\end{proof}

\begin{nota}[Tuple of zeroes]
For any positive integer $n$, the tuple of $n$ zeroes is simply denoted by $0$ when $n$ is implicitly known and there is no ambiguity.
\end{nota}

\begin{prop}\label{prop17}
Let $A$ and $D$ be two $k$-tuples of elements in $\Zn{m}$ and let $\lambda$ be a positive integer. Then, the sequence $\IAP{A}{D}$ is $\lambda k$-periodic if and only if $\pi_{\frac{m}{\gcd(\lambda,m)}}\left(D\right)=0$.
\end{prop}

\begin{proof}
The sequence $\IAP{A}{D}=\left(u_j\right)_{j\in\Z}$ is $\lambda k$-periodic if and only if
$$
u_{j+\lambda k}=u_j,\ \forall j\in\Z\quad\Longleftrightarrow\quad \lambda D = 0\quad\Longleftrightarrow\quad \pi_{\frac{m}{\gcd(\lambda,m)}}\left(D\right)=0.
$$
This completes the proof.
\end{proof}

It is clear that any $k$-interlaced arithmetic progression of $\Zn{m}$ is $mk$-periodic.

\begin{nota}[Zero and identity matrices]
For any positive integers $n_1$ and $n_2$, the {\em zero matrix} of size $n_1\times n_2$ is denoted by $\matr{0_{n_1,n_2}}$ or $\matr{0}$, when the size $n_1\times n_2$ is implicitly known. For any positive integer $n$, the zero square matrix and the identity matrix of order $n$ are denoted by $\matr{0_n}$ and $\matI{n}$, respectively.
\end{nota}

\begin{nota}[Left kernel]
Let $\matM$ be a matrix of size $n_1\times n_2$ over $\Zn{m}$. The {\em left null space}, or {\em left kernel}, of $\matM$ is the $\Zn{m}$--module defined by
$$
\Lker\matM = \left\{ X\in\left(\Zn{m}\right)^{n_1}\ \middle|\ X\matM = \matr{0_{1,n_2}} \right\}.
$$
When $\matM$ is an integer matrix,
$$
\Lker_m\matM = \Lker\pi_m\left(\matM\right),
$$
for all positive integers $m$.
\end{nota}

Now, interlaced arithmetic progressions whose orbit is periodic can be characterized.

\begin{thm}\label{thm*1}
Let $A$ and $D$ be two $k$-tuples of elements in $\Zn{m}$ and let $p_1$ and $p_2$ be two positive integers such that $p_1$ is divisible by $k$. Then, the orbit $\orb{S}$ of the sequence $S=\IAP{A}{D}$ is periodic of period $(p_1,p_2)$ if and only if $\pi_{\frac{m}{\gcd(p_1/k,m)}}(D)=0$ and the block matrix $\left(\begin{array}{c|c} A & D \end{array}\right)$ is in $\Lker_m\Pmat{k}{p_2}$, with
$$
\Pmat{k}{p_2} = \left(\begin{array}{c|c}
\W{k}{p_2} & \matr{0_k} \\
\hline
\T{k}{p_2} &  \W{k}{p_2}
\end{array}\right),
$$
where $\W{k}{p_2}$ is the Wendt matrix $\W{k}{p_2} = \C{k}{p_2} + {(-1)}^{p_2+1}\matI{k}$.
\end{thm}

\begin{proof}
First, we know from Proposition~\ref{prop17} that the sequence $S=\IAP{A}{D}$ is $p_1$-periodic if and only if $\pi_{\frac{m}{\gcd(p_1/k,m)}}(D)=0$. Moreover, from Corollary~\ref{cor1}, we obtain that all the iterated derived sequences $\ider{i}{S}$ are $p_1$-periodic if and only if $\pi_{\frac{m}{\gcd(p_1/k,m)}}(D)=0$. Using Proposition~\ref{prop2}, it follows that the orbit $\orb{S}$ is $(p_1,p_2)$-periodic if and only if $\pi_{\frac{m}{\gcd(p_1/k,m)}}(D)=0$ and
$$
\ider{i+p_2}{S}=\ider{i}{S},\ \forall i\in\N
\begin{array}[t]{l}
\quad\displaystyle\Longleftrightarrow\quad
\ider{p_2}{S}=S \\ \ \\
\quad\displaystyle\Longleftrightarrow\quad
\left\{\begin{array}{l}
{(-1)}^{p_2}\left(A\C{k}{p_2}+D\T{k}{p_2}\right) = A \\
{(-1)}^{p_2}D\C{k}{p_2} = D \\
\end{array}\right. \\ \ \\
\quad\displaystyle\Longleftrightarrow\quad
\left\{\begin{array}{l}
A\left(\C{k}{p_2}+{(-1)}^{p_2+1}\matI{k}\right) + D\T{k}{p_2} = \matr{0} \\
D\left(\C{k}{p_2}+{(-1)}^{p_2+1}\matI{k}\right) = \matr{0} \\
\end{array}\right. \\ \ \\
\quad\displaystyle\Longleftrightarrow\quad
\left\{\begin{array}{l}
A\W{k}{p_2} + D\T{k}{p_2} = \matr{0} \\
D\W{k}{p_2} = \matr{0} \\
\end{array}\right. \\
\end{array}
$$
This completes the proof.
\end{proof}

For example, the orbit of the sequence $\IAP{021}{111}$ of $\Zn{3}$ is periodic of period $(9,9)$ as depicted in Figure~\ref{fig*06}.

\begin{figure}[htbp]
\centering{
\includegraphics{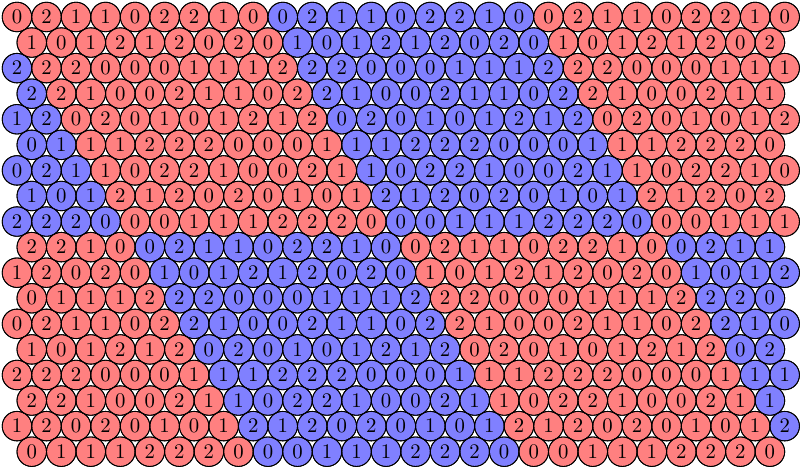}
}
\caption{Periodic orbit of $\IAP{021}{111}$ in $\Zn{3}$}\label{fig*06}
\end{figure}

We end this section by refining this result for antisymmetric sequences.

\begin{nota}[Sum]
For any finite sequence $S=\left(u_j\right)_{j=0}^{n-1}$ of length $n$, the sum of its elements is denoted by $\sigma\left(S\right)=\sum_{j=0}^{n-1}u_j$.
\end{nota}

The antisymmetric structure is preserved under the derivation process.

\begin{prop}\label{prop*1}
The sequence $S$ of length $n$ is antisymmetric if and only if its derived sequence $\der{S}$ is antisymmetric and $2\sigma\left(S\right)=\sigma\left(\der{S}\right)=0$.
\end{prop}

\begin{proof}
The result is clear when $n\le 1$. Suppose that $n\ge2$. Let $S=\left(u_j\right)_{j=0}^{n-1}$ and $\der{S}=\left(v_j\right)_{j=0}^{n-2}=\left(-u_j-u_{j+1}\right)_{j=0}^{n-2}$. First, suppose that $S$ is antisymmetric. Then,
$$
v_j + v_{n-2-j} = -(u_j+u_{j+1})-(u_{n-2-j}+u_{n-1-j}) = -(u_j+u_{n-1-j})-(u_{j+1}+u_{n-1-(j+1)}) = 0,
$$
for all $j\in\{0,\ldots,n-2\}$. Therefore, the sequence $\der{S}$ is antisymmetric. Moreover,
$$
2\sigma\left(S\right) = \sum_{j=0}^{n-1}u_j + \sum_{j=0}^{n-1}u_{n-1-j} = \sum_{j=0}^{n-1}\left(u_j+u_{n-1-j}\right) = 0
$$
and
$$
\sigma\left(\der{S}\right) = \sum_{j=0}^{n-2}v_j = -\sum_{j=0}^{n-2}\left(u_j+u_{j+1}\right) = -\sum_{j=0}^{n-2}u_j-\sum_{j=1}^{n-1}u_j = \left(u_0+u_{n-1}\right)-2\sigma\left(S\right) = 0.
$$
Conversely, suppose that $\der{S}$ is antisymmetric and $2\sigma\left(S\right)=\sigma\left(\der{S}\right)=0$. By induction on $j$. For $j=0$, we already know that
$$
u_0+u_{n-1} = \sigma\left(\der{S}\right) + 2\sigma\left(S\right) = 0.
$$
Suppose that $u_k+u_{n-1-k}=0$ for all $k\in\{0,\ldots,j-1\}$. Since
$$
\sum_{k=j}^{n-2-j}v_k = \sigma\left(\der{S}\right) - \sum_{k=0}^{j-1}\left(v_k+v_{n-2-k}\right) = 0
$$
and
$$
\sum_{k=j}^{n-2-j}v_k \begin{array}[t]{l}
 = -\displaystyle\sum_{k=j}^{n-2-j}\left(u_k+u_{k+1}\right) = -\sum_{k=j}^{n-2-j}u_k-\sum_{k=j+1}^{n-1-j}u_k = (u_j+u_{n-1-j})-2\sum_{k=j}^{n-1-j}u_k \\
 = (u_j+u_{n-1-j})+2\displaystyle\sum_{k=0}^{j-1}\left(u_k+u_{n-1-k}\right) -2\sigma\left(S\right) = u_j+u_{n-1-j},
\end{array}
$$
we obtain that $u_j+u_{n-1-j}=0$. We conclude that the sequence $S$ is antisymmetric.
\end{proof}

The $k$-IAPs with antisymmetric first period are characterized as follows.

\begin{prop}\label{prop3}
Let $A$ and $D$ be two $k$-tuples of elements in $\Zn{m}$, with $m\ge1$, such that the sequence $S=\IAP{A}{D}$ is $dk$-periodic. Its first period $S[dk]$ is antisymmetric if and only if $D=A\matX{k}$, where $\matX{k}$ is the square matrix of size $k$ defined by
$$
\matX{k} = \left(\delta_{r,s}+\delta_{r,k-s+1}\right)_{1\le r,s\le k} = \begin{pmatrix}
1 & & 0 & & & 1 \\
 & \ddots & & & \iddots & \\
 & & \ddots & \iddots & & 0 \\
 0 & & \iddots & \ddots & & \\
 & \iddots & & & \ddots & \\
 1 & & & 0 & & 1 \\ 
\end{pmatrix},
$$
with $\delta_{r,s}$ the Kronecker delta function ($1$ if $r=s$, and $0$ otherwise).
\end{prop}

\begin{proof}
Let $S=\IAP{A}{D}=\left(u_j\right)_{j\in\Z}$. The period $\left(u_j\right)_{j=0}^{dk-1}$ is antisymmetric if and only if
$$
\begin{array}{l}
u_j + u_{dk-1-j} = 0,\ \forall j\in\{0,1,\ldots,dk-1\}, \\
\Longleftrightarrow\quad u_{qk+r} + u_{dk-1-(qk+r)} = 0,\ \forall q\in\{0,1,\ldots,d-1\},\ \forall r\in\{0,1,\ldots,k-1\}, \\
\Longleftrightarrow\quad u_{qk+r} + u_{(d-q-1)k+(k-1-r)} = 0,\ \forall q\in\{0,1,\ldots,d-1\},\ \forall r\in\{0,1,\ldots,k-1\}, \\
\Longleftrightarrow\quad \left(a_r+qd_r\right) + \left(a_{k-1-r}+(d-q-1)d_{k-1-r}\right) = 0,\ \begin{array}[t]{l}\forall q\in\{0,\ldots,d-1\},\\ \forall r\in\{0,\ldots,k-1\},\end{array} \\
\stackrel{dD=0}{\Longleftrightarrow}\quad \left(a_r+a_{k-1-r}-d_{k-1-r}\right) + q\left(d_r-d_{k-1-r}\right) = 0,\ \begin{array}[t]{l}\forall q\in\{0,\ldots,d-1\},\\ \forall r\in\{0,\ldots,k-1\},\end{array} \\
\Longleftrightarrow\quad \left\{\begin{array}{l}
a_r+a_{k-1-r}-d_{k-1-r} = 0,\\
d_r-d_{k-1-r} = 0,\\
\end{array}\right.,\quad \forall r\in\{0,\ldots,k-1\}, \\
\Longleftrightarrow\quad d_{k-1-r} = a_{k-1-r}+a_r,\quad \forall r\in\{0,\ldots,k-1\},\\
\Longleftrightarrow\quad d_{r} = a_{r}+a_{k-1-r},\quad \forall r\in\{0,\ldots,k-1\},\\
\Longleftrightarrow\quad D=A\matX{k}.
\end{array}
$$
This completes the proof.
\end{proof}

Finally, we obtain the following refinement of Theorem~\ref{thm*1} for $k$-IAPs with antisymmetric first period.

\begin{thm}\label{thm1}
Let $A$ be a $k$-tuple of elements in $\Zn{m}$ and let $\lambda$ and $p$ be two positive integers such that $p$ is divisible by $k$. Then, the orbit $\orb{S}$ of the sequence $S=\IAP{A}{A\matX{k}}$ is periodic of period $(p,\lambda p)$ if and only if $\pi_{\frac{m}{\gcd(p/k,m)}}(A\matX{k})=0$ and $A\in\Lker_m\M{k}{\lambda p}$, where
$$
\M{k}{\lambda p} = \W{k}{\lambda p}+\matX{k}\T{k}{\lambda p}
$$
with $\W{k}{\lambda p} = \C{k}{\lambda p} + {(-1)}^{\lambda p+1}\matI{k}$.
\end{thm}

The proof is based on the following

\begin{lem}\label{lem*1}
Let $S$ be a $p$-periodic sequence of $\Zn{m}$. If the first period $S[p]$ is antisymmetric, then $\left(\ider{\lambda p}{S}\right)[\mu p]$ is antisymmetric, for all non-negative integers $\lambda$ and $\mu$.
\end{lem}

\begin{proof}
Suppose that $S[p]$ is antisymmetric. Then, since $S[\mu p]=\left(S[p]\right)^\mu$, we obtain that $S[\mu p]$ is antisymmetric, for all non-negative integers $\mu$. Moreover, since
$$
\left(\ider{\lambda p}{S}\right)[\mu p] = \ider{\lambda p}{\left(S[(\lambda+\mu)p]\right)}
$$
$S[(\lambda+\mu)p]$ is antisymmetric, we deduce from Proposition~\ref{prop*1} that $\left(\ider{\lambda p}{S}\right)[\mu p]$ is antisymmetric, for all non-negative integers $\lambda$ and $\mu$.
\end{proof}

\begin{proof}[Proof of Theorem~\ref{thm1}]
As in the proof of Theorem~\ref{thm*1}, we obtain that the orbit $\orb{S}$ is $(p,\lambda p)$-periodic if and only if $\pi_{\frac{m}{\gcd(p/k,m)}}(A\matX{k})=0$ and $\ider{\lambda p}{S}=S$. Moreover, using Proposition~\ref{prop3} and Lemma~\ref{lem*1}, we know that the first periods $S[p]$ and $\left(\ider{\lambda p}{S}\right)[p]$ are antisymmetric. Therefore, by Proposition~\ref{prop3} again, the common difference of $\ider{\lambda p}{S}$ is $\left(\ider{\lambda p}{S}\right)[k]\matX{k}$. Finally, using Proposition~\ref{prop2}, we obtain that
$$
\ider{\lambda p}{S}=S \begin{array}[t]{l}
\quad\displaystyle\Longleftrightarrow\quad \left(\ider{\lambda p}{S}\right)[k]=S[k]\ \text{and}\ \left(\ider{\lambda p}{S}\right)[k]\matX{k}=S[k]\matX{k} \\ \ \\
\quad\displaystyle\Longleftrightarrow\quad \left(\ider{\lambda p}{S}\right)[k]=S[k] \\ \ \\
\quad\displaystyle\Longleftrightarrow\quad {(-1)}^{\lambda p}\left(A\C{k}{\lambda p}+A\matX{k}\T{k}{\lambda p}\right) = A \\ \ \\
\quad\displaystyle\Longleftrightarrow\quad A\left(\C{k}{\lambda p}+\matX{k}\T{k}{\lambda p}+{(-1)}^{\lambda p+1}\matI{k}\right) = \matr{0}.
\end{array}
$$
This completes the proof.
\end{proof}

\section{Solutions modulo powers of 2}

In this section, we are looking for IAPs of integers with antisymmetric first period such that the projection of their orbit into $\Zn{2^u}$ is periodic and contains infinitely many balanced triangles, for all non-negative integers $u$. We explain how theses integer $k$-IAPs are obtained and justify why $k=24$ is an interesting number of interlaces to consider.

First, we are interested in $p$-periodic sequences $S$ with $(p,p)$-periodic orbit $\orb{S}$ such that the triangles $\ST{S[\lambda p]}$ are balanced, for all non-negative integers $\lambda$. We already know that a necessary condition for having a balanced triangle of $\Zn{m}$ of size $\lambda p$ is that the binomial $\binom{\lambda p+1}{2}$ must be divisible by $m$.

\begin{prop}\label{prop5}
Let $p$ and $m$ be two positive integers. The binomials $\binom{\lambda p + 1}{2}$ are divisible by $m$, for all non-negative integers $\lambda$, if and only if either $p$ is divisible by $m$, when $m$ is odd, or $p$ is divisible by $2m$, when $m$ is even.
\end{prop}

\begin{proof}
First, as depicted in Figure~\ref{fig01} for $\lambda=5$, it is easy to see that
$$
\binom{\lambda p+1}{2} = \lambda\binom{p+1}{2} + \binom{\lambda}{2}p^2,
$$
for all non-negative integers $\lambda$. Moreover, it is clear that $\lambda\binom{p+1}{2} + \binom{\lambda}{2}p^2$ is divisible by $m$, for all non-negative integers $\lambda$, if and only if $\binom{p+1}{2}$ and $p^2$ are divisible by $m$. Therefore, the binomials $\binom{\lambda p+1}{2}$ are divisible by $m$, for all non-negative integers $\lambda$, if and only if $\binom{p+1}{2}$ and $p^2$ are divisible by $m$.
\begin{case}
Suppose that $m$ is odd. Then,
$$
\left\{\begin{array}{l}
\frac{p(p+1)}{2}\equiv0\pmod{m},\\[1.5ex]
p^2\equiv0\pmod{m},
\end{array}\right.
\ \Longleftrightarrow\ 
\left\{\begin{array}{l}
p^2+p\equiv0\pmod{m},\\
p^2\equiv0\pmod{m},
\end{array}\right.
\ \Longleftrightarrow\  p\equiv0\pmod{m}.
$$
\end{case}
\begin{case}
Suppose that $m$ is even. Then,
$$
\left\{\begin{array}{l}
\frac{p(p+1)}{2}\equiv0\pmod{m},\\[1.5ex]
p^2\equiv0\pmod{m},
\end{array}\right.
\ \Longleftrightarrow\ 
\left\{\begin{array}{l}
p^2+p\equiv0\pmod{2m},\\
p^2\equiv0\pmod{m}.
\end{array}\right.
$$
If $p^2+p\equiv0\bmod{2m}$, then $p^2+p\equiv0\bmod{m}$. Moreover, since $p^2\equiv0\bmod{m}$, we obtain that $p\equiv0\bmod{m}$. Since $m$ is even, we know that $m^2$ is divisible by $2m$. It follows that
$$
p\equiv 0\pmod{m}\ \Longrightarrow\ p^2\equiv0\pmod{m^2}\ \Longrightarrow p^2\equiv0\pmod{2m}.
$$
Combining with $p^2+p\equiv0\bmod{2m}$, we conclude that $p\equiv0\bmod{2m}$. Conversely, if $p\equiv0\bmod{2m}$, it is clear that $p^2+p\equiv0\bmod{2m}$ and $p^2\equiv0\bmod{m}$. 
\end{case}
This completes the proof.
\end{proof}

\begin{figure}[htbp]
\centerline{
\includegraphics{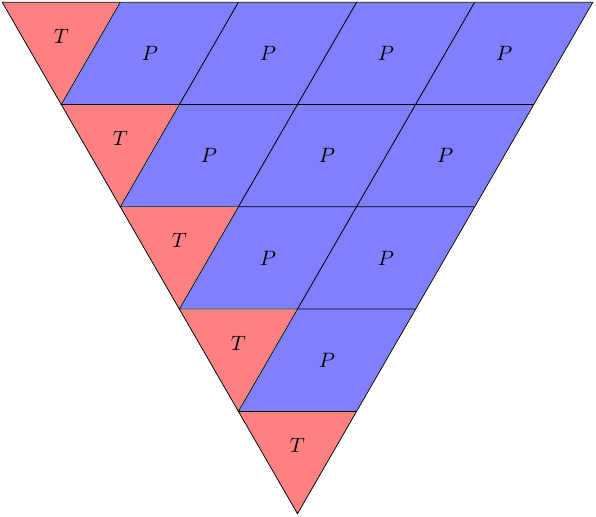}
}
\caption{Structure of periodic triangles $\ST{S[5p]}$}\label{fig01}
\end{figure}

Now, for $p$-periodic sequences $S$ with $(p,p)$-periodic orbit $\orb{S}$, we show that the triangles $\ST{S[\lambda p]}$ are balanced, for all non-negative integers $\lambda$, if and only if it is the case only for two distinct positive values $\lambda_1$ and $\lambda_2$.

\begin{prop}\label{prop4}
Let $S$ be a $p$-periodic sequence of elements in $\Zn{m}$ whose orbit $\orb{S}$ is $(p,p)$-periodic. Moreover, suppose that $m\ge1$ and $p$ is either a multiple of $m$ if $m$ is odd or a multiple of $2m$ if $m$ is even. Then, the triangles $\ST{S[\lambda p]}$ are balanced, for all non-negative integers $\lambda$, if and only if it is the case only for two distinct positive values $\lambda_1$ and $\lambda_2$.
\end{prop}

\begin{proof}
First, as depicted in Figure~\ref{fig01} for $\lambda=5$, by $(p,p)$-periodicity of the orbit $\orb{S}=\left(a_{i,j}\right)_{(i,j)\in\N\times\Z}$, we know that each triangle $\ST{S[\lambda p]}$ can be decomposed into $\lambda$ copies of $T=\ST{S[p]}$ and $\binom{\lambda}{2}$ periods $P=\left\{a_{i,j}\ \middle|\ 0\le i\le p-1 , p-i\le j\le 2p-i-1 \right\}$.

Suppose that there exist two distinct positive integers $\lambda_1$ and $\lambda_2$ such that $\ST{S[\lambda_1p]}$ and $\ST{S[\lambda_2p]}$ are balanced. We prove that $T$ and $P$ are then balanced. Since the multiplicity functions $\mf{\ST{S[\lambda_1p]}}$ and $\mf{\ST{S[\lambda_2p]}}$ are constant, we have
$$
\left\{\begin{array}{l}
\displaystyle\lambda_1 \mf{T} + \binom{\lambda_1}{2} \mf{P} = \frac{1}{m}\binom{\lambda_1p+1}{2}, \\ \ \\
\displaystyle\lambda_2 \mf{T} + \binom{\lambda_2}{2} \mf{P} = \frac{1}{m}\binom{\lambda_2p+1}{2}.
\end{array}\right.
$$
This implies that
\begin{equation*}
\resizebox{\textwidth}{!}{$
\left(\lambda_2\binom{\lambda_1}{2}-\lambda_1\binom{\lambda_2}{2}\right)\mf{P}
\begin{array}[t]{l}
= \displaystyle\frac{1}{m} \left(\lambda_2\binom{\lambda_1p+1}{2}-\lambda_1\binom{\lambda_2p+1}{2}\right) \\ \ \\
= \displaystyle\frac{1}{m} \left(\lambda_2\lambda_1\binom{p+1}{2}+\lambda_2\binom{\lambda_1}{2}p^2-\lambda_1\lambda_2\binom{p+1}{2}-\lambda_1\binom{\lambda_2}{2}p^2\right) \\ \ \\
= \displaystyle\frac{1}{m} \left(\lambda_2\binom{\lambda_1}{2}-\lambda_1\binom{\lambda_2}{2}\right)p^2.
\end{array}
$}
\end{equation*}
Since $\lambda_1\neq\lambda_2$ and are positive, we know that 
$$
\lambda_2\binom{\lambda_1}{2}-\lambda_1\binom{\lambda_2}{2} = \frac{\lambda_1\lambda_2(\lambda_1-\lambda_2)}{2} \neq0
$$
and thus
$$
\mf{P}(x) = \frac{p^2}{m},
$$
for all $x\in\Zn{m}$. Therefore $P$ is balanced. It follows that
$$
\lambda_1\mf{T} = \frac{1}{m}\binom{\lambda_1p+1}{2}-\binom{\lambda_1}{2}\mf{P} = \frac{1}{m}\left(\binom{\lambda_1p+1}{2}-\binom{\lambda_1}{2}p^2\right) = \frac{\lambda_1}{m}\binom{p+1}{2}.
$$
Since $\lambda_1\neq0$, we obtain that
$$
\mf{T}(x) = \frac{1}{m}\binom{p+1}{2},
$$
for all $x\in\Zn{m}$. Therefore $T$ is balanced.

Finally, since $T$ and $P$ are balanced, it is clear that $\ST{S[\lambda p]}$ is balanced, for all non-negative integers $\lambda$. This completes the proof.
\end{proof}

In the sequel, we suppose that $m=2^{u}$, for a certain positive integer $u$, and that $k$ is even (the period $mk$ must be divisible by $2m$, see Proposition~\ref{prop5}). We consider the matrix
$$
\M{k}{\lambda k} = \W{k}{\lambda k}+\matX{k}\T{k}{\lambda k} = \C{k}{\lambda k}-\matI{k}+\matX{k}\T{k}{\lambda k},
$$
for all positive integers $\lambda$. We want to determine the set $\B{k}{2^u}$ of $A\in\left(\Zn{2^{u}}\right)^k$ such that, for the $k$-IAP with antisymmetric first period
$$
S=\IAP{A}{A\matX{k}}
$$
and for all $v\in\{0,1,\ldots,u\}$, the orbit of $\pi_{2^v}(S)$ is $(2^vk,2^vk)$-periodic and the triangles $\ST{\pi_{2^v}(S)[\lambda 2^v k]}$ are balanced, for all non-negative integers $\lambda$. Let $v\in\{0,1,\ldots,u\}$. By Theorem~\ref{thm1}, we know that the orbit of $\pi_{2^v}(S)$ is $(2^vk,2^vk)$-periodic if and only if
$$
\pi_{2^v}(A)\in\Lker_{2^v}\M{k}{2^vk}.
$$
Moreover, by Proposition~\ref{prop4}, the triangles $\ST{\pi_{2^v}(S)[\lambda 2^v k]}$ are balanced, for all non-negative integers $\lambda$, if and only if $\ST{\pi_{2^v}(S)[2^v k]}$ and $\ST{\pi_{2^v}(S)[2^{v+1} k]}$ are balanced. Therefore, the set $\B{k}{2^u}$ is
$$
\B{k}{2^{u}} = \left\{ A\in\left(\Zn{2^{u}}\right)^k\ \middle|\ \begin{array}{l}
\forall v\in\{0,1,\ldots,u\}:\\
\quad\quad\pi_{2^v}(A)\in\Lker_{2^v}\M{k}{2^vk},\\
\quad\quad\ST{\pi_{2^v}(S)[2^v k]}\ \text{and}\ \ST{\pi_{2^v}(S)[2^{v+1}k]}\ \text{balanced},\\
\text{where}\ S=\IAP{A}{A\matX{k}}.
\end{array}
\right\},
$$
for all positive integers $u$.

For any positive integer $u$, it is clear that
$$
\pi_{2^u}\!\left(\B{k}{2^{u+1}}\right)\subset\B{k}{2^u}
$$
and
$$
\B{k}{2^{u+1}} = \left\{ A\in\left(\Zn{2^{u+1}}\right)^k\ \middle|\ \begin{array}{l}
\pi_{2^u}(A)\in\B{k}{2^u},\\
A\in\Lker_{2^{u+1}}\M{k}{2^{u+1}k},\\
\ST{S[2^{u+1} k]}\ \text{and}\ \ST{S[2^{u+2}k]}\ \text{balanced},\\
\text{where}\ S=\IAP{A}{A\matX{k}}.
\end{array}
\right\},
$$
Now, we explain how to determine $\B{k}{2^{u+1}}$ by lifting elements of $\B{k}{2^{u}}$ into $\Zn{2^{u+1}}$. Let $A\in\B{k}{2^{u}}$. First, we know that
$$
A\in\Lker_{2^u}\M{k}{2^{u}k}
$$
and the orbit of $S=\IAP{A}{A\matX{k}}$ is $(2^{u}k,2^{u}k)$-periodic. It is clear that $\orb{S}$ is also $(2^{u+1}k,2^{u+1}k)$-periodic and thus
$$
A\in\Lker_{2^{u}}\M{k}{2^{u+1}k},
$$
by Theorem~\ref{thm1}. Let $A'\in\left(\Zn{2^{u+1}}\right)^k$ such that
$$
\pi_{2^u}(A')=A.
$$
Since $A\in\Lker_{2^{u}}\M{k}{2^{u+1}k}$, we obtain that
$$
A'\,\M{k}{2^{u+1}k} = 2^{u}X_2,
$$
where $X_2\in{\{0,1\}}^k$. If there exists $Y_2\in{\{0,1\}}^k$ such that
\begin{equation}\label{eq*1}
Y_2\,\M{k}{2^{u+1}k} \equiv X_2\pmod{2},
\end{equation}
then we have obtained $\tilde{A}=A'-2^{u}Y_2\in\left(\Zn{2^{u+1}}\right)^k$ such that 
$$
\pi_{2^{u}}(\tilde{A})=A\quad \text{and}\quad \tilde{A}\in\Lker_{2^{u+1}}\M{k}{2^{u+1}k}.
$$
If there are two $k$-tuples $\tilde{A_1}$ and $\tilde{A_2}$ of $\Zn{2^{u+1}}$ such that $\pi_{2^{u}}(\tilde{A_1})=\pi_{2^{u}}(\tilde{A_2})=A$  and
$$
\tilde{A_1},\tilde{A_2}\in\Lker_{2^{u+1}}\M{k}{2^{u+1}k},
$$
then
$$
\tilde{A_1} = \tilde{A_2} + 2^{u}Z_2,
$$
where $Z_2\in\Lker_{2}\M{k}{2^{u+1}k}$. We conclude that
$$
\begin{array}{l}
\displaystyle\left\{ \tilde{A}\in\left(\Zn{2^{u+1}}\right)^k\ \middle|\ \pi_{2^{u}}(\tilde{A})=A\ \text{and}\ \tilde{A}\in\Lker_{2^{u+1}}\M{k}{2^{u+1}k}\right\} \\ \ \\
= \left\{\begin{array}{l}
\emptyset \text{, if \eqref{eq*1} has no solution,}\\ \ \\
\left\{A'-2^{u}Y_2\right\} + 2^{u}\Lker_{2}\M{k}{2^{u+1}k}\text{, otherwise.}
\end{array}
\right.
\end{array}
$$
The determination of $\B{k}{2^{u+1}}$ is now clear. First, as explained above, we determine the set 
$$
\mathcal{P} = \left\{ \tilde{A}\in\left(\Zn{2^{u+1}}\right)^k\ \middle|\ \pi_{2^{u}}(\tilde{A})\in\B{k}{2^{u}}\ \text{and}\ \tilde{A}\in\Lker_{2^{u+1}}\M{k}{2^{u+1}k}\right\}.
$$
After that, for each sequence $S=\IAP{\tilde{A}}{\tilde{A}\matX{k}}$ with $\tilde{A}\in\mathcal{P}$, we test if $\ST{S[2^{u+1}k]}$ and $\ST{S[2^{u+2}k]}$ are balanced in $\Zn{2^{u+1}}$. If this is the case, we obtain that $\tilde{A}\in\B{k}{2^{u+1}}$.

\begin{rem}
If $\B{k}{2^{u}}=\emptyset$, for a certain integer $u$, then it is clear that $\B{k}{2^{u+i}}=\emptyset$, for all non-negative integers $i$.
\end{rem}

In the sequel in this section, for the first few even values of $k$, we give the cardinality of $\B{k}{2^{u}}$, for all positive integers $u$, that has been obtained by computer search.

\begin{rem}
For any positive integers $m\ge2$ and $n$, the triangle generated from the $n$-tuple of zeroes is never balanced in $\Zn{m}$ since it is only constituted by zeroes.
\end{rem}

\paragraph{For $\mathbf{k=2}$:}
Since
$$
\pi_2(\M{2}{4}) = \matI{2},
$$
then $\Lker_{2}\M{2}{4}=\left\{ 00 \right\}$ and
$$
\B{2}{2^{u}}=\emptyset,
$$
for all positive integers $u$.

\paragraph{For $\mathbf{k=4}$:}
Since
$$
\pi_2(\M{4}{8}) = \matI{4},
$$
then $\Lker_{2}\M{4}{8}=\left\{ 0000 \right\}$ and
$$
\B{4}{2^{u}}=\emptyset,
$$
for all positive integers $u$.

\paragraph{For $\mathbf{k=6}$:}
Since
$$
\pi_2(\M{6}{12}) = \left(\begin{array}{cccccc}
\un & 0 & 0 & \un & \un & 0 \\
0 & \un & \un & 0 & 0 & \un \\
0 & \un & \un & 0 & 0 & \un \\
\un & 0 & 0 & \un & \un & 0 \\
\un & 0 & 0 & \un & \un & 0 \\
0 & \un & \un & 0 & 0 & \un \\
\end{array}\right),
$$
of rank $2$, then
$$
\Lker_{2}\M{6}{12}=\left\langle 100010 , 010001 , 001001 , 000110 \right\rangle
$$
of dimension $4$. Since $\B{6}{2}=\emptyset$, we have that
$$
\B{6}{2^{u}}=\emptyset,
$$
for all positive integers $u$.

\paragraph{For $\mathbf{k=8}$:}
Since
$$
\pi_2(\M{8}{16}) = \matI{8},
$$
then $\Lker_{2}\M{8}{16}=\left\{ 00000000 \right\}$ and
$$
\B{8}{2^{u}}=\emptyset,
$$
for all positive integers $u$.

\paragraph{For $\mathbf{k=10}$:}
Since
\begin{equation*}
\resizebox{!}{0.1\textheight}{$
\pi_2(\M{10}{20}) = \left(\begin{array}{cccccccccc}
 \un & 0 & 0 & \un & \un & 0 & 0 & 0 & 0 & 0 \\
 0 & \un & \un & 0 & 0 & \un & 0 & 0 & 0 & 0 \\
 0 & \un & \un & 0 & 0 & 0 & \un & 0 & 0 & 0 \\
 \un & 0 & 0 & \un & 0 & 0 & 0 & \un & 0 & 0 \\
 \un & 0 & 0 & 0 & \un & 0 & 0 & 0 & \un & 0 \\
 0 & \un & 0 & 0 & 0 & \un & 0 & 0 & 0 & \un \\
 0 & 0 & \un & 0 & 0 & 0 & \un & 0 & 0 & \un \\
 0 & 0 & 0 & \un & 0 & 0 & 0 & \un & \un & 0 \\
 0 & 0 & 0 & 0 & \un & 0 & 0 & \un & \un & 0 \\
 0 & 0 & 0 & 0 & 0 & \un & \un & 0 & 0 & \un \\
\end{array}\right),
$}
\end{equation*}
of rank $10$, then $\Lker_{2}\M{10}{20}=\left\{ 0000000000 \right\}$ and
$$
\B{10}{2^{u}}=\emptyset,
$$
for all positive integers $u$.

\paragraph{For $\mathbf{k=12}$:}
Since
\begin{equation*}
\resizebox{!}{0.1\textheight}{$
\pi_2(\M{12}{24}) = \left(\begin{array}{cccccccccccc}
 \un & 0 & 0 & 0 & 0 & 0 & 0 & \un & \un & 0 & 0 & 0 \\
 0 & \un & 0 & 0 & 0 & 0 & \un & 0 & 0 & \un & 0 & 0 \\
 0 & 0 & \un & 0 & 0 & \un & 0 & 0 & 0 & 0 & \un & 0 \\
 0 & 0 & 0 & \un & \un & 0 & 0 & 0 & 0 & 0 & 0 & \un \\
 0 & 0 & 0 & \un & \un & 0 & 0 & 0 & 0 & 0 & 0 & \un \\
 0 & 0 & \un & 0 & 0 & \un & 0 & 0 & 0 & 0 & \un & 0 \\
 0 & \un & 0 & 0 & 0 & 0 & \un & 0 & 0 & \un & 0 & 0 \\
 \un & 0 & 0 & 0 & 0 & 0 & 0 & \un & \un & 0 & 0 & 0 \\
 \un & 0 & 0 & 0 & 0 & 0 & 0 & \un & \un & 0 & 0 & 0 \\
 0 & \un & 0 & 0 & 0 & 0 & \un & 0 & 0 & \un & 0 & 0 \\
 0 & 0 & \un & 0 & 0 & \un & 0 & 0 & 0 & 0 & \un & 0 \\
 0 & 0 & 0 & \un & \un & 0 & 0 & 0 & 0 & 0 & 0 & \un \\
\end{array}\right),
$}
\end{equation*}
of rank $4$, then
$$
\Lker_{2}\M{12}{24} = \left\langle
A_1 , A_2 , \ldots\ldots , A_8
\right\rangle,
$$
of dimension $8$, where the generators $A_1,A_2,\ldots,A_8$ are given in Table~\ref{tab*1}.
Table~\ref{tab1} gives $|\B{12}{2^u}|$, up to automorphisms, for first few values of $u$.
Since $\B{12}{2^6}=\emptyset$, we obtain that
$$
\B{12}{2^{u}}=\emptyset,
$$
for all positive integers $u\ge6$.

\begin{table}[htbp]
\begin{center}
\begin{tabular}{|c|c|}
\hline
$A_1$ & $100000001000$ \\
\hline
$A_2$ & $010000000100$ \\
\hline
$A_3$ & $001000000010$ \\
\hline
$A_4$ & $000100000001$ \\
\hline
\end{tabular}
\begin{tabular}{|c|c|}
\hline
$A_5$ & $000010000001$ \\
\hline
$A_6$ & $000001000010$ \\
\hline
$A_7$ & $000000100100$ \\
\hline
$A_8$ & $000000011000$ \\
\hline
\end{tabular}
\caption{Generators of $\Lker_{2}\M{12}{24}$}\label{tab*1}
\end{center}
\end{table}

\begin{table}[htbp]
\begin{center}
\begin{tabular}{|c||c|c|c|c|c|c|c|c|c|c|c|c|}
\hline
 $2^u$ & 1 & 2 & 4 & 8 & 16 & 32 & 64 \\
\hline
$|\B{12}{2^u}|$ & 1 & 8 & 86 & 455 & 80 & 2 & 0 \\
\hline
\end{tabular}
\caption{The first few values of $|\B{12}{2^u}|$}\label{tab1}
\end{center}
\end{table}

\paragraph{For $\mathbf{k=14}$:}
Since
\begin{equation*}
\resizebox{!}{0.1\textheight}{$
\pi_2(\M{14}{28}) = \left(\begin{array}{cccccccccccccc}
\un & 0 & 0 & \un & \un & 0 & 0 & \un & \un & 0 & 0 & \un & \un & 0 \\
0 & \un & \un & 0 & 0 & \un & \un & 0 & 0 & \un & \un & 0 & 0 & \un \\
0 & \un & \un & 0 & 0 & \un & \un & 0 & 0 & \un & \un & 0 & 0 & \un \\
\un & 0 & 0 & \un & \un & 0 & 0 & \un & \un & 0 & 0 & \un & \un & 0 \\
\un & 0 & 0 & \un & \un & 0 & 0 & \un & \un & 0 & 0 & \un & \un & 0 \\
0 & \un & \un & 0 & 0 & \un & \un & 0 & 0 & \un & \un & 0 & 0 & \un \\
0 & \un & \un & 0 & 0 & \un & \un & 0 & 0 & \un & \un & 0 & 0 & \un \\
\un & 0 & 0 & \un & \un & 0 & 0 & \un & \un & 0 & 0 & \un & \un & 0 \\
\un & 0 & 0 & \un & \un & 0 & 0 & \un & \un & 0 & 0 & \un & \un & 0 \\
0 & \un & \un & 0 & 0 & \un & \un & 0 & 0 & \un & \un & 0 & 0 & \un \\
0 & \un & \un & 0 & 0 & \un & \un & 0 & 0 & \un & \un & 0 & 0 & \un \\
\un & 0 & 0 & \un & \un & 0 & 0 & \un & \un & 0 & 0 & \un & \un & 0 \\
\un & 0 & 0 & \un & \un & 0 & 0 & \un & \un & 0 & 0 & \un & \un & 0 \\
0 & \un & \un & 0 & 0 & \un & \un & 0 & 0 & \un & \un & 0 & 0 & \un \\
\end{array}\right),
$}
\end{equation*}
of rank $2$, then
$$
\Lker_{2}\M{14}{28} = \left\langle
A_1 , A_2 , \ldots\ldots , A_{12}
\right\rangle,
$$
of dimension $12$, where the tuples $A_1,A_2,\ldots,A_{12}$ are given in Table~\ref{tab*2}. Since $\B{14}{2}=\emptyset$, we obtain that
$$
\B{14}{2^{u}}=\emptyset,
$$
for all positive integers $u$.

\begin{table}[htbp]
\begin{center}
\begin{tabular}{|c|c|}
\hline
$A_1$ & $10000000000010$ \\
\hline
$A_2$ & $01000000000001$ \\
\hline
$A_3$ & $00100000000001$ \\
\hline
$A_4$ & $00010000000010$ \\
\hline
\end{tabular}
\begin{tabular}{|c|c|}
\hline
$A_5$ & $00001000000010$ \\
\hline
$A_6$ & $00000100000001$ \\
\hline
$A_7$ & $00000010000001$ \\
\hline
$A_8$ & $00000001000010$ \\
\hline
\end{tabular}
\begin{tabular}{|c|c|}
\hline
$A_9$ & $00000000100010$ \\
\hline
$A_{10}$ & $00000000010001$ \\
\hline
$A_{11}$ & $00000000001001$ \\
\hline
$A_{12}$ & $00000000000110$ \\
\hline
\end{tabular}
\caption{Generators of $\Lker_{2}\M{14}{28}$}\label{tab*2}
\end{center}
\end{table}

\paragraph{For $\mathbf{k=16}$:}
Since
$$
\pi_2(\M{16}{32}) = \matI{16},
$$
then $\Lker_{2}\M{16}{32} = \left\{ 0000000000000000 \right\}$ and
$$
\B{16}{2^{u}}=\emptyset,
$$
for all positive integers $u$.

\paragraph{For $\mathbf{k=18}$:}
Since
\begin{equation*}
\resizebox{!}{0.15\textheight}{$
\pi_2(\M{18}{36}) = \left(\begin{array}{cccccccccccccccccc}
\un & 0 & 0 & \un & \un & 0 & 0 & 0 & 0 & 0 & 0 & 0 & 0 & 0 & 0 & 0 & 0 & 0 \\
0 & \un & \un & 0 & 0 & \un & 0 & 0 & 0 & 0 & 0 & 0 & 0 & 0 & 0 & 0 & 0 & 0 \\
0 & \un & \un & 0 & 0 & 0 & \un & 0 & 0 & 0 & 0 & 0 & 0 & 0 & 0 & 0 & 0 & 0 \\
\un & 0 & 0 & \un & 0 & 0 & 0 & \un & 0 & 0 & 0 & 0 & 0 & 0 & 0 & 0 & 0 & 0 \\
\un & 0 & 0 & 0 & \un & 0 & 0 & 0 & \un & 0 & 0 & 0 & 0 & 0 & 0 & 0 & 0 & 0 \\
0 & \un & 0 & 0 & 0 & \un & 0 & 0 & 0 & \un & 0 & 0 & 0 & 0 & 0 & 0 & 0 & 0 \\
0 & 0 & \un & 0 & 0 & 0 & \un & 0 & 0 & 0 & \un & 0 & 0 & 0 & 0 & 0 & 0 & 0 \\
0 & 0 & 0 & \un & 0 & 0 & 0 & \un & 0 & 0 & 0 & \un & 0 & 0 & 0 & 0 & 0 & 0 \\
0 & 0 & 0 & 0 & \un & 0 & 0 & 0 & \un & 0 & 0 & 0 & \un & 0 & 0 & 0 & 0 & 0 \\
0 & 0 & 0 & 0 & 0 & \un & 0 & 0 & 0 & \un & 0 & 0 & 0 & \un & 0 & 0 & 0 & 0 \\
0 & 0 & 0 & 0 & 0 & 0 & \un & 0 & 0 & 0 & \un & 0 & 0 & 0 & \un & 0 & 0 & 0 \\
0 & 0 & 0 & 0 & 0 & 0 & 0 & \un & 0 & 0 & 0 & \un & 0 & 0 & 0 & \un & 0 & 0 \\
0 & 0 & 0 & 0 & 0 & 0 & 0 & 0 & \un & 0 & 0 & 0 & \un & 0 & 0 & 0 & \un & 0 \\
0 & 0 & 0 & 0 & 0 & 0 & 0 & 0 & 0 & \un & 0 & 0 & 0 & \un & 0 & 0 & 0 & \un \\
0 & 0 & 0 & 0 & 0 & 0 & 0 & 0 & 0 & 0 & \un & 0 & 0 & 0 & \un & 0 & 0 & \un \\
0 & 0 & 0 & 0 & 0 & 0 & 0 & 0 & 0 & 0 & 0 & \un & 0 & 0 & 0 & \un & \un & 0 \\
0 & 0 & 0 & 0 & 0 & 0 & 0 & 0 & 0 & 0 & 0 & 0 & \un & 0 & 0 & \un & \un & 0 \\
0 & 0 & 0 & 0 & 0 & 0 & 0 & 0 & 0 & 0 & 0 & 0 & 0 & \un & \un & 0 & 0 & \un \\
\end{array}\right),
$}
\end{equation*}
of rank $14$, then
$$
\Lker_{2}\M{18}{36} = \left\langle 
A_1 , A_2 , A_3 , A_4
\right\rangle
$$
of dimension $4$, where the generators $A_1,A_2,A_3,A_4$ are given in Table~\ref{tab*3}. Since $\B{18}{2}=\emptyset$, we obtain that
$$
\B{18}{2^{u}}=\emptyset,
$$
for all positive integers $u$.

\begin{table}[htbp]
\begin{center}
\begin{tabular}{|c|c|}
\hline
$A_1$ & $100010010001100010$ \\
\hline
$A_2$ & $010001100010010001$ \\
\hline
\end{tabular}
\begin{tabular}{|c|c|}
\hline
$A_3$ & $001001100100001001$ \\
\hline
$A_4$ & $000110011000000110$ \\
\hline
\end{tabular}
\caption{Generators of $\Lker_{2}\M{18}{36}$}\label{tab*3}
\end{center}
\end{table}

\paragraph{For $\mathbf{k=20}$:}
Since
\begin{equation*}
\resizebox{!}{0.15\textheight}{$
\pi_2(\M{20}{40}) = \left(\begin{array}{cccccccccccccccccccc}
\un & 0 & 0 & 0 & 0 & 0 & 0 & \un & \un & 0 & 0 & 0 & 0 & 0 & 0 & 0 & 0 & 0 & 0 & 0 \\
0 & \un & 0 & 0 & 0 & 0 & \un & 0 & 0 & \un & 0 & 0 & 0 & 0 & 0 & 0 & 0 & 0 & 0 & 0 \\
0 & 0 & \un & 0 & 0 & \un & 0 & 0 & 0 & 0 & \un & 0 & 0 & 0 & 0 & 0 & 0 & 0 & 0 & 0 \\
0 & 0 & 0 & \un & \un & 0 & 0 & 0 & 0 & 0 & 0 & \un & 0 & 0 & 0 & 0 & 0 & 0 & 0 & 0 \\
0 & 0 & 0 & \un & \un & 0 & 0 & 0 & 0 & 0 & 0 & 0 & \un & 0 & 0 & 0 & 0 & 0 & 0 & 0 \\
0 & 0 & \un & 0 & 0 & \un & 0 & 0 & 0 & 0 & 0 & 0 & 0 & \un & 0 & 0 & 0 & 0 & 0 & 0 \\
0 & \un & 0 & 0 & 0 & 0 & \un & 0 & 0 & 0 & 0 & 0 & 0 & 0 & \un & 0 & 0 & 0 & 0 & 0 \\
\un & 0 & 0 & 0 & 0 & 0 & 0 & \un & 0 & 0 & 0 & 0 & 0 & 0 & 0 & \un & 0 & 0 & 0 & 0 \\
\un & 0 & 0 & 0 & 0 & 0 & 0 & 0 & \un & 0 & 0 & 0 & 0 & 0 & 0 & 0 & \un & 0 & 0 & 0 \\
0 & \un & 0 & 0 & 0 & 0 & 0 & 0 & 0 & \un & 0 & 0 & 0 & 0 & 0 & 0 & 0 & \un & 0 & 0 \\
0 & 0 & \un & 0 & 0 & 0 & 0 & 0 & 0 & 0 & \un & 0 & 0 & 0 & 0 & 0 & 0 & 0 & \un & 0 \\
0 & 0 & 0 & \un & 0 & 0 & 0 & 0 & 0 & 0 & 0 & \un & 0 & 0 & 0 & 0 & 0 & 0 & 0 & \un \\
0 & 0 & 0 & 0 & \un & 0 & 0 & 0 & 0 & 0 & 0 & 0 & \un & 0 & 0 & 0 & 0 & 0 & 0 & \un \\
0 & 0 & 0 & 0 & 0 & \un & 0 & 0 & 0 & 0 & 0 & 0 & 0 & \un & 0 & 0 & 0 & 0 & \un & 0 \\
0 & 0 & 0 & 0 & 0 & 0 & \un & 0 & 0 & 0 & 0 & 0 & 0 & 0 & \un & 0 & 0 & \un & 0 & 0 \\
0 & 0 & 0 & 0 & 0 & 0 & 0 & \un & 0 & 0 & 0 & 0 & 0 & 0 & 0 & \un & \un & 0 & 0 & 0 \\
0 & 0 & 0 & 0 & 0 & 0 & 0 & 0 & \un & 0 & 0 & 0 & 0 & 0 & 0 & \un & \un & 0 & 0 & 0 \\
0 & 0 & 0 & 0 & 0 & 0 & 0 & 0 & 0 & \un & 0 & 0 & 0 & 0 & \un & 0 & 0 & \un & 0 & 0 \\
0 & 0 & 0 & 0 & 0 & 0 & 0 & 0 & 0 & 0 & \un & 0 & 0 & \un & 0 & 0 & 0 & 0 & \un & 0 \\
0 & 0 & 0 & 0 & 0 & 0 & 0 & 0 & 0 & 0 & 0 & \un & \un & 0 & 0 & 0 & 0 & 0 & 0 & \un \\
\end{array}\right),
$}
\end{equation*}
of rank $20$, then $\Lker_{2}\M{20}{40} = \left\{ 00000000000000000000 \right\}$ and
$$
\B{20}{2^{u}}=\emptyset,
$$
for all positive integers $u$.

\paragraph{For $\mathbf{k=22}$:}
Since
\begin{equation*}
\resizebox{!}{0.15\textheight}{$
\pi_2(\M{22}{44}) = \left(\begin{array}{cccccccccccccccccccccc}
\un & 0 & 0 & \un & \un & 0 & 0 & \un & \un & 0 & 0 & \un & \un & 0 & 0 & 0 & 0 & 0 & 0 & 0 & 0 & 0 \\
0 & \un & \un & 0 & 0 & \un & \un & 0 & 0 & \un & \un & 0 & 0 & \un & 0 & 0 & 0 & 0 & 0 & 0 & 0 & 0 \\
0 & \un & \un & 0 & 0 & \un & \un & 0 & 0 & \un & \un & 0 & 0 & 0 & \un & 0 & 0 & 0 & 0 & 0 & 0 & 0 \\
\un & 0 & 0 & \un & \un & 0 & 0 & \un & \un & 0 & 0 & \un & 0 & 0 & 0 & \un & 0 & 0 & 0 & 0 & 0 & 0 \\
\un & 0 & 0 & \un & \un & 0 & 0 & \un & \un & 0 & 0 & 0 & \un & 0 & 0 & 0 & \un & 0 & 0 & 0 & 0 & 0 \\
0 & \un & \un & 0 & 0 & \un & \un & 0 & 0 & \un & 0 & 0 & 0 & \un & 0 & 0 & 0 & \un & 0 & 0 & 0 & 0 \\
0 & \un & \un & 0 & 0 & \un & \un & 0 & 0 & 0 & \un & 0 & 0 & 0 & \un & 0 & 0 & 0 & \un & 0 & 0 & 0 \\
\un & 0 & 0 & \un & \un & 0 & 0 & \un & 0 & 0 & 0 & \un & 0 & 0 & 0 & \un & 0 & 0 & 0 & \un & 0 & 0 \\
\un & 0 & 0 & \un & \un & 0 & 0 & 0 & \un & 0 & 0 & 0 & \un & 0 & 0 & 0 & \un & 0 & 0 & 0 & \un & 0 \\
0 & \un & \un & 0 & 0 & \un & 0 & 0 & 0 & \un & 0 & 0 & 0 & \un & 0 & 0 & 0 & \un & 0 & 0 & 0 & \un \\
0 & \un & \un & 0 & 0 & 0 & \un & 0 & 0 & 0 & \un & 0 & 0 & 0 & \un & 0 & 0 & 0 & \un & 0 & 0 & \un \\
\un & 0 & 0 & \un & 0 & 0 & 0 & \un & 0 & 0 & 0 & \un & 0 & 0 & 0 & \un & 0 & 0 & 0 & \un & \un & 0 \\
\un & 0 & 0 & 0 & \un & 0 & 0 & 0 & \un & 0 & 0 & 0 & \un & 0 & 0 & 0 & \un & 0 & 0 & \un & \un & 0 \\
0 & \un & 0 & 0 & 0 & \un & 0 & 0 & 0 & \un & 0 & 0 & 0 & \un & 0 & 0 & 0 & \un & \un & 0 & 0 & \un \\
0 & 0 & \un & 0 & 0 & 0 & \un & 0 & 0 & 0 & \un & 0 & 0 & 0 & \un & 0 & 0 & \un & \un & 0 & 0 & \un \\
0 & 0 & 0 & \un & 0 & 0 & 0 & \un & 0 & 0 & 0 & \un & 0 & 0 & 0 & \un & \un & 0 & 0 & \un & \un & 0 \\
0 & 0 & 0 & 0 & \un & 0 & 0 & 0 & \un & 0 & 0 & 0 & \un & 0 & 0 & \un & \un & 0 & 0 & \un & \un & 0 \\
0 & 0 & 0 & 0 & 0 & \un & 0 & 0 & 0 & \un & 0 & 0 & 0 & \un & \un & 0 & 0 & \un & \un & 0 & 0 & \un \\
0 & 0 & 0 & 0 & 0 & 0 & \un & 0 & 0 & 0 & \un & 0 & 0 & \un & \un & 0 & 0 & \un & \un & 0 & 0 & \un \\
0 & 0 & 0 & 0 & 0 & 0 & 0 & \un & 0 & 0 & 0 & \un & \un & 0 & 0 & \un & \un & 0 & 0 & \un & \un & 0 \\
0 & 0 & 0 & 0 & 0 & 0 & 0 & 0 & \un & 0 & 0 & \un & \un & 0 & 0 & \un & \un & 0 & 0 & \un & \un & 0 \\
0 & 0 & 0 & 0 & 0 & 0 & 0 & 0 & 0 & \un & \un & 0 & 0 & \un & \un & 0 & 0 & \un & \un & 0 & 0 & \un \\
\end{array}\right),
$}
\end{equation*}
of rank $22$, then $\Lker_{2}\M{22}{44} = \left\{ 0000000000000000000000 \right\}$ and
$$
\B{22}{2^{u}}=\emptyset,
$$
for all positive integers $u$.

\paragraph{For $\mathbf{k=24}$:}
Since
\begin{equation*}
\resizebox{!}{0.15\textheight}{$
\pi_2(\M{24}{48}) = \left(\begin{array}{cccccccccccccccccccccccc}
\un & 0 & 0 & 0 & 0 & 0 & 0 & 0 & 0 & 0 & 0 & 0 & 0 & 0 & 0 & \un & \un & 0 & 0 & 0 & 0 & 0 & 0 & 0 \\
0 & \un & 0 & 0 & 0 & 0 & 0 & 0 & 0 & 0 & 0 & 0 & 0 & 0 & \un & 0 & 0 & \un & 0 & 0 & 0 & 0 & 0 & 0 \\
0 & 0 & \un & 0 & 0 & 0 & 0 & 0 & 0 & 0 & 0 & 0 & 0 & \un & 0 & 0 & 0 & 0 & \un & 0 & 0 & 0 & 0 & 0 \\
0 & 0 & 0 & \un & 0 & 0 & 0 & 0 & 0 & 0 & 0 & 0 & \un & 0 & 0 & 0 & 0 & 0 & 0 & \un & 0 & 0 & 0 & 0 \\
0 & 0 & 0 & 0 & \un & 0 & 0 & 0 & 0 & 0 & 0 & \un & 0 & 0 & 0 & 0 & 0 & 0 & 0 & 0 & \un & 0 & 0 & 0 \\
0 & 0 & 0 & 0 & 0 & \un & 0 & 0 & 0 & 0 & \un & 0 & 0 & 0 & 0 & 0 & 0 & 0 & 0 & 0 & 0 & \un & 0 & 0 \\
0 & 0 & 0 & 0 & 0 & 0 & \un & 0 & 0 & \un & 0 & 0 & 0 & 0 & 0 & 0 & 0 & 0 & 0 & 0 & 0 & 0 & \un & 0 \\
0 & 0 & 0 & 0 & 0 & 0 & 0 & \un & \un & 0 & 0 & 0 & 0 & 0 & 0 & 0 & 0 & 0 & 0 & 0 & 0 & 0 & 0 & \un \\
0 & 0 & 0 & 0 & 0 & 0 & 0 & \un & \un & 0 & 0 & 0 & 0 & 0 & 0 & 0 & 0 & 0 & 0 & 0 & 0 & 0 & 0 & \un \\
0 & 0 & 0 & 0 & 0 & 0 & \un & 0 & 0 & \un & 0 & 0 & 0 & 0 & 0 & 0 & 0 & 0 & 0 & 0 & 0 & 0 & \un & 0 \\
0 & 0 & 0 & 0 & 0 & \un & 0 & 0 & 0 & 0 & \un & 0 & 0 & 0 & 0 & 0 & 0 & 0 & 0 & 0 & 0 & \un & 0 & 0 \\
0 & 0 & 0 & 0 & \un & 0 & 0 & 0 & 0 & 0 & 0 & \un & 0 & 0 & 0 & 0 & 0 & 0 & 0 & 0 & \un & 0 & 0 & 0 \\
0 & 0 & 0 & \un & 0 & 0 & 0 & 0 & 0 & 0 & 0 & 0 & \un & 0 & 0 & 0 & 0 & 0 & 0 & \un & 0 & 0 & 0 & 0 \\
0 & 0 & \un & 0 & 0 & 0 & 0 & 0 & 0 & 0 & 0 & 0 & 0 & \un & 0 & 0 & 0 & 0 & \un & 0 & 0 & 0 & 0 & 0 \\
0 & \un & 0 & 0 & 0 & 0 & 0 & 0 & 0 & 0 & 0 & 0 & 0 & 0 & \un & 0 & 0 & \un & 0 & 0 & 0 & 0 & 0 & 0 \\
\un & 0 & 0 & 0 & 0 & 0 & 0 & 0 & 0 & 0 & 0 & 0 & 0 & 0 & 0 & \un & \un & 0 & 0 & 0 & 0 & 0 & 0 & 0 \\
\un & 0 & 0 & 0 & 0 & 0 & 0 & 0 & 0 & 0 & 0 & 0 & 0 & 0 & 0 & \un & \un & 0 & 0 & 0 & 0 & 0 & 0 & 0 \\
0 & \un & 0 & 0 & 0 & 0 & 0 & 0 & 0 & 0 & 0 & 0 & 0 & 0 & \un & 0 & 0 & \un & 0 & 0 & 0 & 0 & 0 & 0 \\
0 & 0 & \un & 0 & 0 & 0 & 0 & 0 & 0 & 0 & 0 & 0 & 0 & \un & 0 & 0 & 0 & 0 & \un & 0 & 0 & 0 & 0 & 0 \\
0 & 0 & 0 & \un & 0 & 0 & 0 & 0 & 0 & 0 & 0 & 0 & \un & 0 & 0 & 0 & 0 & 0 & 0 & \un & 0 & 0 & 0 & 0 \\
0 & 0 & 0 & 0 & \un & 0 & 0 & 0 & 0 & 0 & 0 & \un & 0 & 0 & 0 & 0 & 0 & 0 & 0 & 0 & \un & 0 & 0 & 0 \\
0 & 0 & 0 & 0 & 0 & \un & 0 & 0 & 0 & 0 & \un & 0 & 0 & 0 & 0 & 0 & 0 & 0 & 0 & 0 & 0 & \un & 0 & 0 \\
0 & 0 & 0 & 0 & 0 & 0 & \un & 0 & 0 & \un & 0 & 0 & 0 & 0 & 0 & 0 & 0 & 0 & 0 & 0 & 0 & 0 & \un & 0 \\
0 & 0 & 0 & 0 & 0 & 0 & 0 & \un & \un & 0 & 0 & 0 & 0 & 0 & 0 & 0 & 0 & 0 & 0 & 0 & 0 & 0 & 0 & \un \\
\end{array}\right),
$}
\end{equation*}
of rank $8$, then
$$
\Lker_{2}\M{24}{48} = \left\langle
A_1 , A_2 , \ldots\ldots , A_{16}
\right\rangle,
$$
of dimension $16$, where the tuples $A_1,A_2,\ldots,A_{16}$ are given in Table~\ref{tab*4}. Table~\ref{tab2} gives $|\B{24}{2^u}|$, up to automorphisms, for the first few values of $u$.

\begin{table}[htbp]
\begin{center}
\begin{tabular}{|c|c|}
\hline
$A_1$ & $100000000000000010000000$ \\
\hline
$A_2$ & $010000000000000001000000$ \\
\hline
$A_3$ & $001000000000000000100000$ \\
\hline
$A_4$ & $000100000000000000010000$ \\
\hline
$A_5$ & $000010000000000000001000$ \\
\hline
$A_6$ & $000001000000000000000100$ \\
\hline
$A_7$ & $000000100000000000000010$ \\
\hline
$A_8$ & $000000010000000000000001$ \\
\hline
\end{tabular}
\begin{tabular}{|c|c|}
\hline
$A_9$ & $000000001000000000000001$ \\
\hline
$A_{10}$ & $000000000100000000000010$ \\
\hline
$A_{11}$ & $000000000010000000000100$ \\
\hline
$A_{12}$ & $000000000001000000001000$ \\
\hline
$A_{13}$ & $000000000000100000010000$ \\
\hline
$A_{14}$ & $000000000000010000100000$ \\
\hline
$A_{15}$ & $000000000000001001000000$ \\
\hline
$A_{16}$ & $000000000000000110000000$ \\
\hline
\end{tabular}
\caption{Generators of $\Lker_{2}\M{24}{48}$}\label{tab*4}
\end{center}
\end{table}

\begin{table}[htbp]
\begin{center}
\renewcommand{\arraystretch}{1.5}
\begin{tabular}{|c||c|c|c|c|c|c|c|c|c|c|c|c|}
\hline
 $2^u$ & 1 & 2 & 4 & 8 & 16 & 32 \\
\hline
$|\B{24}{2^u}|$ & 1 & 658 & 178102 & 14237227 & ? & ? \\
\hline
$|\BT{24}{2^u}|$ & 0 & 0 & 0 & 896 & $\ge 896\times 2^{15}$
 & $\ge 896\times 2^{30}$ \\
\hline
\end{tabular}
\renewcommand{\arraystretch}{1}
\caption{The first few values of $|\B{24}{2^u}|$ and $|\BT{24}{2^u}|$}\label{tab2}
\end{center}
\end{table}

The computation of $\B{24}{16}$ is too long. We consider particular subsets of $\B{24}{2^{u}}$. As seen above, for any $A\in\B{24}{2^u}$, we know that
$$
\left|\left\{\tilde{A}\in\B{24}{2^{u+1}}\ \middle|\ \pi_{2^u}(\tilde{A})=A \right\}\right| \le \left| \Lker_{2}\M{24}{24.2^{u+1}} \right|.
$$
For the first few values of $u$, it is easy to verify that
$$
\pi_2(\M{24}{24.2^{u}}) = \pi_2(\M{24}{48})
$$
and then
$$
\left| \Lker_{2}\M{24}{24.2^{u+1}} \right|
=
\left|\Lker_{2}\M{24}{48}\right| = 2^{16}.
$$
This result is true for all positive integers $u$ and will be proved in Lemma~\ref{thm3}. For every positive integer $u$, we consider the subset
$$
\BT{24}{2^u} =
\left\{ A\in\B{24}{2^u} \ \middle|\ \ 
\left|\left\{\tilde{A}\in\B{24}{2^{u+1}}\ \middle|\ \pi_{2^u}(\tilde{A})=A \right\}\right| = 2^{16}
\right\}.
$$
Table~\ref{tab2} gives $|\BT{24}{2^u}|$, up to automorphisms, for the first few values of $u$. First, we obtain that
$$
|\BT{24}{8}| = 3584\ (896\ \text{up to automorphisms}).
$$
For each of the $3584\times 2^{16}$ tuples $A\in\B{24}{16}$ such that $\pi_8(A)\in\BT{24}{8}$, we obtain that
$$
\left|\left\{\tilde{A}\in\B{24}{32}\ \middle|\ \pi_{16}(\tilde{A})=A \right\}\right| = 2^{16}.
$$
Therefore
$$
|\BT{24}{16}| \ge 3584\times2^{16}\ (896\times2^{15}\ \text{up to automorphisms}).
$$
Similarly, for each of the $3584\times 2^{32}$ tuples $A\in\B{24}{32}$ such that $\pi_8(A)\in\BT{24}{8}$, we obtain that
$$
\left|\left\{\tilde{A}\in\B{24}{64}\ \middle|\ \pi_{32}(\tilde{A})=A \right\}\right| = 2^{16}.
$$
Therefore
$$
|\BT{24}{32}| \ge 3584\times2^{32}\ (896\times2^{30}\ \text{up to automorphisms}).
$$
And so on. By analyzing these $3584$ elements of $\BT{24}{8}$, we obtain the following result, which is a detailed version of Theorem~\ref{thm*2} when $m$ is a power of two.

\begin{thm}\label{mainthm}
Let $\setE$ be the infinite set of $24$-tuples of integers defined by
$$
\setE = \setX+\setM,
$$
where $\setX$ is the set of cardinality $14$
$$
\setX = \left\{\pm X_i\ \middle|\ i\in\{1,\ldots,7\}\right\}
$$
and $\setM$ is the free $\Z$-module of rank $16$
$$
\setM = \left\langle 4Y_1,\ldots,4Y_8,8Z_9,\ldots,8Z_{16} \right\rangle,
$$
for which the $24$-tuples $X_i$, $Y_i$ and $Z_i$ are given in Table~\ref{tab4}. For any $A\in\setE$ and any non-negative integer $u$, the orbit of the sequence
$$
S=\IAP{\pi_{2^u}(A)}{\pi_{2^u}(A)\matX{24}}
$$
is $(12.2^u,12.2^u)$-periodic and the triangles $\ST{S[12.2^u\lambda]}$ are balanced in $\Zn{2^{u}}$, for all non-negative integers $\lambda$.
\end{thm}

\begin{table}[htbp]
\begin{center}
\includegraphics[width=\textwidth]{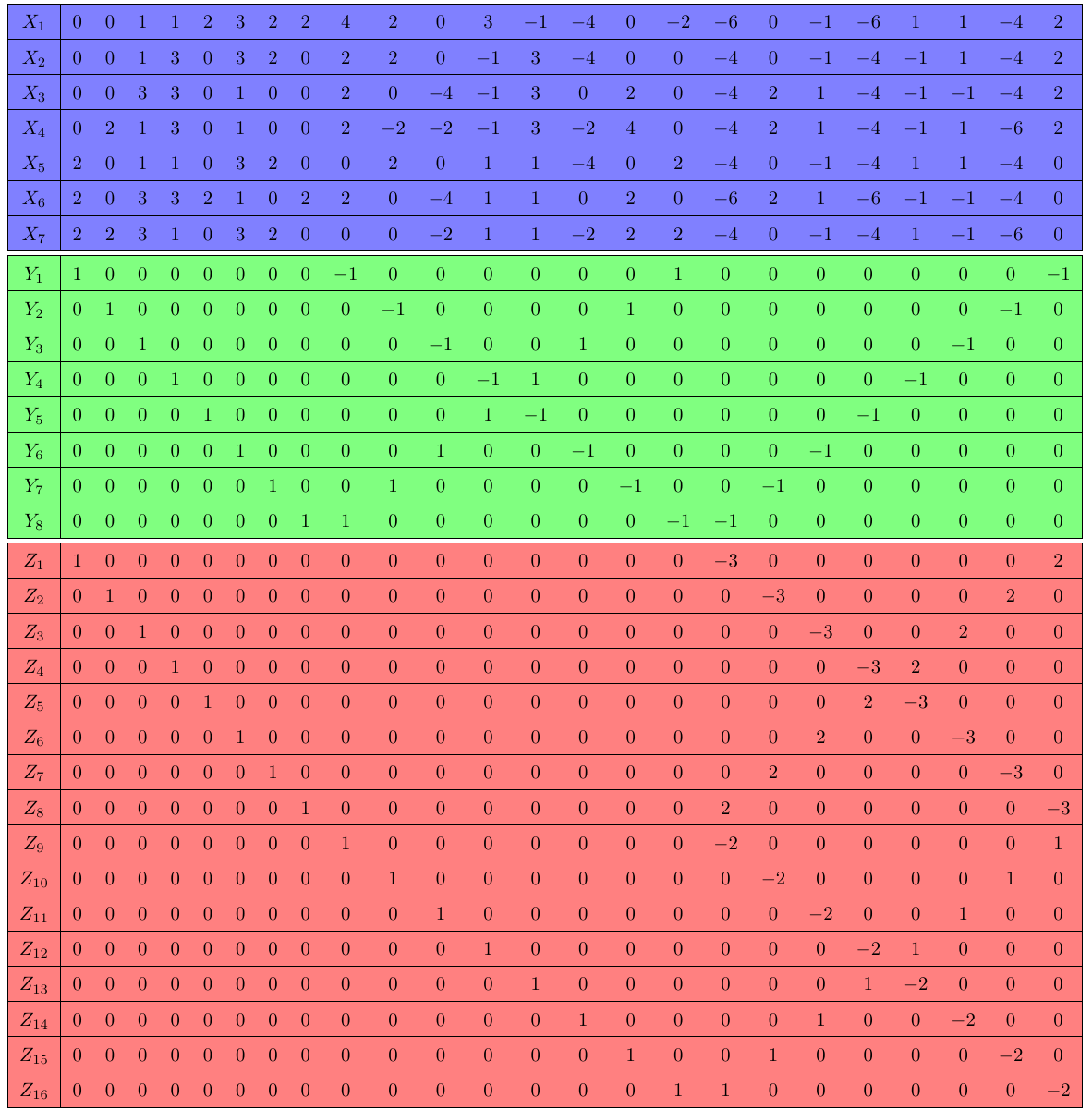}
\caption{The sets $\left\{X_1,\ldots,X_7\right\}$, $\left\{Y_1,\ldots,Y_8\right\}$ and $\left\{Z_1,\ldots,Z_{16}\right\}$}\label{tab4}
\end{center}
\end{table}

The proof of Theorem~\ref{mainthm} is the heart of this paper. We end this section with two results on the set $\setE$ that will be useful in the sequel of this paper. First, it is easy to see that $\setE$ is a subset of the free $\Z$-module $\left\langle Z_1,\ldots,Z_{16}\right\rangle$ of rank $16$.

\begin{prop}\label{prop*3}
For the $24$-tuples of integers $Z_1,\ldots,Z_{16}$ given in Table~\ref{tab4}, we have
$$
\setE\subset\left\langle Z_1,\ldots,Z_{16}\right\rangle.
$$
\end{prop}

\begin{proof}
It is straightforward to verify that
$$
Y_j = Z_j-Z_{8+j}+Z_{17-j},
$$
for all $j\in\{1,\ldots,8\}$, and
$$
\begin{array}{l}
X_1 = Z_3+Z_4+2Z_5+3Z_6+2Z_7+2Z_8+4Z_9+2Z_{10}+3Z_{12}-Z_{13}-4Z_{14}-2Z_{16}, \\
X_2 = Z_3+3Z_4+3Z_6+2Z_7+2Z_9+2Z_{10}-Z_{12}+3Z_{13}-4Z_{14}, \\
X_3 = 3Z_3+3Z_4+Z_6+2Z_9-4Z_{11}-Z_{12}+3Z_{13}+2Z_{15}, \\
X_4 = 2Z_2+Z_3+3Z_4+Z_6+2Z_9-2Z_{10}-2Z_{11}-Z_{12}+3Z_{13}-2Z_{14}+4Z_{15}, \\
X_5 = 2Z_1+Z_3+Z_4+3Z_6+2Z_7+2Z_{10}+Z_{12}+Z_{13}-4Z_{14}+2Z_{16}, \\
X_6 = 2Z_1+3Z_3+3Z_4+2Z_5+Z_6+2Z_8+2Z_9-4Z_{11}+Z_{12}+Z_{13}+2Z_{15}, \\
X_7 = 2Z_1+2Z_2+3Z_3+Z_4+3Z_6+2Z_7-2Z_{11}+Z_{12}+Z_{13}-2Z_{14}+2Z_{15}+2Z_{16}. \\
\end{array}
$$
This completes the proof.
\end{proof}

\begin{prop}\label{prop31}
For all odd numbers $\mu$, we have
$$
\mu\setE\subset\setE.
$$
\end{prop}

\begin{lem}\label{lem*2}
$4\setE\subset\setM$.
\end{lem}

\begin{proof}
From the definition of $X_i$, $Y_i$ and $Z_i$, we have
$$
\begin{array}{l}
X_1 = \left(Y_3+Y_4+2Y_5+3Y_6+2Y_7+2Y_8\right) + 2\left(Z_9-Z_{11}+Z_{12}-Z_{14}+Z_{15}\right),\\
X_2 = \left(Y_3+3Y_4+3Y_6+2Y_7\right) + 2\left(Z_9-Z_{11}+Z_{12}-Z_{14}+Z_{15}\right),\\
X_3 = \left(3Y_3+3Y_4+Y_6\right) + 2\left(Z_9-Z_{11}+Z_{12}-Z_{14}+Z_{15}\right),\\
X_4 = \left(2Y_2+Y_3+3Y_4+Y_6\right) + 2\left(Z_9-Z_{11}+Z_{12}-Z_{14}+Z_{15}\right),\\
X_5 = \left(2Y_1+Y_3+Y_4+3Y_6+2Y_7\right) + 2\left(Z_9-Z_{11}+Z_{12}-Z_{14}+Z_{15}\right),\\
X_6 = \left(2Y_1+3Y_3+3Y_4+2Y_5+Y_6+2Y_8\right) + 2\left(Z_9-Z_{11}+Z_{12}-Z_{14}+Z_{15}\right),\\
X_7 = \left(2Y_1+2Y_2+3Y_3+Y_4+3Y_6+2Y_7\right) + 2\left(Z_9-Z_{11}+Z_{12}-Z_{14}+Z_{15}\right).\\
\end{array}
$$
Therefore,
$$
X_i \in \left\langle Y_1,\ldots,Y_8,2Z_9,\ldots,2Z_{16}\right\rangle
$$
for all $i\in\{1,\ldots,7\}$. It follows that
$$
4\setX\subset\setM.
$$
Moreover, since $4\setM\subset\setM$ by definition of $\setM$, we obtain that
$$
4\setE = 4\setX + 4\setM \subset \setM+\setM = \setM.
$$
This completes the proof.
\end{proof}

\begin{proof}[Proof of Proposition~\ref{prop31}]
Let $A\in\setE$. By definition of $\setE$, it is clear that $-A\in\setE$. Moreover, by Lemma~\ref{lem*2}, we know that $4A\in\setM$. We consider the Euclidean division of the odd number $\mu$ by $4$, that s, $\mu=4q+r$ where $r\in\{1,3\}$. If $r=1$, then
$$
\mu A = (4q+1)A = A + q(4A) \in \setE+\setM = \setE.
$$
Otherwise, if $r=3$, we have that
$$
\mu A = (4q+3)A = -A + (q+1)(4A) \in \setE+\setM = \setE.
$$
This completes the proof.
\end{proof}

The set $\pi_8\left(\setE\right)$ is constituted by $14\times 2^8=3584$ distinct elements. We retrieve the $3584$ elements of $\BT{24}{8}$. More generally, for all integers $u\ge3$, the set $\pi_{2^u}\left(\setE\right)$ is constituted by
$$
14\times\left(2^{u-2}\right)^8\times \left(2^{u-3}\right)^8=7\times2^{16u-39}
$$
distinct elements, i.e., $7\times 2^{15u-38}$ distinct elements, up to automorphisms. Finally, $\left|\pi_4\left(\setE\right)\right|=14$ and $\left|\pi_2\left(\setE\right)\right|=1$ with
$$
\pi_4\left(\setE\right) = \left\{\pm\pi_4(X_i)\ \middle|\ i\in\{1,\ldots,7\}\right\}
$$
and
$$
\pi_2\left(\setE\right) = \left\{ 001101000001100000101100 \right\}.
$$

\section{Arithmetic sums of binomials modulo powers of 2}

In this section, we study the arithmetic sum and the weighted arithmetic sum of binomials.

\begin{nota}[Arithmetic sums]
Let $k$ and $i$ be two non-negative integers. The {\em arithmetic sum} $\SB{0}{k}{i}{j}$ and the {\em weighted arithmetic sum} $\SB{1}{k}{i}{j}$ are the integers defined by
$$
\SB{0}{k}{i}{j} = \sum_{\alpha\in\Z}\binom{i}{\alpha k+j}\quad\text{and}\quad \SB{1}{k}{i}{j} = \sum_{\alpha\in\Z}\alpha\binom{i}{\alpha k+j},
$$
respectively, for all integers $j$.
\end{nota}

With these notations, we consider the circulant matrices
$$
\C{k}{i} = \left(\SB{0}{k}{i}{r-s}\right)_{1\le r,s\le k}
$$
and the Toeplitz matrices
$$
\T{k}{i} = \left(\SB{1}{k}{i}{r-s}\right)_{1\le r,s\le k},
$$
for all non-negative integers $k$ and $i$. In the sequel of this section, the values of the matrices $\C{24}{3.2^u}$ and $\T{24}{3.2^u}$ modulo $2^u$ will be given, for all positive integers $u$. These values are keys in the proof of Theorem~\ref{mainthm}.

\begin{nota}[Circulant matrices]
Let $S=\left(u_j\right)_{j=0}^{k-1}$ be a $k$-tuple of elements in $\Zn{m}$. The {circulant matrix} associated to $S$ is the matrix
$$
\Circ{S} = \left(u_{(s-r\bmod{k})}\right)_{1\le r,s\le k}.
$$
For instance, the circulant matrix $\Circ{012}$ is
$$
\Circ{012} = \begin{pmatrix}
 0 & 1 & 2 \\
 2 & 0 & 1 \\
 1 & 2 & 0 \\
\end{pmatrix}.
$$
\end{nota}

We begin by determining, for all positive integers $u$, the value of the circulant matrix
$$
\C{24}{3.2^u} = \left(\SB{0}{24}{3.2^u}{r-s}\right)_{1\le r,s\le 24}
$$
modulo $2^u$. For $u\le4$, we have
$$
\begin{array}{l}
\C{24}{6} \equiv \Circ{1, 0, 0, 0, 0, 0, 0, 0, 0, 0, 0, 0, 0, 0, 0, 0, 0, 0, 1, 0, 1, 0, 1, 0} \pmod{2},\\[1.5ex]
\C{24}{12} \equiv \Circ{1, 0, 0, 0, 0, 0, 0, 0, 0, 0, 0, 0, 1, 0, 2, 0, 3, 0, 0, 0, 3, 0, 2, 0} \pmod{4},\\[1.5ex]
\C{24}{24} \equiv \Circ{2, 0, 4, 0, 2, 0, 4, 0, 7, 0, 0, 0, 4, 0, 0, 0, 7, 0, 4, 0, 2, 0, 4, 0} \pmod{8},\\[1.5ex]
\C{24}{48} \equiv \Circ{14, 0, 8, 0, 12, 0, 8, 0, 1, 0, 0, 0, 8, 0, 0, 0, 1, 0, 8, 0, 12, 0, 8, 0} \pmod{16}.
\end{array}
$$
For $u\ge4$, we obtain the following

\begin{thm}\label{prop19}
For all positive integers $u\ge 4$, we have
$$
3\C{24}{3.2^u} \equiv \Cd{0} + 2^{u-2}\Cd{1} \pmod{2^u},
$$
where
$$
\Cd{0} = \Circ{2,0,0,0,0,0,0,0,-1,0,0,0,0,0,0,0,-1,0,0,0,0,0,0,0}
$$
and
$$
\Cd{1} = \Circ{2,0,2,0,1,0,2,0,1,0,0,0,2,0,0,0,1,0,2,0,1,0,2,0}.
$$
\end{thm}

\begin{proof}
By induction on $u\ge4$. For $u=4$, it is clear that
$$
3\C{24}{48}
\begin{array}[t]{l}
\displaystyle\equiv \Circ{10, 0, 8, 0, 4, 0, 8, 0, 3, 0, 0, 0, 8, 0, 0, 0, 3, 0, 8, 0, 4, 0, 8, 0} \pmod{16} \\
\displaystyle\equiv \Cd{0} + 4\Cd{1} \pmod{16},
\end{array}
$$
as announced. Suppose now that the result is true for some positive integer $u\ge4$ and prove it for $u+1$. First, we know from Proposition~\ref{prop*2} that
$$
\C{24}{3.2^{u+1}} = {\C{24}{3.2^u}}^2.
$$
Since
$$
3\C{24}{3.2^u} \equiv \Cd{0}+2^{u-2}\Cd{1} \pmod{2^u},
$$
by induction hypothesis, it follows that there exists a circulant matrix $\Cd{2}$ of integers such that
$$
3\C{24}{3.2^u} = \Cd{0}+2^{u-2}\Cd{1}+2^u\Cd{2}.
$$
Moreover, since the product of circulant matrices is commutative, we obtain that
$$
9\C{24}{3.2^{u+1}} 
\begin{array}[t]{l}
= \displaystyle\left(3\C{24}{3.2^u}\right)^2 \\[2ex]
= \displaystyle\left(\Cd{0}+2^{u-2}\Cd{1}+2^u\Cd{2}\right)^2 \\[2ex]
= \displaystyle\left(\Cd{0}+2^{u-2}\Cd{1}\right)^2+2^{u+1}\left(\Cd{0}+2^{u-2}\Cd{1}\right)\Cd{2} + 2^{2u}{\Cd{2}}^2.\\
\end{array}
$$
Since $u\ge4$, it follows that $2u\ge u+1$ and
$$
9\C{24}{3.2^{u+1}} \equiv \left(\Cd{0}+2^{u-2}\Cd{1}\right)^2 \equiv \Cd{0}^2 + 2^{u-1}\Cd{0}\Cd{1} + 2^{2u-4}\Cd{1}^2 \pmod{2^{u+1}}.
$$
By direct computation, we obtain that
$$
\begin{array}{l}
\Cd{0}^2 = \Circ{6,0,0,0,0,0,0,0,-3,0,0,0,0,0,0,0,-3,0,0,0,0,0,0,0} = 3\Cd{0}, \\ \ \\
\Cd{0}\Cd{1}
\begin{array}[t]{l}
= \Circ{2,0,2,0,-1,0,2,0,-1,0,-4,0,2,0,-4,0,-1,0,2,0,-1,0,2,0}\\[1.5ex]
\equiv 3\Cd{1} \pmod{4},
\end{array} \\ \ \\
\Cd{1}^2
\begin{array}[t]{l}
= \Circ{28,0,20,0,22,0,24,0,18,0,20,0,20,0,20,0,18,0,24,0,22,0,20,0} \\[2ex]
 = 2\Cd{3},
\end{array}
\end{array}
$$
where
$$
\Cd{3} = \Circ{14,0,10,0,11,0,12,0,9,0,10,0,10,0,10,0,9,0,12,0,11,0,10,0}.
$$
Therefore,
$$
9\C{24}{3.2^{u+1}} \equiv 3\Cd{0} + 3.2^{u-1}\Cd{1} + 2^{2u-3}\Cd{3} \pmod{2^{u+1}}.
$$
Since $u\ge4$, it follows that $2u-3\ge u+1$ and
$$
9\C{24}{3.2^{u+1}} \equiv 3\Cd{0} + 3.2^{u-1}\Cd{1} \pmod{2^{u+1}}.
$$
Finally, since $\gcd(3,2^{u+1})=1$, we conclude that
$$
3\C{24}{3.2^{u+1}} \equiv \Cd{0} + 2^{u-1}\Cd{1} \pmod{2^{u+1}}.
$$
This completes the proof.
\end{proof}

\begin{nota}[Toeplitz matrices]
Let $S_1=\left(u_j\right)_{j=0}^{k-1}$ and $S_2=\left(u_{-j}\right)_{j=1}^{k-1}$ be a $k$-tuple and a $(k-1)$-tuple of elements in $\Zn{m}$. The Toeplitz matrix associated with $S_1$ and $S_2$ is the matrix
$$
\Toepl{S_1}{S_2} = \left(u_{r-s}\right)_{1\le r,s\le k}.
$$
For instance, the Toeplitz matrix $\Toepl{012}{34}$ is the matrix
$$
\Toepl{012}{34} =
\begin{pmatrix}
0 & 3 & 4 \\
1 & 0 & 3 \\
2 & 1 & 0 \\
\end{pmatrix}.
$$
\end{nota}

We continue by determining, for all positive integers $u$, the value of the Toeplitz matrix
$$
\T{24}{3.2^u} = \left(\SB{1}{24}{3.2^u}{r-s}\right)_{1\le r,s\le 24}
$$
modulo $2^u$. For $u\le 5$, we have
\begin{equation*}
\resizebox{\textwidth}{!}{$
\begin{array}{l}
\T{24}{6} \equiv \Toepl{(0,\ldots,0)}{(0,0,0,0,0,0,0,0,0,0,0,0,0,0,0,0,0,1,0,1,0,1,0)} \pmod{2}, \\[1.5ex]
\T{24}{12} \equiv \Toepl{(0,\ldots,0)}{(0,0,0,0,0,0,0,0,0,0,0,1,0,2,0,3,0,0,0,3,0,2,0)} \pmod{4}, \\[1.5ex]
\T{24}{24} \equiv \Toepl{(1,0,\ldots,0)}{(0,4,0,2,0,4,0,7,0,0,0,4,0,0,0,7,0,4,0,2,0,4,0)} \pmod{8}, \\[1.5ex]
\T{24}{48} \equiv \mathbf{Toepl}\!
\begin{array}[t]{l}
\!\left( (14,0,0,0,8,0,0,0,15,0,8,0,4,0,8,0,2,0,8,0,4,0,8,0), \right.\\
\ \!\left.(0,0,0,0,0,0,0,3,0,8,0,12,0,8,0,0,0,8,0,4,0,8,0) \right)
\end{array} \pmod{16}, \\[4ex]
\T{24}{96} \equiv \mathbf{Toepl}\!
\begin{array}[t]{l}
\!\left( (12,0,0,0,16,0,0,0,5,0,16,0,24,0,16,0,2,0,16,0,24,0,16,0), \right.\\
\ \!\left.(0,0,0,16,0,0,0,15,0,16,0,8,0,16,0,18,0,16,0,8,0,16,0) \right)
\end{array} \pmod{32}.
\end{array}
$}
\end{equation*}
For $u\ge5$, we obtain the following

\begin{thm}\label{prop18}
For all positive integers $u\ge5$, we have
$$
9\T{24}{3.2^u} \equiv \Td{0} + 2^{u-3}\Td{1} \pmod{2^u},
$$
where
$$
\Td{0} = \mathbf{Toepl}\!
\begin{array}[t]{l}
\!\left((0,0,0,0,0,0,0,0,1,0,0,0,0,0,0,0,2,0,0,0,0,0,0,0), \right.\\
\ \!\left.(0,0,0,0,0,0,0,-1,0,0,0,0,0,0,0,-2,0,0,0,0,0,0,0)\right),
\end{array}
$$
and
$$
\Td{1} = \mathbf{Toepl}\!
\begin{array}[t]{l}
\!\left((3, 0, 0, 0, 4, 0, 0, 0, 3, 0, 4, 0, 6, 0, 4, 0, 4, 0, 4, 0, 6, 0, 4, 0), \right.\\
\ \!\left.(0, 0, 0, 4, 0, 0, 0, 2, 0, 4, 0, 2, 0, 4, 0, 1, 0, 4, 0, 2, 0, 4, 0)\right).
\end{array}
$$
\end{thm}

The proof of this result is based on the following lemmas.

\begin{lem}\label{lem*4}
Let $k$ and $i$ be two non-negative integers. Then,
$$
\SB{1}{k}{i}{-j} = \SB{1}{k}{i}{k-j}+\SB{0}{k}{i}{k-j},
$$
for all integers $j$.
\end{lem}

\begin{proof}
By definition, we have
$$
\SB{1}{k}{i}{-j}
\begin{array}[t]{l}
= \displaystyle\sum_{\alpha\in\Z}\alpha\binom{i}{\alpha k-j} = \sum_{\alpha\in\Z}\alpha\binom{i}{(\alpha-1)k+(k-j)} \\ \ \\
= \displaystyle\sum_{\alpha\in\Z}(\alpha+1)\binom{i}{\alpha k+(k-j)} = \SB{1}{k}{i}{k-j} + \SB{0}{k}{i}{k-j}. \\
\end{array}
$$
This completes the proof.
\end{proof}

\begin{nota}[SUT]
Let $\matr{\mathrm{M}}=\left(a_{r,s}\right)_{1\le r,s\le n}$ be a square matrix of order $n$ over $\Zn{m}$. The {\em strictly upper triangular part} of $\matr{\mathrm{M}}$ is the matrix $\SUT{\matr{\mathrm{M}}}=\left(b_{r,s}\right)_{1\le r,s\le n}$ defined by
$$
b_{r,s} = \left\{\begin{array}{ll}
a_{r,s} & \text{if}\ s>r,\\
0 & \text{otherwise},
\end{array}\right.
$$
for all integers $r,s\in\{1,\ldots,n\}$.
\end{nota}

It follows that the Toeplitz matrix $\T{k}{i}$ can be seen as the sum of a circulant matrix and the strictly upper triangular part of $\C{k}{i}$.

\begin{lem}\label{lem*5}
Let $k$ and $i$ be two non-negative integers. Then,
$$
\T{k}{i} = \CT{k}{i} + \SUT{\C{k}{i}},
$$
where
$$
\CT{k}{i}
\begin{array}[t]{l}
 = \left(\SB{1}{k}{i}{(r-s\bmod{k})}\right)_{1\le r,s\le k}\\[2ex]
 = \Circ{\SB{1}{k}{i}{0}, \SB{1}{k}{i}{k-1},\ldots\ldots,\SB{1}{k}{i}{1}}.
\end{array}
$$
\end{lem}

\begin{proof}
From Lemma~\ref{lem*4}, we obtain that
$$
\T{k}{i}
\begin{array}[t]{l}
 = \displaystyle\left(\SB{1}{k}{i}{r-s}\right)_{1\le r,s\le k} \\ \ \\
 = \displaystyle\left(\SB{1}{k}{i}{(r-s\bmod{k})}\right)_{1\le r,s\le k} + \SUT{\left(\SB{0}{k}{i}{r-s}\right)_{1\le r,s\le k}} \\ \ \\
 = \CT{k}{i} + \SUT{\C{k}{i}},
\end{array}
$$
where
$$
\CT{k}{i} = \Circ{\SB{1}{k}{i}{0}, \SB{1}{k}{i}{k-1},\ldots\ldots,\SB{1}{k}{i}{1}}.
$$
This completes the proof.
\end{proof}

\begin{lem}\label{lem*6}
Let $\Md{1}$ and $\Md{2}$ be two circulant matrices of integers such that
$$
\Md{1} \equiv \Md{2} \pmod{2^u},
$$
for a certain positive integer $u$. Then,
$$
\Md{1}\SUT{\Md{1}}+\SUT{\Md{1}}\Md{1} \equiv \Md{2}\SUT{\Md{2}}+\SUT{\Md{2}}\Md{2} \pmod{2^{u+1}}.
$$
\end{lem}

\begin{proof}
Suppose that
$$
\Md{1} = \left(a_{r,s}\right)_{1\le r,s\le k} = \Circ{a_0,a_1,\ldots,a_{k-1}},
$$
with $a_{r,s}=a_{(s-r\bmod{k})}$, for all integers $r,s\in\{1,\ldots,k\}$. Then, by definition, we have
$$
\SUT{\Md{1}} = \left(a_{r,s}'\right)_{1\le r,s\le k},
$$
with
$$
a_{r,s}' =
\left\{\begin{array}{ll}
 a_{r,s} = a_{s-r} & \text{if}\ s>r,\\
 0 & \text{otherwise},
\end{array}\right.
$$
for all integers $r,s\in\{1,\ldots,k\}$. Let
$$
\matr{\mathrm{P}_1} = \Md{1}\SUT{\Md{1}}+\SUT{\Md{1}}\Md{1}
$$
and let $r,s\in\{1,\ldots,k\}$.
Since
$$
\left(\Md{1}\SUT{\Md{1}}\right)_{r,s}
\begin{array}[t]{l}
 = \displaystyle\sum_{i=1}^{k}\left(\Md{1}\right)_{r,i}\left(\SUT{\Md{1}}\right)_{i,s}\\ \ \\
 = \displaystyle\sum_{i=1}^{s-1}a_{(i-r\bmod{k})}a_{s-i} = \sum_{i=1}^{s-1}a_{i}a_{(s-r-i\bmod{k})}
\end{array}
$$
and
$$
\left(\SUT{\Md{1}}\Md{1}\right)_{r,s}
\begin{array}[t]{l}
 = \displaystyle\sum_{i=1}^{k}\left(\SUT{\Md{1}}\right)_{r,i}\left(\Md{1}\right)_{i,s}\\ \ \\
 = \displaystyle\sum_{i=r+1}^{k}a_{i-r}a_{(s-i\bmod{k})} = \sum_{i=1}^{k-r}a_{i}a_{(s-r-i\bmod{k})},
\end{array}
$$
we obtain that
$$
\left(\matr{\mathrm{P}_1}\right)_{r,s} = \sum_{i=1}^{s-1}a_{i}a_{(s-r-i\bmod{k})} + \sum_{i=1}^{k-r}a_{i}a_{(s-r-i\bmod{k})}.
$$
We distinguish the both cases $s-1\le k-r$ and $s-1\ge k-r$.
\setcounter{case}{0}
\begin{case}[$s-1\le k-r$]
Then,
$$
\left(\matr{\mathrm{P}_1}\right)_{r,s} = 2\sum_{i=1}^{s-1}a_{i}a_{(s-r-i\bmod{k})} + \sum_{i=s}^{k-r}a_{i}a_{k+s-r-i}.
$$
If $k+s-r$ is odd, then
$$
\sum_{i=s}^{k-r}a_{i}a_{k+s-r-i} = \sum_{i=s}^{\frac{k+s-r-1}{2}}a_{i}a_{k+s-r-i} + \sum_{i=\frac{k+s-r+1}{2}}^{k-r}a_{i}a_{k+s-r-i} = 2\!\!\!\!\sum_{i=s}^{\frac{k+s-r-1}{2}}a_{i}a_{k+s-r-i}
$$
and
$$
\left(\matr{\mathrm{P}_1}\right)_{r,s} = 2\!\!\!\!\sum_{i=1}^{\frac{k+s-r-1}{2}}a_{i}a_{(s-r-i\bmod{k})}.
$$
Otherwise, if $k+s-r$ is even, then
$$
\sum_{i=s}^{k-r}a_{i}a_{k+s-r-i} \begin{array}[t]{l}
 = \displaystyle\sum_{i=s}^{\frac{k+s-r}{2}-1}a_{i}a_{k+s-r-i} + {a_{\frac{k+s-r}{2}}}^2+\sum_{i=\frac{k+s-r}{2}+1}^{k-r}a_{i}a_{k+s-r-i} \\ \ \\
 = 2\!\!\!\!\displaystyle\sum_{i=s}^{\frac{k+s-r}{2}-1}a_{i}a_{k+s-r-i}+{a_{\frac{k+s-r}{2}}}^2
\end{array}
$$
and
$$
\left(\matr{\mathrm{P}_1}\right)_{r,s} = 2\!\!\!\!\sum_{i=1}^{\frac{k+s-r}{2}-1}a_{i}a_{(s-r-i\bmod{k})}+{a_{\frac{k+s-r}{2}}}^2.
$$
\end{case}
\begin{case}[$s-1\ge k-r$]
Then,
$$
\left(\matr{\mathrm{P}_1}\right)_{r,s} = 2\sum_{i=1}^{k-r}a_{i}a_{(s-r-i\bmod{k})} + \sum_{i=k-r+1}^{s-1}a_{i}a_{k+s-r-i}.
$$
If $k+s-r$ is odd, then
$$
\sum_{i=k-r+1}^{s-1}a_{i}a_{k+s-r-i} = \sum_{i=k-r+1}^{\frac{k+s-r-1}{2}}a_{i}a_{k+s-r-i} + \sum_{i=\frac{k+s-r+1}{2}}^{s-1}a_{i}a_{k+s-r-i} = 2\!\!\!\!\sum_{i=k-r+1}^{\frac{k+s-r-1}{2}}a_{i}a_{k+s-r-i}
$$
and
$$
\left(\matr{\mathrm{P}_1}\right)_{r,s} = 2\!\!\!\!\sum_{i=1}^{\frac{k+s-r-1}{2}}a_{i}a_{(s-r-i\bmod{k})}.
$$
Otherwise, if $k+s-r$ is even, then
$$
\sum_{i=k-r+1}^{s-1}a_{i}a_{k+s-r-i} \begin{array}[t]{l}
 = \displaystyle\sum_{i=k-r+1}^{\frac{k+s-r}{2}-1}a_{i}a_{k+s-r-i} + {a_{\frac{k+s-r}{2}}}^2+\sum_{i=\frac{k+s-r}{2}+1}^{s-1}a_{i}a_{k+s-r-i} \\ \ \\
 = 2\!\!\!\!\displaystyle\sum_{i=k-r+1}^{\frac{k+s-r}{2}-1}a_{i}a_{k+s-r-i}+{a_{\frac{k+s-r}{2}}}^2
\end{array}
$$
and
$$
\left(\matr{\mathrm{P}_1}\right)_{r,s} = 2\!\!\!\!\sum_{i=1}^{\frac{k+s-r}{2}-1}a_{i}a_{(s-r-i\bmod{k})}+{a_{\frac{k+s-r}{2}}}^2.
$$
\end{case}
\noindent In all cases, we obtain that
\begin{equation}\label{eq*2}
\left(\matr{\mathrm{P}_1}\right)_{r,s} = \left\{\begin{array}{ll}
2\!\!\!\!\displaystyle\sum_{i=1}^{\frac{k+s-r-1}{2}}a_{i}a_{(s-r-i\bmod{k})} & \text{if}\ k+s-r\ \text{odd},\\ \ \\
2\!\!\!\!\displaystyle\sum_{i=1}^{\frac{k+s-r}{2}-1}a_{i}a_{(s-r-i\bmod{k})}+{a_{\frac{k+s-r}{2}}}^2 & \text{if}\ k+s-r\ \text{even}.
\end{array}\right.
\end{equation}
Similarly, for
$$
\Md{2} = \left(b_{(s-r\bmod{k})}\right)_{1\le r,s\le k} = \Circ{b_0,b_1,\ldots,b_{k-1}}
$$
and
$$
\matr{\mathrm{P}_2} = \Md{2}\SUT{\Md{2}}+\SUT{\Md{2}}\Md{2},
$$
we obtain that
\begin{equation}\label{eq*3}
\left(\matr{\mathrm{P}_2}\right)_{r,s} = \left\{\begin{array}{ll}
2\!\!\!\!\displaystyle\sum_{i=1}^{\frac{k+s-r-1}{2}}b_{i}b_{(s-r-i\bmod{k})} & \text{if}\ k+s-r\ \text{odd},\\ \ \\
2\!\!\!\!\displaystyle\sum_{i=1}^{\frac{k+s-r}{2}-1}b_{i}b_{(s-r-i\bmod{k})}+{b_{\frac{k+s-r}{2}}}^2 & \text{if}\ k+s-r\ \text{even},
\end{array}\right.
\end{equation}
for all $r,s\in\{1,\ldots,k\}$. Moreover, for any integers $a$ and $b$ such that $a\equiv b\pmod{2^u}$, it is easy to see that
$$
2a\equiv 2b\pmod{2^{u+1}}\quad \text{and}\quad a^2\equiv b^2\pmod{2^{u+1}}
$$
since $u\ge1$. Finally, since $a_i\equiv b_i\pmod{2^u}$ for all $i\in\{1,\ldots,k\}$ by hypothesis, we deduce from \eqref{eq*2} and \eqref{eq*3} that
$$
\left(\matr{\mathrm{P}_1}\right)_{r,s} \equiv \left(\matr{\mathrm{P}_2}\right)_{r,s} \pmod{2^{u+1}},
$$
for all $r,s\in\{1,\ldots,k\}$. This completes the proof.
\end{proof}

\begin{lem}\label{lem*7}
For all positive integers $u\ge5$, we have
$$
9\CT{24}{3.2^u} \equiv \CTd{0} + 2^{u-3}\CTd{1} \pmod{2^u},
$$
where
$$
\CTd{0} = \Circ{0,0,0,0,0,0,0,0,2,0,0,0,0,0,0,0,1,0,0,0,0,0,0,0}
$$
and
$$
\CTd{1} = \Circ{3,0,4,0,6,0,4,0,4,0,4,0,6,0,4,0,3,0,0,0,4,0,0,0}.
$$
\end{lem}

\begin{proof}
By induction on $u\ge 5$. For $u=5$, we know that
\begin{equation*}
\resizebox{\textwidth}{!}{$
\CT{24}{96}
\begin{array}[t]{l}
= \left(\SB{1}{24}{96}{(r-s\bmod{24})}\right)_{1\le r,s\le 24}\\[2ex]
\equiv \Circ{12,0,16,0,24,0,16,0,2,0,16,0,24,0,16,0,5,0,0,0,16,0,0,0} \pmod{32}
\end{array}
$}
\end{equation*}
and thus
$$
9\CT{24}{96}
\begin{array}[t]{l}
\equiv \Circ{12, 0, 16, 0, 24, 0, 16, 0, 18, 0, 16, 0, 24, 0, 16, 0, 13, 0, 0, 0, 16, 0, 0, 0} \\[2ex]
\equiv \CTd{0} + 4\CTd{1} \pmod{32},
\end{array}
$$
as announced. Suppose now that the result is true for some positive integer $u\ge5$ and prove it for $u+1$. First, we know from Lemma~\ref{lem*5} and Proposition~\ref{prop*2} that
\begin{equation*}
\resizebox{\textwidth}{!}{$
\CT{24}{3.2^{u+1}}
\begin{array}[t]{l}
 = \T{24}{3.2^{u+1}} - \SUT{\C{24}{3.2^{u+1}}} \\[2ex]
 = \T{24}{3.2^u}\C{24}{3.2^u} + \C{24}{3.2^u}\T{24}{3.2^u} - \SUT{\C{24}{3.2^{u+1}}} \\[2ex]
 = \begin{array}[t]{l}
 \left(\CT{24}{3.2^u}+\SUT{\C{24}{2^u}}\right)\C{24}{3.2^u} \\[2ex]
 + \C{24}{3.2^u}\left(\CT{24}{3.2^u}+\SUT{\C{24}{2^u}}\right)- \SUT{\C{24}{3.2^{u+1}}}.
\end{array} \\[2ex]
\end{array}
$}
\end{equation*}
Moreover, since the product of circulant matrices is commutative, we obtain that
\begin{equation}\label{eq*4}
\CT{24}{3.2^{u+1}} =
\begin{array}[t]{l}
2\CT{24}{3.2^u}\C{24}{3.2^u} + \SUT{\C{24}{3.2^u}}\C{24}{3.2^u} \\[2ex]
+ \C{24}{3.2^u}\SUT{\C{24}{3.2^u}} - \SUT{\C{24}{3.2^{u+1}}}.
\end{array}
\end{equation}
Since
$$
9\CT{24}{3.2^u} \equiv \CTd{0}+2^{u-3}\CTd{1} \pmod{2^u},
$$
by induction hypothesis, and
$$
3\C{24}{3.2^u} \equiv \Cd{0}+2^{u-2}\Cd{1} \pmod{2^u},
$$
by Theorem~\ref{prop19}, we obtain that
\begin{equation*}
\resizebox{\textwidth}{!}{$
27\CT{24}{3.2^u}\C{24}{3.2^u} \equiv \CTd{0}\Cd{0} + 2^{u-3}\left(2\CTd{0}\Cd{1}+\CTd{1}\Cd{0}\right) + 2^{2u-5}\CTd{1}\Cd{1} \pmod{2^u}.
$}
\end{equation*}
Since $2u-5\ge u$ when $u\ge 5$, we have
$$
27\CT{24}{3.2^u}\C{24}{3.2^u} \equiv \CTd{0}\Cd{0} + 2^{u-3}\left(2\CTd{0}\Cd{1}+\CTd{1}\Cd{0}\right) \pmod{2^u}
$$
and thus
\begin{equation}\label{eq*5}
54\CT{24}{3.2^u}\C{24}{3.2^u} \equiv 2\CTd{0}\Cd{0} + 2^{u-2}\left(2\CTd{0}\Cd{1}+\CTd{1}\Cd{0}\right)\pmod{2^{u+1}}.
\end{equation}
Moreover, since
$$
3\C{24}{3.2^u} \equiv \Cd{0}+2^{u-2}\Cd{1} \pmod{2^u},
$$
by Theorem~\ref{prop19}, we obtain that
$$
3\SUT{\C{24}{3.2^u}} \equiv \SUT{\Cd{0}}+2^{u-2}\SUT{\Cd{1}} \pmod{2^u}.
$$
It follows, from Lemma~\ref{lem*6}, that
$$
\begin{array}{l}
9\left(\SUT{\C{24}{3.2^u}}\C{24}{3.2^u}+\C{24}{3.2^u}\SUT{\C{24}{3.2^u}}\right) \\[2ex]
\equiv \begin{array}[t]{l}
\left(\SUT{\Cd{0}}+2^{u-2}\SUT{\Cd{1}}\right)\left(\Cd{0}+2^{u-2}\Cd{1}\right) \\[2ex]
+ \left(\Cd{0}+2^{u-2}\Cd{1}\right)\left(\SUT{\Cd{0}}+2^{u-2}\SUT{\Cd{1}}\right) \\[2ex] 
\end{array}\\
\equiv \begin{array}[t]{l}
\left(\SUT{\Cd{0}}\Cd{0}+\Cd{0}\SUT{\Cd{0}}\right) \\[2ex]
+ 2^{u-2}\left(\SUT{\Cd{0}}\Cd{1}+\Cd{1}\SUT{\Cd{0}}+\SUT{\Cd{1}}\Cd{0}+\Cd{0}\SUT{\Cd{1}}\right) \\[2ex]
+ 2^{2u-4}\left(\SUT{\Cd{1}}\Cd{1}+\Cd{1}\SUT{\Cd{1}}\right) \pmod{2^{u+1}}.
\end{array}
\end{array}
$$
Since $2u-4\ge u+1$ when $u\ge5$, we obtain that
\begin{equation}\label{eq*6}
\resizebox{\textwidth}{!}{$
\begin{array}{l}
9\left(\SUT{\C{24}{3.2^u}}\C{24}{3.2^u}+\C{24}{3.2^u}\SUT{\C{24}{3.2^u}}\right) \\[2ex]
\equiv \begin{array}[t]{l}
\left(\SUT{\Cd{0}}\Cd{0}+\Cd{0}\SUT{\Cd{0}}\right) \\[2ex]
+ 2^{u-2}\left(\SUT{\Cd{0}}\Cd{1}+\Cd{1}\SUT{\Cd{0}}+\SUT{\Cd{1}}\Cd{0}+\Cd{0}\SUT{\Cd{1}}\right) \pmod{2^{u+1}}.
\end{array}
\end{array}
$}
\end{equation}
Finally, since
$$
3\C{24}{3.2^{u+1}} \equiv \Cd{0}+2^{u-1}\Cd{1} \pmod{2^{u+1}},
$$
by Theorem~\ref{prop19}, we obtain that
\begin{equation}\label{eq*7}
3\SUT{\C{24}{3.2^{u+1}}} \equiv \SUT{\Cd{0}}+2^{u-1}\SUT{\Cd{1}} \pmod{2^{u+1}}.
\end{equation}
Combining \eqref{eq*4}, \eqref{eq*5}, \eqref{eq*6} and \eqref{eq*7}, this leads to
$$
27\CT{24}{3.2^{u+1}} \equiv \matr{\mathrm{X}_0} + 2^{u-2}\matr{\mathrm{X}_1} \pmod{2^{u+1}},
$$
where
$$
\matr{\mathrm{X}_0} = 2\CTd{0}\Cd{0}+3\SUT{\Cd{0}}\Cd{0}+3\Cd{0}\SUT{\Cd{0}}-9\SUT{\Cd{0}}
$$
and
$$
\matr{\mathrm{X}_1} =
\begin{array}[t]{l}
2\CTd{0}\Cd{1}+\CTd{1}\Cd{0}+3\SUT{\Cd{0}}\Cd{1}+3\Cd{1}\SUT{\Cd{0}} \\[1ex] +3\SUT{\Cd{1}}\Cd{0}+3\Cd{0}\SUT{\Cd{1}}-18\SUT{\Cd{1}}.
\end{array}
$$
By direct computation, we obtain that
$$
\matr{\mathrm{X}_0} = \Circ{0, 0, 0, 0, 0, 0, 0, 0, 6, 0, 0, 0, 0, 0, 0, 0, 3, 0, 0, 0, 0, 0, 0, 0} = 3\CTd{0}
$$
and
\begin{equation*}
\resizebox{\textwidth}{!}{$
\matr{\mathrm{X}_1} \begin{array}[t]{l}
= \Circ{-7, 0, -12, 0, -14, 0, -12, 0, -12, 0, -20, 0, -22, 0, -20, 0, -23, 0, -40, 0, -36, 0, -40, 0} \\
\equiv 3\CTd{1} \pmod{8}.
\end{array}
$}
\end{equation*}
Therefore,
$$
27\CT{24}{3.2^{u+1}} \equiv 3\CTd{0} + 3.2^{u-2}\CTd{1} \pmod{2^{u+1}}.
$$
Finally, since $\gcd(3,2^{u+1})=1$, we conclude that
$$
9\CT{24}{3.2^{u+1}} \equiv \CTd{0} + 2^{u-2}\CTd{1} \pmod{2^{u+1}}.
$$
This completes the proof.
\end{proof}

We are now ready to prove Theorem~\ref{prop18}.

\begin{proof}[Proof of Theorem~\ref{prop18}]
First, by Lemma~\ref{lem*5}, we have
$$
\T{24}{3.2^u} = \CT{24}{3.2^u} + \SUT{\C{24}{3.2^u}}.
$$
Moreover, we know from Lemma~\ref{lem*7} that
$$
9\CT{24}{3.2^u} \equiv \CTd{0} + 2^{u-3}\CTd{1} \pmod{2^u}
$$
and from Theorem~\ref{prop19} that
$$
3\C{24}{3.2^u} \equiv \Cd{0} + 2^{u-2}\Cd{1} \pmod{2^u}
$$
and thus
$$
3\SUT{\C{24}{3.2^u}} \equiv \SUT{\Cd{0}} + 2^{u-2}\SUT{\Cd{1}} \pmod{2^u}.
$$
It follows that
$$
9\T{24}{3.2^u} \equiv \left(\CTd{0}+3\SUT{\Cd{0}}\right) + 2^{u-3}\left(\CTd{1}+6\SUT{\Cd{1}}\right) \pmod{2^u}.
$$
By direct computation, we obtain that
$$
\begin{array}{l}
\CTd{0}+3\SUT{\Cd{0}} \\[2ex]
= \begin{array}[t]{l}
\Circ{0,0,0,0,0,0,0,0,2,0,0,0,0,0,0,0,1,0,0,0,0,0,0,0} \\
+ 3\Toepl{(0\ldots,0)}{(0,0,0,0,0,0,0,-1,0,0,0,0,0,0,0,-1,0,0,0,0,0,0,0)} \\[2ex]
\end{array}\\
= \mathbf{Toepl}\!
\begin{array}[t]{l}
\!\left((0,0,0,0,0,0,0,0,1,0,0,0,0,0,0,0,2,0,0,0,0,0,0,0), \right.\\
\ \!\left.(0,0,0,0,0,0,0,-1,0,0,0,0,0,0,0,-2,0,0,0,0,0,0,0)\right)
\end{array} \\[2ex]
= \Td{0}
\end{array}
$$
and
$$
\begin{array}{l}
\CTd{1} + 6\SUT{\Cd{1}} \\[2ex]
 = \begin{array}[t]{l}
\Circ{3,0,4,0,6,0,4,0,4,0,4,0,6,0,4,0,3,0,0,0,4,0,0,0} \\
+ 6\Toepl{(0,\ldots,0)}{(0,2,0,1,0,2,0,1,0,0,0,2,0,0,0,1,0,2,0,1,0,2,0)} \\[2ex]
\end{array} \\
= \mathbf{Toepl}\!
\begin{array}[t]{l}
\!\left((3, 0, 0, 0, 4, 0, 0, 0, 3, 0, 4, 0, 6, 0, 4, 0, 4, 0, 4, 0, 6, 0, 4, 0), \right.\\
\ \!\left.(0, 16, 0, 12, 0, 16, 0, 10, 0, 4, 0, 18, 0, 4, 0, 9, 0, 12, 0, 10, 0, 12, 0)\right) \\[2ex]
\end{array} \\
\equiv \mathbf{Toepl}\!
\begin{array}[t]{l}
\!\left((3, 0, 0, 0, 4, 0, 0, 0, 3, 0, 4, 0, 6, 0, 4, 0, 4, 0, 4, 0, 6, 0, 4, 0), \right.\\
\ \!\left.(0, 0, 0, 4, 0, 0, 0, 2, 0, 4, 0, 2, 0, 4, 0, 1, 0, 4, 0, 2, 0, 4, 0)\right) \\[2ex]
\end{array} \\
\equiv \Td{1} \pmod{8}.
\end{array}
$$
Therefore,
$$
9\T{24}{3.2^u} \equiv \Td{0} + 2^{u-3}\Td{1} \pmod{2^u}.
$$
This completes the proof.
\end{proof}

\section{Periodicity of solutions}

For any positive integer $m$, let $\setP{m}$ be the set of $24$-tuples of integers $A$ such that the orbit of the sequence $\IAP{\pi_m(A)}{\pi_m(A)\matX{24}}$ is $(24m,24m)$-periodic in $\Zn{m}$. From Theorem~\ref{thm1}, it is straightforward to see that
$$
\setP{m} = \left\{ A\in\Z^{24}\ \middle|\ \pi_{m}(A)\in\Lker_{m}\M{24}{24m}\right\},
$$
where $\M{24}{24m} = \W{24}{24m}+\matX{24}\T{24}{24m} = \C{24}{24m}-\matI{24}+\matX{24}\T{24}{24m}$.

In this section, we determine the set $\setF$ of $24$-tuples of integers $A$ such that the orbit of $S = \IAP{\pi_{2^u}(A)}{\pi_{2^u}(A)\matX{k}}$ is $(24.2^u,24.2^u)$-periodic, for all non-negative integers $u$, that is,
$$
\setF = \bigcap_{u\in\N}\setP{2^u} = \left\{ A\in\Z^{24}\ \middle|\ \pi_{2^u}(A)\in\Lker_{2^u}\M{24}{24.2^u},\ \forall u\in\N \right\}.
$$
First, by definition, it is clear that $\setF$ is a submodule of the $\Z$-module of $24$-tuples of integers. Moreover, using results of the previous section, we show that $\setF$ is the free $\Z$-module of rank $16$ whose $\left\{Z_1,\ldots,Z_{16}\right\}$ is a basis.

\begin{thm}\label{thm2}
For the $24$-tuples $Z_1,\ldots,Z_{16}$ given in Table~\ref{tab4}, we have
$$
\setF = \left\langle Z_1,\ldots,Z_{16}\right\rangle.
$$
\end{thm}

The proof is based on the following three lemmas.

\begin{lem}\label{thm3}
For all non-negative integers $u$, we have
$$
9\M{24}{24.2^u} \equiv \Nd{0} \pmod{2^u},
$$
where $\Nd{0}$ is the integer matrix defined by
\begin{center}
\includegraphics[width=\textwidth]{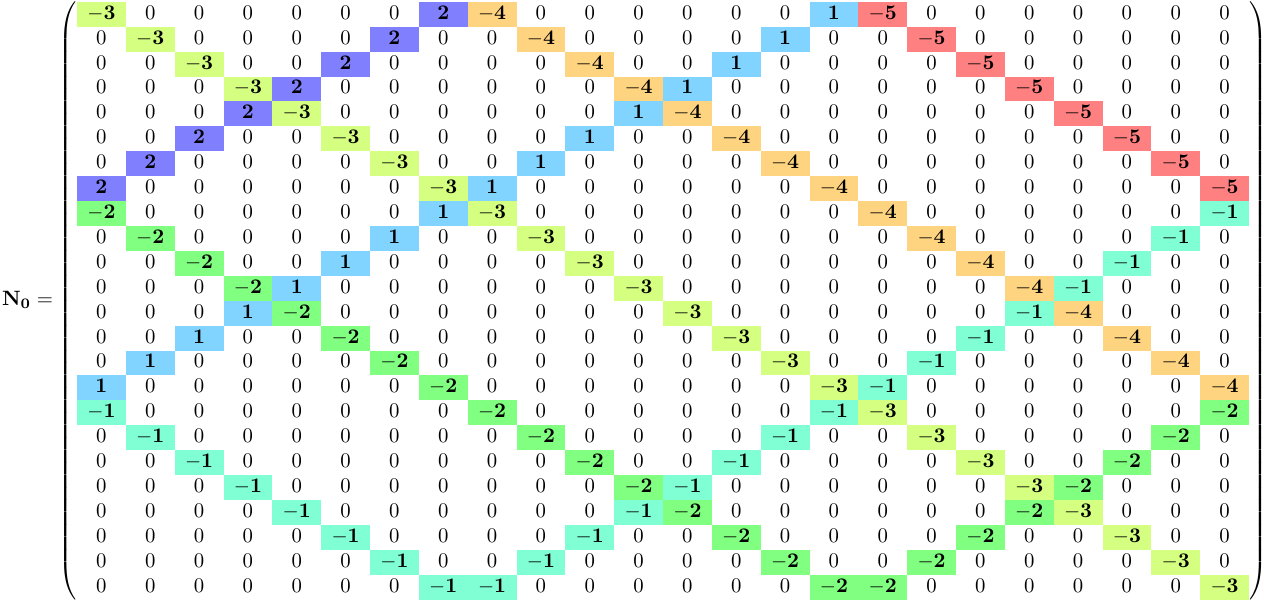}
\end{center}
\end{lem}

\begin{proof}
First, from Theorem~\ref{prop19}, we know that
$$
3\C{24}{24.2^u} = 3\C{24}{3.2^{u+3}} \equiv \Cd{0}+2^{u+1}\Cd{1} \pmod{2^{u+3}}
$$
and thus
$$
3\C{24}{24.2^u} \equiv \Cd{0} \pmod{2^u},
$$
for all $u\ge 1$. The result is clear for $u=0$. Moreover, From Theorem~\ref{prop18}, we know that
$$
9\T{24}{24.2^u} = 9\T{24}{3.2^{u+3}} \equiv \Td{0}+2^{u}\Td{1} \pmod{2^{u+3}}
$$
and thus
$$
9\T{24}{24.2^u} \equiv \Td{0} \pmod{2^u},
$$
for all $u\ge 2$. The result is also true for $u=1$, since
$$
\T{24}{48} \equiv \mathbf{Toepl}\!
\begin{array}[t]{l}
\!\left( (14,0,0,0,8,0,0,0,15,0,8,0,4,0,8,0,2,0,8,0,4,0,8,0), \right.\\
\ \!\left.(0,0,0,0,0,0,0,3,0,8,0,12,0,8,0,0,0,8,0,4,0,8,0) \right)
\end{array} \pmod{16},
$$
and is trivial for $u=0$. Therefore,
$$
9\M{24}{24.2^u}
\begin{array}[t]{l}
= 9\C{24}{24.2^u}-9\matI{24}+9\matX{24}\T{24}{24.2^u} \\[2ex]
\equiv 3\Cd{0}-9\matI{24}+\matX{24}\Td{0} \pmod{2^u}. \\
\end{array}
$$
By direct computation, we obtain that
$$
3\Cd{0}-9\matI{24}+\matX{24}\Td{0} = \Nd{0}.
$$
This completes the proof.
\end{proof}

\begin{lem}
$\setF=\Lker\Nd{0}$.
\end{lem}

\begin{proof}
By definition of $\setF$, we have
$$
\setF = \left\{ A\in\Z^{24}\ \middle|\ A\M{24}{24.2^u}\equiv\matr{0}\pmod{2^u},\ \forall u\in\N \right\}.
$$
From Lemma~\ref{thm3}, we know that
$$
9\M{24}{24.2^u} \equiv \Nd{0} \pmod{2^u},
$$
for all non-negative integers $u$. Moreover, for any $A\in\Z^{24}$, since $\gcd(9,2^u)=1$, we have
$$
A\M{24}{24.2^u}\equiv\matr{0}\pmod{2^u}\quad\Longleftrightarrow\quad 9A\M{24}{24.2^u}\equiv\matr{0}\pmod{2^u},
$$
for all non-negative integers $u$. This leads to
$$
\setF = \left\{ A\in\Z^{24}\ \middle|\ A\Nd{0}\equiv\matr{0}\pmod{2^u},\ \forall u\in\N \right\}.
$$
Finally, for any $A\in\Z^{24}$, since
$$
A\Nd{0}\equiv\matr{0}\pmod{2^u},\ \forall u\in\N\quad\Longleftrightarrow\quad A\Nd{0} = \matr{0},
$$
we conclude that
$$
\setF = \left\{ A\in\Z^{24}\ \middle|\ A\Nd{0} = \matr{0}\right\}=\Lker\Nd{0}.
$$
This completes the proof.
\end{proof}

\begin{nota}[Rows and columns of a matrix]
Let $\matM$ be a matrix of size $n_1\times n_2$ over $\Zn{m}$. The $i$\textsuperscript{th} row of $\matM$ is denoted by $R_i(\matM)$, for all $i\in\{1,\ldots,n_1\}$, and the $j$\textsuperscript{th} column of $\matM$ is denoted by $C_j(\matM)$, for all $j\in\{1,\ldots,n_2\}$.
\end{nota}

\begin{lem}\label{lem*8}
$\Lker\Nd{0} = \left\langle Z_1,\ldots,Z_{16}\right\rangle$.
\end{lem}

\begin{proof}
First, all the rows of $\Nd{0}$ are linear combination of its last eight rows. Indeed, for all $i\in\{1,\ldots,8\}$, we have
\begin{equation}\label{eq*8}
R_i(\Nd{0}) = 3R_{16+i}(\Nd{0})-2R_{25-i}(\Nd{0}),
\end{equation}
and
\begin{equation}\label{eq*9}
R_{8+i}(\Nd{0}) = 2R_{16+i}(\Nd{0})-R_{25-i}(\Nd{0}).
\end{equation}
Therefore the rank of $\Nd{0}$ is at most $8$. Moreover, since the submatrix obtained by removing the first $16$ rows and the last $16$ columns is the matrix $-\matI{8}$ of rank $8$, we obtain that $\Nd{0}$ is of rank $8$. From the rank-nullity theorem, we deduce that
$$
\dim\Lker\Nd{0} = 24-\rank\Nd{0} = 16.
$$
Finally, from \eqref{eq*8} and \eqref{eq*9}, we deduce that,
$$
Z_i = E_i-3E_{16+i}+2E_{25-i} \in\Lker\Nd{0}
$$
and
$$
Z_{8+i} = E_{8+i}-2E_{16+i}+E_{25-i} \in\Lker\Nd{0},
$$
for all $i\in\{1,\ldots,8\}$. Since the family $\left\{Z_1,\ldots,Z_{16}\right\}$ is clearly linearly independent (by considering the $16$ first elements of each one of them) and since $\Lker\Nd{0}$ is of dimension $16$, we conclude that $\left\{Z_1,\ldots,Z_{16}\right\}$ is a basis of $\Lker\Nd{0}$. Therefore,
$$
\Lker\Nd{0} = \left\langle Z_1,\ldots,Z_{16}\right\rangle,
$$
as announced.
\end{proof}

It follows that
$$
\setE\subset\setF
$$
from Proposition~\ref{prop*3}. Therefore, for all $A\in\setE$, the orbit of the sequence
$$
S = \IAP{\pi_{2^u}(A)}{\pi_{2^u}(A)\matX{k}}
$$
is $(24.2^u,24.2^u)$-periodic, for all non-negative integers $u$. To enhance this result and establish Theorems~\ref{mainthm} and \ref{thm*2}, the interlaced doubly arithmetic orbit structure is introduced in the following section.

\section{Interlaced doubly arithmetic orbits of IAP}

In this section, the interlaced doubly arithmetic orbit structure is introduced and conditions for obtaining interlaced doubly arithmetic orbits of IAP are given.

\begin{defn}[DAP]
Let $a$, $d_1$ and $d_2$ be three elements in $\Zn{m}$. The {\em doubly arithmetic progression} $\DAP{a}{d_1}{d_2}$ is the sequence $\left(u_{i,j}\right)_{(i,j)\in\N\times\Z}$ of $\Zn{m}$ where $u_{0,0}=a$, each row $\left(u_{i,j}\right)_{j\in\Z}$ is an arithmetic progression with common difference $d_1$, for all $i\in\N$, and each column $\left(u_{i,j}\right)_{i\in\N}$ is an arithmetic progression with common difference $d_2$, for all $j\in\Z$, i.e., the sequence $\left(u_{i,j}\right)_{(i,j)\in\N\times\Z}$ defined by
$$
u_{i,j} = a + id_2 + jd_1,
$$
for all $(i,j)\in\N\times\Z$. A doubly infinite sequence $S=\left(u_{i,j}\right)_{(i,j)\in\N\times\Z}$ of $\Zn{m}$ is said to be {\em doubly arithmetic} if $S=\DAP{u_{0,0}}{u_{0,1}-u_{0,0}}{u_{1,0}-u_{0,0}}$, i.e.,  if
$$
u_{i,j} = u_{0,0} + i\left(u_{1,0}-u_{0,0}\right) + j\left(u_{0,1}-u_{0,0}\right),
$$
for all $(i,j)\in\N\times\Z$.
\end{defn}

For instance, the doubly arithmetic progression $\DAP{0}{1}{3}$ of $\Zn{8}$ is depicted in Figure~\ref{fig*07}.

\begin{figure}[htbp]
\centering{
\includegraphics{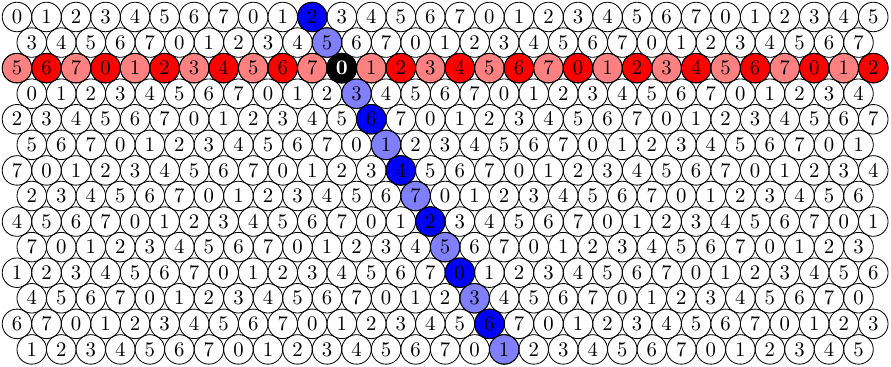}
}
\caption{The doubly arithmetic progression $\DAP{0}{1}{3}$ of $\Zn{8}$}\label{fig*07}
\end{figure}

\begin{prop}\label{prop15}
The doubly infinite sequence $\left(u_{i,j}\right)_{(i,j)\in\N\times\Z}$ of $\Zn{m}$ is doubly arithmetic if and only if
\begin{enumerate}[i)]
\item\label{prop15item1}
each row $\left(u_{i,j}\right)_{j\in\Z}$ is arithmetic with common difference $u_{0,1}-u_{0,0}$, for all $i\in\N$,
\item\label{prop15item2}
the column $\left(u_{i,0}\right)_{i\in\N}$ is arithmetic.
\end{enumerate}
\end{prop}

\begin{proof}
First, if the sequence $\left(u_{i,j}\right)_{(i,j)\in\N\times\Z}$ is doubly arithmetic, then
$$
\left(u_{i,j}\right)_{(i,j)\in\N\times\Z}=\DAP{u_{0,0}}{u_{0,1}-u_{0,0}}{u_{1,0}-u_{0,0}}
$$
and it is clear that each row $\left(u_{i,j}\right)_{j\in\Z}$ is arithmetic with common difference $u_{0,1}-u_{0,0}$, for all $i\in\N$, and each column $\left(u_{i,j}\right)_{i\in\N}$ is arithmetic with common difference $u_{1,0}-u_{0,0}$, for all $j\in\Z$.

Conversely, suppose that \ref{prop15item1}) and \ref{prop15item2}) are verified. Let $(r,s)\in\N\times\Z$. Since the row $\left(u_{r,j}\right)_{j\in\Z}$ is arithmetic with common difference $u_{0,1}-u_{0,0}$, we have
$$
u_{r,s} = u_{r,0} + s\left(u_{0,1}-u_{0,0}\right).
$$
Moreover, since the column $\left(u_{i,0}\right)_{i\in\N}$ is arithmetic, we have
$$
u_{r,0} = u_{0,0} + r\left(u_{1,0}-u_{0,0}\right).
$$
Therefore,
$$
u_{r,s} = u_{0,0} + r\left(u_{1,0}-u_{0,0}\right) + s\left(u_{0,1}-u_{0,0}\right),
$$
for all $(r,s)\in\N\times\Z$, i.e., the sequence $\left(u_{i,j}\right)_{(i,j)\in\N\times\Z}$ is doubly arithmetic.
\end{proof}

\begin{defn}[IDAP]
Let $k_1$ and $k_2$ be two positive integers and let
$$
\matA=\left(a_{i,j}\right)_{\substack{0\le i\le k_2-1\\ 0\le j\le k_1-1}},\quad \matD{1}=\left(d^{\,(1)}_{i,j}\right)_{\substack{0\le i\le k_2-1\\ 0\le j\le k_1-1}}\quad\text{and}\quad \matD{2}=\left(d^{\,(2)}_{i,j}\right)_{\substack{0\le i\le k_2-1\\ 0\le j\le k_1-1}}
$$
be three $(k_2\times k_1)$-matrices of elements in $\Zn{m}$. The {\em interlaced doubly arithmetic progression} $\IDAP{\matA}{\matD{1}}{\matD{2}}$ is the sequence $\left(u_{i,j}\right)_{(i,j)\in\N\times\Z}$ of $\Zn{m}$ where each subsequence $\left(u_{i_0+ik_2,j_0+jk_1}\right)_{(i,j)\in\N\times\Z}$ verifies
$$
\left(u_{i_0+ik_2,j_0+jk_1}\right)_{(i,j)\in\N\times\Z} = \DAP{a_{i_0,j_0}}{d^{\,(1)}_{i_0,j_0}}{d^{\,(2)}_{i_0,j_0}},
$$
for all $(i_0,j_0)\in\{0,\ldots,k_2-1\}\times\{0,\ldots,k_1-1\}$, i.e., the sequence $\left(u_{i,j}\right)_{(i,j)\in\N\times\Z}$ defined by
$$
u_{i_0+ik_2,j_0+jk_1} = a_{i_0,j_0} + id^{\,(2)}_{i_0,j_0} + jd^{\,(1)}_{i_0,j_0}
$$
for all $(i_0,j_0)\in\{0,\ldots,k_2-1\}\times\{0,\ldots,k_1-1\}$ and all $(i,j)\in\N\times\Z$.
A doubly infinite sequence $\left(u_{i,j}\right)_{(i,j)\in\N\times\Z}$ of $\Zn{m}$ is said to be {\em $(k_1,k_2)$-interlaced doubly arithmetic} if
$$
\left(u_{i,j}\right)_{(i,j)\in\N\times\Z} = \IDAP{\left(u_{i,j}\right)_{\substack{0\le i\le k_2-1\\ 0\le j\le k_1-1}}}{\left(u_{i+k_2,j}-u_{i,j}\right)_{\substack{0\le i\le k_2-1\\ 0\le j\le k_1-1}}}{\left(u_{i,j+k_1}-u_{i,j}\right)_{\substack{0\le i\le k_2-1\\ 0\le j\le k_1-1}}},
$$
i.e., if
$$
u_{i_0+ik_2,j_0+jk_1} = u_{i_0,j_0}+i\left(u_{i_0+k_2,j_0}-u_{i_0,j_0}\right)+j\left(u_{i_0,j_0+k_1}-u_{i_0,j_0}\right),
$$
for all $(i_0,j_0)\in\{0,\ldots,k_2-1\}\times\{0,\ldots,k_1-1\}$ and all $(i,j)\in\N\times\Z$.
\end{defn}

For instance, the interlaced doubly arithmetic progression
$$
\IDAP{
\begin{pmatrix}
2 & 0 & 1 \\
3 & 2 & 2 \\
\end{pmatrix}
}{
\begin{pmatrix}
1 & 0 & 2 \\
2 & 3 & 1 \\
\end{pmatrix}
}{
\begin{pmatrix}
3 & 1 & 2 \\
0 & 3 & 1 \\
\end{pmatrix}
}
$$
of $\Zn{4}$ is depicted in Figure~\ref{fig*08}.

\begin{figure}[htbp]
\centering{
\includegraphics{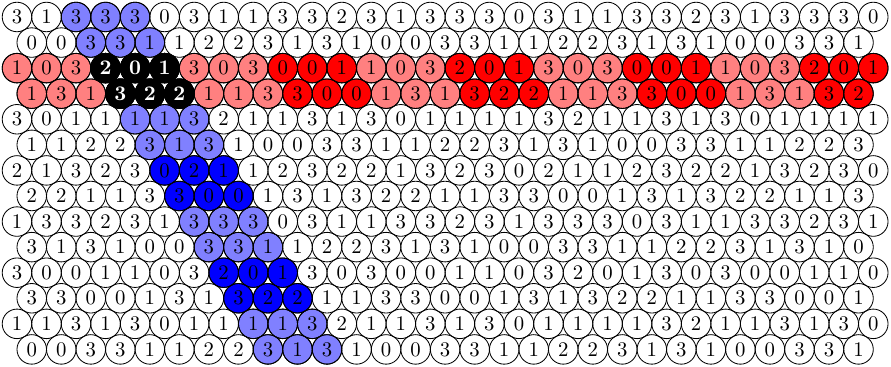}
}
\caption{$\IDAP{
\begin{pmatrix}
2 & 0 & 1 \\
3 & 2 & 2 \\
\end{pmatrix}
}{
\begin{pmatrix}
1 & 0 & 2 \\
2 & 3 & 1 \\
\end{pmatrix}
}{
\begin{pmatrix}
3 & 1 & 2 \\
0 & 3 & 1 \\
\end{pmatrix}
}$ of $\Zn{4}$}\label{fig*08}
\end{figure}

\begin{prop}\label{prop16}
The sequence $\left(u_{i,j}\right)_{(i,j)\in\N\times\Z}$ of $\Zn{m}$ is $(k_1,k_2)$-interlaced doubly arithmetic if and only if
\begin{enumerate}[i)]
\item
each row $\left(u_{i,j}\right)_{j\in\Z}$ is $k_1$-interlaced arithmetic, for all $i\in\N$, and the sequence of their common differences $\left(\left(u_{i,j+k_1}-u_{i,j}\right)_{j=0}^{k_1-1}\right)_{i\in\N}$ is $k_2$-periodic,
\item
the sequence $\left(\left(u_{i,j}\right)_{j=0}^{k_1-1}\right)_{i\in\N}$ is $k_2$-interlaced arithmetic.
\end{enumerate} 
\end{prop}

\begin{proof}
Since, by definition, the sequence $\left(u_{i,j}\right)_{(i,j)\in\N\times\Z}$ is $(k_1,k_2)$-interlaced arithmetic if and only each subsequence $\left(u_{i_0+ik_2,j_0+jk_1}\right)_{(i,j)\in\N\times\Z}$ is doubly arithmetic, for all $(i_0,j_0)\in\{0,\ldots,k_2-1\}\times\{0,\ldots,k_1-1\}$. For any $(i_0,j_0)\in\{0,\ldots,k_2-1\}\times\{0,\ldots,k_1-1\}$, we know from Proposition~\ref{prop15} that the subsequence $\left(u_{i_0+ik_2,j_0+jk_1}\right)_{(i,j)\in\N\times\Z}$ is doubly arithmetic if and only if
\begin{itemize}
\item
the sequence $\left(u_{i_0+ik_2,j_0+jk_1}\right)_{j\in\Z}$ is arithmetic with common difference $u_{i_0,j_0+k_1}-u_{i_0,j_0}$, for all $i\in\N$,
\item
the sequence $\left(u_{i_0+ik_2,j_0}\right)_{i\in\N}$ is arithmetic.
\end{itemize}
Since
$$
\begin{array}{l}
\displaystyle\left(u_{i_0+ik_2,j_0+jk_1}\right)_{j\in\Z}\ \text{arithmetic},\ \forall j_0\in\{0,\ldots,k_1-1\},\\[2ex]
\displaystyle\Longleftrightarrow\ \left(u_{i_0+ik_2,j}\right)_{j\in\Z}\ k_1\text{-interlaced arithmetic}, \\
\end{array}
$$
for all $i_0\in\{0,\ldots,k_2-1\}$ and $i\in\N$,
\begin{equation*}
\resizebox{\textwidth}{!}{$
\begin{array}{l}
u_{i_0+ik_2,j_0+k_1}-u_{i_0+ik_2,j_0} = u_{i_0,j_0+k_1}-u_{i_0,j_0},\ \forall i\in\N,\ \forall i_0\in\{0,\ldots,k_2-1\},\ \forall j_0\in\{0,\ldots,k_1-1\},\\[2ex]
\Longleftrightarrow\ \left(u_{i,j_0+k_1}-u_{i,j_0}\right)_{i\in\N}\ k_2\text{-periodic},\ \forall j_0\in\{0,\ldots,k_1-1\},\\[2ex]
\Longleftrightarrow\ \left(\left(u_{i,j+k_1}-u_{i,j}\right)_{j=0}^{k_1-1}\right)_{i\in\N}\ k_2\text{-periodic},\\
\end{array}
$}
\end{equation*}
and
$$
\begin{array}{l}
\left(u_{i_0+ik_2,j_0}\right)_{i\in\N}\ \text{arithmetic},\ \forall i_0\in\{0,\ldots,k_2-1\},\ \forall j_0\in\{0,\ldots,k_1-1\},\\[2ex]
\Longleftrightarrow\ \left(u_{i,j_0}\right)_{i\in\N}\ k_2\text{-interlaced arithmetic},\ \forall j_0\in\{0,\ldots,k_1-1\},\\[2ex]
\Longleftrightarrow\ \left(\left(u_{i,j}\right)_{j=0}^{k_1-1}\right)_{i\in\N}\ k_2\text{-interlaced arithmetic},\\
\end{array}
$$
the result follows.
\end{proof}

Now, the sequences $S$ whose orbit $\orb{S}$ is $(k_1,k_2)$-doubly arithmetic are characterized. First, it is clear that the sequence $S$ is necessarily a $k_1$-interlaced arithmetic progression and we already know from Proposition~\ref{prop2} that the orbit $\orb{S}$ only contains $k_1$-interlaced arithmetic progressions.

\begin{thm}\label{thm4}
Let $A$ and $D$ be two $k_1$-tuples of elements in $\Zn{m}$. Then, the orbit $\orb{S}$ of $S=\IAP{A}{D}$ is $(k_1,k_2)$-interlaced doubly arithmetic if and only if the block matrix $\left(\begin{array}{c|c} A & D \end{array}\right)$ is in $\Lker_{m}\IA{k_1}{k_2}$, with
$$
\IA{k_1}{k_2} = 
\left(\begin{array}{c|c}
{\W{k_1}{k_2}}^2 & \matr{0_{k_1}} \\
\hline
\T{k_1}{k_2}\W{k_1}{k_2} & \W{k_1}{k_2} \\
\end{array}\right)
$$
where $\W{k_1}{k_2}$ is the Wendt matrix
$$
\W{k_1}{k_2} = \C{k_1}{k_2}+{(-1)}^{k_2+1}\matI{k_1}.
$$
If the orbit $\orb{S}$ of $S=\IAP{A}{D}$ is $(k_1,k_2)$-interlaced doubly arithmetic, then we have $\orb{S} = \IDAP{\matA}{\matD{1}}{\matD{2}}$, where the rows of $\matA$, $\matD{1}$ and $\matD{2}$ satisfy
$$
\begin{array}{l}
\displaystyle R_i(\matA) = {(-1)}^{i}\left(A\C{k_1}{i}+D\T{k_1}{i}\right), \\ \ \\
\displaystyle R_i(\matD{1}) = {(-1)}^{i}D\C{k_1}{i}, \\  \ \\
\displaystyle R_i(\matD{2}) = {(-1)}^{i+k_2}\left(A\W{k_1}{k_2}+D\T{k_1}{k_2}\right)\C{k_1}{i}, \\
\end{array}
$$
for all $i\in\left\{0,\ldots,k_2-1\right\}$.
\end{thm}

\begin{proof}
Let $S=\IAP{A}{D}$ and $\orb{S}=\left(u_{i,j}\right)_{(i,j)\in\N\times\Z}$. From Proposition~\ref{prop2}, we know that each row of $\orb{S}$ is the $k_1$-interlaced arithmetic progression
$$
\left(u_{i,j}\right)_{j\in\Z} = \IAP{{(-1)}^{i}\left(A\C{k_1}{i}+D\T{k_1}{i}\right)}{{(-1)}^{i}D\C{k_1}{i}},
$$
for all $i\in\N$. It follows, from Proposition~\ref{prop16}, that $\orb{S}$ is $(k_1,k_2)$-interlaced doubly arithmetic if and only if
\begin{itemize}
\item
the sequence $\left(\left(u_{i,j+k_1}-u_{i,j}\right)_{j=0}^{k_1-1}\right)_{i\in\N} = \left({(-1)}^{i}D\C{k_1}{i}\right)_{i\in\N}$ is $k_2$-periodic,
\item
the sequence $\left(\left(u_{i,j}\right)_{j=0}^{k_1-1}\right)_{i\in\N} = \left({(-1)}^{i}\left(A\C{k_1}{i}+D\T{k_1}{i}\right)\right)_{i\in\N}$ is $k_2$-interlaced arithmetic.
\end{itemize}
First, the sequence $\left({(-1)}^{i}D\C{k_1}{i}\right)_{i\in\N}$ is $k_2$-periodic if and only if
$$
\begin{array}{l}
{(-1)}^{i+k_2}D\C{k_1}{i+k_2} = {(-1)}^{i}D\C{k_1}{i},\ \forall i\in\N, \\[2ex]
\Longleftrightarrow\ D\left(\C{k_1}{i+k_2}+{(-1)}^{k_2+1}\C{k_1}{i}\right) = \matr{0},\ \forall i \in\N.\\
\end{array}
$$
Since $\C{k_1}{i+k_2}=\C{k_1}{k_2}\C{k_1}{i}$ by Proposition~\ref{prop*2}, we have
$$
\C{k_1}{i+k_2}+{(-1)}^{k_2+1}\C{k_1}{i} = \left(\C{k_1}{k_2}+{(-1)}^{k_2+1}\matI{k_1}\right)\C{k_1}{i} = \W{k_1}{k_2}\C{k_1}{i},
$$
for all non-negative integers $i$. Therefore, the sequence $\left({(-1)}^{i}D\C{k_1}{i}\right)_{i\in\N}$ is $k_2$-periodic if and only if
$$
D\W{k_1}{k_2}\C{k_1}{i} = \matr{0},\ \forall i\in\N\quad\stackrel{\C{k_1}{0}=\matr{I_{k_1}}}{\Longleftrightarrow}\quad D\W{k_1}{k_2}= \matr{0}.
$$
Note that,
$$
D\W{k_1}{k_2} = \matr{0}\ \Longleftrightarrow\ D\C{k_1}{k_2} = {(-1)}^{k_2}D.
$$
Moreover, the sequence
$$
\left(\left(u_{i,j}\right)_{j=0}^{k_1-1}\right)_{i\in\N}=\left({(-1)}^{i}\left(A\C{k_1}{i}+D\T{k_1}{i}\right)\right)_{i\in\N}
$$
is $k_2$-interlaced arithmetic if and only if
$$
\begin{array}{l}
\begin{array}{lcl}
{(-1)}^{i+2k_2}\displaystyle\left(A\C{k_1}{i+2k_2}+D\T{k_1}{i+2k_2}\right)\\[2ex]
-2{(-1)}^{i+k_2}\displaystyle\left(A\C{k_1}{i+k_2}+D\T{k_1}{i+k_2}\right)\\[2ex]
+{(-1)}^{i}\displaystyle\left(A\C{k_1}{i}+D\T{k_1}{i}\right) & = & \matr{0_{1,k_1}},\ \forall i\in\N,\\[2ex]
\end{array}\\
\displaystyle\Longleftrightarrow\ 
\begin{array}[t]{l}
A\displaystyle\left(\C{k_1}{i+2k_2}+2{(-1)}^{k_2+1}\C{k_1}{i+k_2}+\C{k_1}{i}\right)\\[2ex]
+ D\displaystyle\left(\T{k_1}{i+2k_2}+2{(-1)}^{k_2+1}\T{k_1}{i+k_2}+\T{k_1}{i}\right) = \matr{0},\ \forall i\in\N.\\
\end{array}
\end{array}
$$
Since
$$
\C{k_1}{i+2k_2}={\C{k_1}{k_2}}^2\C{k_1}{i}\quad\text{and}\quad\C{k_1}{i+k_2}=\C{k_1}{k_2}\C{k_1}{i}
$$
by Proposition~\ref{prop*2}, we obtain that
$$
\begin{array}{l}
A\displaystyle\left(\C{k_1}{i+2k_2}+2{(-1)}^{k_2+1}\C{k_1}{i+k_2}+\C{k_1}{i}\right)\\[2ex]
= A\displaystyle\left({\C{k_1}{k_2}}^2+2{(-1)}^{k_2+1}\C{k_1}{k_2}+\matr{I_{k_1}}\right)^2\C{k_1}{i}\\[2ex]
= A\displaystyle\left(\C{k_1}{k_2}+{(-1)}^{k_2+1}\matr{I_{k_1}}\right)^2\C{k_1}{i}\\[2ex]
= A\displaystyle{\W{k_1}{k_2}}^2\C{k_1}{i},
\end{array}
$$
for all non-negative integers $i$. Since
$$\
\T{k_1}{i+2k_2} 
\begin{array}[t]{l}
= \T{k_1}{2k_2}\C{k_1}{i} + \C{k_1}{2k_2}\T{k_1}{i} \\[2ex]
= \left(\T{k_1}{k_2}\C{k_1}{k_2}+\C{k_1}{k_2}\T{k_1}{k_2}\right)\C{k_1}{i} + \left(\C{k_1}{k_2}\right)^2\T{k_1}{i} \\
\end{array}
$$
and
$$
\T{k_1}{i+k_2} = \T{k_1}{k_2}\C{k_1}{i} + \C{k_1}{k_2}\T{k_1}{i}
$$
by Proposition~\ref{prop*2} again, it follows that
$$
\begin{array}{l}
\displaystyle\T{k_1}{i+2k_2}+2{(-1)}^{k_2+1}\T{k_1}{i+k_2}+\T{k_1}{i} \\[2ex]
= 
\begin{array}[t]{l}
\displaystyle\left(\T{k_1}{k_2}\C{k_1}{k_2}+\C{k_1}{k_2}\T{k_1}{k_2}+2{(-1)}^{k_2+1}\T{k_1}{k_2}\right)\C{k_1}{i}\\[2ex]
+ \displaystyle\left(\left(\C{k_1}{k_2}\right)^2+2{(-1)}^{k_2+1}\C{k_1}{k_2}+\matr{I_{k_1}}\right)\T{k_1}{i}\\[2ex]
\end{array}\\[2ex]
= 
\begin{array}[t]{l}
\displaystyle\left(\T{k_1}{k_2}\left(\C{k_1}{k_2}+{(-1)}^{k_2+1}\matr{I_{k_1}}\right) + \left(\C{k_1}{k_2}+{(-1)}^{k_2+1}\matr{I_{k_1}}\right)\T{k_1}{k_2}\right)\C{k_1}{i}\\[2ex]
+ \displaystyle\left(\C{k_1}{k_2}+{(-1)}^{k_2+1}\matr{I_{k_1}}\right)^2\T{k_1}{i}\\[2ex]
\end{array}\\
= \displaystyle\left(\T{k_1}{k_2}\W{k_1}{k_2}+\W{k_1}{k_2}\T{k_1}{k_2}\right)\C{k_1}{i} + {\W{k_1}{k_2}}^2\T{k_1}{i},\\
\end{array}
$$
for all non-negative integers $i$. If $D\W{k_1}{k_2}=\matr{0}$, we obtain that
$$
D\left(\T{k_1}{i+2k_2}+2{(-1)}^{k_2+1}\T{k_1}{i+k_2}+\T{k_1}{i}\right) = D\T{k_1}{k_2}\W{k_1}{k_2}\C{k_1}{i},
$$
and thus
$$
\begin{array}{l}
\begin{array}{l}
A\displaystyle\left(\C{k_1}{i+2k_2}+2{(-1)}^{k_2+1}\C{k_1}{i+k_2}+\C{k_1}{i}\right)\\[2ex]
+ D\displaystyle\left(\T{k_1}{i+2k_2}+2{(-1)}^{k_2+1}\T{k_1}{i+k_2}+\T{k_1}{i}\right)\\[2ex]
\end{array} \\
= \left(A\displaystyle{\W{k_1}{k_2}}^2+D\T{k_1}{k_2}\W{k_1}{k_2}\right)\C{k_1}{i},
\end{array}
$$
for all non-negative integers $i$. Since $\C{k_1}{0}=\matr{I_{k_1}}$, we know that
$$
\begin{array}{l}
\left(A\displaystyle{\W{k_1}{k_2}}^2+D\T{k_1}{k_2}\W{k_1}{k_2}\right)\C{k_1}{i} = \matr{0},\ \forall i\in\N,\\[2ex]
\Longleftrightarrow\ A\displaystyle{\W{k_1}{k_2}}^2+D\T{k_1}{k_2}\W{k_1}{k_2} = \matr{0}.
\end{array}
$$
We conclude that the orbit $\orb{S}$ is $(k_1,k_2)$-interlaced doubly arithmetic if and only if
$$
\left\{\begin{array}{l}
D\displaystyle\W{k_1}{k_2} = \matr{0}, \\[2ex]
A\displaystyle{\W{k_1}{k_2}}^2+D\T{k_1}{k_2}\W{k_1}{k_2} = \matr{0}, \\
\end{array}\right.
$$
i.e., if and only if the block matrix $\left(\begin{array}{c|c} A & D \end{array}\right)$ is in $\Lker_{m}\IA{k_1}{k_2}$.

Suppose now that the orbit of $S=\IAP{A}{D}$ is $(k_1,k_2)$-interlaced doubly arithmetic. Let
$$
\orb{S} = \IDAP{\matA}{\matD{1}}{\matD{2}} = \left(u_{i,j}\right)_{(i,j)\in\N\times\Z}.
$$
We know from Proposition~\ref{prop2} that
$$
\left(u_{i,j}\right)_{j\in\Z} = \ider{i}{\IAP{A}{D}} = {(-1)}^{i}\,\IAP{A\C{k_1}{i}+D\T{k_1}{i}}{D\C{k_1}{i}},
$$
for all non-negative integers $i$. Therefore, for all $i\in\{0,\ldots,k_2-1\}$, we have
$$
R_i(\matA) = {(-1)}^{i}\left(A\C{k_1}{i}+D\T{k_1}{i}\right),
$$
$$
R_i(\matD{1}) = {(-1)}^{i}D\C{k_1}{i},
$$
and
\begin{equation*}
\resizebox{\textwidth}{!}{$
R_i(\matD{2}) 
\begin{array}[t]{l}
= {(-1)}^{i+k_2}\left(A\C{k_1}{i+k_2}+D\T{k_1}{i+k_2}\right) - {(-1)}^{i}\left(A\C{k_1}{i}+D\T{k_1}{i}\right) \\[2ex]
= {(-1)}^{i+k_2}\left(A\left(\C{k_1}{i+k_2}+{(-1)}^{k_2+1}\C{k_1}{i}\right)+D\left(\T{k_1}{i+k_2}+{(-1)}^{k_2+1}\T{k_1}{i}\right)\right).
\end{array}
$}
\end{equation*}
First, since
$$
\C{k_1}{i+k_2} = \C{k_1}{k_2}\C{k_1}{i}
$$
by Proposition~\ref{prop*2}, we have
$$
\C{k_1}{i+k_2}+{(-1)}^{k_2+1}\C{k_1}{i} = \W{k_1}{k_2}\C{k_1}{i}.
$$
Moreover, since
$$
\T{k_1}{i+k_2} = \T{k_1}{k_2}\C{k_1}{i}+\C{k_1}{k_2}\T{k_1}{i}
$$
by Proposition~\ref{prop*2} again, we have
$$
\T{k_1}{i+k_2} + {(-1)}^{k_2+1}\T{k_1}{i} = \T{k_1}{k_2}\C{k_1}{i} + \W{k_1}{k_2}\T{k_1}{i}.
$$
Finally, since $\orb{S}$ is $(k_1,k_2)$-interlaced doubly arithmetic, we know that
$$
D\W{k_1}{k_2} = \matr{0}
$$
and we conclude that
$$
R_i(\matD{2}) = {(-1)}^{i+k_2}\left(A\W{k_1}{k_2}+D\T{k_1}{k_2}\right)\C{k_1}{i},
$$
for all $i\in\{0,\ldots,k_2-1\}$. This completes the proof.
\end{proof}

For example, as depicted in Figure~\ref{fig*09}, the orbit of the sequence $S=\IAP{021}{111}$ of $\Zn{3}$ is $(3,3)$-interlaced doubly arithmetic
$$
\orb{S} = \IDAP{
\begin{pmatrix}
0 & 2 & 1 \\
1 & 0 & 1 \\
2 & 2 & 0 \\
\end{pmatrix}
}{
\begin{pmatrix}
1 & 1 & 1 \\
1 & 1 & 1 \\
1 & 1 & 1 \\
\end{pmatrix}
}{
\begin{pmatrix}
2 & 2 & 2 \\
2 & 2 & 2 \\
2 & 2 & 2 \\
\end{pmatrix}
}.
$$

\begin{figure}[htbp]
\centering{
\includegraphics{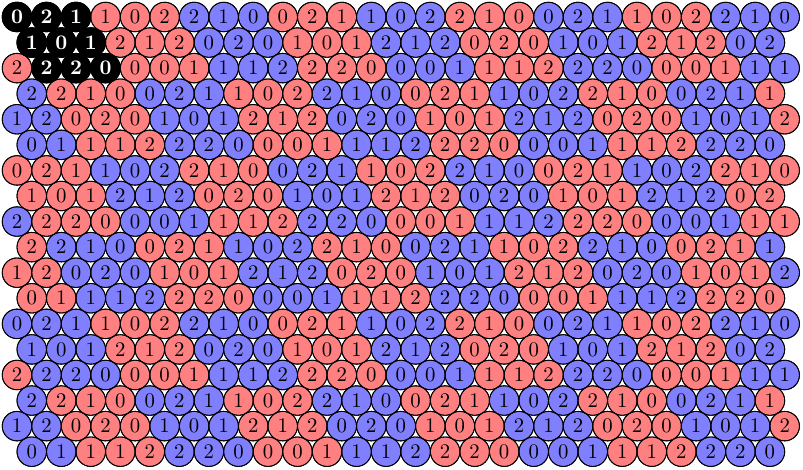}
}
\caption{$(3,3)$-interlaced doubly arithmetic orbit of $\IAP{021}{111}$ in $\Zn{3}$}\label{fig*09}
\end{figure}

Now, we refine Theorem~\ref{thm4} for interlaced arithmetic progressions with antisymmetric first periods.

\begin{cor}\label{cor*1}
Let $A$ be a $k_1$-tuple of $\Zn{m}$. Then, the orbit of $\IAP{A}{A\matX{k_1}}$ is $(k_1,k_2)$-interlaced doubly arithmetic if and only if $A$ is in $\Lker_m\AIA{k_1}{k_2}$, where
$$
\AIA{k_1}{k_2} = 
\left(\begin{array}{c|c}
\M{k_1}{k_2}\W{k_1}{k_2}
&
\matX{k_1}\W{k_1}{k_2}
\end{array}\right),
$$
with
$$
\W{k_1}{k_2} = \C{k_1}{k_2}+{(-1)}^{k_2+1}\matI{k_1}
$$
and
$$
\M{k_1}{k_2} = \W{k_1}{k_2}+\matX{k_1}\T{k_1}{k_2}.
$$
If the orbit $\orb{S}$ of $S=\IAP{A}{A\matX{k_1}}$ is $(k_1,k_2)$-interlaced doubly arithmetic, then we have $\orb{S} = \IDAP{\matA}{\matD{1}}{\matD{2}}$, where the rows of $\matA$, $\matD{1}$ and $\matD{2}$ satisfy
$$
\begin{array}{l}
\displaystyle R_i(\matA) = {(-1)}^{i}A\left(\C{k_1}{i}+\matX{k_1}\T{k_1}{i}\right), \\[2ex]
\displaystyle R_i(\matD{1}) = {(-1)}^{i}A\matX{k_1}\C{k_1}{i}, \\[2ex]
\displaystyle R_i(\matD{2}) = {(-1)}^{i+k_2}A\M{k_1}{k_2}\C{k_1}{i}, \\
\end{array}
$$
for all $i\in\left\{0,\ldots,k_2-1\right\}$.
\end{cor}

\begin{proof}
From Theorem~\ref{thm4} and Proposition~\ref{prop3}.
\end{proof}

\section{Interlaced doubly arithmeticity of solutions}

In this section, we show that, for any $A\in\setE$, the orbit of the sequence
$$
S = \IAP{\pi_{2^u}(A)}{\pi_{2^u}(A)\matX{24}}
$$
is $(24,3.2^u)$-interlaced doubly arithmetic in $\Zn{2^u}$, for all integers $u\ge3$. Moreover, the common differences $\matD{1}$ and $\matD{2}$ of
$$
\orb{S} = \IDAP{\matA}{\matD{1}}{\matD{2}}
$$
are determined, for all integers $u\ge3$.

\begin{defn}[$\setG{u_0}$]
For any positive integers $u_0$, let $\mathcal{G}_{u_0}$ be the set of $24$-tuples of integers $A$ such that the orbit of the sequence
$$
S=\IAP{\pi_{2^u}(A)}{\pi_{2^u}(A)\matr{X_{24}}}
$$
is $(24,3.2^u)$-interlaced arithmetic in $\Zn{2^u}$, for all integers $u\ge u_0$, that is,
$$
\mathcal{G}_{u_0} = \left\{ A\in\Z^{24}\ \middle|\ A\AIA{24}{3.2^u}\equiv\matr{0_{1,48}}\pmod{2^u},\ \forall u\ge u_0  \right\}.
$$
\end{defn}

In the sequel, the sets $\setG{5}$, $\setG{4}$ and $\setG{3}$ are determined and we will show that
$$
\mathcal{E} \subset \setG{3} \subset \setG{4} \subset \setG{5} \subset \mathcal{F}.
$$
First, using Theorems~\ref{prop19} and \ref{prop18}, we obtain the following expressions of the matrices $\M{24}{3.2^u}$ and $\AIA{24}{3.2^u}$ modulo $2^u$, for all integers $u\ge 5$.

\begin{prop}\label{prop*5}
For all integers $u\ge5$, we have
$$
9\M{24}{3.2^u} \equiv \Nd{0} + 2^{u-3}\Nd{1} \pmod{2^u},
$$
where
\begin{center}
\includegraphics[width=\textwidth]{MAT_01}
\end{center}
and
\begin{center}
\includegraphics[width=\textwidth]{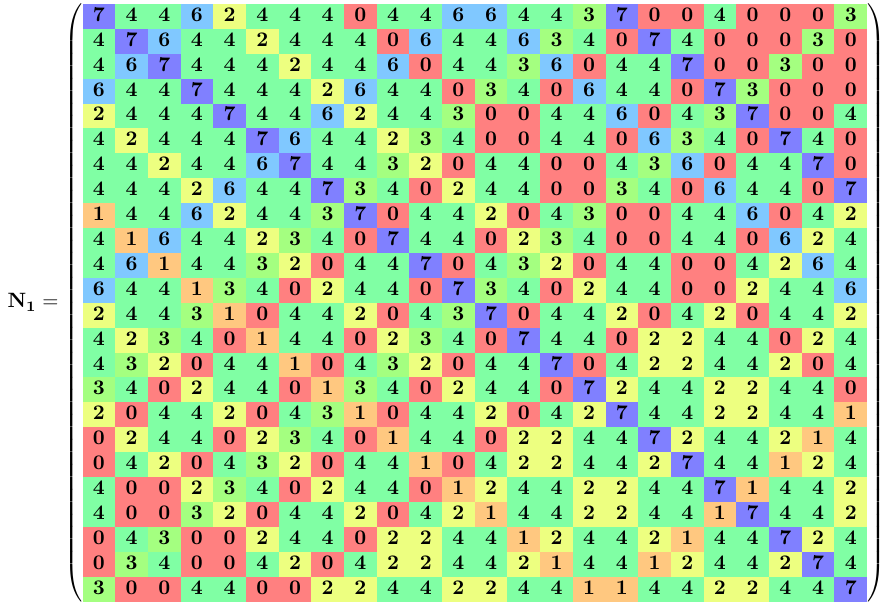}
\end{center}
\end{prop}

\begin{proof}
By definition,
$$
\M{24}{3.2^u} = \W{24}{3.2^u} + \matX{24}\T{24}{3.2^u},
$$
with
$$
\W{24}{3.2^u} = \C{24}{3.2^u}-\matI{24}.
$$
Let $u$ be an integer such that $u\ge5$. From Theorem~\ref{prop19}, we know that
$$
3\C{24}{3.2^u} \equiv \Cd{0} + 2^{u-2}\Cd{1} \pmod{2^u}
$$
and thus
\begin{equation}\label{eq*10}
3\W{24}{3.2^u} \equiv \left(\Cd{0}-3\matI{24}\right)+2^{u-2}\Cd{1}\pmod{2^u}.
\end{equation}
Moreover, from Theorem~\ref{prop18}, we know that
$$
9\T{24}{3.2^u} \equiv \Td{0} + 2^{u-3}\Td{1} \pmod{2^u}.
$$
It follows that
$$
9\M{24}{3.2^u} \equiv \left(3\Cd{0}-9\matI{24}+\matX{24}\Td{0}\right) + 2^{u-3}\left(6\Cd{1}+\matX{24}\Td{1}\right) \pmod{2^u}.
$$
By direct computation, we obtain that
$$
3\Cd{0}-9\matI{24}+\matX{24}\Td{0} = \Nd{0}
$$
and
$$
6\Cd{1}+\matX{24}\Td{1} \equiv \Nd{1} \pmod{8}.
$$
This completes the proof.
\end{proof}

\begin{prop}\label{prop21}
For all integers $u\ge5$, we have
$$
9\AIA{24}{3.2^u} \equiv \Md{0} + 2^{u-2}\Md{1} \pmod{2^u},
$$
where
\begin{center}
\includegraphics[width=\textwidth]{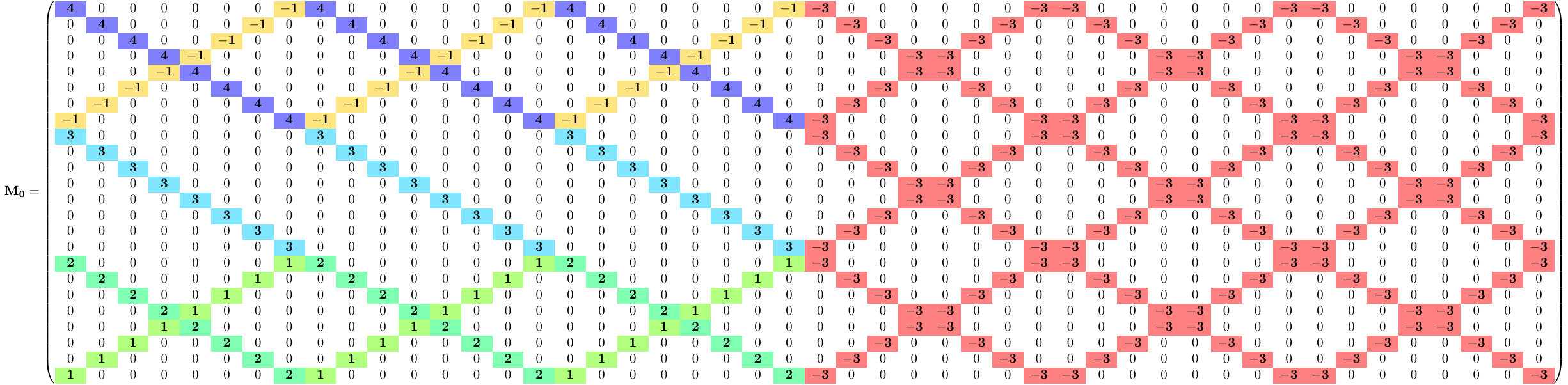}
\end{center}
and
\begin{center}
\includegraphics[width=\textwidth]{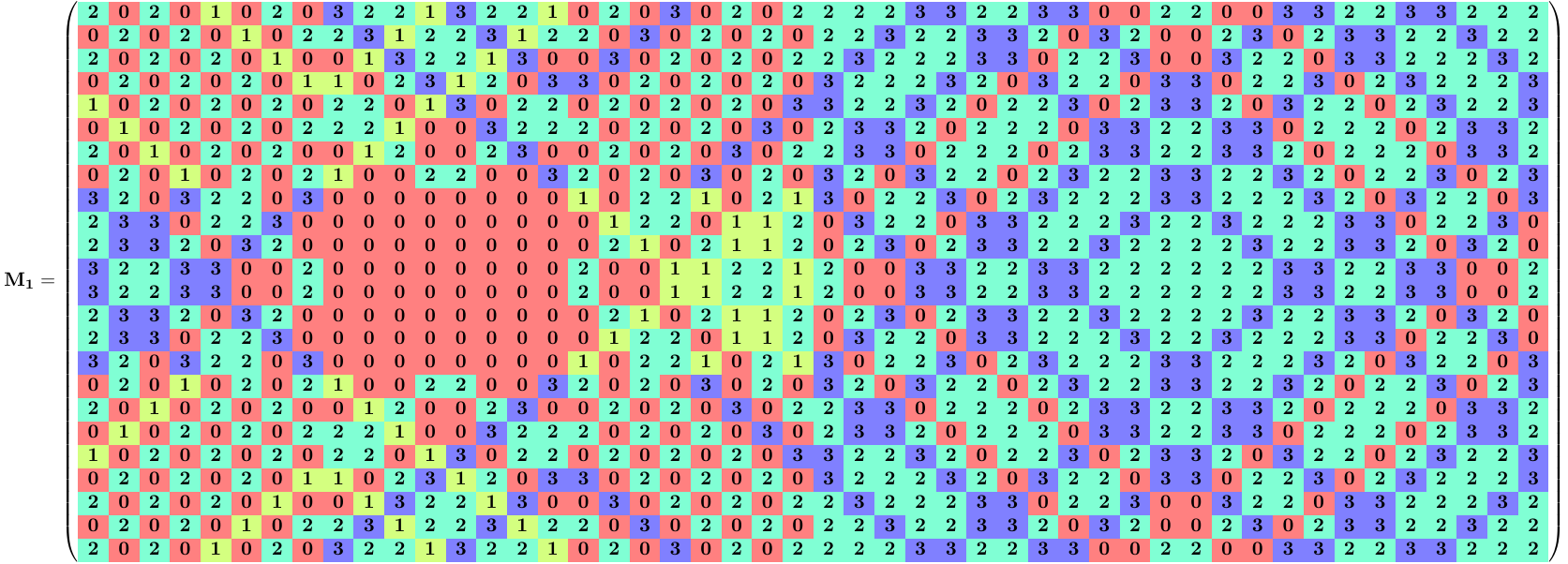}
\end{center}
\end{prop}

\begin{proof}
By definition,
$$
\AIA{24}{3.2^u} = 
\left(\begin{array}{c|c}
\M{24}{3.2^u}\W{24}{3.2^u}
&
\matr{X_{24}}\W{24}{3.2^u}
\end{array}\right).
$$
Let $u$ be an integer such that $u\ge5$. From Proposition~\ref{prop*5} and \eqref{eq*10}, we know that
$$
\begin{array}[t]{l}
27\M{24}{3.2^u}\W{24}{3.2^u} \\[2ex]
\equiv \left(\Nd{0}+2^{u-3}\Nd{1}\right)\left(\left(\Cd{0}-3\matI{24}\right)+2^{u-2}\Cd{1}\right) \\[2ex]
\equiv \Nd{0}\left(\Cd{0}-3\matI{24}\right) + 2^{u-3}\left(\Nd{1}\left(\Cd{0}-3\matI{24}\right)+2\Nd{0}\Cd{1}\right) + 2^{2u-5}\Nd{1}\Cd{1} \\[2ex]
\stackrel{2u-5\ge u}{\equiv} \Nd{0}\left(\Cd{0}-3\matI{24}\right) + 2^{u-3}\left(\Nd{1}\left(\Cd{0}-3\matI{24}\right)+2\Nd{0}\Cd{1}\right) \pmod{2^u}.
\end{array}
$$
Since
$$
\Cd{0}^2=3\Cd{0},\quad \Cd{0}\Cd{1}=\Cd{1}\Cd{0}\equiv\Cd{1} \pmod{2},\quad\text{and}\quad \Cd{1}^2 \equiv \matr{0_{24}} \pmod{2},
$$
as already seen in the proof of Theorem~\ref{prop19}, we obtain that
$$
\Nd{0}\left(\Cd{0}-3\matI{24}\right)
\begin{array}[t]{l}
= \left(3\Cd{0}-9\matI{24}+\matX{24}\Td{0}\right)\left(\Cd{0}-3\matI{24}\right) \\[2ex]
= \left(\matX{24}\Td{0}-9\matI{24}\right)\left(\Cd{0}-3\matI{24}\right)
\end{array}
$$
and
$$
\Nd{1}\left(\Cd{0}-3\matI{24}\right)+2\Nd{0}\Cd{1}
\begin{array}[t]{l}
\equiv \left(6\Cd{1}+\matX{24}\Td{1}\right)\left(\Cd{0}-3\matI{24}\right) + 2\left(3\Cd{0}-9\matI{24}+\matX{24}\Td{0}\right)\Cd{1} \\[2ex]
\equiv 12\Cd{1}\left(\Cd{0}-3\matI{24}\right) + \matX{24}\left(\Td{1}\left(\Cd{0}-3\matI{24}\right)+2\Td{0}\Cd{1}\right) \\[2ex]
\equiv \matX{24}\left(\Td{1}\left(\Cd{0}-3\matI{24}\right)+2\Td{0}\Cd{1}\right) \pmod{8}.
\end{array}
$$
Therefore, we have
$$
\begin{array}{l}
27\M{24}{3.2^u}\W{24}{3.2^u} \\[2ex]
\equiv \left(\matX{24}\Td{0}-9\matI{24}\right)\left(\Cd{0}-3\matI{24}\right) + 2^{u-3}\matX{24}\left(2\Td{0}\Cd{1}+\Td{1}\left(\Cd{0}-3\matI{24}\right)\right) \pmod{2^u}
\end{array}
$$
and
$$
27\matX{24}\W{24}{3.2^u} \equiv 9\matX{24}\left(\left(\Cd{0}-3\matI{24}\right)+2^{u-2}\Cd{1}\right) \pmod{2^u}.
$$
Finally, by direct computation, we obtain that
$$
\left(\begin{array}{c|c}
\left(\matr{X_{24}}\Td{0}-9\matI{24}\right)\left(\Cd{0}-3\matr{I_{24}}\right) & 9\matr{X_{24}}\left(\Cd{0}-3\matr{I_{24}}\right)
\end{array}\right)
 = 3\Md{0}
$$
and
$$
\left(\begin{array}{c|c}
\matr{X_{24}}\left( 2\Td{0}\Cd{1}+\Td{1}\left(\Cd{0}-3\matr{I_{24}}\right)\right) & 18\matr{X_{24}}\Cd{1}
\end{array}\right)
\equiv 6\Md{1} \pmod{8}.
$$
Therefore
$$
27\AIA{24}{3.2^u} \equiv 3\Md{0}+3.2^{u-2}\Md{1} \pmod{2^u}.
$$
Since $\gcd(3,2^{u+1})=1$, we conclude that
$$
9\AIA{24}{3.2^{u}} \equiv \Md{0}+2^{u-2}\Md{1}\pmod{2^{u}},
$$
for all $u\ge5$. This completes the proof.
\end{proof}

\begin{prop}\label{prop20}
Let $U_1,\ldots,U_{16}$ be the $24$-tuples of integers given in Table~\ref{tab3}. Then, the free $\Z$-module $\setG{5}$ is of rank $16$ and admits $\{U_1,\ldots,U_{16}\}$ as basis, i.e.,
$$
\setG{5} = \left\langle U_1,\ldots,U_{16}\right\rangle.
$$
\end{prop}

\begin{table}[htbp]
\begin{center}
\includegraphics[width=\textwidth]{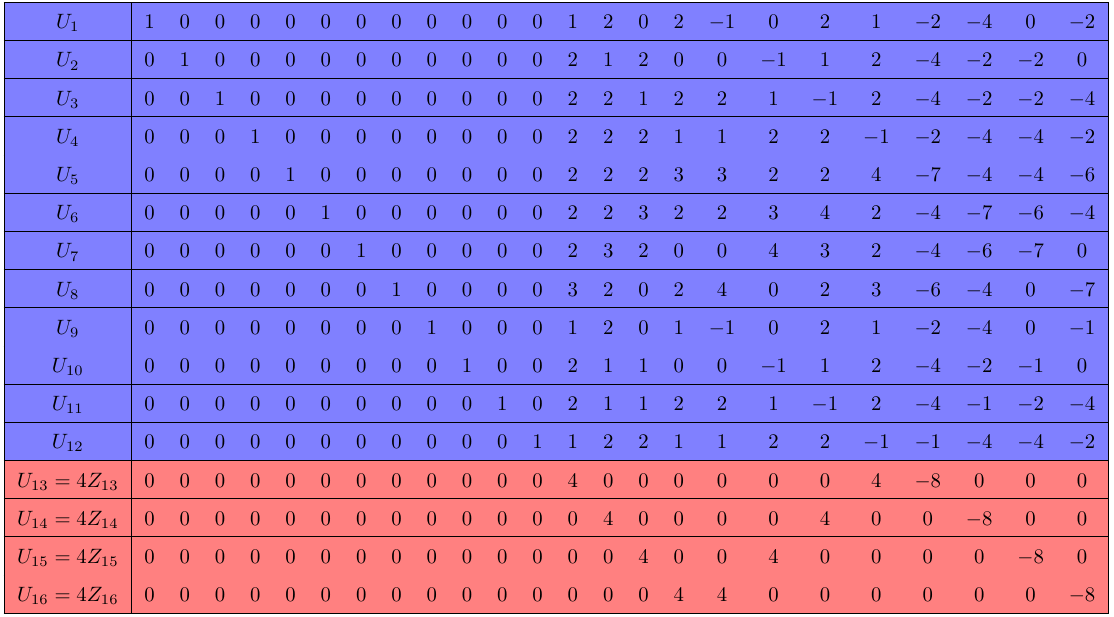}
\caption{The basis $\left\{U_1,\ldots,U_{16}\right\}$ of $\setG{5}$}\label{tab3}
\end{center}
\end{table}

The proof is based on the following two lemmas.

\begin{lem}\label{lem10}
$$
\setG{5} = \left\{ A\in\Lker\Md{0}\ \middle|\ A\Md{1}\equiv\matr{0}\pmod{4} \right\}.
$$
\end{lem}

\begin{proof}
From Proposition~\ref{prop21}, we know that
$$
9\AIA{24}{3.2^u} \equiv \Md{0} + 2^{u-2}\Md{1} \pmod{2^u},
$$
for all $u\ge5$. Moreover, since $\gcd(9,2^u)=1$, we have
$$
A\AIA{24}{3.2^u} \equiv \matr{0}\pmod{2^u}\ \Longleftrightarrow\ 9A\AIA{24}{3.2^u} \equiv \matr{0}\pmod{2^u},
$$
for all $24$-tuples of integers $A$. Therefore,
$$
\setG{5} = \left\{ A\in\Z^{24}\ \middle|\ A\Md{0}+2^{u-2}A\Md{1}\equiv \matr{0}\pmod{2^u},\ \forall u\ge5 \right\}.
$$
First, if we suppose that
$$
A\Md{0}+2^{u-2}A\Md{1}\equiv \matr{0}\pmod{2^u},
$$
for all $u\ge 5$, then
$$
A\Md{0}\equiv \matr{0}\pmod{2^{u-2}},
$$
for all $u\ge5$. This leads to $A\Md{0}=\matr{0}$ and thus $A\in\Lker\Md{0}$. If $A\in\Lker\Md{0}$, then,
$$
A\Md{0}+2^{u-2}A\Md{1}\equiv \matr{0}\pmod{2^u}\ \Longleftrightarrow\ A\Md{1}\equiv\matr{0}\pmod{4},
$$
for all $u\ge5$. Conversely, if we suppose that $A\in\Lker\Md{0}$ and $A\Md{1}\equiv\matr{0}\pmod{4}$, we obtain that
$$
A\Md{0}+2^{u-2}A\Md{1}\equiv \matr{0}\pmod{2^u},
$$
for all $u\ge 5$. Therefore,
$$
\setG{5} = \left\{ A\in\Lker\Md{0}\ \middle|\ A\Md{1}\equiv\matr{0}\pmod{4} \right\},
$$
as announced.
\end{proof}

\begin{lem}\label{lem9}
The left kernel of $\Md{0}$ is
$$
\Lker\Md{0} = \left\langle Z_1,\ldots,Z_{16}\right\rangle = \setF.
$$
\end{lem}

\begin{proof}
First, all the rows of $\Md{0}$ are linear combination of its last eight rows. Indeed, for all $i\in\{1,\ldots,8\}$, we have
\begin{equation}\label{eq1}
R_i(\Md{0}) = 3R_{16+i}(\Md{0})-2R_{25-i}(\Md{0}),
\end{equation}
and
\begin{equation}\label{eq2}
R_{8+i}(\Md{0}) = 2R_{16+i}(\Md{0})-R_{25-i}(\Md{0}).
\end{equation}
Therefore the rank of $\Md{0}$ is at most $8$. Moreover, since the submatrix obtained by removing the first $16$ rows and the last $40$ columns, i.e., the matrix
\begin{center}
\includegraphics{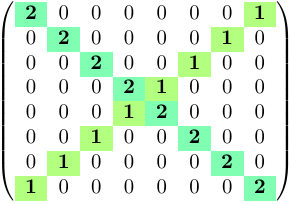}
\end{center}
is of rank $8$ in $\Z$, we obtain that $\Md{0}$ is of rank $8$. From the rank-nullity theorem, we deduce that
$$
\dim\Lker\Md{0} = 24-\rank\Md{0} = 16.
$$
Finally, from \eqref{eq1} and \eqref{eq2}, we deduce that,
$$
Z_i = E_i-3E_{16+i}+2E_{25-i} \in\Lker\Md{0}
$$
and
$$
Z_{8+i} = E_{8+i}-2E_{16+i}+E_{25-i} \in\Lker\Md{0},
$$
for all $i\in\{1,\ldots,8\}$. Since the family $\left\{Z_1,\ldots,Z_{16}\right\}$ is clearly linearly independent (just considering the $16$ first elements of each one of them) and since $\Lker\Md{0}$ is of dimension $16$, we conclude that $\left\{Z_1,\ldots,Z_{16}\right\}$ is a basis of $\Lker\Md{0}$. Therefore,
$$
\Lker\Md{0} = \left\langle Z_1,\ldots,Z_{16}\right\rangle = \setF,
$$
as announced.
\end{proof}

\begin{rem}
From Lemmas~\ref{lem10} and \ref{lem9}, it is straightforward that
$$
\setG{5}\subset\setF.
$$
\end{rem}

We are now ready to prove Proposition~\ref{prop20}.

\begin{proof}[Proof of Proposition~\ref{prop20}]
From Lemma~\ref{lem10}, we know that
\begin{equation}\label{eq4}
\setG{5} = \left\{ A\in\Lker\Md{0}\ \middle|\ A\Md{1}\equiv\matr{0}\pmod{4} \right\}.
\end{equation}
Let $\matZ$ be the $16\times 24$ integer matrix whose $i$th row is $Z_i$, for all $i\in\{1,\ldots,16\}$, i.e.,
\begin{center}
\includegraphics[width=0.75\textwidth]{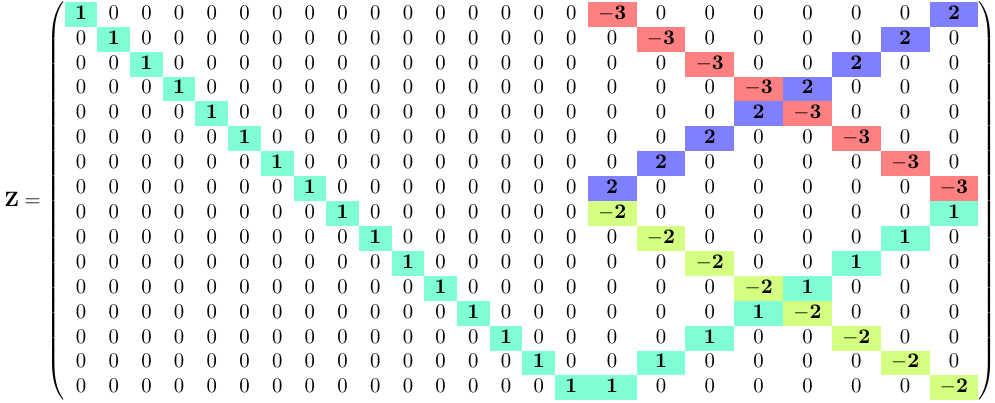}
\end{center}
Now, we consider the $\Zn{4}$-matrix $\Md{2}=(\matZ\Md{1}\bmod{4})$, i.e.,
\begin{center}
\includegraphics[width=\textwidth]{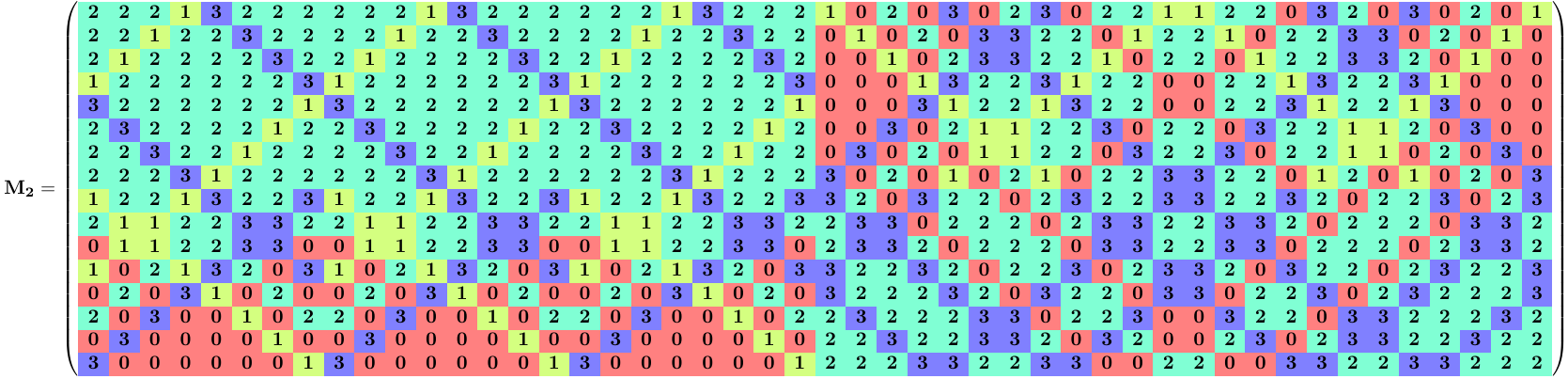}
\end{center}
First, all the rows of $\Md{2}$ are linear combination of its four last rows. Indeed, for all $i\in\{1,\ldots,12\}$, there exists $(\alpha_i,\beta_i,\gamma_i,\delta_i)\in\{0,1,2,3\}^4$ such that
\begin{equation}\label{eq3}
R_i(\Md{2}) + \alpha_i R_{13}(\Md{2}) + \beta_i R_{14}(\Md{2}) + \gamma_i R_{15}(\Md{2}) + \delta_i R_{16}(\Md{2}) = \matr{0}.
\end{equation}
The coefficients $(\alpha_i,\beta_i,\gamma_i,\delta_i)$ are given in Table~\ref{tab5}.
\begin{table}[htbp]
\begin{center}
\begin{tabular}{ccccc}
\begin{tabular}{|c||c|c|c|c|}
\hline
 $i$ & $\alpha_i$ & $\beta_i$ & $\gamma_i$ & $\delta_i$ \\
\hline
 $1$ & $1$ & $2$ & $0$ & $2$ \\
\hline
 $2$ & $2$ & $1$ & $2$ & $0$ \\
\hline
 $3$ & $2$ & $2$ & $1$ & $2$ \\
\hline
 $4$ & $2$ & $2$ & $2$ & $1$ \\
\hline
\end{tabular}
&
&
\begin{tabular}{|c||c|c|c|c|}
\hline
 $i$ & $\alpha_i$ & $\beta_i$ & $\gamma_i$ & $\delta_i$ \\
\hline
 $5$ & $2$ & $2$ & $2$ & $3$ \\
\hline
 $6$ & $2$ & $2$ & $3$ & $2$ \\
\hline
 $7$ & $2$ & $3$ & $2$ & $0$ \\
\hline
 $8$ & $3$ & $2$ & $0$ & $2$ \\
\hline
\end{tabular}
&
&
\begin{tabular}{|c||c|c|c|c|}
\hline
 $i$ & $\alpha_i$ & $\beta_i$ & $\gamma_i$ & $\delta_i$ \\
\hline
 $9$ & $1$ & $2$ & $0$ & $1$ \\
\hline
 $10$ & $2$ & $1$ & $1$ & $0$ \\
\hline
 $11$ & $2$ & $1$ & $1$ & $2$ \\
\hline
 $12$ & $1$ & $2$ & $2$ & $1$ \\
\hline
\end{tabular}
\end{tabular}
\caption{The coefficients $(\alpha_i,\beta_i,\gamma_i,\delta_i)$}\label{tab5}
\end{center}
\end{table}
From \eqref{eq3}, we deduce that
$$
U_i = Z_i+\alpha_i Z_{13}+\beta_i Z_{14}+\gamma_i Z_{15}+\delta_i Z_{16} \in\setG{5},
$$
for all $i\in\{1,\ldots,12\}$. Moreover, from \eqref{eq4}, we have
$$
U_i = 4Z_i\in\setG{5},
$$
for all $i\in\{13,14,15,16\}$. This leads to
$$
\left\langle U_1,\ldots,U_{16}\right\rangle \subset\setG{5}.
$$


Now, we prove that $\{U_1,\ldots,U_{16}\}$ is a generating set of $\setG{5}$. From \eqref{eq4}, we know that $\setG{5}\subset\setF$. Let
$$
A = \sum_{i=1}^{16}\alpha_i Z_i \in\setG{5},
$$
with $(\alpha_1,\ldots,\alpha_{16})\in\Z^{16}$. From the definitions of $Z_i$ and $U_i$, it is clear that
$$
A - \sum_{i=1}^{12}\alpha_i U_i = \sum_{i=13}^{16}\beta_i Z_i,
$$
where
$$
\begin{array}{l}
\beta_{13} = \alpha_{13}-\alpha_1-2\alpha_2-2\alpha_3-2\alpha_4-2\alpha_5-2\alpha_6-2\alpha_7-3\alpha_8-\alpha_9-2\alpha_{10}-2\alpha_{11}-\alpha_{12}, \\
\beta_{14} = \alpha_{14}-2\alpha_1-\alpha_2-2\alpha_3-2\alpha_4-2\alpha_5-2\alpha_6-3\alpha_7-2\alpha_8-2\alpha_9-\alpha_{10}-\alpha_{11}-2\alpha_{12}, \\
\beta_{15} = \alpha_{15}-2\alpha_2-\alpha_3-2\alpha_4-2\alpha_5-3\alpha_6-2\alpha_7-\alpha_{10}-\alpha_{11}-2\alpha_{12}, \\
\beta_{16} = \alpha_{16}-2\alpha_1-2\alpha_3-\alpha_4-3\alpha_5-2\alpha_6-2\alpha_8-\alpha_9-2\alpha_{11}-\alpha_{12}. \\
\end{array}
$$
Since $A,U_1,\ldots,U_{12}\in\setG{5}$, we deduce that $\sum_{i=13}^{16}\beta_i Z_i\in\setG{5}$. Therefore
$$
\left(\sum_{i=13}^{16}\beta_i Z_i\right)\Md{1} \equiv \matr{0} \pmod{4}
$$
and
$$
\sum_{i=13}^{16}\beta_i R_i(\Md{2}) \equiv \sum_{i=13}^{16}\beta_i Z_i\Md{1} \equiv \left(\sum_{i=13}^{16}\beta_i Z_i\right)\Md{1} \equiv \matr{0} \pmod{4}.
$$
If we consider the four first columns of $\Md{2}$, we obtain that
$$
\left\{\begin{array}{ll}
2\beta_{14} + 3\beta_{16} \equiv 0 & \pmod{4},\\
2\beta_{13} + 3\beta_{15} \equiv 0 & \pmod{4},\\
3\beta_{14}\equiv 0 & \pmod{4},\\
3\beta_{13}\equiv 0 & \pmod{4},\\
\end{array}\right.
\quad \Longleftrightarrow\quad \beta_i\equiv0\pmod{4},\ \forall i\in\{13,14,15,16\}.
$$
Therefore $\frac{\beta_i}{4}\in\Z$, for all $i\in\{13,14,15,16\}$, and
$$
A = \sum_{i=1}^{12}\alpha_i U_i + \sum_{i=13}^{16}\frac{\beta_i}{4}(4Z_i) = \sum_{i=1}^{12}\alpha_i U_i + \sum_{i=13}^{16}\frac{\beta_i}{4}U_i.
$$
This leads to $\setG{5}\subset\left\langle U_1,\ldots,U_{16}\right\rangle$ and thus,
$$
\setG{5} = \left\langle U_1,\ldots,U_{16}\right\rangle.
$$
Since $\{U_1,\ldots,U_{16}\}$ is a generating set of $\setG{5}$ and is clearly an independent set, by definitions of the $U_i$, we deduce that $\setG{5}$ is a free $\Z$-module of rank $16$ for which $\{U_1,\ldots,U_{16}\}$ is a basis.
\end{proof}

For obtaining $\setG{4}$, we need to consider the matrix $\AIA{24}{48}$ modulo $16$. By direct computation, we obtain that $\left(\AIA{24}{48}\bmod{16}\right)$ is equal to
\begin{center}
\includegraphics[width=\textwidth]{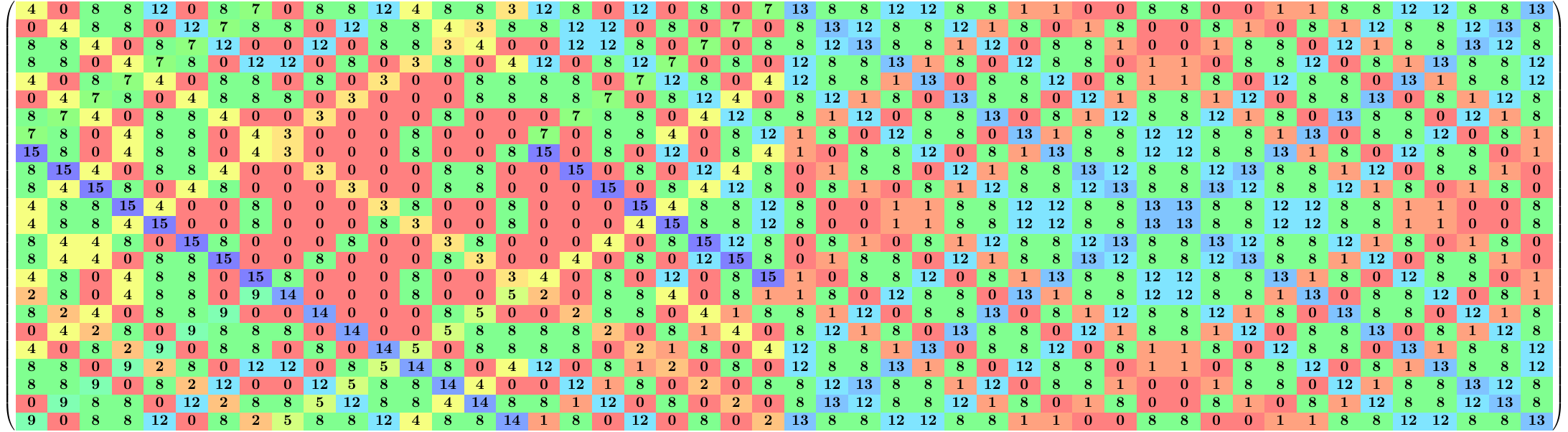}
\end{center}

\begin{prop}\label{prop22}
We have
$$
\setG{4} = \setG{5} = \displaystyle\left\langle U_1,\ldots,U_{16}\right\rangle.
$$
\end{prop}

\begin{proof}
By definition, we know that $\setG{4}\subset\setG{5}$ and
$$
\setG{4} = \left\{ A\in\setG{5}\ \middle|\ A\AIA{24}{48}\equiv\matr{0}\pmod{16} \right\}.
$$
Since, by direct computation, the $24$-tuples of integers $U_i$ verify that
$$
U_i\AIA{24}{48}\equiv\matr{0} \pmod{16},
$$
for all $i\in\{1,\ldots,16\}$, we obtain that
$$
\setG{5}\subset\setG{4}.
$$
This completes the proof.
\end{proof}

For obtaining $\setG{3}$, we need to consider the matrix $\AIA{24}{24}$ modulo $8$. By direct computation, we obtain that $\left(\AIA{24}{24}\bmod{8}\right)$ is equal to
\begin{center}
\includegraphics[width=\textwidth]{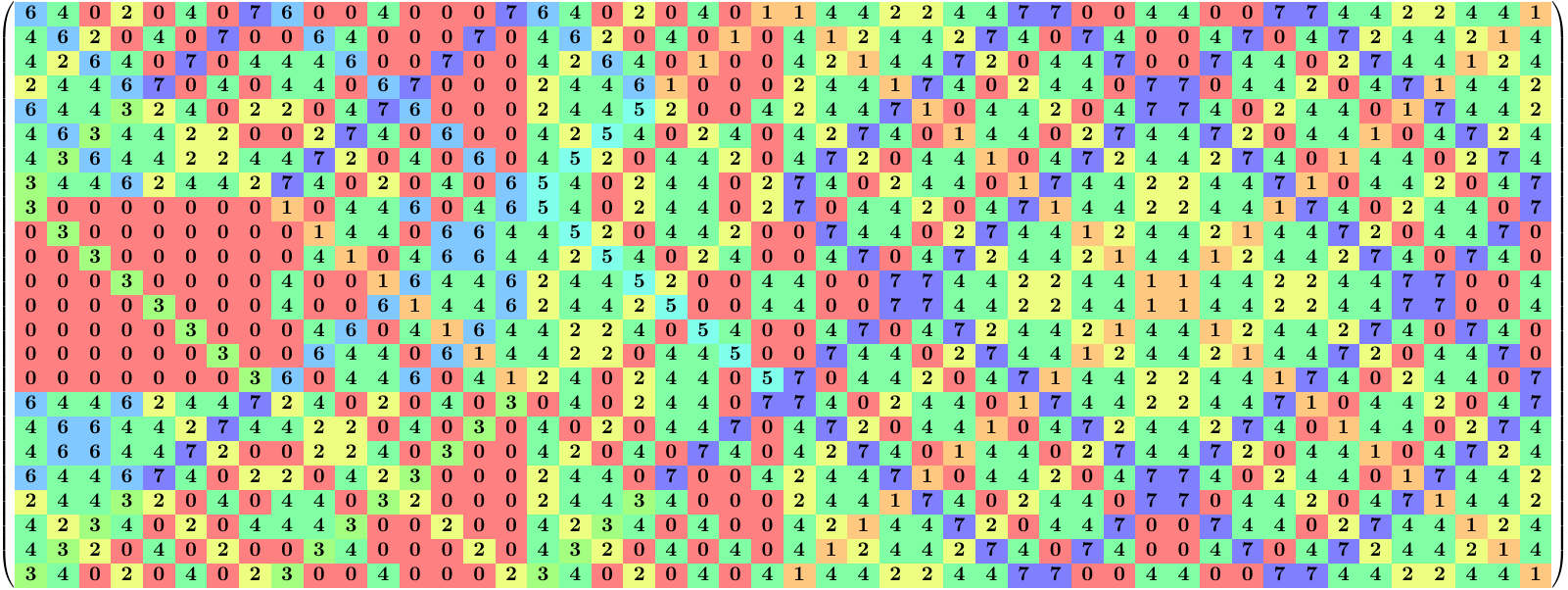}
\end{center}

\begin{prop}\label{prop24}
Let $V_1,\ldots,V_{16}$ be the $24$-tuples of integers given in Table~\ref{tab6}. Then, the free $\Z$-module $\setG{3}$ is of rank $16$ and admits $\{V_1,\ldots,V_{16}\}$ as basis, i.e.,
$$
\setG{3} = \left\langle V_1,\ldots,V_{16}\right\rangle.
$$
\end{prop}

\begin{table}[htbp]
\begin{center}
\includegraphics[width=\textwidth]{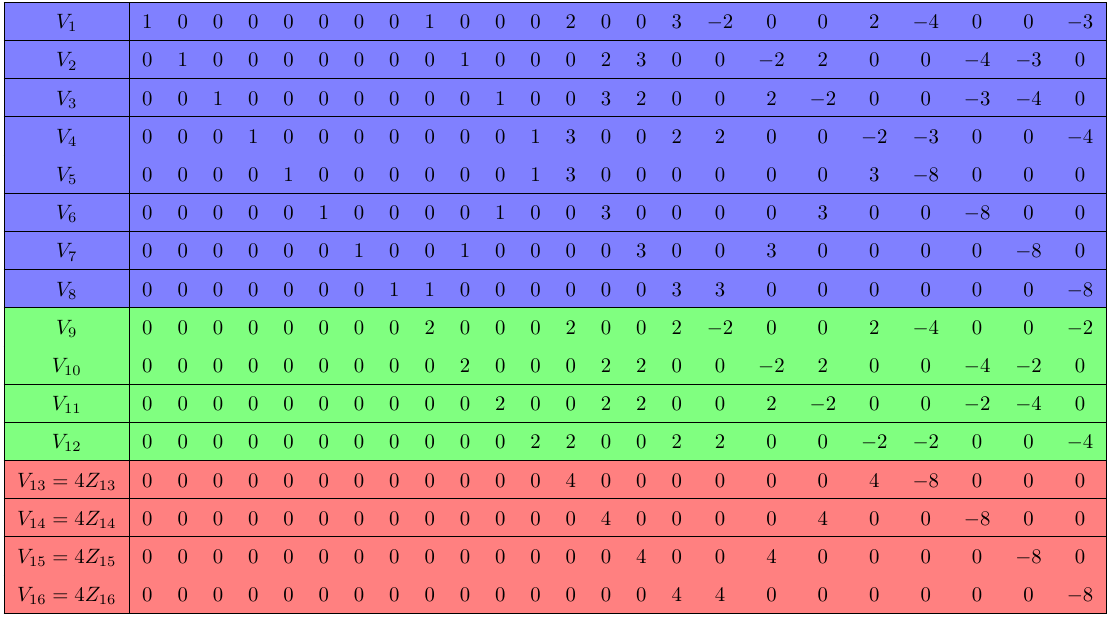}
\caption{The basis $\left\{V_1,\ldots,V_{16}\right\}$ of $\setG{3}$}\label{tab6}
\end{center}
\end{table}

\begin{proof}
By definition, we know that $\setG{3}\subset\setG{4}$ and
$$
\setG{3} = \left\{ A\in\setG{4}\ \middle|\ A\AIA{24}{24}\equiv\matr{0}\pmod{8} \right\}.
$$
Moreover, from Proposition~\ref{prop22}, we have
$$
\mathcal{G}_4 = \displaystyle\left\langle U_1,\ldots,U_{16}\right\rangle.
$$
Let $\matU$ be the $16\times 24$ integer matrix whose $i$th row is $U_i$, for all $i\in\{1,\ldots,16\}$, i.e.,
\begin{center}
\includegraphics[width=0.75\textwidth]{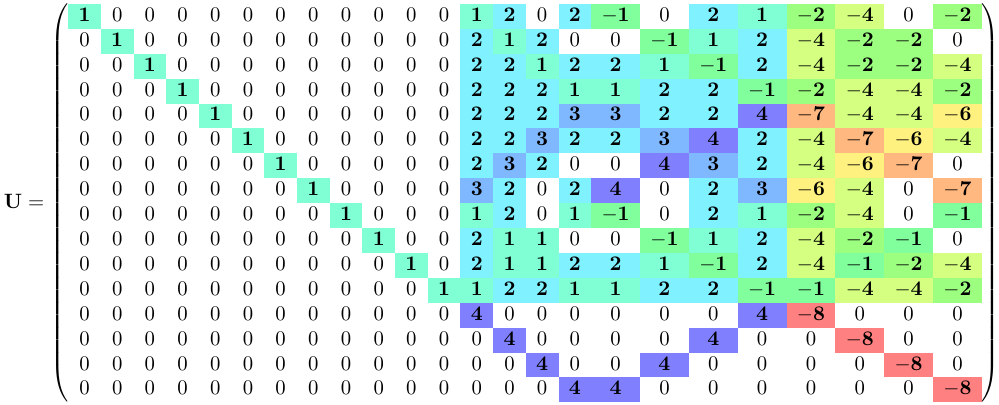}
\end{center}
Now, we consider the $\Zn{8}$-matrix $\Md{3}=(\matU\AIA{24}{24}\bmod{8})$, i.e.,
\begin{center}
\includegraphics[width=\textwidth]{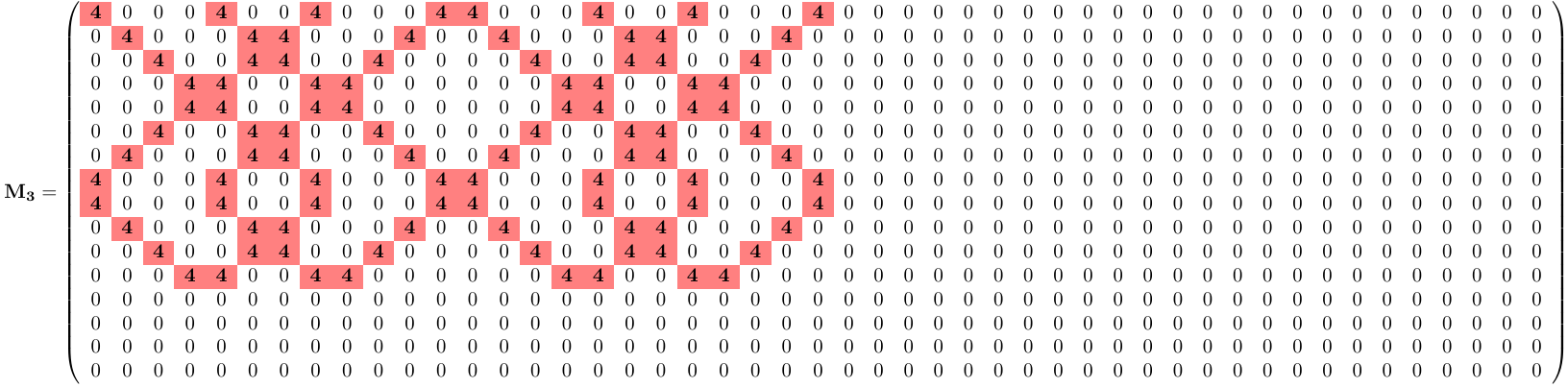}
\end{center}
First, it is clear that $R_i(\Md{3})=\matr{0}$, for all $i\in\{13,14,15,16\}$. Therefore,
$$
V_i = U_i = 4Z_i \in\setG{3},\ \forall i\in\{13,14,15,16\}.
$$
Moreover, it is clear that
$$
R_i(\Md{3}) = R_{9-i}(\Md{3}) = R_{8+i}(\Md{3}),
$$
for all $i\in\{1,2,3,4\}$, and thus
$$
R_i(\Md{3})+R_{8+i}(\Md{3}) = R_{4+i}(\Md{3})+R_{13-i}(\Md{3}) = \matr{0},
$$
for all $i\in\{1,2,3,4\}$. We deduce that
$$
V_i' = U_i+U_{8+i} \in\setG{3}\quad\text{and}\quad V_{i+4}' = U_{i+4}+U_{13-i} \in\setG{3},
$$
for all $i\in\{1,2,3,4\}$. The $24$-tuples of integers $V_1',\ldots,V_{8}'$ are depicted in Table~\ref{tab7}.
\begin{table}[htbp]
\begin{center}
\includegraphics[width=\textwidth]{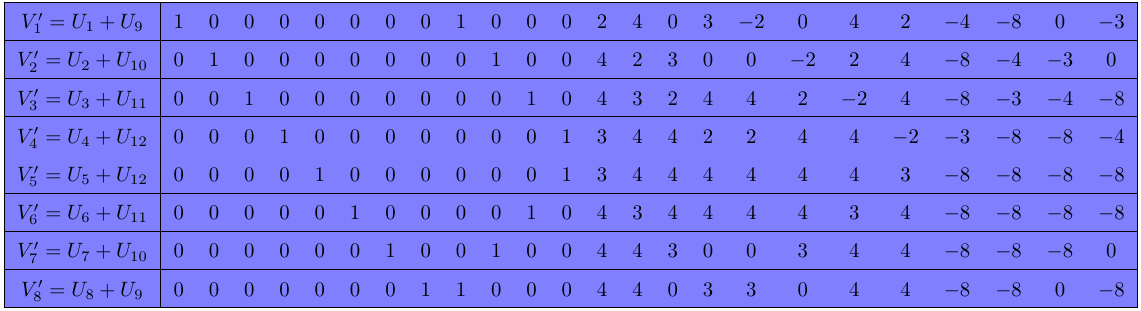}
\caption{The $24$-tuples $\left\{V_1',\ldots,V_{8}'\right\}$}\label{tab7}
\end{center}
\end{table}
Since
$$
\begin{array}{l}
V_1 = V_1'-V_{14},\ V_2 = V_2'-V_{13},\ V_3 = V_3'-V_{13}-V_{16},\ V_4 = V_4'-V_{14}-V_{15}, \\
V_5=V_5'-V_{14}-V_{15}-V_{16},\ V_6=V_6'-V_{13}-V_{15}-V_{16},\ V_7=V_7'-V_{13}-V_{14}, \\
V_8=V_8'-V_{13}-V_{14}, \\
\end{array}
$$
we deduce that $V_i\in\setG{3}$, for all $i\in\{1,\ldots,8\}$. It is clear that $2R_i(\Md3)=\matr{0}$ and thus,
$$
2U_i \in\setG{3},
$$
for all $i\in\{9,10,11,12\}$. Since
\begin{equation}\label{eq5}
V_9 = 2U_9-V_{14},\ V_{10} = 2U_{10}-V_{13},\ V_{11} = 2U_{11}-V_{13}-V_{16},\ V_{12} = 2U_{12}-V_{14}-V_{15},
\end{equation}
we obtain that $V_i\in\setG{3}$, for all $i\in\{9,10,11,12\}$. Therefore, we have $V_i\in\setG{3}$, for all $i\in\{1,\ldots,16\}$, and
$$
\left\langle V_1,\ldots,V_{16}\right\rangle \subset\setG{3}.
$$

Now, we prove that $\{V_1,\ldots,V_{16}\}$ is a generating set of $\setG{3}$. First, we know that $\setG{3}\subset\setG{4}$. Let
$$
A = \sum_{i=1}^{16}\alpha_i U_i \in\setG{3},
$$
with $(\alpha_1,\ldots,\alpha_{16})\in\Z^{16}$. From the definitions of $U_i$ and $V_i$, it is clear that
$$
A - \sum_{i=1}^{8}\alpha_i V_i = \sum_{i=9}^{16}\beta_i U_i,
$$
where
$$
\beta_i = \alpha_i-\alpha_{i-8}-\alpha_{17-i},
$$
for all $i\in\{9,10,11,12\}$, and
$$
\begin{array}{l}
\beta_{13} = \alpha_{13}+\alpha_2+\alpha_3+\alpha_6+\alpha_7+\alpha_8,\\
\beta_{14} = \alpha_{14}+\alpha_1+\alpha_4+\alpha_5+\alpha_7+\alpha_8,\\
\beta_{15} = \alpha_{15}+\alpha_4+\alpha_5+\alpha_6,\\
\beta_{16} = \alpha_{16}+\alpha_3+\alpha_5+\alpha_6.
\end{array}
$$
Since $V_i=U_i$, for all $i\in\{13,14,15,16\}$, we have
$$
A - \sum_{i=1}^{8}\alpha_i V_i - \sum_{i=13}^{16}\beta_iV_i = \sum_{i=9}^{12}\beta_i U_i,
$$
Since $A,V_1,\ldots,V_{8},V_{13},\ldots,V_{16}\in\setG{3}$, we deduce that $\sum_{i=9}^{12}\beta_i U_i\in\setG{3}$. Therefore
$$
\left(\sum_{i=9}^{12}\beta_i U_i\right)\AIA{24}{24} \equiv \matr{0} \pmod{8}
$$
and
$$
\sum_{i=9}^{12}\beta_i R_i(\Md{3}) \equiv \sum_{i=9}^{12}\beta_i U_i\AIA{24}{24} \equiv \left(\sum_{i=9}^{12}\beta_i U_i\right)\AIA{24}{24} \equiv \matr{0} \pmod{8}.
$$
If we consider the four fist columns of $\Md{3}$, we obtain that
$$
4\beta_i\equiv0\pmod{8},\ \forall i\in\{9,10,11,12\},\quad \Longleftrightarrow\quad \beta_i\equiv0\pmod{2},\ \forall i\in\{9,10,11,12\}.
$$
Therefore $\frac{\beta_i}{2}\in\Z$, for all $i\in\{9,10,11,12\}$, and
$$
A = \sum_{i=1}^{8}\alpha_i V_i + \sum_{i=9}^{12}\frac{\beta_i}{2}(2U_i) + \sum_{i=13}^{16}\beta_i V_i.
$$
Moreover, from \eqref{eq5}, we have
$$
A = \sum_{i=1}^{8}\alpha_i V_i + \sum_{i=9}^{12}\frac{\beta_i}{2}V_i + \sum_{i=13}^{16}\gamma_i V_i,
$$
where
$$
\gamma_{13}=\beta_{13}+\frac{\beta_{10}}{2}+\frac{\beta_{11}}{2},\ \gamma_{14}=\beta_{14}+\frac{\beta_{9}}{2}+\frac{\beta_{12}}{2},\ \gamma_{15}=\beta_{15}+\frac{\beta_{12}}{2},\ \gamma_{16}=\beta_{16}+\frac{\beta_{11}}{2}.
$$
This leads to $\setG{3}\subset\left\langle V_1,\ldots,V_{16}\right\rangle$ and thus,
$$
\setG{3} = \left\langle V_1,\ldots,V_{16}\right\rangle.
$$
Since $\{V_1,\ldots,V_{16}\}$ is a generating set of $\setG{3}$ and is clearly an independent set, by definitions of the $V_i$, we deduce that $\setG{3}$ is a free $\Z$-module of rank $16$ for which $\{V_1,\ldots,V_{16}\}$ is a basis.
\end{proof}

We are now ready to show the main result of this section, that for any $A\in\setE$, the orbit of the sequence
$$
S = \IAP{\pi_{2^u}(A)}{\pi_{2^u}(A)\matX{k}}
$$
is $(24,3.2^u)$-interlaced doubly arithmetic in $\Zn{2^u}$, for all integers $u\ge3$, that is the following

\begin{thm}\label{prop23}
$$
\mathcal{E} \subset \setG{3} \subset \setF.
$$
\end{thm}

The proof is based on the following

\begin{lem}\label{lem11}
$$
4\setF \subset \setG{3} \subset \setF.
$$
\end{lem}

\begin{proof}
First, from Lemmas~\ref{lem10} and \ref{lem9}, it is clear that
$$
\setG{3}\subset\setG{5}\subset\setF.
$$
Moreover, since, for all $i\in\{1,2,3,4\}$, we have
$$
\begin{array}{l}
4Z_i = 4V_i-2V_{8+i}-V_{12+i}-2V_{17-i}, \\
4Z_{4+i} = 4V_{4+i}-2V_{13-i}-2V_{12+i}+V_{17-i}, \\
4Z_{8+i} = 2V_{8+i}-V_{12+i}-V_{17-i}, \\
4Z_{12+i} = V_{12+i},
\end{array}
$$
we deduce from Proposition~\ref{prop24} that $4Z_i\in\setG{3}$, for all $i\in\{1,\ldots,16\}$, and thus
$$
4\setF = \left\langle 4Z_1,\ldots,4Z_{16}\right\rangle \subset \setG{3},
$$
as announced.
\end{proof}

\begin{proof}[Proof of Proposition~\ref{prop23}]
It is clear that $\setG{3}\subset\setF$. For $\setE\subset\setG{3}$, since $4\setF\subset\setG{3}$ from Lemma~\ref{lem11}, it follows that
$$
\setM\subset 4\setF\subset\setG{3}.
$$
Therefore, since $\setE=\setX+\setM$, we have

$$
\setE\subset \setG{3}\quad\Longleftrightarrow\quad \setX\subset\setG{3}.
$$
Moreover, by Proposition~\ref{prop24}, since
$$
-X_i\in\setG{3}\quad\Longleftrightarrow\quad X_i\in\setG{3},
$$
for all $i\in\{1,\ldots,7\}$, it is sufficient to prove that
$$
X_i\in\setG{3},\ \forall i\in\{1,\ldots,7\}.
$$
Finally, since $\setG{3} = \left\langle V_1,\ldots,V_{16}\right\rangle$ by Proposition~\ref{prop24} and
$$
\begin{array}{l}
X_1 = V_{3}+V_{4}+2V_{5}+3V_{6}+2V_{7}+2V_{8}+V_{9}-2V_{11}-3V_{13}-3V_{14}-V_{15}-3V_{16},\\
X_2 = V_{3}+3V_{4}+3V_{6}+2V_{7}+V_{9}-2V_{11}-2V_{12}-V_{13}-3V_{14}-V_{15}-V_{16},\\
X_3 = 3V_{3}+3V_{4}+V_{6}+V_{9}-4V_{11}-2V_{12}-V_{13}-V_{14}+V_{15}-V_{16},\\
X_4 = 2V_{2}+V_{3}+3V_{4}+V_{6}+V_{9}-2V_{10}-2V_{11}-2V_{12}-V_{13}-V_{14}+V_{15}-V_{16},\\
X_5 = 2V_{1}+V_{3}+V_{4}+3V_{6}+2V_{7}-V_{9}-2V_{11}-V_{13}-3V_{14}-V_{15}-V_{16},\\
X_6 = 2V_{1}+3V_{3}+3V_{4}+2V_{5}+V_{6}+2V_{8}-V_{9}-4V_{11}-2V_{12}-3V_{13}-V_{14}+V_{15}-3V_{16},\\
X_7 = 2V_{1}+2V_{2}+3V_{3}+V_{4}+3V_{6}+2V_{7}-V_{9}-2V_{10}-4V_{11}-V_{13}-3V_{14}-V_{15}-V_{16},\\
\end{array}
$$
the result follows.
\end{proof}

For any $A\in\setE$, we know from Theorem~\ref{prop23} that the orbit of the sequence
$$
S=\IAP{\pi_{2^u}(A)}{\pi_{2^u}(A)\matX{24}}
$$
is $(24,3.2^u)$-interlaced doubly arithmetic in $\Zn{2^u}$, for all integers $u\ge3$. In the last part of this section,  the common differences, that are the $(3.2^u\times24)$-matrices $\matD{1}$ and $\matD{2}$ of $\Zn{2^u}$, in
$$
\orb{S} = \IDAP{\matA}{\matD{1}}{\matD{2}}
$$
are determined, for all integers $u\ge3$.

\begin{nota}[$\matr{\mathrm{\Delta}}$ matrices]
Let $k_1$ and $k_2$ be two positive integers. For any $A\in\left(\Zn{m}\right)^{k_1}$, let $\Dmat{k_2}{A}$ be the $(k_2\times k_1)$-matrix of $\Zn{m}$ defined by
$$
R_i\left(\Dmat{k_2}{A}\right) = (-1)^{i-1}A\C{k_1}{i-1},
$$
for all $i\in\{1,\ldots,k_2\}$.
\end{nota}

\begin{rem}
Since $\C{k_1}{i}=\C{k_1}{i-1}\C{k_1}{1}$ from Proposition~\ref{prop*2}, the rows of $\Dmat{k_2}{A}$ are recursively defined as
$$
R_{i+1}\left(\Dmat{k_2}{A}\right) = -R_i\left(\Dmat{k_2}{A}\right)\C{k_1}{1},
$$
for all $i\in\{1,\ldots,k_2-1\}$, with $R_1\left(\Dmat{k_2}{A}\right)=A$.
\end{rem}

For instance, for the $3$-tupe $A=112$ of $\Zn{3}$, we obtain
$$
\Dmat{4}{A} =
\begin{pmatrix}
1 & 1 & 2 \\
1 & 0 & 0 \\
2 & 0 & 2 \\
1 & 1 & 2 \\
\end{pmatrix}
$$
since, for $\C{3}{1}=\begin{pmatrix}
1 & 0 & 1 \\
1 & 1 & 0 \\
0 & 1 & 1 \\
\end{pmatrix}$, we have
$$
\begin{array}{l}
-\begin{pmatrix} 1 & 1 & 2 \end{pmatrix}\C{3}{1} = \begin{pmatrix} 1 & 0 & 0 \end{pmatrix}, \\[2ex]
-\begin{pmatrix} 1 & 0 & 0 \end{pmatrix}\C{3}{1} = \begin{pmatrix} 2 & 0 & 2 \end{pmatrix}, \\[2ex]
-\begin{pmatrix} 2 & 0 & 2 \end{pmatrix}\C{3}{1} = \begin{pmatrix} 1 & 1 & 2 \end{pmatrix}. \\
\end{array}
$$

Using this notation, it is straightforward from Corollary~\ref{cor*1} to obtain the following

\begin{prop}\label{prop*6}
Let $A$ be a $k_1$-tuple of $\Zn{m}$. If the orbit of $S=\IAP{A}{A\matX{k_1}}$ is $(k_1,k_2)$-interlaced doubly arithmetic, then we have $\orb{S} = \IDAP{\matA}{\matD{1}}{\matD{2}}$, where the matrices $\matD{1}$ and $\matD{2}$ satisfy
$$
\matD{1} = \Dmat{k_2}{A\matX{k_1}}\quad\text{and}\quad \matD{2} = \Dmat{k_2}{{(-1)}^{k_2}A\M{k_1}{k_2}}.
$$
\end{prop}

\begin{prop}\label{prop*4}
Let $i_0\in\{1,\ldots,7\}$ and let $(\alpha_1,\ldots,\alpha_{16})\in\Z^{16}$. Then, for the $24$-tuple of integers
$$
A = X_{i_0} + 4\sum_{i=1}^{8}\alpha_i Y_i + 8\sum_{i=9}^{16}\alpha_i Z_i \in\setE,
$$
the orbit of the sequence
$$
S=\IAP{\pi_{2^u}(A)}{\pi_{2^u}(A)\matX{24}}
$$
is $(24,3.2^u)$-interlaced doubly arithmetic in $\Zn{2^u}$, with common differences $\matD{1}$ and $\matD{2}$ verifying
$$
\matD{1} = \Dmat{3.2^u}{\pi_{2^u}\left((2,-4,2)^8+8\sum_{i=9}^{16}\alpha_iZ_i\matX{24}\right)}
$$
and
$$
\matD{2} = \Dmat{3.2^u}{2^{u-2}(3,3,2)^8},
$$
for all positive integers $u\ge3$.
\end{prop}

\begin{proof}
By Theorem~\ref{prop23}, the orbit $\orb{S}$ is $(24,3.2^u)$-interlaced doubly arithmetic and by Proposition~\ref{prop*6}, we know that $\orb{S} = \IDAP{\matA}{\matD{1}}{\matD{2}}$, where the matrices $\matD{1}$ and $\matD{2}$ satisfy
$$
\matD{1} = \Dmat{3.2^u}{\pi_{2^u}(A\matX{24})}\quad\text{and}\quad \matD{2} = \Dmat{3.2^u}{\pi_{2^u}(A\M{24}{3.2^u})}.
$$
First, the common differences of $S$ are
$$
A\matr{X_{24}} = X_{i_0}\matX{24} + 4\sum_{i=1}^{8}\alpha_iY_i\matX{24} + 8\sum_{i=9}^{16}\alpha_iZ_i\matX{24}.
$$
It is easy to verify that
$$
X_{i_0}\matX{24} = (2,-4,2)^8,
$$
for all $i_0\in\{1,\ldots,7\}$, and
$$
Y_i\matX{24} = \matr{0},
$$
for all $i\in\{1,\ldots,8\}$. Therefore,
$$
A\matX{24} = (2,-4,2)^8 + 8\sum_{i=9}^{16}\alpha_iZ_i\matX{24}
$$
and
$$
\matD{1} = \Dmat{3.2^u}{\pi_{2^u}(A\matX{24})} = \Dmat{3.2^u}{\pi_{2^u}\left((2,-4,2)^8+8\sum_{i=9}^{16}\alpha_iZ_i\matX{24}\right)},
$$
for all positive integers $u\ge3$.

Moreover, from Proposition~\ref{prop*5}, we have that
$$
9\M{24}{3.2^u}\equiv\Nd{0}+2^{u-3}\Nd{1}\pmod{2^u},
$$
for all integers $u\ge5$. Since $\setE\subset\setF=\Lker\Nd{0}$ by Proposition~\ref{prop*3} and Lemma~\ref{lem*8}, we obtain that
$$
A\Nd{0} = \matr{0}.
$$
Since
$$
X_{i_0}\Nd{1} \equiv (6,6,4)^8 \pmod{8},
$$
for all $i_0\in\{1,\ldots,7\}$, and
$$
Y_i\Nd{1} \equiv \matr{0} \pmod{2},
$$
for all $i\in\{1,\ldots,8\}$,
we have that
$$
A\left(\Nd{0}+2^{u-3}\Nd{1}\right) \equiv 2^{u-3}(6,6,4)^8 \pmod{2^u}.
$$
Therefore,
$$
9A\M{24}{3.2^u} \equiv 2^{u-2}(3,3,2)^8 \equiv 9.2^{u-2}(3,3,2)^8 \pmod{2^{u}}.
$$
Since $\gcd(9,2^u)=1$, we conclude that
$$
A\M{24}{3.2^u} \equiv 2^{u-2}(3,3,2)^8 \pmod{2^{u}}
$$
and thus,
$$
\matD{2} = \Dmat{3.2^u}{2^{u-2}(3,3,2)^8},
$$
for all integers $u\ge 5$. For $u=4$, we consider the matrix $\M{24}{48}$ modulo $16$. By direct computation, we obtain that $\left(\M{24}{48}\bmod{16}\right)$ is equal to
\begin{center}
\includegraphics[width=0.75\textwidth]{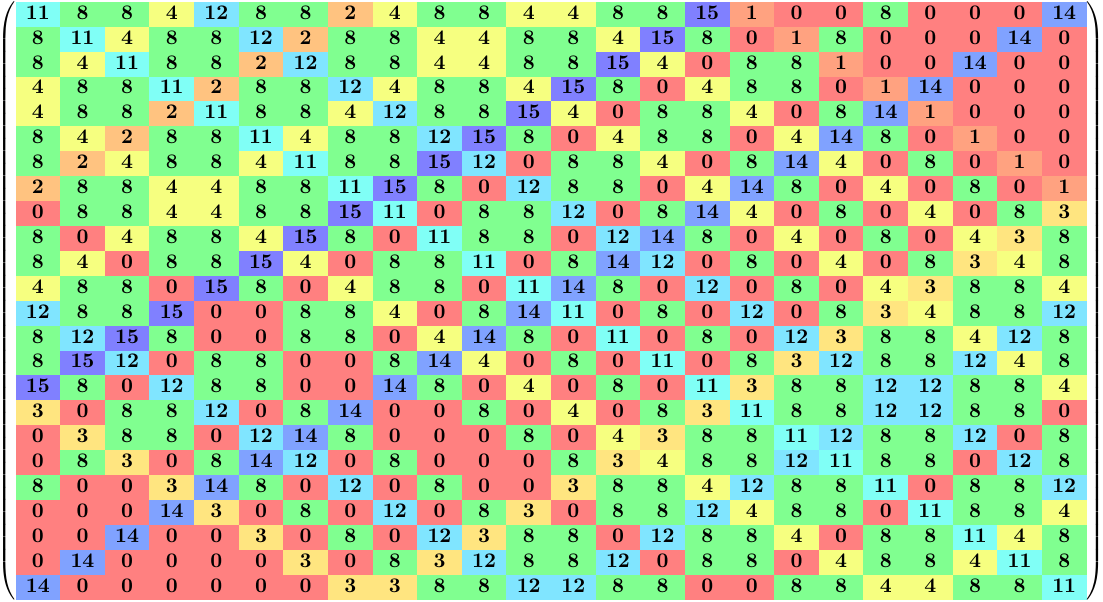}
\end{center}
Since
$$
X_{i_0}\M{24}{48} \equiv (12,12,8)^8 \pmod{16},
$$
for all $i_0\in\{1,\ldots,7\}$,
$$
Y_{i}\M{24}{48} \equiv \matr{0} \pmod{4},
$$
for all $i\in\{1,\ldots,8\}$, and
$$
Z_{i}\M{24}{48} \equiv \matr{0} \pmod{2},
$$
for all $i\in\{1,\ldots,16\}$, we deduce that
$$
A\M{24}{48} \equiv 4(3,3,2)^8 \pmod{16},
$$
and thus,
$$
\matD{2} = \Dmat{48}{4(3,3,2)^8},
$$
in this case. Finally, for $u=3$, we consider the matrix $\M{24}{24}$ modulo $8$. By direct computation, we obtain that $\left(\M{24}{24}\bmod{8}\right)$ is equal to
\begin{center}
\includegraphics[width=0.65\textwidth]{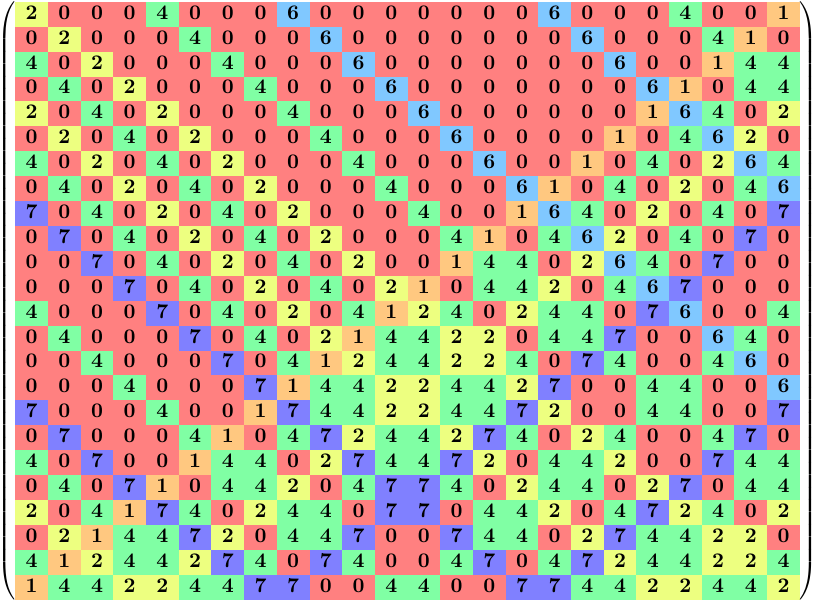}
\end{center}
Since
$$
X_{i_0}\M{24}{24} \equiv (6,6,4)^8 \pmod{8},
$$
for all $i_0\in\{1,\ldots,7\}$, and
$$
Y_{i}\M{24}{24} \equiv \matr{0} \pmod{2},
$$
for all $i\in\{1,\ldots,8\}$, we deduce that
$$
A\M{24}{24} \equiv 2(3,3,2)^8 \pmod{8},
$$
and thus,
$$
\matD{2} = \Dmat{24}{2(3,3,2)^8},
$$
in this case. This completes the proof.
\end{proof}

In the end of this section, we show that a $(k_1,k_2)$-interlaced doubly arithmetic progression can also be seen as a $(\lambda k_1,\mu k_2)$-interlaced arithmetic progression and the associated common differences are determined. Using this, for any $A\in\setE$, the common differences of the $(3.2^u,3.2^u)$-interlaced doubly arithmetic orbit of the sequence $\IAP{\pi_{2^u}(A)}{\pi_{2^u}(A)\matX{24}}$ are determined, for all integers $u\ge3$. 

\begin{prop}\label{prop25}
Let $k_1$ and $k_2$ be two positive integers. Let $\matA$, $\matD{1}$ and $\matD{2}$ be three $(k_2\times k_1)$-matrices of elements in $\Zn{m}$. For any positive integers $\lambda$ and $\mu$, the interlaced doubly arithmetic progression $\IDAP{\matA}{\matD{1}}{\matD{2}}$ is also $(\lambda k_1,\mu k_2)$-interlaced doubly arithmetic. There exist three $(\mu k_2\times \lambda k_1)$-matrices $\matAp$, $\matDp{1}$ and $\matDp{2}$ such that
$$
\IDAP{\matA}{\matD{1}}{\matD{2}} = \IDAP{\matAp}{\matDp{1}}{\matDp{2}},
$$
where
$$
R_{ik_2+j}(\matDp{1}) = \lambda\left(R_j(\matD{1})\right)^{\lambda}
$$
and
$$
R_{ik_2+j}(\matDp{2}) = \mu\left(R_j(\matD{2})\right)^{\lambda}
$$
for all $i\in\{0,\ldots,\mu-1\}$ and $j\in\{0,\ldots,k_2-1\}$.
\end{prop}

\begin{proof}
Let $S=\IDAP{\matA}{\matD{1}}{\matD{2}} = \left(u_{i,j}\right)_{(i,j)\in\N\times\Z}$. Since $S$ is $(k_1,k_2)$-interlaced doubly arithmetic, we know that
$$
u_{i_0+ik_2,j_0+jk_1} = u_{i_0,j_0} + i\left(u_{i_0+k_2,j_0}-u_{i_0,j_0}\right) + j\left(u_{i_0,j_0+k_1}-u_{i_0,j_0}\right),
$$
for all $(i_0,j_0)\in\N\times\Z$ and $(i,j)\in\N\times\Z$. Then,
$$
u_{i_0+i\mu k_2,j_0+j\lambda k_1}
\begin{array}[t]{l}
 = u_{i_0,j_0} + i\mu\left(u_{i_0+k_2,j_0}-u_{i_0,j_0}\right) + j\lambda\left(u_{i_0,j_0+k_1}-u_{i_0,j_0}\right) \\[2ex]
 = u_{i_0,j_0} + i\left(u_{i_0+\mu k_2,j_0}-u_{i_0,j_0}\right) + j\left(u_{i_0,j_0+\lambda k_1}-u_{i_0,j_0}\right),
\end{array}
$$
for all $(i_0,j_0)\in\N\times\Z$ and $(i,j)\in\N\times\Z$. Therefore, $S$ is also $(\lambda k_1,\mu k_2)$-interlaced doubly arithmetic.

Let $i\in\{0,\ldots,\mu-1\}$. For all $j\in\{0,\ldots,k_2-1\}$, we have
$$
R_{ik_2+j}(\matDp{1})
\begin{array}[t]{l}
 = \left(u_{ik_2+j,r+\lambda k_1}-u_{ik_2+j,r}\right)_{0\le r\le \lambda k_1-1} = \left(\lambda\left(u_{ik_2+j,r+k_1}-u_{ik_2+j,r}\right)\right)_{0\le r\le \lambda k_1-1} \\[2ex]
 = \lambda\left(u_{ik_2+j,r+k_1}-u_{ik_2+j,r}\right)_{0\le r\le \lambda k_1-1} = \lambda\left(u_{j,r+k_1}-u_{j,r}\right)_{0\le r\le \lambda k_1-1} \\[2ex]
 = \lambda\left(R_j(\matD{1})\right)^{\lambda},
\end{array}
$$
and
$$
R_{ik_2+j}(\matDp{2})
\begin{array}[t]{l}
 = \left(u_{(\mu+i)k_2+j,r}-u_{ik_2+j,r}\right)_{0\le r\le \lambda k_1-1} = \left(\mu\left(u_{(i+1)k_2+j,r}-u_{ik_2+j,r}\right)\right)_{0\le r\le \lambda k_1-1}  \\[2ex]
 = \mu\left(u_{(i+1)k_2+j,r}-u_{ik_2+j,r}\right)_{0\le r\le \lambda k_1-1} = \mu\left(u_{k_2+j,r}-u_{j,r}\right)_{0\le r\le \lambda k_1-1} \\[2ex]
 = \mu\left(R_j(\matD{2})\right)^{\lambda}.
\end{array}
$$
This completes the proof.
\end{proof}

\begin{prop}\label{prop26}
Let $k_1$ and $k_2$ be two positive integers. Let $A$ and $D$ be two $k_1$-tuples of $\Zn{m}$ such that the orbit $\orb{S}$ of $S=\IAP{A}{D}$ is $(k_1,k_2)$-interlaced doubly arithmetic. Let $\matA$, $\matD{1}$ and $\matD{2}$ be the three $(k_2\times k_1)$-matrices of $\Zn{m}$ such that
$$
\orb{S} = \IDAP{\matA}{\matD{1}}{\matD{2}},
$$
with
$$
\matD{1} = \Dmat{k_2}{D}\quad\text{and}\quad \matD{2} = (-1)^{k_2}\Dmat{k_2}{A\W{k_1}{k_2}+D\T{k_1}{k_2}}.
$$
Then, for any positive integers $\lambda$ and $\mu$, the orbit $\orb{S}$ is also $(\lambda k_1,\mu k_2)$-interlaced doubly arithmetic. There exist three $(\mu k_2\times \lambda k_1)$-matrices $\matAp$, $\matDp{1}$ and $\matDp{2}$ such that
$$
\IDAP{\matA}{\matD{1}}{\matD{2}} = \IDAP{\matAp}{\matDp{1}}{\matDp{2}},
$$
where
$$
\matDp{1} = \lambda\Dmat{\mu k_2}{D^{\lambda}}
$$
and
$$
\matDp{2} = (-1)^{k_2}\mu\Dmat{\mu k_2}{\left(A\W{k_1}{k_2}+D\T{k_1}{k_2}\right)^{\lambda}}.
$$
\end{prop}

\begin{proof}
From Theorem~\ref{thm4} and Proposition~\ref{prop25}.
\end{proof}

\begin{prop}\label{prop27}
Let $i_0\in\{1,\ldots,7\}$ and let $(\alpha_1,\ldots,\alpha_{16})\in\Z^{16}$. Then, for the $24$-tuple of integers
$$
A = X_{i_0} + 4\sum_{i=1}^{8}\alpha_i Y_i + 8\sum_{i=9}^{16}\alpha_i Z_i \in\setE,
$$
the orbit of the sequence
$$
S=\IAP{\pi_{2^u}(A)}{\pi_{2^u}(A)\matX{24}}
$$
is $(3.2^u,3.2^u)$-interlaced doubly arithmetic in $\Zn{2^u}$, with common differences $\matD{1}$ and $\matD{2}$ defined by
$$
\matD{1} = 2^{u-2}\Circ{(1,2,1)^{2^u}}
$$
and
$$
\matD{2} = 2^{u-2}\Circ{(3,3,2)^{2^u}},
$$
for all positive integers $u\ge3$.
\end{prop}

The proof is based on the following


\begin{lem}\label{lem*9}
Let $\alpha$ and $\beta$ be in $\Zn{m}$. Then, we have
$$
\Dmat{3\lambda}{(\alpha,\beta,-\alpha-\beta)^{\lambda}} = \Circ{(\alpha,\beta,-\alpha-\beta)^{\lambda}},
$$
for all positive integers $\lambda$.
\end{lem}

\begin{proof}
Since
$$
\left\{\begin{array}{l}
-(\alpha,\beta,-\alpha-\beta)^{\lambda} .\C{3\lambda}{1} = (-\alpha-\beta,\alpha,\beta)^{\lambda}, \\
-(-\alpha-\beta,\alpha,\beta)^{\lambda} .\C{3\lambda}{1} = (\beta,-\alpha-\beta,\alpha)^{\lambda}, \\
-(\beta,-\alpha-\beta,\alpha)^{\lambda} .\C{3\lambda}{1} = (\alpha,\beta,-\alpha-\beta)^{\lambda}, \\
\end{array}\right.
$$
the result follows.
\end{proof}

\begin{proof}[Proof of Proposition~\ref{prop27}]
Let $u\ge3$ be an integer. From Proposition~\ref{prop*4}, we know that the orbit $\orb{S}$ is $(24,3.2^u)$-interlaced doubly arithmetic with common differences $\matDp{1}$ and $\matDp{2}$ defined by
$$
\matDp{1} = \Dmat{3.2^u}{\pi_{2^u}\left((2,-4,2)^8+8\sum_{i=9}^{16}\alpha_iZ_i\matX{24}\right)}
$$
and
$$
\matDp{2} = \Dmat{3.2^u}{2^{u-2}(3,3,2)^8}.
$$
Moreover, from Proposition~\ref{prop26}, we know that $\orb{S}$ is also $(3.2^u,3.2^u)$-interlaced doubly arithmetic with common differences $\matD{1}$ and $\matD{2}$ defined by
\begin{equation*}
\resizebox{\textwidth}{!}{$
\matD{1} = 2^{u-3}\Dmat{3.2^u}{\pi_{2^u}\left((2,-4,2)^8+8\sum_{i=9}^{16}\alpha_iZ_i\matX{24}\right)^{2^{u-3}}} = \Dmat{3.2^u}{2^{u-2}(1,2,1)^{2^u}},
$}
\end{equation*}
and
$$
\matD{2} = \Dmat{3.2^u}{\left(2^{u-2}(3,3,2)^8\right)^{2^{u-3}}} = \Dmat{3.2^u}{2^{u-2}(3,3,2)^{2^u}}.
$$
Since, from Lemma~\ref{lem*9}, we have
$$
\Dmat{3.2^u}{2^{u-2}(1,2,1)^{2^u}} = 2^{u-2}\mathbf{Circ}\left((1,2,1)^{2^u}\right)
$$
and
$$
\Dmat{3.2^u}{2^{u-2}(3,3,2)^{2^u}} = 2^{u-2}\mathbf{Circ}\left((3,3,2)^{2^u}\right),
$$
this completes the proof.
\end{proof}

\section{Arithmetic triangles}

In this section, we define and analyze {\em arithmetic triangles}, that are finite triangles appearing in doubly arithmetic progressions. The interest of arithmetic triangles is that there are elementary balanced structures that can be used in the proof of the main theorem of this paper.

\begin{defn}[Arithmetic triangles]
Let $a$, $d_1$ and $d_2$ be three elements of $\Zn{m}$ and let $n$ be a positive integer. The {\em arithmetic triangle} $\AT{a}{d_1}{d_2}{n}$ is the triangle of size $n$ of elements of $\Zn{m}$, with first element $a$ and where each row and each diagonal are arithmetic progressions with respective common differences $d_1$ and $d_2$, i.e., the triangle of $t_{n}$ elements of $\Zn{m}$ defined by
$$
\AT{a}{d_1}{d_2}{n} = \left( a+id_2+jd_1\right)_{(i,j)\in\Tn{n}}.
$$
Note that the anti-diagonals of $\AT{a}{d_1}{d_2}{n}$ are also arithmetic progressions with common difference $d_2-d_1$.
\end{defn}

For instance, the arithmetic triangle $\AT{0}{2}{3}{5}$ of $\Zn{5}$ is depicted in Figure~\ref{fig02}.

\begin{figure}[htbp]
\centering{
\includegraphics{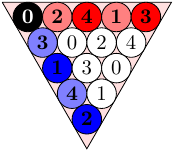}
}
\caption{$\AT{0}{2}{3}{5}$ in $\Zn{5}$}\label{fig02}
\end{figure}

The main result of this section is the following

\begin{thm}\label{thm5}
Let $a$, $d_1$ and $d_2$ be three elements of $\Zn{m}$ such that $\gcd(d_1,m)=\gcd(d_2,m)$ and let $n$ be a positive integer such that $n\equiv 0$ or $-1\pmod{\frac{m}{\gcd(d_1,m)}}$. Then, in the arithmetic triangle $\nabla=\AT{a}{d_1}{d_2}{n}$, we have
$$
\mf{\nabla}(x+\gcd(d_2-d_1,m)) = \mf{\nabla}(x),
$$
for all $x\in\Zn{m}$.
\end{thm}

The proof is based on the following two lemmas.

\begin{lem}\label{lem13}
Let $a$, $d_1$ and $d_2$ be three elements of $\Zn{m}$ and let $n$ be a positive integer. Then, in the arithmetic triangle $\nabla=\AT{a}{d_1}{d_2}{n}$, we have
$$
\mf{\nabla}(x+d_1)-\mf{\nabla}(x+d_2) = \mf{\APf{a}{d_2}{n}}(x+d_1) - \mf{\APf{a}{d_1}{n}}(x+d_2),
$$
for all $x\in\Zn{m}$.
\end{lem}

\begin{proof}
The first row of $\nabla$ is the sequence $\APf{a}{d_1}{n}$ and, if we remove this row from $\nabla$, we obtain the arithmetic triangle
$$
\nabla\setminus\APf{a}{d_1}{n} = \AT{a+d_2}{d_1}{d_2}{n-1},
$$
which is the translate of $\AT{a}{d_1}{d_2}{n-1}$ by $d_2$. Moreover, the first diagonal of $\nabla$ is the sequence $\APf{a}{d_2}{n}$ and, if we remove this diagonal from $\nabla$, we obtain the arithmetic triangle
$$
\nabla\setminus\APf{a}{d_2}{n} = \AT{a+d_1}{d_1}{d_2}{n-1},
$$
which is the translate of $\AT{a}{d_1}{d_2}{n-1}$ by $d_1$. This leads to
$$
\mf{\nabla}(x+d_1)-\mf{\APf{a}{d_2}{n}}(x+d_1)
\begin{array}[t]{l}
 = \mf{\nabla\setminus\APf{a}{d_2}{n}}(x+d_1) \\[1.5ex]
 = \mf{\AT{a+d_1}{d_1}{d_2}{n-1}}(x+d_1) \\[1.5ex]
 = \mf{\AT{a}{d_1}{d_2}{n-1}}(x) \\[1.5ex]
 = \mf{\AT{a+d_2}{d_1}{d_2}{n-1}}(x+d_2) \\[1.5ex]
 = \mf{\nabla\setminus\APf{a}{d_1}{n}}(x+d_2) \\[1.5ex]
 = \mf{\nabla}(x+d_2)-\mf{\APf{a}{d_1}{n}}(x+d_2),
\end{array} 
$$
for all $x\in\Zn{m}$. This completes the proof.
\end{proof}

\begin{lem}\label{lem12}
Let $a$ and $d$ be two elements of $\Zn{m}$ and let $n$ be a positive integer divisible by $\frac{m}{\gcd(d,m)}$. Then, in the arithmetic progression $S=\APf{a}{d}{n}$, we have
$$
\mf{S}(x+\gcd(d,m)) = \mf{S}(x),
$$
for all $x\in\Zn{m}$. More precisely, we have
$$
\mf{S}(x) = \left\{\begin{array}{ll}
\displaystyle\frac{\gcd(d,m)n}{m} & \text{if}\ \pi_{\gcd(d,m)}(x)=\pi_{\gcd(d,m)}(a),\\[2ex]
0 & \text{otherwise}.
\end{array}\right.
$$
\end{lem}

\begin{proof}
From the Bézout's identity, we know that there exists an integer $\alpha$ such that
$$
\alpha d \equiv \gcd(d,m)\pmod{m}.
$$
Then, since $n$ is a multiple of $\frac{m}{\gcd(d,m)}$, we have
$$
\left\{a+id\ \middle|\ i\in\{0,\ldots,n-1\}\right\} = \left\{a+i\gcd(d,m)\ \middle|\ i\in\left\{0,\ldots,\frac{m}{\gcd(d,m)}-1\right\}\right\}.
$$
Moreover, it is clear that
$$
a+i_1\gcd(d,m) \equiv a+i_2\gcd(d,m)\pmod{m}\quad\Longleftrightarrow\quad i_1\equiv i_2\pmod{\frac{m}{\gcd(d,m)}},
$$
for all integers $i_1$ and $i_2$. Therefore,
$$
\mf{S}(x) = \left\{\begin{array}{ll}
\displaystyle\frac{\gcd(d,m)n}{m} & \text{if}\ \pi_{\gcd(d,m)}(x)=\pi_{\gcd(d,m)}(a),\\[2ex]
0 & \text{otherwise},
\end{array}\right.
$$
as announced. It follows that
$$
\mf{S}(x+\gcd(d,m)) = \mf{S}(x),
$$
for all $x\in\Zn{m}$.
\end{proof}

We are now ready to prove Theorem~\ref{thm5}, the main result of this section.

\begin{proof}[Proof of Theorem~\ref{thm5}]
First, suppose that $n$ is divisible by $\frac{m}{\gcd(d_1,m)}$. Since $\gcd(d_1,m)=\gcd(d_2,m)$, we have
$$
\pi_{\gcd(d_2,m)}(x+d_1)=\pi_{\gcd(d_1,m)}(x+d_2)
$$
Moreover, since $n$ is divisible by $\frac{m}{\gcd(d_1,m)}$, we obtain from Lemma~\ref{lem12} that
$$
\mf{\APf{a}{d_2}{n}}(x+d_1) = \mf{\APf{a}{d_1}{n}}(x+d_2),
$$
for all $x\in\Zn{m}$. It follows, from Lemma~\ref{lem13} that
$$
\mf{\nabla}(x+d_1)=\mf{\nabla}(x+d_2),
$$
for all $x\in\Zn{m}$. Therefore,
$$
\mf{\nabla}(x) = \mf{\nabla}(x+d_2-d_1) = \mf{\nabla}(x+\gcd(d_2-d_1,m)),
$$
for all $x\in\Zn{m}$.

Now, suppose that $n\equiv -1\pmod{\frac{m}{\gcd(d_1,m)}}$. We already know, from above, that the arithmetic triangle $\nabla'=\AT{a}{d_1}{d_2}{n+1}$ is such that
$$
\mf{\nabla'}(x) = \mf{\nabla'}(x+\gcd(d_2-d_1,m)),
$$
for all $x\in\Zn{m}$. Moreover, the arithmetic triangle $\nabla=\AT{a}{d_1}{d_2}{n}$ can be obtained from $\nabla'=\AT{a}{d_1}{d_2}{n+1}$ by removing its first antidiagonal, that is the sequence $\APf{a-d_1}{d_2-d_1}{n+1}$,
$$
\AT{a}{d_1}{d_2}{n} = \AT{a}{d_1}{d_2}{n+1}\setminus\APf{a-d_1}{d_2-d_1}{n+1}.
$$
Since $\gcd(d_1,m)$ divides $\gcd(d_2-d_1,m)$, we have that $n+1$ is divisible by $\frac{m}{\gcd(d_1-d_2)}$. It follows, by Lemma~\ref{lem12}, that the sequence $S=\APf{a-d_1}{d_2-d_1}{n+1}$ is such that
$$
\mf{S}(x) = \mf{S}(x+\gcd(d_2-d_1,m)),
$$
for all $x\in\Zn{m}$. We conclude that
$$
\mf{\nabla}(x+\gcd(d_2-d_1,m))
\begin{array}[t]{l}
= \mf{\nabla'}(x+\gcd(d_2-d_1,m))-\mf{S}(x+\gcd(d_2-d_1,m)) \\
= \mf{\nabla'}(x)-\mf{S}(x) \\
= \mf{\nabla}(x),
\end{array}
$$
for all $x\in\Zn{m}$.
\end{proof}

Now, a necessary condition for having a balanced arithmetic triangle is determined.

\begin{prop}
Let $a$, $d_1$ and $d_2$ be three elements of $\Zn{m}$ and let $n$ be a positive integer. If the arithmetic triangle $\AT{a}{d_1}{d_2}{n}$ is balanced, then the common differences $d_1$, $d_2$ and $d_2-d_1$ are all invertible.
\end{prop}

\begin{proof}
Without loss of generality, suppose that $d_1$ is not invertible. Then, we consider the projection of $\nabla=\AT{a}{d_1}{d_2}{n}$ into $\Zn{\alpha}$, where $\alpha=\gcd(d_1,m)\ge2$, i.e.,
$$
\pi_{\alpha}(\nabla) = \AT{\pi_{\alpha}(a)}{0}{\pi_{\alpha}(d_2)}{n}.
$$
Then, the row $R_{i}\left(\pi_{\alpha}(\nabla)\right)$ is constituted by $n-i$ copies of $\pi_{\alpha}(a)+i\pi_{\alpha}(d_2)$, for all $i\in\{0,1,\ldots,n-1\}$.

If $\pi_{\alpha}(d_2)\neq 0$, it is easy to see that $R_{i}\left(\pi_{\alpha}(\nabla)\right)$ is constituted by $n-i$ copies of $\pi_{\alpha}(a)$ if and only if $R_{i+1}\left(\pi_{\alpha}(\nabla)\right)$ is constituted by $n-i-1$ copies of $\pi_{\alpha}(a)+\pi_{\alpha}(d_2)$, for all $i\in\{0,1,\ldots,n-1\}$, where we suppose that the $n$\textsuperscript{th} row is defined as being empty.
Therefore,
$$
\mf{\pi_{\alpha}(\nabla)}(\pi_{\alpha}(a)) > \mf{\pi_{\alpha}(\nabla)}(\pi_{\alpha}(a)+\pi_{\alpha}(d_2)).
$$
It follows that $\pi_{\alpha}(\nabla)$ is not balanced in $\Zn{\alpha}$ and thus, by Theorem~\ref{thm6}, the triangle $\nabla$ is not balanced in $\Zn{m}$.

Otherwise, if $\pi_{\alpha}(d_2)=0$, we have
$$
\pi_{\alpha}(\nabla) = \AT{\pi_{\alpha}(a)}{0}{0}{n},
$$
which is the constant triangle uniquely composed by elements $\pi_{\alpha}(a)$. It follows that $\pi_{\alpha}(\nabla)$ is not balanced in $\Zn{\alpha}$ and thus, by Theorem~\ref{thm6}, the triangle $\nabla$ is not balanced in $\Zn{m}$.
\end{proof}

\begin{rem}
If $\gcd(d_1,m)=\gcd(d_2,m)$ and $\gcd(d_2-d_1,m)=1$, it is easy to see that $d_1$, $d_2$ and $d_2-d_1$ are all invertible.
\end{rem}

\begin{rem}
When $m$ is even, there is no balanced arithmetic triangles in $\Zn{m}$ since $d_1$, $d_2$ and $d_1-d_2$ cannot be all invertible in $\Zn{m}$.
\end{rem}

When $m$ is odd, Theorem~\ref{thm5} can then be refined in the case where $d_1$, $d_2$ and $d_1-d_2$ are all invertible.

\begin{thm}\label{thm7}
Let $m$ be an odd number and let $a$, $d_1$ and $d_2$ be three elements of $\Zn{m}$ such that $d_1$, $d_2$ and $d_2-d_1$ are all invertible. Then, the arithmetic triangle $\AT{a}{d_1}{d_2}{n}$ is balanced for all non-negative integers $n\equiv 0$ or $-1\pmod{m}$.
\end{thm}

\begin{proof}
Since $\gcd(d_1,m)=\gcd(d_2,m)=\gcd(d_2-d_1,m)=1$, we deduce from Theorem~\ref{thm5} that
$$
\mf{\AT{a}{d_1}{d_2}{n}}(x) = \mf{\AT{a}{d_1}{d_2}{n}}(x+1),
$$
for all $x\in\Zn{m}$. Therefore, the arithmetic triangle $\AT{a}{d_1}{d_2}{n}$ is balanced in $\Zn{m}$.
\end{proof}

For instance, the arithmetic triangle $\AT{0}{2}{3}{5}$ of $\Zn{5}$ depicted in Figure~\ref{fig02} is balanced since $n=5$ is divisible by $5$ and $d_1=2$, $d_2=3$ and $d_2-d_1=1$ are all invertible in $\Zn{5}$. Indeed, this triangle contains three times each element of $\Zn{5}$. 

We conclude this section by explaining the relationship between arithmetic triangles and the triangles appearing in interlaced doubly arithmetic orbits. Indeed, it is possible to decompose a triangle from an interlaced doubly arithmetic orbit into subtriangles that are arithmetic triangles.

\begin{prop}\label{prop28}
Let $k$ be a positive integer and let
$$
\matA=\left(a_{i,j}\right)_{0\le i,j\le k-1},\quad \matD{1}=\left(d^{\,(1)}_{i,j}\right)_{0\le i,j\le k-1}\quad\text{and}\quad\matD{2}=\left(d^{\,(2)}_{i,j}\right)_{0\le i,j\le k-1}
$$
be three square matrices of order $k$ of elements in $\Zn{m}$. Let $n$ be a positive integer and let $n=qk+r$ be the Euclidean division of $n$ by $k$. Then the triangle $\nabla=(u_{i,j})_{(i,j)\in T_n}$, of size $n$, appearing in the $(k,k)$-interlaced doubly arithmetic progression
$$
\IDAP{\matA}{\matD{1}}{\matD{2}}=\left(u_{i,j}\right)_{(i,j)\in\N\times\Z},
$$
can be decomposed into $k^2$ subtriangles
$$
\nabla_{i_0,j_0} = \left(u_{i_0+ik,j_0+jk}\right)_{(i,j)\in \mathrm{I}_{i_0,j_0}},
$$
where
$$
\mathrm{I}_{i_0,j_0} = \left\{ (i,j)\in\N\times\Z\ \middle|\ (i_0+ik,j_0+jk)\in\Tn{n} \right\},
$$
for all $(i_0,j_0)\in\{0,\ldots,k-1\}^2$. Moreover, these subtriangles are arithmetic triangles
$$
\nabla_{i_0,j_0} = \AT{u_{i_0,j_0}}{d^{\,(1)}_{i_0,j_0}}{d^{\,(2)}_{i_0,j_0}}{q+\varepsilon},
$$
of size $q+\varepsilon$, for all $(i_0,j_0)\in\{0,\ldots,k-1\}^2$, where
$$
\varepsilon = \left\{\begin{array}{cl}
 1 & \text{if}\ 0\le i_0+j_0\le r-1,\\
 0 & \text{if}\ r\le i_0+j_0\le k+r-1,\\
 -1 & \text{if}\ k+r\le i_0+j_0\le 2k-2.\\
\end{array}\right.
$$
\end{prop}

\begin{proof}
First, since $\IDAP{\matA}{\matD{1}}{\matD{2}}=\left(u_{i,j}\right)_{(i,j)\in\N\times\Z}$, we know that
$$
u_{i_0+ik,j_0+jk} = u_{i_0,j_0} + id^{\,(2)}_{i_0,j_0} + jd^{\,(1)}_{i_0,j_0},
$$
for all $(i_0,j_0)\in\{0,\ldots,k-1\}^2$ and all $(i,j)\in\N\times\Z$. Therefore, the subtriangle $\nabla_{i_0,j_0}$ is a subtriangle of the form
$$
\nabla_{i_0,j_0} = \AT{u_{i_0,j_0}}{d^{\,(1)}_{i_0,j_0}}{d^{\,(2)}_{i_0,j_0}}{n_{i_0,j_0}},
$$
where $n_{i_0,j_0}$ is a positive integer to be determined, for all $(i_0,j_0)\in\{0,\ldots,k-1\}^2$. Let $n=qk+r$, with $0\le r\le k-1$, and let $(i_0,j_0)\in\{0,\ldots,k-1\}^2$. Then, for all $(i,j)\in\N^2$, we have
$$
(i_0+ik)+(j_0+jk)\le n-1 \Longleftrightarrow i+j\le q+\frac{r-1-(i_0+j_0)}{k}.
$$
Therefore,
$$
n_{i_0,j_0} = q+1+\left\lfloor\frac{r-1-(i_0+j_0)}{k}\right\rfloor = \left\{\begin{array}{cl}
 q+1 & \text{if}\ 0\le i_0+j_0\le r-1,\\
 q & \text{if}\ r\le i_0+j_0\le k+r-1,\\
 q-1 & \text{if}\ k+r\le i_0+j_0\le 2k-2.\\
\end{array}\right.
$$
This completes the proof.
\end{proof}

\section{Proof for the powers of 2}

In this section, a proof of Theorem~\ref{mainthm} is given. For each $24$-tuples of integers
$$
A = \pm X_{i_0} + 4\sum_{i=1}^{8}\alpha_i Y_i + 8\sum_{i=9}^{16}\alpha_i Z_i \in\mathcal{E},
$$
where $i_0\in\{1,\ldots,7\}$ and $\alpha_i\in\Z$ for all $i\in\{1,\ldots,16\}$, we will prove that, for all non-negative integers $u$, the orbit of the sequence
$$
S = \IAP{\pi_{2^u}(A)}{\pi_{2^u}(A)\matX{24}}
$$
is $(12.2^u,12.2^u)$-periodic and the triangles $\ST{S[12.2^u\lambda]}$ are balanced in $\Zn{2^{u}}$, for all non-negative integers $\lambda$. This then proves Theorem~\ref{thm*2}, the main result of this paper, when $m$ is a power of two.

First, it is straightforward to see that the results of Theorem~\ref{mainthm} are verified for $A\in\setE$ if and only if it is also the case for $-A\in\setE$. Indeed, for all non-negative integers $u$, for the sequences
$$
S = \IAP{\pi_{2^u}(A)}{\pi_{2^u}(A)\matX{24}}\quad\text{and}\quad -S = \IAP{\pi_{2^u}(-A)}{\pi_{2^u}(-A)\matX{24}},
$$
it is clear that the orbit $\orb{S}$ is $(12.2^u,12.2^u)$-periodic if and only if the orbit $\orb{-S}=-\orb{S}$ is also $(12.2^u,12.2^u)$-periodic, and the triangles $\ST{S[12.2^u\lambda]}$ are balanced if and only if $\ST{(-S)[12.2^u\lambda]}=-\ST{S[12.2^u\lambda]}$ are as well, for all non-negative integers $\lambda$. Therefore, we only consider
$$
A = X_{i_0} + 4\sum_{i=1}^{8}\alpha_i Y_i + 8\sum_{i=9}^{16}\alpha_i Z_i \in\mathcal{E},
$$
with $i_0\in\{1,\ldots,7\}$ and $\alpha_i\in\Z$ for all $i\in\{1,\ldots,16\}$, and the sequence
$$
S = \IAP{\pi_{2^u}(A)}{\pi_{2^u}(A)\matX{24}},
$$
for all non-negative integers $u$, in the sequel of this proof.

We proceed by induction on $u$. For $u=0$, the result is clear, since any orbit of $\Zn{1}$ is $(p,p)$-periodic, for all positive integers $p$, and any triangle is balanced in $\Zn{1}$.

For $u=1$, we have
$$
\pi_2(A) = A_2 = 001101000001100000101100.
$$
Since $A_2\matX{24}=\matr{0}$, we know that $S={A_2}^{\infty}$ is $24$-periodic and thus
$$
\pi_2(S)[24\lambda] = {A_2}^{\lambda},
$$
for all non-negative integers $\lambda$. It follows that all its derived sequences $\ider{i}{S}$ are also $24$-periodic, for all non-negative integers $i$, by Corollary~\ref{cor1}. It is easy to verify that
$$
\ider{24}{\left({A_2}^2\right)} = A_2\quad \text{and thus}\quad \ider{24}{S}=S.
$$
Therefore the orbit $\orb{S}$ is $(24,24)$-periodic. Finally, from Proposition~\ref{prop4}, we know that the triangles $\ST{S[24\lambda]}$ are balanced for all non-negative integers $\lambda$ if and only if it is the case only for two distinct positive values $\lambda_1$ and $\lambda_2$. As showed in Figures~\ref{fig03} and \ref{fig04}, the triangles $\ST{S[24]}=\ST{A_2}$ and $\ST{S[48]}=\ST{{A_2}^2}$ are balanced in $\Zn{2}$. This completes the proof for $u=1$.

\begin{figure}[htbp]
\centering{
\includegraphics{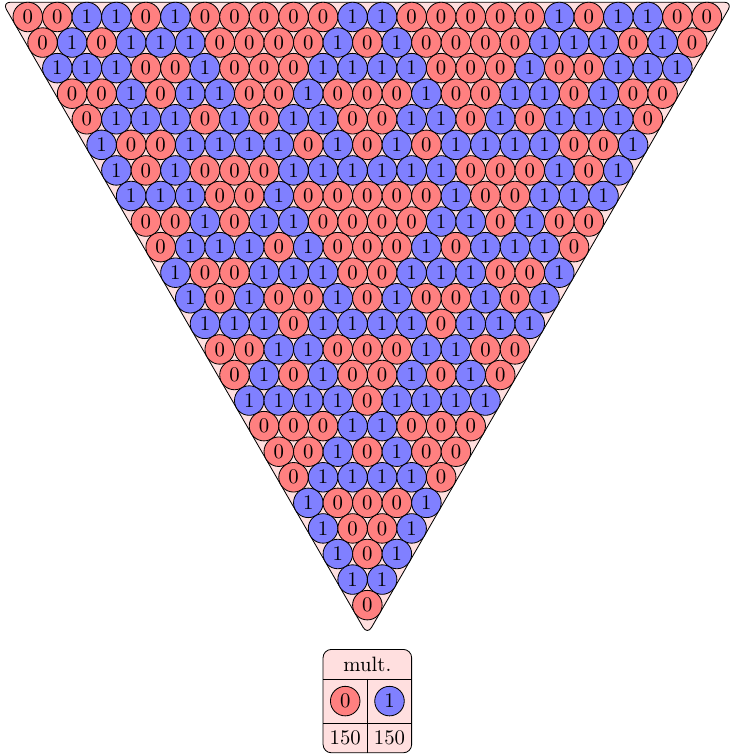}
}
\caption{$\ST{A_2}$ in $\Zn{2}$}\label{fig03}
\end{figure}

\begin{figure}[htbp]
\centering{
\begin{tabular}{c}
\includegraphics[width=0.974\textwidth]{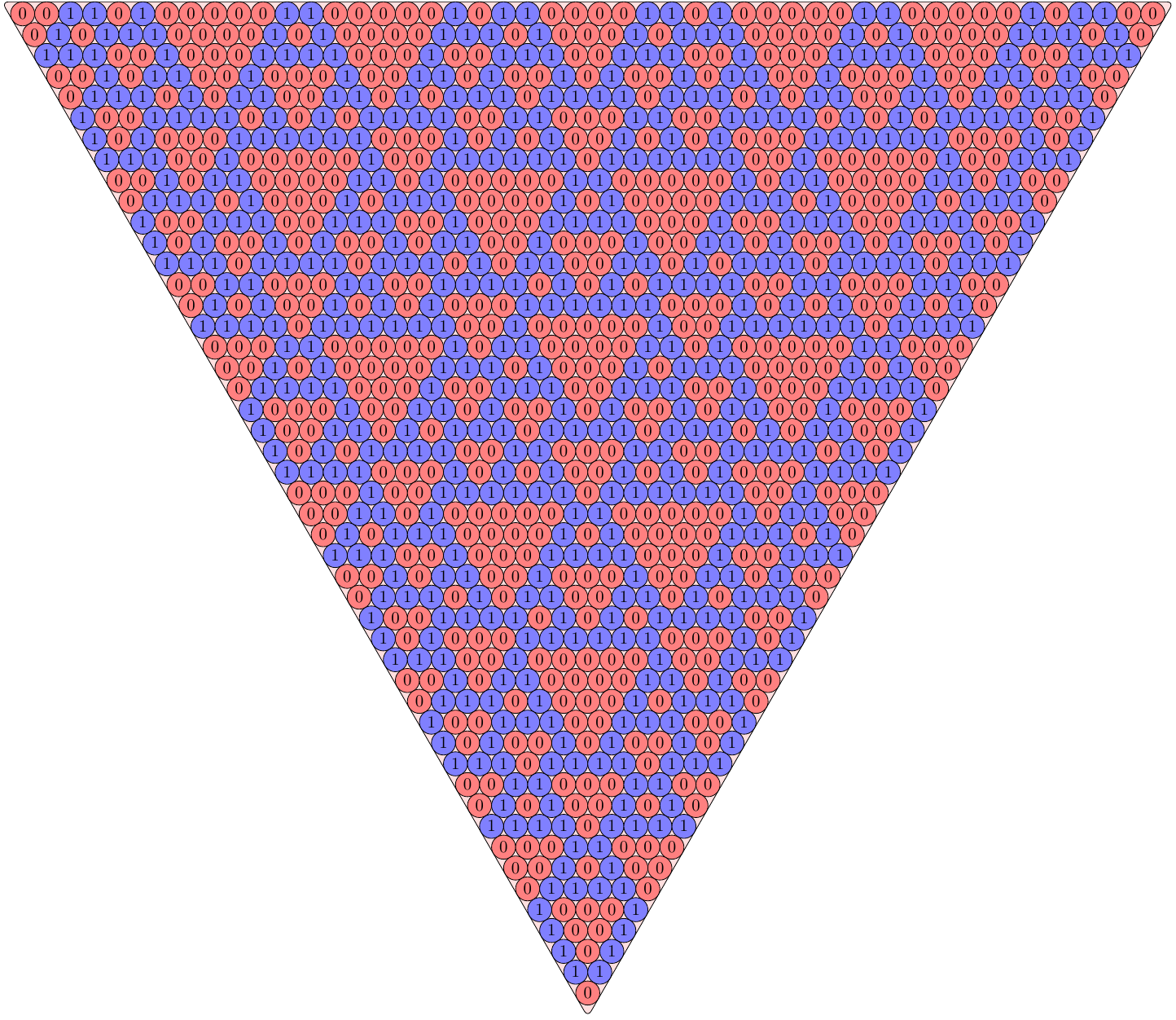} \\
\includegraphics{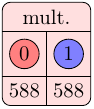} \\
\end{tabular}
}
\caption{$\ST{{A_2}^2}$ in $\Zn{2}$}\label{fig04}
\end{figure}

For $u=2$, we have
$$
\pi_4(A) \in \left\{ \pi_4(X_i)\ \middle|\ i\in\{1,\ldots,7\}\right\} = \left\{
\begin{array}{l}
001123220203300220321102,\\
001303202203300000303102,\\
003301002003302002103302,\\
021301002223320002103122,\\
201103200201100200301100,\\
203321022001102022123300,\\
223103200021122200301320\\
\end{array}\right\}.
$$
Since
$$
\pi_4(X_i)\matX{24} = (2,0,2)^8,
$$
for all $i\in\{1,\ldots,7\}$, we know that $S$ is $48$-periodic and thus
$$
S = {A_4}^{\infty}\quad\text{and}\quad S[48\lambda] = {A_4}^{\lambda},
$$
for all positive integers $\lambda$, where $A_4$ is the $48$-tuple $A_4 = S[48]$.
It follows that all its derived sequences $\ider{i}{S}$ are $48$-periodic, for all non-negative integers $i$, by Corollary~\ref{cor1}. It is easy to verify that
$$
\ider{48}{{A_4}^2} = A_4\quad \text{and thus}\quad \ider{48}{S}=S.
$$
Therefore, the orbit $\orb{S}$ is $(48,48)$-periodic. Finally, from Proposition~\ref{prop4}, we know that the triangles $\ST{S[48\lambda]}$ are balanced for all non-negative integers $\lambda$ if and only if it is the case only for two distinct positive values $\lambda_1$ and $\lambda_2$. As showed in Figures~\ref{fig05} and \ref{fig06} when $\pi_4(A)=\pi_4(X_1)$, i.e.,
$$
A_4 = S[48] = 001123220203300220321102203321022001102022123300,
$$
the triangles $\ST{S[48]}=\ST{A_4}$ and $\ST{S[96]}=\ST{{A_4}^2}$ are balanced in $\Zn{4}$. The six other cases, when $\pi_4(A)=\pi_4(X_i)$ for $i\in\{2,\ldots,7\}$, can be treated by a similar way. This completes the proof for $u=2$.

\begin{figure}[htbp]
\centering{
\begin{tabular}{c}
\includegraphics[width=0.974\textwidth]{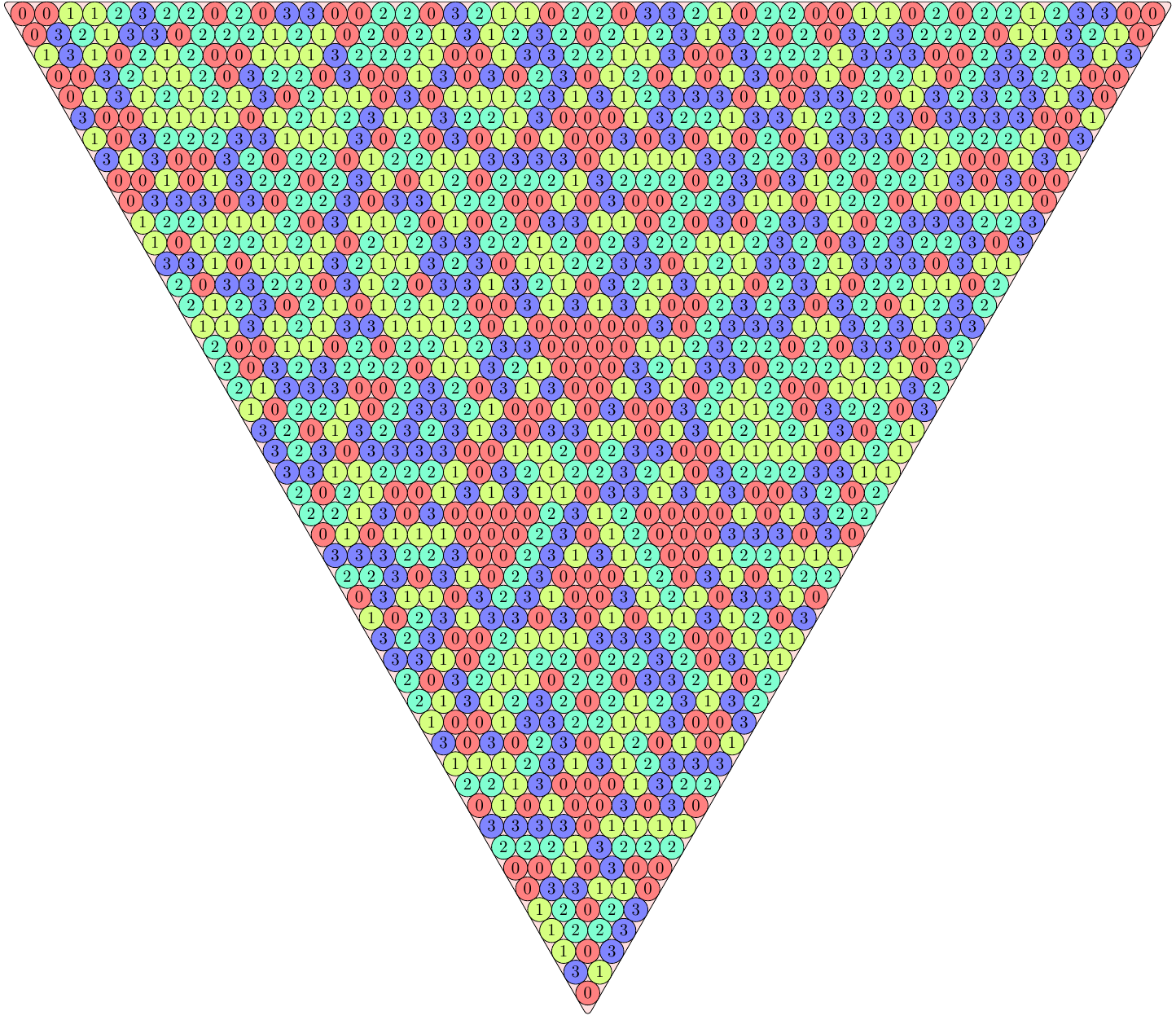} \\
\includegraphics{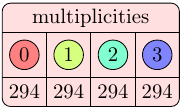} \\
\end{tabular}
}
\caption{$\ST{A_4}$ in $\Zn{4}$}\label{fig05}
\end{figure}

\begin{figure}[htbp]
\centering{
\begin{tabular}{c}
\includegraphics[width=0.974\textwidth]{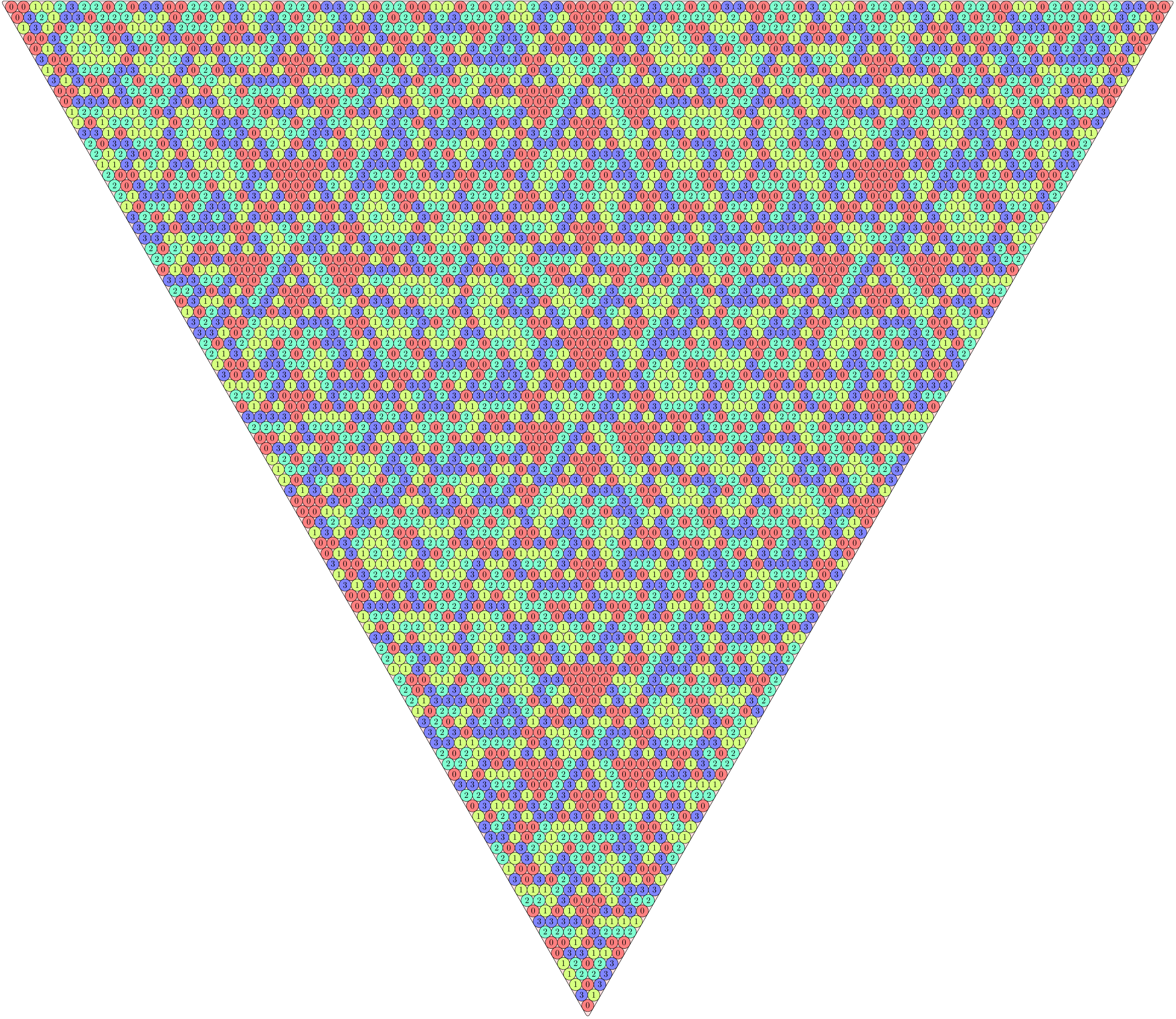} \\
\includegraphics{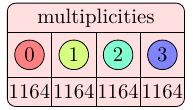} \\
\end{tabular}
}
\caption{$\ST{{A_4}^2}$ in $\Zn{4}$}\label{fig06}
\end{figure}

Suppose now that the result is true for some positive integer $u-1\ge 2$ and prove it for $u$. First, we know from Proposition~\ref{prop27} that $\orb{S}$ is $(3.2^u,3.2^u)$-interlaced doubly arithmetic with common differences
\begin{equation}\label{eq*11}
\matD{1} = 2^{u-2}\Circ{(1,2,1)^{2^u}}\quad\text{and}\quad \matD{2} = 2^{u-2}\Circ{(3,3,2)^{2^u}}.
\end{equation}
Since
$$
4\matD{1} \equiv \matr{0} \pmod{2^{u}}\quad\text{and}\quad 4\matD{2} \equiv \matr{0} \pmod{2^{u}},
$$
we obtain that $\orb{S}$ is $(12.2^u,12.2^u)$-periodic, by Proposition~\ref{prop17}. Let $\lambda$ be a positive integer. We consider the triangle
$$
\nabla = \ST{S[12.2^u\lambda]} = \left(a_{i,j}\right)_{(i,j)\in T_{12.2^u\lambda}}.
$$
From Proposition~\ref{prop28}, we know that $\nabla$ can be decomposed into $9.2^{2u}$ subtriangles $\nabla_{r,s}$ that are the arithmetic triangles
$$
\nabla_{r,s} = \left(a_{r+3.2^ui,s+3.2^uj}\right)_{(i,j)\in T_{n_{r,s}}} = \AT{a_{r,s}}{\left(\matD{1}\right)_{r,s}}{\left(\matD{2}\right)_{r,s}}{n_{r,s}},
$$
where
$$
n_{r,s} = \left\{\begin{array}{ll}
 4\lambda & \text{if}\ 0\le r+s\le 3.2^u-1,\\
 4\lambda-1 & \text{if}\ 3.2^u\le r+s\le 6.2^u-2,\\
\end{array}\right.
$$
for all $(r,s)\in\{0,\ldots,3.2^u-1\}^2$ and with
$$
\matD{1} = \left(d^{\,(1)}_{r,s}\right)_{0\le r,s\le 3.2^u-1}\quad\text{and}\quad\matD{2} = \left(d^{\,(2)}_{r,s}\right)_{0\le r,s\le 3.2^u-1}.
$$
Moreover, from \eqref{eq*11}, we deduce that
$$
\left(d^{\,(1)}_{r,s},
d^{\,(2)}_{r,s}\right) = 
\left\{\begin{array}{cl}
 2^{u-2}(1,3) & \text{if}\ r-s\equiv0\pmod{3}, \\
 2^{u-2}(1,2) & \text{if}\ r-s\equiv1\pmod{3}, \\
 2^{u-2}(2,3) & \text{if}\ r-s\equiv2\pmod{3}, \\
\end{array}\right.
$$
for all $(r,s)\in\{0,\ldots,3.2^u-1\}^2$. Let $(r,s)\in\{0,\ldots,3.2^u-1\}^2$. We distinguish different cases according to the residue class of $r-s$ modulo $3$. 

\setcounter{case}{0}
\begin{case}
If $r-s\equiv0\pmod{3}$, then
$$
d^{\,(1)}_{r,s} = 2^{u-2},\ d^{\,(2)}_{r,s} = 3.2^{u-2}\ \text{and}\ d^{\,(2)}_{r,s}-d^{\,(1)}_{r,s} = 2^{u-1}.
$$
Since
$$
\gcd(d^{\,(1)}_{r,s},2^u) = \gcd(d^{\,(2)}_{r,s},2^u) = 2^{u-2},\ \gcd(d^{\,(2)}_{r,s}-d^{\,(1)}_{r,s},2^u) = 2^{u-1},\ \frac{2^u}{\gcd(d^{\,(1)}_{r,s},2^u)}=4
$$
and $n_{r,s}\in\left\{4\lambda-1,4\lambda\right\}$, we deduce from Theorem~\ref{thm5} that
$$
\mf{\nabla_{r,s}}(x+2^{u-1}) = \mf{\nabla_{r,s}}(x),
$$
for all $x\in\Zn{2^u}$, in this case.
\end{case}

\begin{case}
If $r-s\equiv1\pmod{3}$, then
$$
d^{\,(1)}_{r,s} = 2^{u-2},\ d^{\,(2)}_{r,s} = 2^{u-1}\ \text{and}\ d^{\,(2)}_{r,s}-d^{\,(1)}_{r,s} = 2^{u-2}.
$$
Since
$$
\gcd(d^{\,(1)}_{r,s},2^u) = \gcd(d^{\,(2)}_{r,s}-d^{\,(1)}_{r,s},2^u) = 2^{u-2},\ \gcd(d^{\,(2)}_{r,s},2^u) = 2^{u-1},\ \frac{2^u}{\gcd(d^{\,(1)}_{r,s},2^u)}=4
$$
and $n_{r,s}\in\left\{4\lambda-1,4\lambda\right\}$, we deduce from Theorem~\ref{thm5} that
$$
\mf{\nabla_{r,s}}(x+2^{u-1}) = \mf{\nabla_{r,s}}(x),
$$
for all $x\in\Zn{2^u}$, in this case.
\end{case}

\begin{case}
If $r-s\equiv2\pmod{3}$, then
$$
d^{\,(1)}_{r,s} = 2^{u-1},\ d^{\,(2)}_{r,s} = 3.2^{u-2}\ \text{and}\ d^{\,(2)}_{r,s}-d^{\,(1)}_{r,s} = 2^{u-2}.
$$
Since
$$
\gcd(d^{\,(2)}_{r,s},2^u) = \gcd(d^{\,(2)}_{r,s}-d^{\,(1)}_{r,s},2^u) = 2^{u-2},\ \gcd(d^{\,(1)}_{r,s},2^u) = 2^{u-1},\ \frac{2^u}{\gcd(d^{\,(2)}_{r,s},2^u)}=4
$$
and $n_{r,s}\in\left\{4\lambda-1,4\lambda\right\}$, we deduce from Theorem~\ref{thm5} that
$$
\mf{\nabla_{r,s}}(x+2^{u-1}) = \mf{\nabla_{r,s}}(x),
$$
for all $x\in\Zn{2^u}$, in this case.
\end{case}

Therefore, in all cases, we have
$$
\mf{\nabla_{r,s}}(x+2^{u-1}) = \mf{\nabla_{r,s}}(x),
$$
for all $x\in\Zn{2^u}$, in any subtriangle $\nabla_{r,s}$, for all $(r,s)\in\{0,\ldots,3.2^u-1\}^2$. Since
$$
\nabla = \bigsqcup_{(r,s)\in\{0,\ldots,3.2^u-1\}^2} \nabla_{r,s},
$$
we deduce that
$$
\mf{\nabla}(x+2^{u-1}) = \mf{\nabla}(x),
$$
for all $x\in\Zn{2^u}$.

Finally, since $\pi_{2^{u-1}}(\nabla)$ is balanced in $\Zn{2^{u-1}}$, by induction hypothesis, and since $\mf{\nabla}(x+2^{u-1}) = \mf{\nabla}(x)$, for all $x\in\Zn{2^u}$, by the previous result, we obtain from Theorem~\ref{thm6} that $\nabla$ is balanced in $\Zn{2^u}$.

This completes the proof of Theorem~\ref{mainthm} and Theorem~\ref{thm*2}, the main result of this paper, when $m$ is a power of two. \qed

\

We shall conclude this section by exhibiting a few simple elements of the set $\setE$.
We search in the set $\setE$ for the $24$-tuples $A=(a_1,\ldots,a_{24})$ whose elements $a_i$ are all, in absolute value, smaller than a certain $\alpha\in\N$, i.e., the subset
$$
\setE_\alpha = \left\{ A=(a_1,\ldots,a_{24})\in\setE\ \middle|\ |a_i|\le\alpha,\ \forall i\in\{1,\ldots,24\} \right\}.
$$
By exhaustive search for $\alpha\le 4$, we obtain that $\setE_\alpha=\emptyset$ for $\alpha\le 3$ and $|\setE_4|=330$. For instance, $A= X_1-4Y_5-4Y_8\in\setE_4$. In details, we have
$$
A = (0, 0, 1, 1, -2, 3, 2, -2, 0, 2, 0, -1, 3, -4, 0, 2, -2, 0, -1, -2, 1, 1, -4, 2)
$$
with
$$
A\matX{24} = (2,-4,2)^8.
$$
For all positive integers $u$, we know that the sequence $S_{2^u}=\IAP{\pi_{2^u}(A)}{\pi_{2^u}(A)\matX{24}}$ is $(12.2^u)$-periodic, i.e.,
$$
S_{2^u}=\IAP{\pi_{2^u}(A)}{\pi_{2^u}(A)\matX{24}} = {A_{2^u}}^\infty,
$$
where $A_{2^u}$ is the first period of $S_{2^u}$. From Theorem~\ref{mainthm}, we know that $\orb{S_{2^u}}$ is $(12.2^u,12.2^u)$-periodic and that the triangles $\ST{{A_{2^u}}^\lambda}$, of size $12.2^u\lambda$, are balanced in $\Zn{2^u}$, for all non-negative integers $\lambda$. For $u\in\{1,2,3\}$, we obtain that
\begin{equation*}
\resizebox{\textwidth}{!}{$
\begin{array}{l}
A_2 = 001101000001100000101100, \\
A_4 = 001123220203300220321102203321022001102022123300, \\
A_8 = 001163260207340260761142243325422441502422123304405567664603744664365546647721026045106026527700. \\
\end{array}
$}
\end{equation*}
The triangles $\ST{A_2}$ and $\ST{{A_2}^2}$ in $\Zn{2}$ are depicted in Figures~\ref{fig03} and \ref{fig04}, respectively, the triangles $\ST{A_4}$ and $\ST{{A_4}^2}$ in $\Zn{4}$ in Figures~\ref{fig05} and \ref{fig06}, respectively, and the triangles $\ST{A_8}$ and $\ST{{A_8}^2}$ in $\Zn{8}$ in Figures~\ref{fig07} and \ref{fig08}, respectively. 

\begin{figure}[htbp]
\centering{
\begin{tabular}{c}
\includegraphics[width=0.974\textwidth]{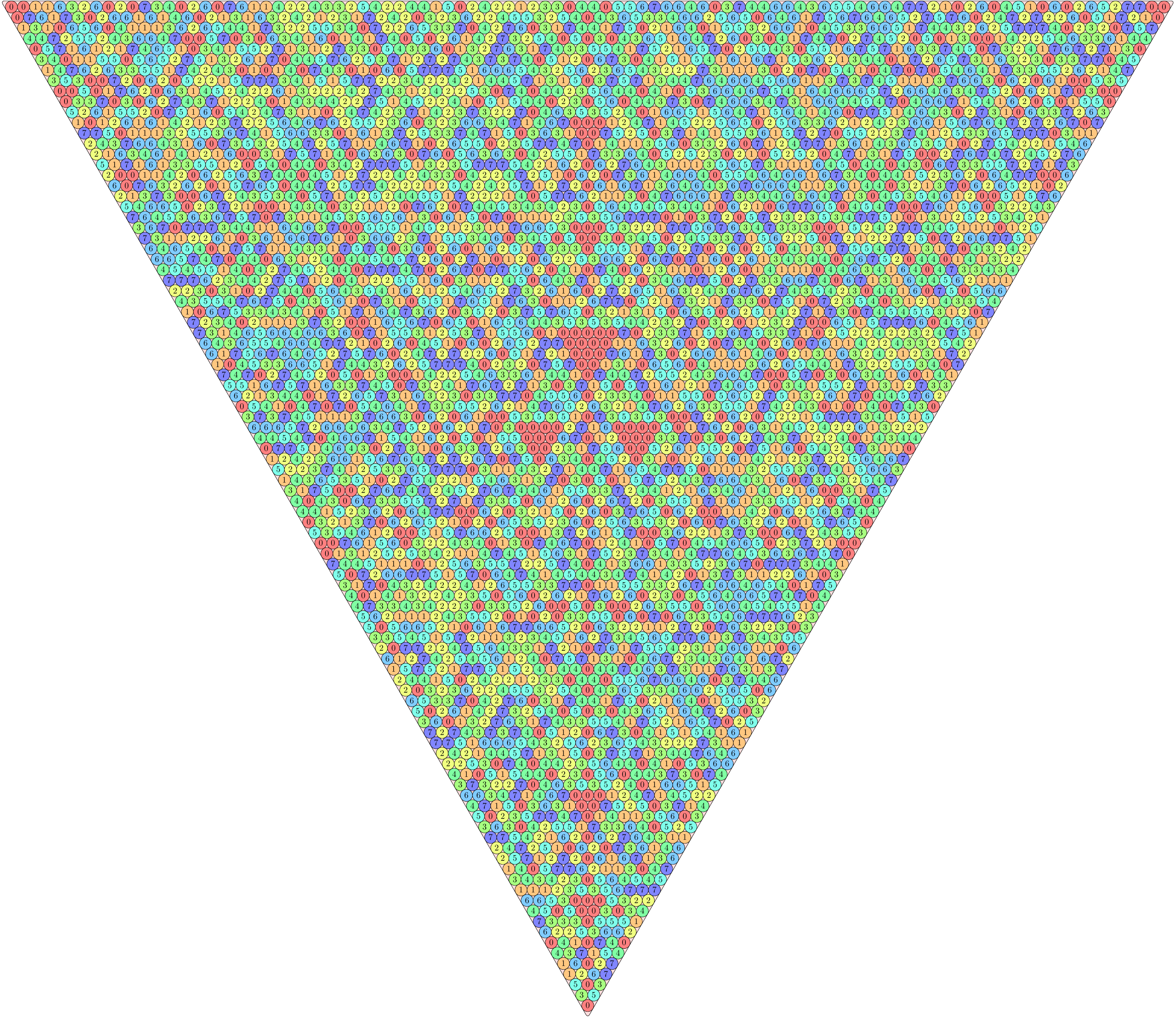} \\
\includegraphics{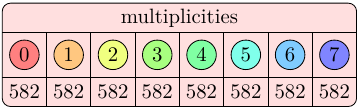} \\
\end{tabular}
}
\caption{$\ST{A_8}$ in $\Zn{8}$}\label{fig07}
\end{figure}

\begin{figure}[htbp]
\centering{
\begin{tabular}{c}
\includegraphics[width=0.974\textwidth]{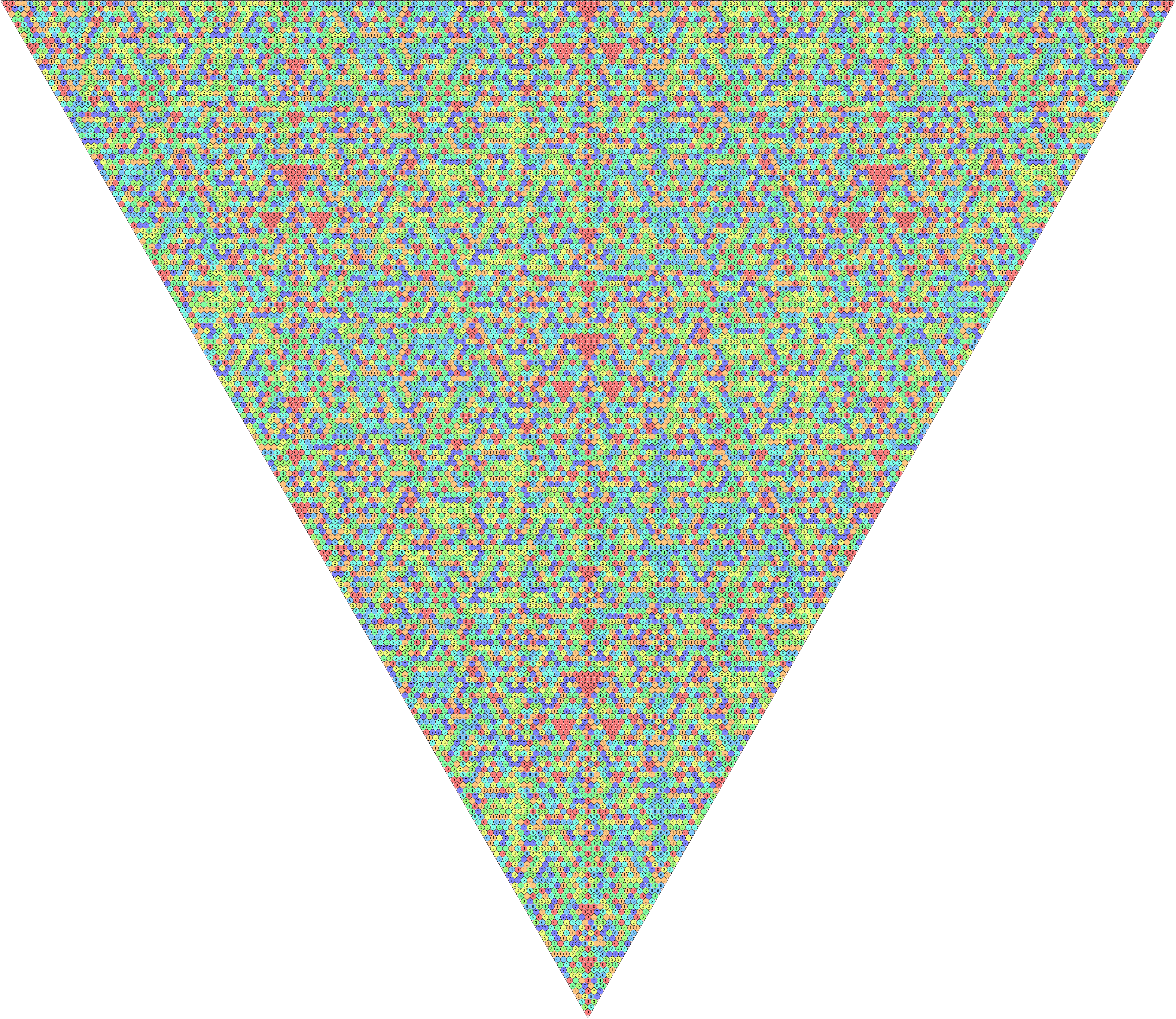} \\
\includegraphics{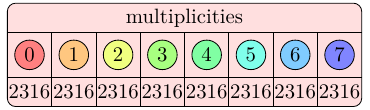} \\
\end{tabular}
}
\caption{$\ST{{A_8}^2}$ in $\Zn{8}$}\label{fig08}
\end{figure}

\section{The odd case}

In this section, we determine a set $\setO$ of 24-tuples of integers $A$ such that, for every odd number $m$, the orbit of the sequence
$$
S=\IAP{\pi_m(A)}{\pi_m(A)\matX{24}}
$$
is $(3m,3m)$-periodic in $\Zn{m}$ and the triangles
$$
\ST{S[3\lambda m]}
$$
are balanced in $\Zn{m}$, for all non-negative integers $\lambda$.

Recall that, for every positive integer $m$, the set $\setP{m}$ of $24$-tuples of integers $A$ such that the orbit of the sequence $\IAP{\pi_m(A)}{\pi_m(A)\matr{X_{24}}}$ is $(24m,24m)$-periodic in $\Zn{m}$ is
$$
\setP{m} = \left\{ A\in\Z^{24}\ \middle|\ \pi_m(A)\in\Lker_{m}\M{24}{24m}\right\},
$$
where $\M{24}{24m} = \W{24}{24m}+\matX{24}\T{24}{24m} = \C{24}{24m}-\matI{24}+\matX{24}\T{24}{24m}$.

We begin by showing that
$$
\bigcap_{m\in\N^*}\setP{m} = \bigcap_{\stackrel{m\in\N^*}{m\ \text{odd}}}\setP{m} = \left\langle A_1 , A_2 \right\rangle,
$$
where $A_1$ and $A_2$ are the $24$-tuples of integers
\begin{equation*}
\resizebox{\textwidth}{!}{$
\begin{array}{l}
A_1 = (1,0,-1)^8, \\[2ex]
A_2 = (0,1,-1,-1,3,-2,-2,5,-3,-3,7,-4,-4,9,-5,-5,11,-6,-6,13,-7,-7,15,-8). \\
\end{array}
$}
\end{equation*}
First, for all positive integers $m$, the inclusion $\left\langle A_1 , A_2 \right\rangle \subset \setP{m}$ holds.

\begin{prop}\label{prop30}
For all positive integers $m$, we have
$$
\left\langle A_1 , A_2 \right\rangle \subset \setP{m}.
$$ 
\end{prop}

\begin{proof}
First, it is easy to see that $\IAP{A_i}{A_i\matr{X_{24}}}$ is a $3$-interlaced arithmetic progression, for all $i\in\{1,2\}$. Indeed, we have
$$
\IAP{A_1}{A_1\matr{X_{24}}} = \IAP{(1,0,-1)}{(0,0,0)}
$$
and
$$
\IAP{A_2}{A_2\matr{X_{24}}} = \IAP{(0,1,-1)}{(-1,2,-1)}.
$$
Let $A_1'=(1,0,-1)$ and $A_2'=(0,1,-1)$ in this proof. Moreover, since
$$
\C{3}{3} =
\begin{pmatrix}
2 & 3 & 3 \\
3 & 2 & 3 \\
3 & 3 & 2 \\
\end{pmatrix}
\quad\text{and}\quad
\T{3}{3} =
\begin{pmatrix}
1 & 3 & 3 \\
0 & 1 & 3 \\
0 & 0 & 1 \\
\end{pmatrix},
$$
we obtain that
$$
\IA{3}{3} = 
\begin{pmatrix}
\left(\C{3}{3}+\matr{I_{3}}\right)^2 & \matr{0} \\
\T{3}{3}\left(\C{3}{3}+\matr{I_{3}}\right) &  \C{3}{3}+\matr{I_{3}}
\end{pmatrix} = 
\begin{pmatrix}
27 & 27 & 27 & 0 & 0 & 0 \\
27 & 27 & 27 & 0 & 0 & 0 \\
27 & 27 & 27 & 0 & 0 & 0 \\
21 & 21 & 21 & 3 & 3 & 3 \\
12 & 12 & 12 & 3 & 3 & 3 \\
3 & 3 & 3 & 3 & 3 & 3 \\
\end{pmatrix}
$$
and thus
$$
\begin{pmatrix}
A_i' & A_i'\matr{X_{3}}
\end{pmatrix}
\IA{3}{3} = \matr{0},
$$
for all $i\in\{1,2\}$. It follows from Theorem~\ref{thm4} that the orbit of $\IAP{A_i'}{A_i'\matr{X_{3}}}$ is $(3,3)$-interlaced doubly arithmetic, for all $i\in\{1,2\}$. Therefore, for all $i\in\{1,2\}$, the orbit of the sequence $\IAP{\pi_m(A_i')}{\pi_m(A_i')\matr{X_{3}}}$ is $(3m,3m)$-periodic and thus $(24m,24m)$-periodic. This leads to
$$
A_i \M{24}{24m} \equiv \matr{0}\pmod{m}\quad\text{and}\quad A_i\in\setP{m},
$$
for all $i\in\{1,2\}$. This completes the proof.
\end{proof}

Next, with the aid of a computer, we determine $\Lker_p\M{24}{24p}$ for the first prime values of $p$.

\begin{comp}\label{comp1}
For all prime number $p<3000$, we have
$$
\Lker_p\M{24}{24p} = \left\langle \pi_p(A_1) , \pi_p(A_2) \right\rangle,
$$
except for the eight values $p\in\{2, 3, 5, 7, 13, 17, 73, 241\}$, where $\dim\Lker_p\M{24}{24p}>2$ is given in Table~\ref{tab8}.
\end{comp}

\begin{table}[htbp]
\begin{center}
\renewcommand{\arraystretch}{1.5}
\begin{tabular}{|c||c|c|c|c|c|c|c|c|c|c|c|c|c|}
\hline
 $p$ & $2$ & $3$ & $5$ & $7$ & $13$ & $17$ & $73$ & $241$ \\
\hline
$\dim\Lker_p\M{24}{24p}$ & $16$ & $21$ & $23$ & $11$ & $11$ & $5$ & $8$ & $5$ \\
\hline
\end{tabular}
\renewcommand{\arraystretch}{1}
\caption{The values of $\dim\Lker_p\M{24}{24p}$ for $p\in\{2, 3, 5, 7, 13, 17, 73, 241\}$}\label{tab8}
\end{center}
\end{table}

Now, we consider $p=11$, the least prime number such that
$$
\Lker_p\M{24}{24p} = \left\langle \pi_p(A_1) , \pi_p(A_2) \right\rangle
$$
and we show the following

\begin{prop}\label{prop29}
For $p=11$, we have
$$
\bigcap_{u\in\N}\setP{{11}^u} = \left\langle A_1 , A_2 \right\rangle.
$$
\end{prop}

The proof is based on the following lemmas.

\begin{lem}\label{lem14}
For all positive integers $u$, we have
$$
\C{24}{24.11^u} \equiv \C{24}{24.11} \pmod{11}.
$$
\end{lem}

\begin{proof}
By induction on $u\ge1$. By direct computation, we obtain that
\begin{equation*}
\resizebox{\textwidth}{!}{$
\C{24}{24.11^2} \equiv \C{24}{24.11} \equiv \mathbf{Circ}\left((2, 2, 1, 0, 0, 0, 0, 0, 0, 0, 0, 2, 4, 2, 0, 0, 0, 0, 0, 0, 0, 0, 1, 2)\right) \pmod{11}
$}
\end{equation*}
and the equivalence is true for $u\in\{1,2\}$. Suppose now that
\begin{equation}\label{eq*12}
\C{24}{24.11^u} \equiv \C{24}{24.11} \pmod{11},
\end{equation}
for some positive integer $u\ge1$. Then, for $u+1$, from Corollary~\ref{cor*2}, we know that
$$
\C{24}{24.11^{u+1}} = \C{24}{24.11^u}^{11}\quad\text{and}\quad \C{24}{24.11^2}=\C{24}{24.11}^2.
$$
Using \eqref{eq*12}, this leads to
$$
\C{24}{24.11^{u+1}} = \C{24}{24.11^u}^{11}
\begin{array}[t]{l}
\equiv \C{24}{24.11}^{11} \\[1.5ex]
= \C{24}{24.11^2} \equiv \C{24}{24.11} \pmod{11}.
\end{array}
$$
This completes the proof.
\end{proof}

\begin{lem}\label{lem15}
For all positive integers $u$, we have
$$
\T{24}{24.11^u} \equiv \T{24}{24.11} \pmod{11}.
$$
\end{lem}

\begin{proof}
By induction on $u\ge1$. By direct computation, we obtain that
$$
\T{24}{24.11^2} \equiv \T{24}{24.11} \equiv \mathbf{Toepl}\left(S_1,S_2\right) \pmod{11},
$$
where
$$
\begin{array}{l}
S_1 = (0, 10, 10, 0, 0, 0, 0, 0, 0, 0, 0, 0, 9, 9, 0, 0, 0, 0, 0, 0, 0, 0, 0, 10),\\
S_2 = ( 1, 1, 0, 0, 0, 0, 0, 0, 0, 0, 0, 2, 2, 0, 0, 0, 0, 0, 0, 0, 0, 0, 1), \\
\end{array}
$$
and the equivalence is true for $u\in\{1,2\}$.
Suppose now that
\begin{equation}\label{eq*13}
\T{24}{24.11^u} \equiv \T{24}{24.11} \pmod{11},
\end{equation}
for some positive integer $u\ge1$. Then, for $u+1$, from Corollary~\ref{cor*2}, we know that
$$
\T{24}{24.11^{u+1}} = \sum_{i=0}^{10}\C{24}{24.11^u}^i\T{24}{24.11^u}\C{24}{24.11^u}^{10-i}
$$
and
$$
\T{24}{24.11^{2}} = \sum_{i=0}^{10}\C{24}{24.11}^i\T{24}{24.11}\C{24}{24.11}^{10-i}.
$$
Using \eqref{eq*13} and Lemma~\ref{lem14}, we obtain that
$$
\T{24}{24.11^{u+1}}
\begin{array}[t]{l}
= \displaystyle\sum_{i=0}^{10}\C{24}{24.11^u}^{i}\T{24}{24.11^u}\C{24}{24.11^u}^{10-i} \\ \ \\
\equiv \displaystyle\sum_{i=0}^{10}\C{24}{24.11}^{i}\T{24}{24.11}\C{24}{24.11}^{10-i} \\ \ \\
= \T{24}{24.11^2} \equiv \T{24}{24.11} \pmod{11}. \\
\end{array}
$$
This completes the proof.
\end{proof}

\begin{lem}\label{lem*10}
For all positive integers $u$, we have
$$
\M{24}{24.11^u} \equiv \M{24}{24.11} \pmod{11}.
$$
\end{lem}

\begin{proof}
From Lemmas~\ref{lem14} and \ref{lem15}, we obtain that
$$
\M{24}{24.11^{u}}
\begin{array}[t]{l}
= \C{24}{24.11^{u}}-\matI{24}+\matX{24}\T{24}{24.11^{u}} \\
\equiv \C{24}{24.11}-\matI{24}+\matX{24}\T{24}{24.11} = \M{24}{24.11} \pmod{11}. \\
\end{array}
$$
This completes the proof.
\end{proof}

We are now ready to prove Proposition~\ref{prop29}.

\begin{proof}[Proof of Proposition~\ref{prop29}]
First, we know from Proposition~\ref{prop30} that
$$
\left\langle A_1,A_2 \right\rangle \subset \bigcap_{u\in\N}\setP{11^u}.
$$
Now, let $A=(a_1,\ldots,a_{24})\in\bigcap_{u\in\N}\setP{11^u}$ and consider
$$
T = A-a_1A_1-a_2A_2 \in\bigcap_{u\in\N}\setP{11^u}.
$$
If $T=(t_1,\ldots,t_{24})$, it is easy to see that $t_1=t_2=0$, by definition of $T$. We will show, by induction on $u\ge1$, that $T$ is divisible by $11^u$. For $u=1$, since $T\in\setP{11}$, we have
$$
T\M{24}{24.11} \equiv \matr{0}\pmod{11}.
$$
Moreover, by Computational Result~\ref{comp1}, we have
$$
\Lker_{11}\M{24}{24.11} = \left\langle \pi_{11}(A_1) , \pi_{11}(A_2) \right\rangle.
$$
Since $t_1=t_2=0$, we deduce that
$$
T\equiv \matr{0}\pmod{11}.
$$
Suppose now that $T$ is divisible by $11^u$, for some positive integer $u\ge1$. Let
$$
T = 11^uT_{11^u},
$$
with $T_{11^u}\in\Z^{24}$. Since $T\in\setP{11^{u+1}}$, we have
$$
11^uT_{11^u}\M{24}{24.11^{u+1}} \equiv \matr{0}\pmod{11^{u+1}}
$$
and thus
$$
T_{11^u}\M{24}{24.11^{u+1}} \equiv \matr{0}\pmod{11}.
$$
Moreover, from Lemmas~\ref{lem*10}, we know that
$$
\M{24}{24.11^{u+1}} \equiv \M{24}{24.11} \pmod{11}.
$$
It follows that
$$
T_{11^u}\M{24}{24.11} \equiv \matr{0}\pmod{11}
$$
and thus
$$
\pi_{11}(T_{11^u})\in\Lker_{11}\M{24}{24.11} = \left\langle \pi_{11}(A_1) , \pi_{11}(A_2) \right\rangle.
$$
Since $t_1=t_2=0$, we deduce that
$$
T_{11^u}\equiv\matr{0}\pmod{11}\quad\text{and}\quad T\equiv\matr{0}\pmod{11^{u+1}}.
$$
Finally, since $T$ is divisible by $11^u$, for all positive integers $u$, we deduce that $T=\matr{0}$. Therefore,
$$
A = a_1A_1+a_2A_2 \in\left\langle A_1,A_2 \right\rangle.
$$
This completes the proof.
\end{proof}

From Propositions~\ref{prop30} and \ref{prop29}, it is straightforward to obtain the following

\begin{thm}\label{thm*3}
$$
\bigcap_{m\in\N^*}\setP{m} = \bigcap_{\stackrel{m\in\N^*}{m\ \text{odd}}}\setP{m} = \left\langle A_1 , A_2 \right\rangle.
$$
\end{thm}

\begin{proof}
First, we know from Proposition~\ref{prop30} that
$$
\left\langle A_1,A_2 \right\rangle \subset \bigcap_{m\in\N^*}\setP{m}.
$$
Moreover, from Proposition~\ref{prop29}, we have
$$
\bigcap_{m\in\N^*}\setP{m} \subset \bigcap_{\stackrel{m\in\N^*}{m\ \text{odd}}}\setP{m} \subset \bigcap_{u\in\N}\setP{{11}^u} = \left\langle A_1 , A_2 \right\rangle.
$$
This completes the proof.
\end{proof}

\begin{rem}\label{rem*1}
For any $A\in\left\langle A_1 , A_2 \right\rangle$ and any positive integer $m$, the orbit of
$$
\IAP{\pi_m(A)}{\pi_m(A)\matr{X_{24}}}
$$
is $(3m,3m)$-periodic in $\Zn{m}$.
\end{rem}

Now, we show that for any $A=a_1A_1+a_2A_2\in\left\langle A_1 , A_2 \right\rangle$ and any even number $m$, the triangles $\ST{S[3\lambda m]}$, where $S=\IAP{\pi_m(A)}{\pi_m(A)\matX{24}}$, are never balanced in $\Zn{m}$, for all positive integers $\lambda$.

\begin{prop}\label{prop*7}
Let $A=a_1A_1+a_2A_2\in\left\langle A_1 , A_2 \right\rangle$. For all positive integers $\lambda$, the triangle
$$
\ST{\IAP{\pi_{2}(A)}{\pi_2(A)\matX{24}}[6\lambda]}
$$
is never balanced in $\Zn{2}$.
\end{prop}

\begin{proof}
First, we consider the sequence
$$
S=\IAP{\pi_2(A)}{\pi_2(A)\matX{24}}.
$$
Let $\lambda$ be a positive integer. As mentioned in Remark~\ref{rem*1}, the orbit $\orb{S}$ is $(6,6)$-periodic. Therefore, as in the proof of Proposition~\ref{prop4}, the triangle $\ST{S[6\lambda]}$ can be decomposed into $\lambda$ copies of $T=\ST{S[6]}$ and $\binom{\lambda}{2}$ periods $P$ (see Figure~\ref{fig01}). Moreover, it is clear that there are only four possibilities for the projection of $A$ into $\Zn{2}$:
$$
\begin{array}{ll}
 & \pi_2(A) = (0\cdots0) ,\\
\text{or} & \pi_2(A) = \pi_2(A_1) = (101)^8, \\
\text{or} & \pi_2(A) = \pi_2(A_2) = (011110)^4, \\
\text{or} & \pi_2(A) = \pi_2(A_1+A_2) = (110011)^4. \\
\end{array}
$$
\setcounter{case}{0}
\begin{case}
If $\pi_2(A) = (0\cdots0)$, then it is clear that the triangle $\ST{S[6\lambda]}$ uniquely contains $0$'s and cannot be balanced in $\Zn{2}$.
\end{case}
\begin{case}
If $\pi_2(A) = \pi_2(A_1) = (101)^8$, then $T=\ST{S[6]}$ and the period $P$ are
\begin{center}
\includegraphics{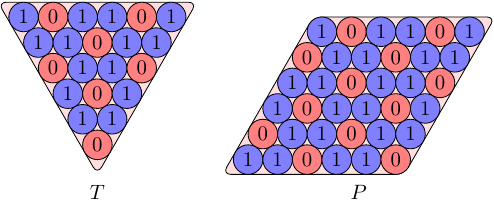}
\end{center}
Since $\mf{T}(1)=2\mf{T}(0)$ and $\mf{P}(1)=2\mf{P}(0)$, we obtain that the multiplicity function of $\nabla=\ST{S[6\lambda]}$ verifies
$$
\mf{\nabla}(1) = 2\mf{\nabla}(0).
$$
Therefore, the triangle $\ST{S[6\lambda]}$ is not balanced in $\Zn{2}$, in this case.
\end{case}
\begin{case}
If $\pi_2(A) = \pi_2(A_2) = (011110)^4$, then $T=\ST{S[6]}$ and the period $P$ are
\begin{center}
\includegraphics{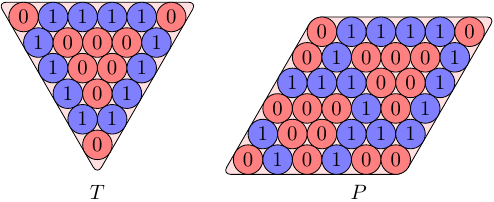}
\end{center}
Since $\mf{T}(1)=\mf{T}(0)+3$ and $\mf{P}(1)=\mf{P}(0)$, we obtain that the multiplicity function of $\nabla=\ST{S[6\lambda]}$ verifies
$$
\mf{\nabla}(1) = \mf{\nabla}(0)+3\lambda > \mf{\nabla}(0).
$$
Therefore, the triangle $\ST{S[6\lambda]}$ is not balanced in $\Zn{2}$, in this case.
\end{case}
\begin{case}
If $\pi_2(A) = \pi_2(A_1+A_2) = (110011)^4$, then $T=\ST{S[6]}$ and the period $P$ are
\begin{center}
\includegraphics{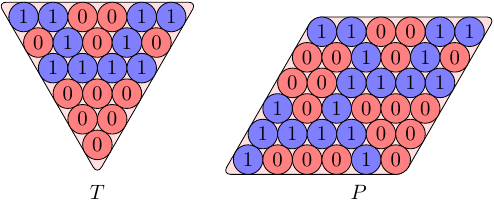}
\end{center}
Since $\mf{T}(0)=\mf{T}(1)+1$ and $\mf{P}(0)=\mf{P}(1)$, we obtain that the multiplicity function of $\nabla=\ST{S[6\lambda]}$ verifies
$$
\mf{\nabla}(0) = \mf{\nabla}(1)+\lambda > \mf{\nabla}(1).
$$
Therefore, the triangle $\ST{S[6\lambda]}$ is not balanced in $\Zn{2}$, in this case.
\end{case}
This completes the proof.
\end{proof}

\begin{cor}\label{cor*3}
Let $A=a_1A_1+a_2A_2\in\left\langle A_1 , A_2 \right\rangle$ and let $m$ be an even number. For all positive integers $\lambda$, the triangle
$$
\ST{\IAP{\pi_{m}(A)}{\pi_{m}(A)\matX{24}}[3\lambda m]}
$$
is never balanced in $\Zn{m}$.
\end{cor}

\begin{proof}
If $\nabla=\ST{\IAP{\pi_{m}(A)}{\pi_{m}(A)\matX{24}}[3\lambda m]}$ is balanced in $\Zn{m}$, with $m$ even, we know from Theorem~\ref{thm6} that
$$
\pi_2(\nabla) = \ST{\IAP{\pi_{2}(A)}{\pi_{2}(A)\matX{24}}[3\lambda m]}
$$
is balanced in $\Zn{2}$, in contradiction with Proposition~\ref{prop*7}.
\end{proof}

\begin{cor}\label{cor*4}
There does not exist $A\in\Z^{24}$ such that, for every positive integer $m$, the orbit of the sequence
$$
S=\IAP{\pi_m(A)}{\pi_m(A)\matX{24}}
$$
is $(24m,24m)$-periodic in $\Zn{m}$ and the triangles $\ST{S[24\lambda m]}$ are balanced in $\Zn{m}$, for all non-negative integers $\lambda$.
\end{cor}

\begin{proof}
Let $A\in\Z^{24}$ such that, for every positive integer $m$, the orbit of the sequence
$$
S=\IAP{\pi_m(A)}{\pi_m(A)\matX{24}}
$$
is $(24m,24m)$-periodic in $\Zn{m}$, i.e.,
$$
A \in \bigcap_{m\in\N^*}\setP{m}.
$$
From Theorem~\ref{thm*3}, we obtain that
$$
A \in \bigcap_{m\in\N^*}\setP{m} = \left\langle A_1,A_2\right\rangle.
$$
We deduce that $\orb{S}$ is in fact $(3m,3m)$-periodic, for all positive integers $m$. Finally, when $m$ is even, we know from Corollary~\ref{cor*3} that the triangles $\ST{S[3\lambda m]}$ are never balanced in $\Zn{m}$, for all positive integers $\lambda$.
\end{proof}

Let $A=a_1A_1+a_2A_2\in\left\langle A_1 , A_2 \right\rangle$. We continue by determining a necessary condition on integers $a_1$ and $a_2$ such that the triangles $\ST{S[3\lambda m]}$, where $S=\IAP{\pi_m(A)}{\pi_m(A)\matX{24}}$, are balanced in $\Zn{m}$, for all non-negative integers $\lambda$ and all odd numbers $m$.

\begin{prop}\label{prop*8}
Let $m$ be an odd number. For all positive integers $\lambda$, the triangle
$$
\ST{\IAP{\pi_{m}(A_1)}{\pi_m(A_1)\matX{24}}[3\lambda]}
$$
is never balanced in $\Zn{m}$, except for $m\in\{1,3\}$.
\end{prop}

\begin{proof}
We consider the sequence
$$
S=\IAP{A_1}{A_1\matX{24}} = {A_1}^\infty = (1,0,-1)^\infty.
$$
By induction on $\alpha\ge0$, it is easy to see that
$$
\begin{array}{l}
\ider{3\alpha}{S} = S = (1,0,-1)^\infty, \\
\ider{3\alpha+1}{S} = (-1,1,0)^\infty, \\
\ider{3\alpha+2}{S} = (0,-1,1)^\infty. \\
\end{array}
$$
Therefore, the orbit $\orb{S}$ is $(3,3)$-periodic.

When $m\ge5$, since any element of the orbit of $\pi_m(S)$ is in $\{\pi_m(-1),\pi_m(0),\pi_m(1)\}$, we deduce that there is no balanced triangle in $\orb{\pi_m(S)}$.

%
When $m=3$, for any positive integer $\lambda$, since the orbit $\orb{\pi_3(S)}$ is $(3,3)$-periodic, we know from the proof of Proposition~\ref{prop4} that the triangle $\ST{\pi_3(S)[3\lambda]}$ can be decomposed into $\lambda$ copies of $T=\ST{\pi_3(S)[3]}$ and $\binom{\lambda}{2}$ periods $P$, as depicted in Figure~\ref{fig01}.
\begin{center}
\includegraphics{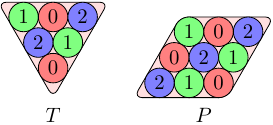}
\end{center}
Since $\mf{T}(0)=\mf{T}(1)=\mf{T}(2)$ and $\mf{P}(0)=\mf{P}(1)=\mf{P}(2)$, we obtain that the multiplicity function of $\nabla=\ST{\pi_{3}(S)[3\lambda]}$ verifies
$$
\mf{\nabla}(0) = \mf{\nabla}(1) = \mf{\nabla}(2).
$$
Therefore, the triangle $\ST{\pi_{3}(S)[3\lambda]}$ is balanced in $\Zn{3}$, for all positive integers $\lambda$.

Finally, for $m=1$, the result is clear since any triangle is balanced in $\Zn{1}$.
\end{proof}

\begin{cor}
Let $A=a_1A_1+a_2A_2\in\left\langle A_1,A_2\right\rangle$ and let $m$ be an odd number. If the triangle
$$
\ST{\IAP{\pi_m(A)}{\pi_m(A)\matX{24}}[3\lambda]}
$$
is balanced in $\Zn{m}$, for a certain positive integer $\lambda$, then we have $\gcd(a_1,a_2,m)=1$ and $\gcd(a_2,m)\in\{1,3\}$.
\end{cor}

\begin{proof}
Suppose that the triangle 
$$
\nabla = \ST{\IAP{\pi_m(A)}{\pi_m(A)\matX{24}}[3\lambda]}
$$
is balanced and let $q=\gcd(a_1,a_2,m)$. Since
$$
\pi_q(A) = (0\cdots0)
$$
we deduce that the triangle $\pi_q(\nabla)$ uniquely contains $0$'s. Moreover, since $\nabla$ is balanced in $\Zn{m}$, we know from Theorem~\ref{thm6} that $\pi_q(\nabla)$ is balanced in $\Zn{q}$. The only possibility is that $q=\gcd(a_1,a_2,m)=1$.

Let $q'=\gcd(a_2,m)$. Similarly, from Theorem~\ref{thm6}, since $\nabla$ is balanced in $\Zn{m}$, we know that $\pi_{q'}(\nabla)$ is balanced in $\Zn{q'}$. Moreover,
$$
\pi_{q'}(A) = \pi_{q'}(a_1)\pi_{q'}(A_1)
$$
and thus
$$
\pi_{q'}(\nabla) = \pi_{q'}(a_1)\ST{\IAP{\pi_{q'}(A_1)}{\pi_{q'}(A_1)\matX{24}}[3\lambda]}.
$$
Since $\gcd(a_1,q')=1$, we know that $\pi_{q'}(a_1)$ is invertible in $\Zn{q'}$. Therefore, since $\pi_{q'}(\nabla)$ is balanced, we deduce that the triangle
$$
\ST{\IAP{\pi_{q'}(A_1)}{\pi_{q'}(A_1)\matX{24}}[3\lambda]}
$$
is balanced in $\Zn{q'}$. Using Proposition~\ref{prop*8}, we conclude that $q'=\gcd(a_2,m)\in\{1,3\}$, as announced.
\end{proof}

\begin{thm}\label{thm8}
Let $A=a_1A_1+a_2A_2\in\left\langle A_1,A_2\right\rangle$ and let $m$ be an odd number such that $\gcd(a_1,a_2,m)=1$ and $\gcd(a_2,m)\in\{1,3\}$. Then, the triangle
$$
\ST{\IAP{\pi_m(A)}{\pi_m(A)\matX{24}}[3\lambda m]}
$$
is balanced in $\Zn{m}$, for all non-negative integers $\lambda$.
\end{thm}

\begin{proof}
Let $S=\IAP{A}{A\matX{24}}$. As already seen in the proof of Proposition~\ref{prop30}, we have
$$
S=\IAP{(a_1,a_2,-a_1-a_2)}{(-a_2,2a_2,-a_2)}
$$
and we know that its orbit is $(3,3)$-interlaced doubly arithmetic. Moreover, since
\begin{equation*}
\resizebox{\textwidth}{!}{$
\C{3}{3}+\matI{3}+\matX{3}\T{3}{3} = 
\begin{pmatrix}
 2 & 3 & 3 \\
 3 & 2 & 3 \\
 3 & 3 & 2 \\
\end{pmatrix}
+
\begin{pmatrix}
1 & 0 & 0 \\
0 & 1 & 0 \\
0 & 0 & 1 \\
\end{pmatrix}
+
\begin{pmatrix}
1 & 0 & 1 \\
0 & 2 & 0 \\
1 & 0 & 1 \\
\end{pmatrix}
\begin{pmatrix}
1 & 3 & 3 \\
0 & 1 & 3 \\
0 & 0 & 1 \\
\end{pmatrix}
=
\begin{pmatrix}
4 & 6 & 7 \\
3 & 5 & 9 \\
4 & 6 & 7 \\
\end{pmatrix}
$}
\end{equation*}
and thus
$$
\begin{pmatrix}
a_1 & a_2 & -a_1-a_2 \\
\end{pmatrix}
\left(\C{3}{3}+\matI{3}+\matX{3}\T{3}{3}\right) =
\begin{pmatrix}
-a_2 & -a_2 & 2a_2 \\
\end{pmatrix},
$$
we deduce, from Theorem~\ref{thm4} and Lemma~\ref{lem*9}, that 
$$
\orb{S} = \IDAP{\matA}{\matD{1}}{\matD{2}},
$$
where
$$
\matD{1} = \Dmat{3}{(-a_2,2a_2,-a_2)} = \Circ{-a_2,2a_2,-a_2}
$$
and
$$
\matD{2} = \Dmat{3}{(a_2,a_2,-2a_2)} = \Circ{a_2,a_2,-2a_2}.
$$
Let $\lambda$ be a positive integer. We consider the Steinhaus triangle
$$
\nabla = \ST{\pi_m(S)}[3\lambda m]= \left(a_{i,j}\right)_{(i,j)\in T_{3\lambda m}}.
$$
From Proposition~\ref{prop28}, we know that $\nabla$ can be decomposed into $9$ subtriangles $\nabla_{r,s}$ that are the arithmetic triangles
$$
\nabla_{r,s} = \left(a_{r+3i,s+3j}\right)_{(i,j)\in T_{n_{r,s}}} = \AT{a_{r,s}}{d^{\,(1)}_{r,s}}{d^{\,(2)}_{r,s}}{n_{r,s}},
$$
where
$$
n_{r,s} = \left\{\begin{array}{ll}
 \lambda m & \text{if}\ r+s\in\{0,1,2\},\\
 \lambda m-1 & \text{if}\ r+s\in\{3,4\},\\
\end{array}\right.
$$
for all $(r,s)\in\{0,1,2\}^2$. Moreover, since
$$
\matr{D_1} = \left(d^{\,(1)}_{r,s}\right)_{0\le r,s\le 2} = \Circ{-a_2,2a_2,-a_2}
$$
and
$$
\matr{D_2} = \left(d^{\,(2)}_{r,s}\right)_{0\le r,s\le 2} = \Circ{a_2,a_2,-2a_2}),
$$
we deduce that
$$
\left(d^{\,(1)}_{r,s},d^{\,(2)}_{r,s}\right) = 
\left\{\begin{array}{cl}
 (-a_2,a_2) & \text{if}\ r-s\equiv0\pmod{3}, \\
 (-a_2,-2a_2) & \text{if}\ r-s\equiv1\pmod{3}, \\
 (2a_2,a_2) & \text{if}\ r-s\equiv2\pmod{3}, \\
\end{array}\right.
$$
for all $(r,s)\in\{0,1,2\}^2$. We distinguish different cases according to the residue class of $r-s$ modulo $3$.

\setcounter{case}{0}
\begin{case}
If $r-s\equiv0\pmod{3}$, then
$$
d^{\,(1)}_{r,s} = -a_2,\ d^{\,(2)}_{r,s} = a_2\ \text{and}\ d^{\,(2)}_{r,s}-d^{\,(1)}_{r,s} = 2a_2.
$$
Since
$$
\gcd(d^{\,(1)}_{r,s},m) = \gcd(d^{\,(2)}_{r,s},m) = \gcd(d^{\,(2)}_{r,s}-d^{\,(1)}_{r,s},m) = \gcd(a_2,m)
$$
and $n_{r,s}\in\left\{\lambda m-1,\lambda m\right\}$, we deduce from Theorem~\ref{thm5} that
$$
\mf{\nabla_{r,s}}(x+\gcd(a_2,m)) = \mf{\nabla_{r,s}}(x),
$$
for all $x\in\Zn{m}$, in this case.
\end{case}

\begin{case}
If $r-s\equiv1\pmod{3}$, then
$$
d^{\,(1)}_{r,s} = -a_2,\ d^{\,(2)}_{r,s} = -2a_2\ \text{and}\ d^{\,(2)}_{r,s}-d^{\,(1)}_{r,s} = -a_2.
$$
Since
$$
\gcd(d^{\,(1)}_{r,s},m) = \gcd(d^{\,(2)}_{r,s},m) = \gcd(d^{\,(2)}_{r,s}-d^{\,(1)}_{r,s},m) = \gcd(a_2,m)
$$
and $n_{r,s}\in\left\{\lambda m-1,\lambda m\right\}$, we deduce from Theorem~\ref{thm5} that
$$
\mf{\nabla_{r,s}}(x+\gcd(a_2,m)) = \mf{\nabla_{r,s}}(x),
$$
for all $x\in\Zn{m}$, in this case.
\end{case}

\begin{case}
If $r-s\equiv2\pmod{3}$, then
$$
d^{\,(1)}_{r,s} = 2a_2,\ d^{\,(2)}_{r,s} = a_2\ \text{and}\ d^{\,(2)}_{r,s}-d^{\,(1)}_{r,s} = -a_2.
$$
Since
$$
\gcd(d^{\,(1)}_{r,s},m) = \gcd(d^{\,(2)}_{r,s},m) = \gcd(d^{\,(2)}_{r,s}-d^{\,(1)}_{r,s},m) = \gcd(a_2,m)
$$
and $n_{r,s}\in\left\{\lambda m-1,\lambda m\right\}$, we deduce from Theorem~\ref{thm5} that
$$
\mf{\nabla_{r,s}}(x+\gcd(a_2,m)) = \mf{\nabla_{r,s}}(x),
$$
for all $x\in\Zn{m}$, in this case.
\end{case}

Therefore, we have
$$
\mf{\nabla_{r,s}}(x+\gcd(a_2,m)) = \mf{\nabla_{r,s}}(x),
$$
for all $x\in\Zn{m}$, in any subtriangle $\nabla_{r,s}$, for all $(r,s)\in\{0,1,2\}^2$. Since
$$
\nabla = \bigsqcup_{(r,s)\in\{0,1,2\}^2} \nabla_{r,s},
$$
we deduce that
$$
\mf{\nabla}(x+\gcd(a_2,m)) = \mf{\nabla}(x),
$$
for all $x\in\Zn{m}$.

Let $q=\gcd(a_2,m)$. Since
$$
\pi_q(A) = \pi_q(a_1)\pi_q(A_1),
$$
we have
$$
\pi_q(\nabla) = \pi_q(a_1)\ST{\IAP{\pi_q(A_1)}{\pi_q(A_1)\matX{24}}[3\lambda m]}.
$$
Since $q\in\{1,3\}$, we know from Proposition~\ref{prop*8} that
$$
\ST{\IAP{\pi_q(A_1)}{\pi_q(A_1)\matX{24}}[3\lambda m]}
$$
is balanced $\Zn{q}$. Moreover, $\pi_q(a_1)$ is invertible since $\gcd(a_1,q)=\gcd(a_1,a_2,m)=1$. It follows that $\pi_q(\nabla)$ is balanced in $\Zn{q}$.
 
Finally, since $\pi_q(\nabla)$ is balanced in $\Zn{q}$ and since $\mf{\nabla}(x+q)=\mf{\nabla}(x)$, for all $x\in\Zn{m}$, by the previous result, we obtain from Theorem~\ref{thm6} that the triangle $\nabla$ is balanced in $\Zn{m}$.
\end{proof}

\begin{cor}\label{cor*5}
Let $\setO$ be the infinite set of $24$-tuples of integers defined by
$$
\setO = \left\{ a_1A_1+a_2A_2\ \middle|\ a_1\in\Z\ \text{and}\ a_2=\pm 3^\alpha2^\beta,\text{with}\ \alpha\in\{0,1\},\beta\in\N,\text{s.t.}\ 3\nmid\gcd(a_1,a_2)\right\}.
$$
Then, for any $A\in\mathcal{O}$ and for all odd numbers $m$, the orbit of the sequence
$$
S=\IAP{\pi_m(A)}{\pi_m(A)\matX{24}}
$$
is $(3m,3m)$-periodic and the triangles $\ST{S[3\lambda m]}$ are balanced in $\Zn{m}$, for all non-negative integers $\lambda$.
\end{cor}

\begin{proof}
Let $A\in\setO$ and let $m$ be an odd number. As mentioned in Remark~\ref{rem*1}, the orbit $\orb{S}$ is $(3m,3m)$-periodic in $\Zn{m}$. Since $a_2=\pm 3^\alpha2^\beta$, for some $\alpha\in\{0,1\}$ and $\beta\in\N$, it is clear that
$$
\gcd(a_2,m)\in\{1,3\}.
$$
Moreover, since $3\nmid\gcd(a_1,a_2)$ and $\gcd(a_2,m)\in\{1,3\}$, we obtain that
$$
\gcd(a_1,a_2,m)=1.
$$
Therefore, by Theorem~\ref{thm8}, we obtain that the triangles $\ST{S[3\lambda m]}$ are balanced in $\Zn{m}$, for all non-negative integers $\lambda$.
\end{proof}

\section{The even case}

From Corollary~\ref{cor*4}, we know that there does not exist $A\in\Z^{24}$ such that, for every positive integer $m$, the orbit of the sequence
$$
S=\IAP{\pi_m(A)}{\pi_m(A)\matX{24}}
$$
is $(24m,24m)$-periodic in $\Zn{m}$ and the triangles $\ST{S[24\lambda m]}$ are balanced in $\Zn{m}$, for all positive integers $\lambda$. Conversely, for powers of two or for odd numbers, we know that this result works when $A\in\setE$ (see Theorem~\ref{mainthm}) or when $A\in\setO$ (see Corollary~\ref{cor*5}), respectively. In this section, we show a similar result for even numbers with the same odd part.

\begin{thm}\label{thm9}
Let $m$ be an even number and let $\mu$ be its odd part, i.e., the odd number $\mu$ such that $m=2^u\mu$ for a certain positive integer $u$. Let $\setU{\mu}$ be the set defined by
$$
\setU{\mu} = \mu\setE+4\setO',
$$
where
$$
\setO' = \left\{ a_1A_1+a_2A_2\ \middle|\ a_1\in\Z\ \text{and}\ a_2=\pm2^{\alpha},\text{with}\ \alpha\in\N \right\} \subset \setO.
$$
Then, for any $A\in\setU{\mu}$, the orbit of the sequence
$$
S=\IAP{\pi_{m}(A)}{\pi_{m}(A)\matX{24}}
$$
is $(12m,12m)$-periodic and the triangles $\ST{S[12\lambda m]}$ are balanced in $\Zn{m}$, for all non-negative integers $\lambda$.
\end{thm}

The proof is based on the following results.

\begin{prop}\label{prop33}
For any odd number $\mu$, we have
$$
\setU{\mu}\subset\setE.
$$
\end{prop}

\begin{proof}
First, from Proposition~\ref{prop31}, we know that
$$
\mu\setE\subset\setE.
$$
Moreover, since
$$
\begin{array}{l}
A_1 = Y_1-Y_3+Y_4-Y_6+Y_7, \\[1.5ex]
A_2 = Y_2-Y_3-Y_4+3Y_5-2Y_6-2Y_7+5Y_8-8Z_9+8Z_{11}-8Z_{12}+8Z_{14}-8Z_{15},
\end{array}
$$
we have
$$
\setO' \subset \left\langle A_1,A_2\right\rangle \subset \left\langle Y_1,\ldots,Y_8,2Z_9,\ldots,2Z_{16} \right\rangle.
$$
Therefore,
$$
4\setO'\subset\left\langle 4Y_1,\ldots,4Y_8,8Z_9,\ldots,8Z_{16} \right\rangle = \setM.
$$
This leads to
$$
\setU{\mu} = \mu\setE+4\setO' \subset \setE+\setM = \setE.
$$
This completes the proof.
\end{proof}

\begin{prop}\label{prop32}
Let $p$ and $q$ be two positive integers such that $\gcd(p,q)=1$. Then, a sequence $S$ of $\Zn{pq}$ is arithmetic if and only if $\pi_p(S)$ and $\pi_q(S)$ are arithmetic in $\Zn{p}$ and $\Zn{q}$, respectively. Moreover, we have
$$
S=\AP{a}{d}\quad\Longleftrightarrow\quad\left\{\begin{array}{l}
\pi_p(S)=\AP{\pi_p(a)}{\pi_p(d)},\\
\pi_q(S)=\AP{\pi_q(a)}{\pi_q(d)}.
\end{array}\right.
$$
\end{prop}

\begin{proof}
Let $S=\left(u_j\right)_{j\in\Z}\in\Zn{pq}$. First, suppose that $S$ is arithmetic, i.e.,
$$
S=\AP{a}{d}\quad\text{where}\quad a=u_0,\ d=u_1-u_0.
$$
Then,
$$
u_j = a+jd,
$$
for all $j\in\Z$. Let $i\in\{p,q\}$. We have
$$
\pi_i(u_j) = \pi_i(a+jd) = \pi_i(a)+j\pi_i(d),
$$
for all $j\in\Z$. It follows that
$$
\pi_p(S)=\AP{\pi_p(a)}{\pi_p(d)}\quad\text{and}\quad\pi_q(S)=\AP{\pi_q(a)}{\pi_q(d)}.
$$
Conversely, suppose that $\pi_p(S)$ and $\pi_q(S)$ are arithmetic, i.e.,
$$
\pi_p(S)=\AP{a_p}{d_p}\quad\text{and}\quad\pi_q(S)=\AP{a_q}{d_q},
$$
with $(a_p,d_p)\in\left(\Zn{p}\right)^2$ and $(a_q,d_q)\in\left(\Zn{q}\right)^2$. By the Chinese Remainder Theorem, we know that there uniquely exists $(a,d)\in\left(\Zn{pq}\right)^2$ such that
$$
\left\{\begin{array}{l}
 \pi_p(a)=a_p, \\
 \pi_q(a)=a_q, \\
\end{array}\right.
\quad\text{and}\quad
\left\{\begin{array}{l}
 \pi_p(d)=d_p, \\
 \pi_q(d)=d_q. \\
\end{array}\right.
$$
Then, since
$$
\left\{\begin{array}{l}
\pi_p(a+jd)=\pi_p(a)+j\pi_p(d)=a_p+jd_p=\pi_p(u_j),\\
\pi_q(a+jd)=\pi_q(a)+j\pi_q(d)=a_q+jd_q=\pi_q(u_j),
\end{array}\right.
$$
for all $j\in\Z$, we obtain by the CRT again that
$$
u_j = a+jd,
$$
for all $j\in\Z$. Therefore $S=\AP{a}{d}$.
\end{proof}

\begin{cor}\label{cor9}
Let $p$ and $q$ be two positive integers such that $\gcd(p,q)=1$ and let $k_1$ and $k_2$ be two positive integers. Then, a doubly indexed sequence $S$ of $\Zn{pq}$ is $(k_1,k_2)$-interlaced doubly arithmetic if and only if $\pi_p(S)$ and $\pi_q(S)$ are $(k_1,k_2)$-interlaced doubly arithmetic in $\Zn{p}$ and $\Zn{q}$, respectively. Moreover, we have
$$
S=\IDAP{\matA}{\matD{1}}{\matD{2}}\quad\Longleftrightarrow\quad\left\{\begin{array}{l}
\pi_p(S)=\IDAP{\pi_p(\matA)}{\pi_p(\matD{1})}{\pi_p(\matD{2})},\\
\pi_q(S)=\IDAP{\pi_q(\matA)}{\pi_q(\matD{1})}{\pi_q(\matD{2})}.\\
\end{array}\right.
$$
\end{cor}

\begin{proof}
Let $S=\left(u_{i,j}\right)_{(i,j)\in\N\times\Z}$. The result is obtained by applying  Proposition~\ref{prop32} in each row
$$
\left(u_{i,j_0+jk_1}\right)_{j\in\Z},
$$
for all $i\in\N$ and all $j_0\in\{0,\ldots,k_1-1\}$, and in each column
$$
\left(u_{i_0+ik_2,j}\right)_{i\in\N},
$$
for all $j\in\Z$ and all $i_0\in\{0,\ldots,k_2-1\}$.
\end{proof}

We are now ready to prove the main theorem of this section.

\begin{proof}[Proof of Theorem~\ref{thm9}]
Let $m$ be an even number with odd part $\mu$ and let $u$ be the positive integer such that $m=2^u\mu$. Let $A\in\setU{\mu}$ and let $S$ be the sequence of $\Zn{m}$
$$
S=\IAP{\pi_{m}(A)}{\pi_{m}(A)\matX{24}}.
$$
First, from Proposition~\ref{prop33}, we know that $A\in\setE$. Thus, from Theorem~\ref{mainthm}, we know that the orbit of $\pi_{2^u}(S)$ is $(12.2^u,12.2^u)$-periodic. Therefore, $\pi_{2^u}\!\left(\orb{S}\right)$ is $(12.2^u,12.2^u)$-interlaced doubly arithmetic with common differences
$$
\matD{e_1} = \matD{e_2} = \matr{0_{12.2^u}}.
$$
Moreover, by definition of $\setU{\mu}$ and since $4\setO'\subset\setO'$, we have
$$
\pi_\mu(A)\in \pi_\mu(4\setO')\subset\pi_\mu(\setO').
$$
Let $a_1$ and $a_2$ be two integers such that
$$
\pi_\mu(A) = \pi_\mu(a_1A_1+a_2A_2).
$$
We know from Theorem~\ref{thm8} and its proof that the orbit of $\pi_\mu(S)$ is $(3,3)$-interlaced doubly arithmetic with common differences
$$
\matD{o_1} = a_2\Circ{-1,2,-1}\quad\text{and}\quad \matD{o_2} = a_2\Circ{1,1,-2}.
$$
Thus, using Proposition~\ref{prop26}, we have that $\pi_\mu\!\left(\orb{S}\right)$ is also $(12.2^u,12.2^u)$-interlaced doubly arithmetic with common differences
$$
\matD{o_1}' = 2^{u+2}a_2\Circ{(-1,2,-1)^{2^{u+2}}}\quad\text{and}\quad \matD{o_2}' = 2^{u+2}a_2\Circ{(1,1,-2)^{2^{u+2}}}.
$$
Since $\pi_{2^u}\!\left(\orb{S}\right)$ and $\pi_{\mu}\!\left(\orb{S}\right)$ are both $(12.2^u,12.2^u)$-interlaced doubly arithmetic in $\Zn{2^u}$ and $\Zn{\mu}$, respectively, we deduce from Corollary~\ref{cor9} that the orbit $\orb{S}$ is $(12.2^u,12.2^u)$-interlaced doubly arithmetic in $\Zn{m}$, with common differences $\matD{1}$ and $\matD{2}$ verifying
$$
\left\{\begin{array}{l}
\pi_{2^u}\left(\matD{1}\right) = \matD{e_1} = \matr{0},\\
\pi_{\mu}\left(\matD{1}\right) = \matD{o_1}' = 2^{u+2}a_2\Circ{(-1,2,-1)^{2^{u+2}}},
\end{array}\right.
$$
and
$$
\left\{\begin{array}{l}
\pi_{2^u}\left(\matD{2}\right) = \matD{e_2} = \matr{0},\\
\pi_{\mu}\left(\matD{2}\right) = \matD{o_2}' = 2^{u+2}a_2\Circ{(1,1,-2)^{2^{u+2}}}.
\end{array}\right.
$$
By the Chinese Remainder Theorem, it is easy to see that
$$
\matD{1} = 2^{u+2}a_2\Circ{(-1,2,-1)^{2^{u+2}}}\quad\text{and}\quad \matD{2} = 2^{u+2}a_2\Circ{(1,1,-2)^{2^{u+2}}}.
$$
Since
$$
\mu\matD{1} \equiv \mu\matD{2} \equiv \matr{0} \pmod{m},
$$
we obtain that $\orb{S}$ is $(12m,12m)$-periodic.

Let $\lambda$ be a positive integer. We consider the triangle
$$
\nabla = \ST{S[12\lambda m]} = \left(a_{i,j}\right)_{(i,j)\in T_{12\lambda m}}.
$$
From Proposition~\ref{prop28}, we know that $\nabla$ can be decomposed into $9.2^{2u+4}$ subtriangles $\nabla_{r,s}$ that are the arithmetic triangles
$$
\nabla_{r,s} = \left(a_{r+12.2^ui,s+12.2^uj}\right)_{(i,j)\in T_{n_{r,s}}} = \AT{a_{r,s}}{d^{\,(1)}_{r,s}}{d^{\,(2)}_{r,s}}{n_{r,s}},
$$
where
$$
n_{r,s} = \left\{\begin{array}{ll}
 \lambda \mu & \text{if}\ 0\le r+s\le 12.2^u-1,\\
 \lambda \mu-1 & \text{if}\ 12.2^u\le r+s\le 24.2^u-2,\\
\end{array}\right.
$$
for all $(r,s)\in\{0,\ldots,12.2^u-1\}^2$. Moreover, since
$$
\matD{1} = \left(d^{\,(1)}_{r,s}\right)_{0\le r,s\le 12.2^u-1} = 2^{u+2}a_2\Circ{(-1,2,-1)^{2^{u+2}}}
$$
and
$$
\matD{2} = \left(d^{\,(2)}_{r,s}\right)_{0\le r,s\le 12.2^u-1} = 2^{u+2}a_2\Circ{(1,1,-2)^{2^{u+2}}},
$$
we deduce that
$$
\left(d^{\,(1)}_{r,s},d^{\,(2)}_{r,s}\right) = 
\left\{\begin{array}{cl}
 2^{u+2}a_2(-1,1) & \text{if}\ r-s\equiv0\pmod{3}, \\
 2^{u+2}a_2(-1,-2) & \text{if}\ r-s\equiv1\pmod{3}, \\
 2^{u+2}a_2(2,1) & \text{if}\ r-s\equiv2\pmod{3}, \\
\end{array}\right.
$$
for all $(r,s)\in\{0,\ldots,12.2^u-1\}^2$. We distinguish different cases according to the residue class of $r-s$ modulo $3$. Recall, by definition, that $\gcd(a_2,\mu)=1$.

\setcounter{case}{0}
\begin{case}
If $r-s\equiv0\pmod{3}$, then
$$
d^{\,(1)}_{r,s} = -2^{u+2}a_2,\ d^{\,(2)}_{r,s} = 2^{u+2}a_2\ \text{and}\ d^{\,(2)}_{r,s}-d^{\,(1)}_{r,s} = 2^{u+3}a_2.
$$
Since
$$
\gcd(d^{\,(1)}_{r,s},m) = \gcd(d^{\,(2)}_{r,s},m) = \gcd(d^{\,(2)}_{r,s}-d^{\,(1)}_{r,s},m) = 2^{u},\ \frac{m}{\gcd(d^{\,(1)}_{r,s},m)}=\mu
$$
and $n_{r,s}\in\left\{\lambda\mu-1,\lambda\mu\right\}$, we deduce from Theorem~\ref{thm5} that
$$
\mf{\nabla_{r,s}}(x+2^{u}) = \mf{\nabla_{r,s}}(x),
$$
for all $x\in\Zn{m}$, in this case.
\end{case}

\begin{case}
If $r-s\equiv1\pmod{3}$, then
$$
d^{\,(1)}_{r,s} = -2^{u+2}a_2,\ d^{\,(2)}_{r,s} = -2^{u+3}a_2\ \text{and}\ d^{\,(2)}_{r,s}-d^{\,(1)}_{r,s} = -2^{u+2}a_2.
$$
Since
$$
\gcd(d^{\,(1)}_{r,s},m) = \gcd(d^{\,(2)}_{r,s},m) = \gcd(d^{\,(2)}_{r,s}-d^{\,(1)}_{r,s},m) = 2^{u},\ \frac{m}{\gcd(d^{\,(1)}_{r,s},m)}=\mu
$$
and $n_{r,s}\in\left\{\lambda\mu-1,\lambda\mu\right\}$, we deduce from Theorem~\ref{thm5} that
$$
\mf{\nabla_{r,s}}(x+2^{u}) = \mf{\nabla_{r,s}}(x),
$$
for all $x\in\Zn{m}$, in this case.
\end{case}

\begin{case}
If $r-s\equiv2\pmod{3}$, then
$$
d^{\,(1)}_{r,s} = 2^{u+3}a_2,\ d^{\,(2)}_{r,s} = 2^{u+2}a_2\ \text{and}\ d^{\,(2)}_{r,s}-d^{\,(1)}_{r,s} = -2^{u+2}a_2.
$$
Since
$$
\gcd(d^{\,(1)}_{r,s},m) = \gcd(d^{\,(2)}_{r,s},m) = \gcd(d^{\,(2)}_{r,s}-d^{\,(1)}_{r,s},m) = 2^{u},\ \frac{m}{\gcd(d^{\,(1)}_{r,s},m)}=\mu
$$
and $n_{r,s}\in\left\{\lambda\mu-1,\lambda\mu\right\}$, we deduce from Theorem~\ref{thm5} that
$$
\mf{\nabla_{r,s}}(x+2^{u}) = \mf{\nabla_{r,s}}(x),
$$
for all $x\in\Zn{m}$, in this case.
\end{case}

Therefore, we have
$$
\mf{\nabla_{r,s}}(x+2^{u}) = \mf{\nabla_{r,s}}(x),
$$
for all $x\in\Zn{m}$, in any subtriangle $\nabla_{r,s}$, for all $(r,s)\in\{0,\ldots,12.2^u-1\}^2$. Since
$$
\nabla = \bigsqcup_{(r,s)\in\{0,\ldots,12.2^u-1\}^2} \nabla_{r,s},
$$
we deduce that
$$
\mf{\nabla}(x+2^{u}) = \mf{\nabla}(x),
$$
for all $x\in\Zn{m}$.

Finally, since $\pi_{2^{u}}(\nabla)=\ST{\pi_{2^u}(S)[12\lambda m]}$ is balanced in $\Zn{2^{u}}$ by Theorem~\ref{mainthm} and since $\mf{\nabla}(x+2^{u}) = \mf{\nabla}(x)$, for all $x\in\Zn{m}$, by the previous result, we obtain from Theorem~\ref{thm6} that $\nabla$ is balanced in $\Zn{m}$.

This completes the proof.
\end{proof}

\begin{cor}
Let $m$ be a positive integer and let $\mu$ its odd part, i.e., the odd number $\mu$ such that $m=2^u\mu$ for a certain non-negative integer $u$. Then, for the $24$-tuple of integers
\begin{equation*}
\resizebox{\textwidth}{!}{$
A =
\begin{array}[t]{l}
\left(0,4,\mu-4,\mu-4,12-2\mu,3\mu-8,2\mu-8,20-2\mu,-12,2\mu-12,28,-\mu-12,3\mu-12,\right. \\
\left.36-4\mu,-20,2\mu-20,44-2\mu,-24,-\mu-24,52-2\mu,\mu-28,\mu-28,60-4\mu,2\mu-32\right), \\
\end{array}
$}
\end{equation*}
the orbit of the sequence $S=\IAP{\pi_{m}(A)}{\pi_{m}(A)\matX{24}}$ is
\begin{itemize}
\item
$(12m,12m)$-periodic and the triangles $\ST{S[12\lambda m]}$ are balanced in $\Zn{m}$, for all non-negative integers $\lambda$, if $m$ is even ($u\ge1$),
\item
$(3m,3m)$-periodic and the triangles $\ST{S[3\lambda m]}$ are balanced in $\Zn{m}$, for all non-negative integers $\lambda$, if $m$ is odd ($u=0$).
\end{itemize}
\end{cor}

\begin{proof}
If we consider the $24$-tuples of integers
$$
A_0
\begin{array}[t]{l}
= X_1-4Y_5-4Y_8 \\
= (0, 0, 1, 1, -2, 3, 2, -2, 0, 2, 0, -1, 3, -4, 0, 2, -2, 0, -1, -2, 1, 1, -4, 2) \in\setE_4\subset\setE
\end{array}
$$
and
\begin{equation*}
\resizebox{\textwidth}{!}{$
A_2 = (0,1,-1,-1,3,-2,-2,5,-3,-3,7,-4,-4,9,-5,-5,11,-6,-6,13,-7,-7,15,-8) \in\setO',
$}
\end{equation*}
we obtain that
$$
A = \mu A_0 + 4A_2 \in \mu\setE+4\setO'=\setU{\mu}.
$$
Therefore, when $m$ is even ($u\ge1$ and $\mu<m$), by Theorem~\ref{thm9}, we obtained that $\orb{S}$ is  $(12m,12m)$-periodic and the triangles $\ST{S[12\lambda m]}$ are balanced in $\Zn{m}$, for all non-negative integers $\lambda$. Finally, when $m$ is odd ($u=0$ and $\mu=m$), we know that $\pi_m(A)=\pi_m(4A_2)$ and the result directly comes from Corollary~\ref{cor*5}.
\end{proof}

\begin{cor}\label{cor*6}
Let $m_0$ be an even number and let $\mu$ be the least common multiple of the first $\frac{m_0}{2}$ even numbers, i.e., $\mu=\lcm(1,3,\ldots,m_0-1)$. Then, for any $A\in\setU{\mu}$ and any positive integer $m\le m_0$, the orbit of the sequence $S=\IAP{\pi_{m}(A)}{\pi_{m}(A)\matX{24}}$ is
\begin{itemize}
\item
$(12m,12m)$-periodic and the triangles $\ST{S[12\lambda m]}$ are balanced in $\Zn{m}$, for all non-negative integers $\lambda$, if $m$ is even,
\item
$(3m,3m)$-periodic and the triangles $\ST{S[3\lambda m]}$ are balanced in $\Zn{m}$, for all non-negative integers $\lambda$, if $m$ is odd.
\end{itemize}
\end{cor}

\begin{proof}
Let $m$ be a even integer such that $m\le m_0$ and let $\nu$ be its odd part. Since $\nu\le m_0-1$, it is clear that $\nu$ divides $\mu=\lcm(1,3,\ldots,m_0-1)$. Since $\mu = \nu\times\frac{\mu}{\nu}$, where $\frac{\mu}{\nu}$ is odd, and since $\frac{\mu}{\nu}\setE\subset\setE$ by Proposition~\ref{prop31}, we obtain that
$$
\setU{\mu} = \mu\setE+4\setO' = \nu\left(\frac{\mu}{\nu}\setE\right)+4\setO' \subset \nu\setE+4\setO' = \setU{\nu}.
$$
Since $A\in\setU{\nu}$, we conclude by Theorem~\ref{thm9} that $\orb{S}$ is  $(12m,12m)$-periodic and the triangles $\ST{S[12\lambda m]}$ are balanced in $\Zn{m}$, for all non-negative integers $\lambda$.

Finally, if $m$ is odd, since $m$ divides $\mu$, we know that $\pi_m(A)\in\pi_m(4\setO')\subset\pi_m(\setO')\subset\pi_m(\setO)$ and the result directly comes from Corollary~\ref{cor*5}.
\end{proof}

For instance, for $m_0=10$, we have
$$
\mu = \lcm(1,3,5,7,9) = 3^2.5.7 = 315.
$$
If we consider
\begin{equation*}
\resizebox{\textwidth}{!}{$
A
\begin{array}[t]{l}
=
\begin{array}[t]{l}
\left(0,4,\mu-4,\mu-4,12-2\mu,3\mu-8,2\mu-8,20-2\mu,-12,2\mu-12,28,-\mu-12,3\mu-12,\right. \\
\left.36-4\mu,-20,2\mu-20,44-2\mu,-24,-\mu-24,52-2\mu,\mu-28,\mu-28,60-4\mu,2\mu-32\right), \\
\end{array} \\ \ \\
=
\begin{array}[t]{l}
(0, 4, 311, 311, -618, 937, 622, -610, -12, 618, 28, -331, 929, -1224, -20, 610,  -586,\\
 -24, -339, -578, 287, 287, -1200, 598) \\
\end{array} \\
\end{array}
$}
\end{equation*}
of $\setU{315}$, we know from Corollary~\ref{cor*6} that, for all positive integers $m\le 10$, the orbit of the sequence $S=\IAP{\pi_{m}(A)}{\pi_{m}(A)\matX{24}}$ is
\begin{itemize}
\item
$(12m,12m)$-periodic and the triangles $\ST{S[12\lambda m]}$ are balanced in $\Zn{m}$, for all non-negative integers $\lambda$, if $m$ is even,
\item
$(3m,3m)$-periodic and the triangles $\ST{S[3\lambda m]}$ are balanced in $\Zn{m}$, for all non-negative integers $\lambda$, if $m$ is odd.
\end{itemize}
Moreover, we know that
$$
S=\IAP{\pi_{m}(A)}{\pi_{m}(A)\matX{24}} = {A_{m}}^\infty,
$$
where $A_{m}$ is the first period of $S$ of length $12m$ if $m$ is even, or of length $3m$ if $m$ is odd. For all $m\in\{1,\ldots,10\}$, we obtain that
\begin{equation*}
\resizebox{\textwidth}{!}{$
\begin{array}{l}
A_1 = 000, \\
A_2 = 001101000001100000101100, \\
A_3 = 012201120, \\
A_4 = 003321220201100220123302201123022003302022321100, \\
A_5 = 041122203334410, \\
A_6 = 045501420045504420345504423345204423342204123342201123042201120042501120, \\
A_7 = 043356662205511124430, \\
A_8 = 047761664245104260567706605527422003742026325544443365260641500664163302201123026407346422721140, \\
A_9 = 045531126612207783378864450, \\
A_{10} = 041127208889960046127708889965046627708884965546627703884465546622703384465541622203384460541122203389460041122208389960.
\end{array}
$}
\end{equation*}
We know thus that, for all $m\in\{1,\ldots,10\}$, the triangle $\ST{{A_m}^\lambda}$ is a balanced triangle in $\Zn{m}$, of size $12\lambda m$ when $m$ is even or of size $3\lambda m$ when $m$ is odd, for all non-negative integers $\lambda$. The balanced triangles $\ST{{A_m}^\lambda}$ in $\Zn{m}$ are depicted for $\lambda\in\{1,2\}$ and for $m=2$ in Figures~\ref{fig03} and \ref{fig04}, for $m\in\{3,4\}$ in Figure~\ref{fig*10}, for $m\in\{5,6\}$ in Figure~\ref{fig*11}, for $m\in\{7,8\}$ in Figure~\ref{fig*12} and for $m\in\{9,10\}$ in Figure~\ref{fig*13}.

\begin{figure}[htbp]
\centering{
\resizebox{\textwidth}{!}{
\begin{tabular}{cc}
\includegraphics[width=0.5\textwidth]{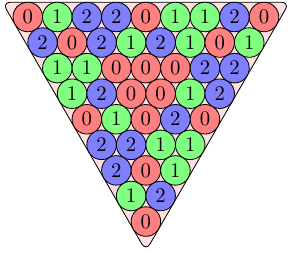}
&
\includegraphics[width=0.5\textwidth]{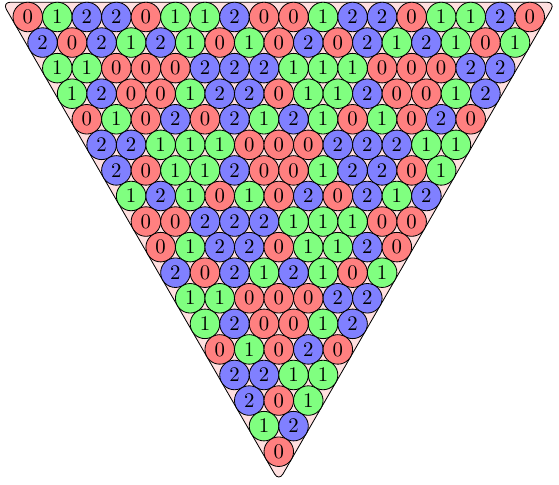} \\
\includegraphics{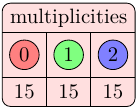}
&
\includegraphics{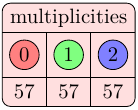} \\
$\ST{A_3}$ & $\ST{{A_3}^2}$ \\ \ \\
\includegraphics[width=0.5\textwidth]{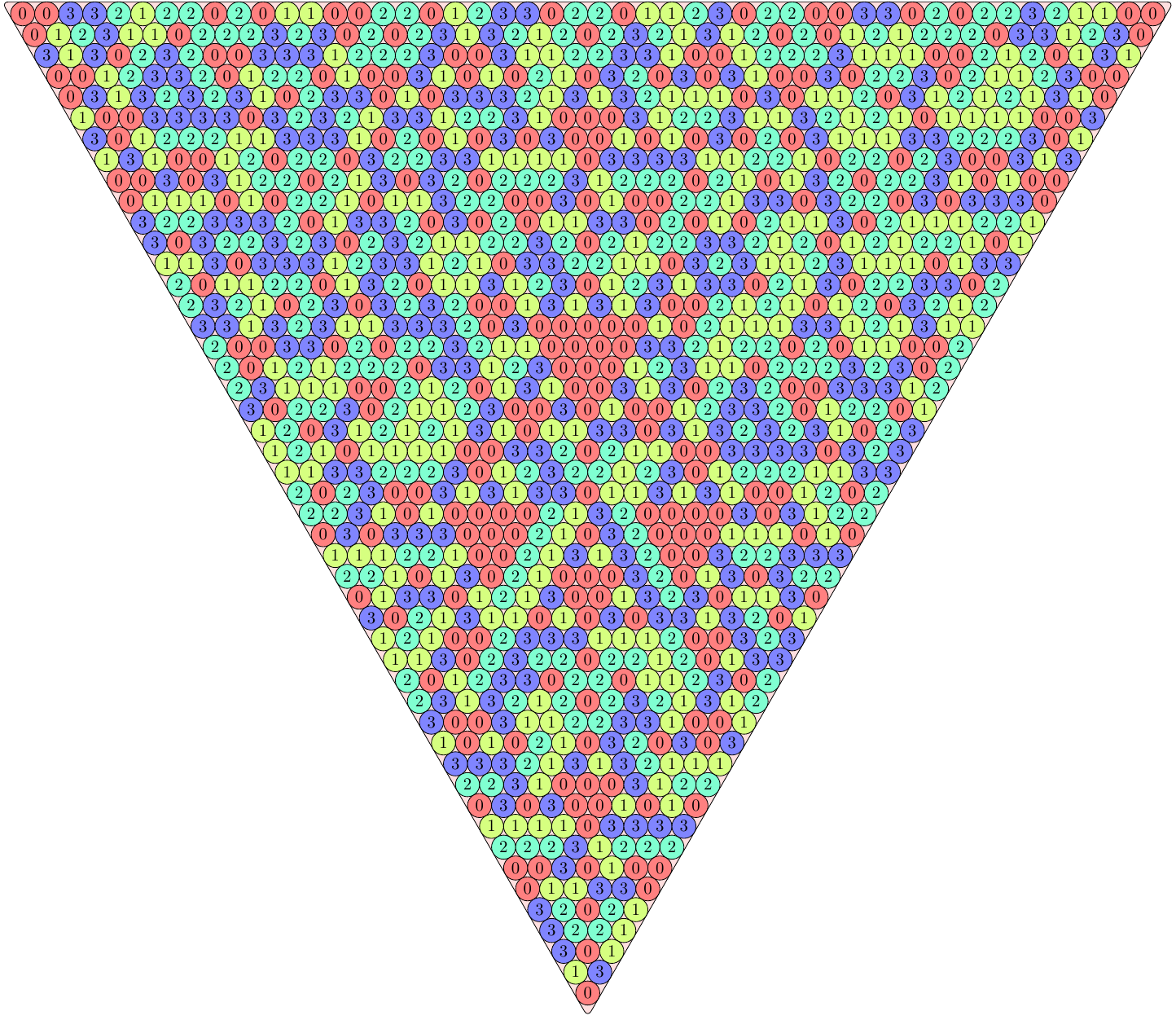}
&
\includegraphics[width=0.5\textwidth]{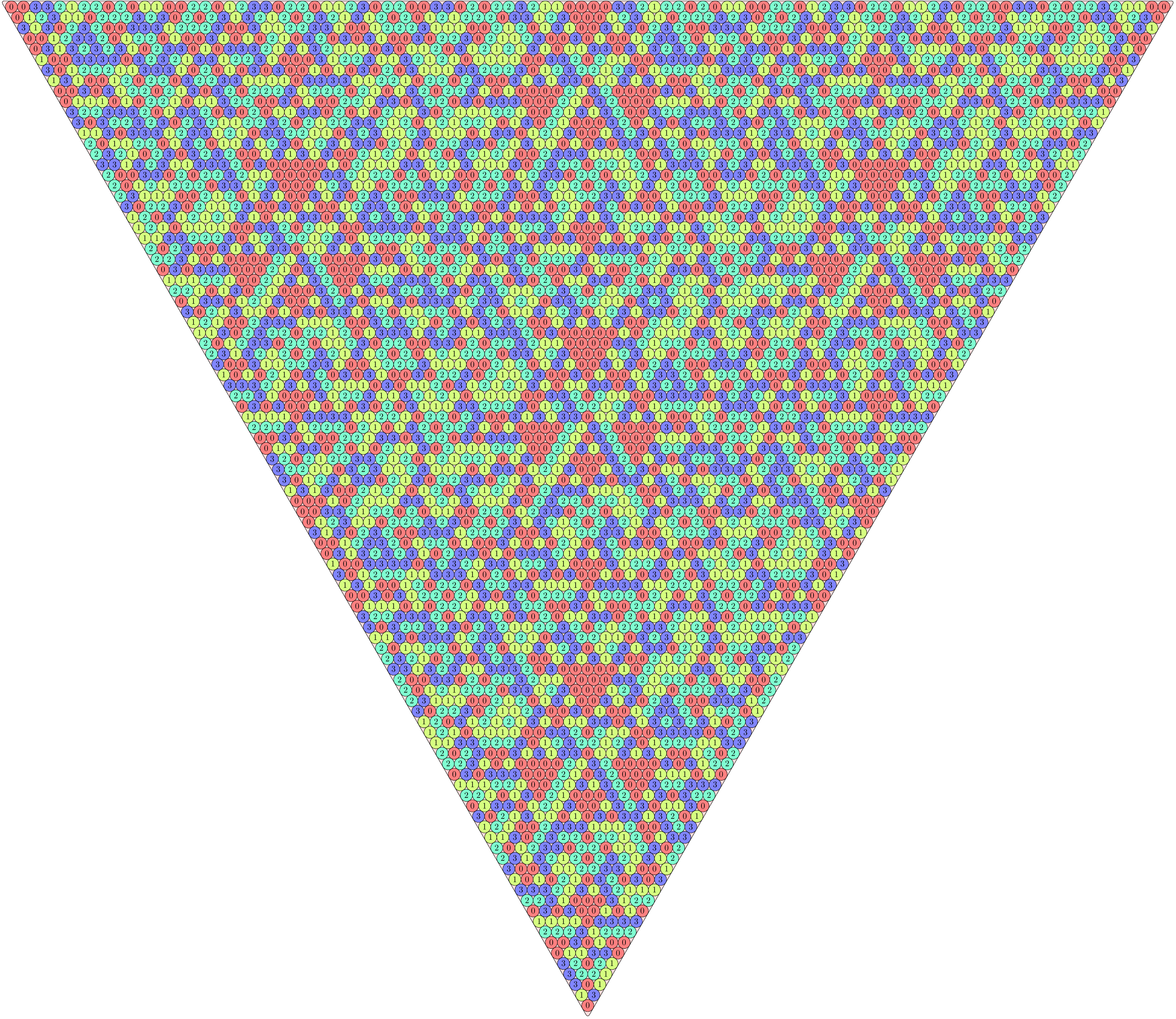} \\
\includegraphics{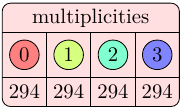}
&
\includegraphics{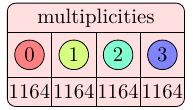} \\
$\ST{A_4}$ & $\ST{{A_4}^2}$ \\
\end{tabular}
}}
\caption{The balanced triangles $\ST{{A_m}^{\lambda}}$ in $\Zn{m}$ for $\lambda\in\{1,2\}$ and $m\in\{3,4\}$}\label{fig*10}
\end{figure}

\begin{figure}[htbp]
\centering{
\resizebox{\textwidth}{!}{
\begin{tabular}{cc}
\includegraphics[width=0.5\textwidth]{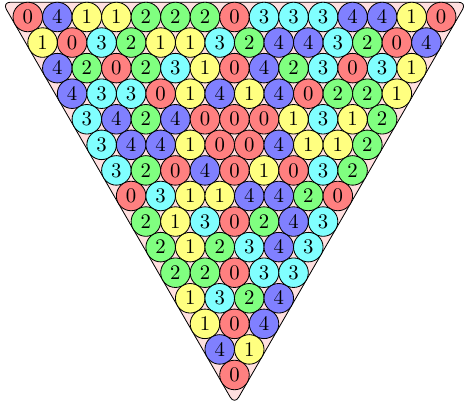}
&
\includegraphics[width=0.5\textwidth]{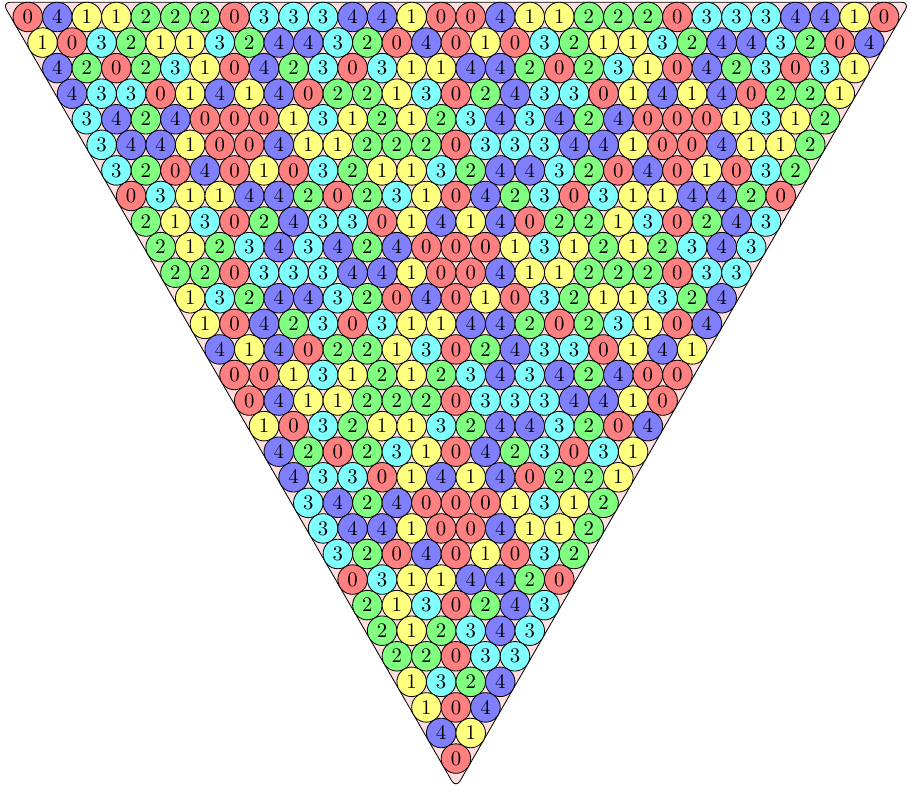} \\
\includegraphics{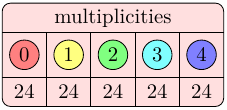}
&
\includegraphics{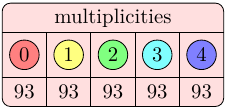} \\
$\ST{A_5}$ & $\ST{{A_5}^2}$ \\ \ \\
\includegraphics[width=0.5\textwidth]{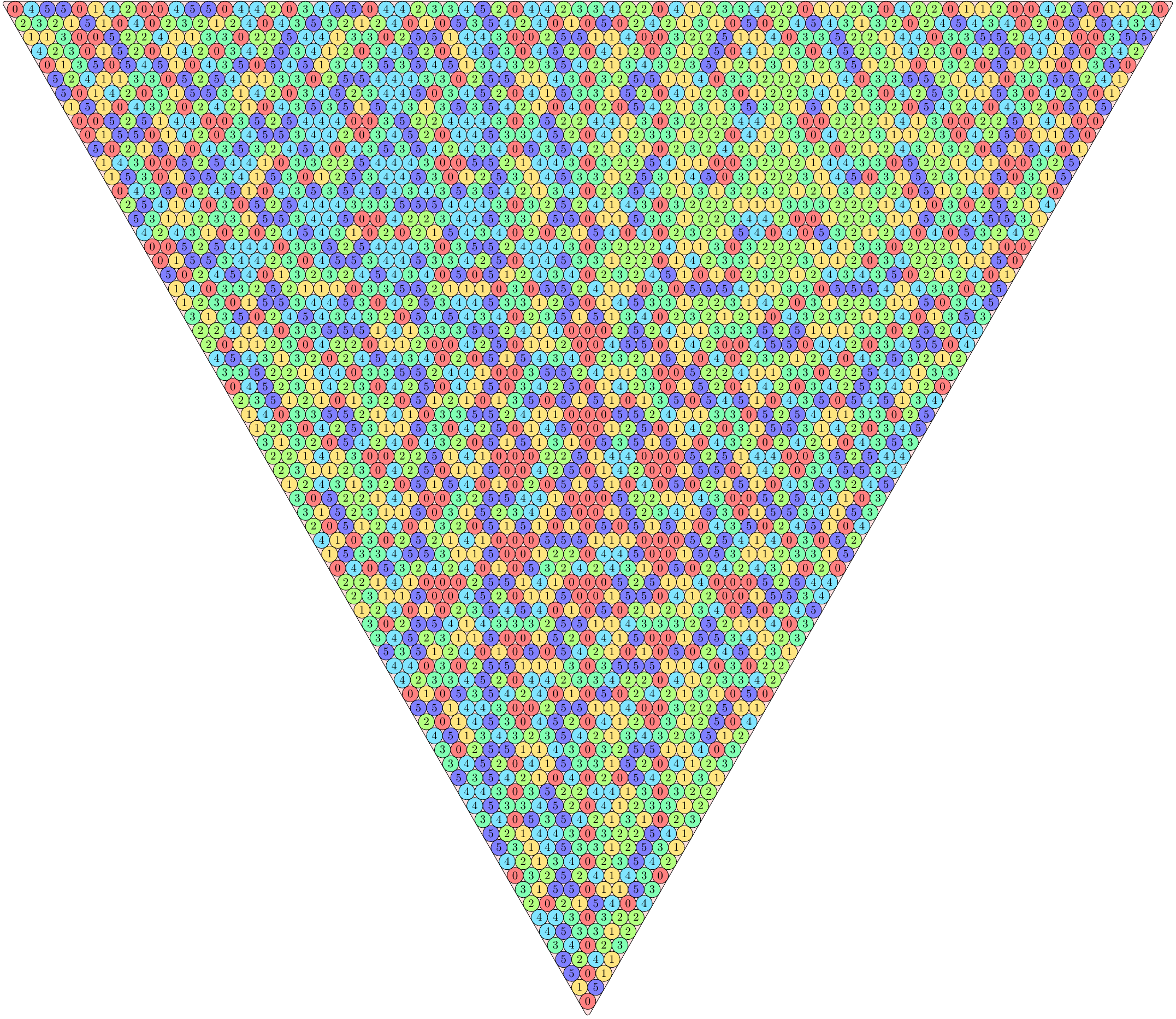}
&
\includegraphics[width=0.5\textwidth]{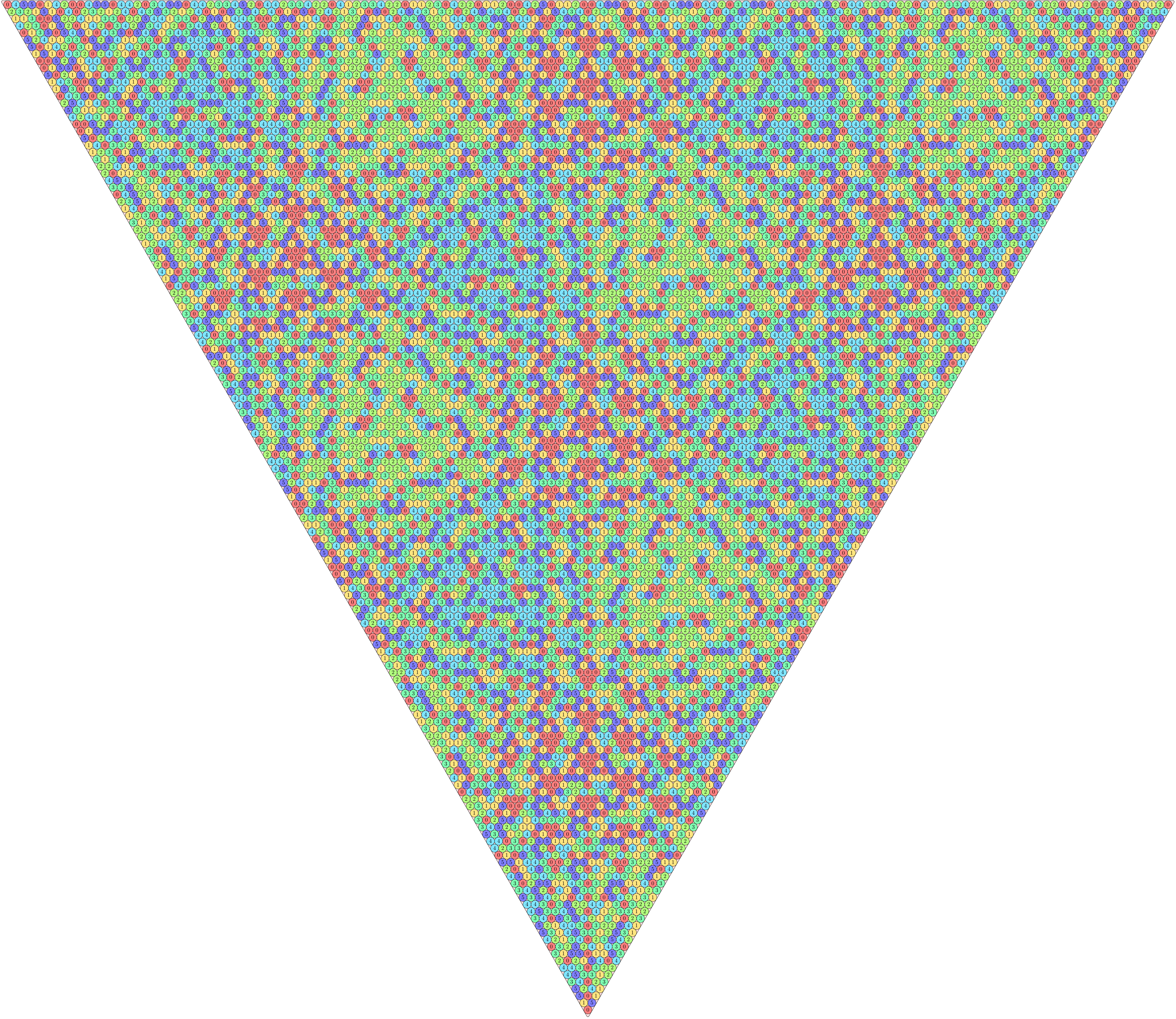} \\
\includegraphics{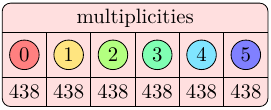}
&
\includegraphics{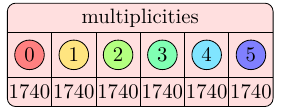} \\
$\ST{A_6}$ & $\ST{{A_6}^2}$ \\
\end{tabular}
}}
\caption{The balanced triangles $\ST{{A_m}^{\lambda}}$ in $\Zn{m}$ for $\lambda\in\{1,2\}$ and $m\in\{5,6\}$}\label{fig*11}
\end{figure}

\begin{figure}[htbp]
\centering{
\resizebox{\textwidth}{!}{
\begin{tabular}{cc}
\includegraphics[width=0.5\textwidth]{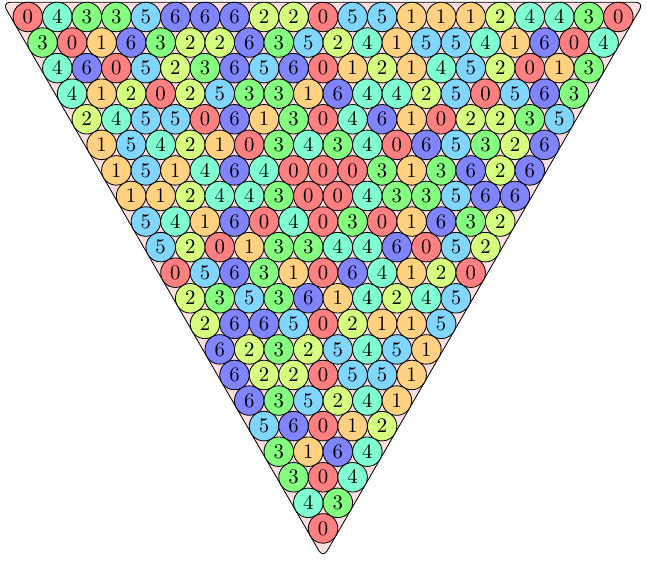}
&
\includegraphics[width=0.5\textwidth]{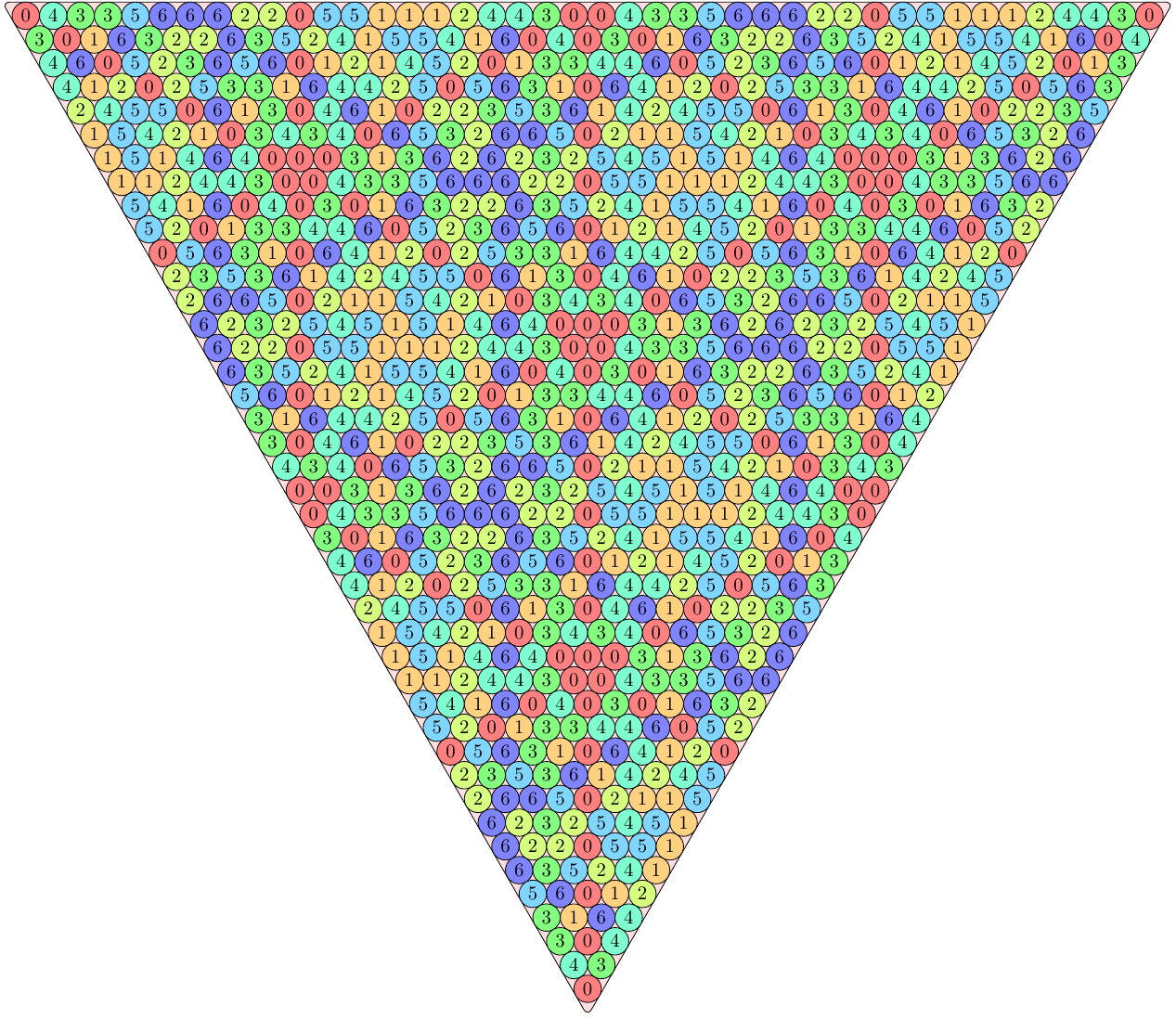} \\
\includegraphics{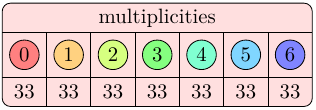}
&
\includegraphics{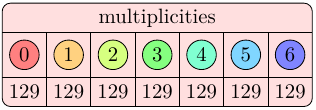} \\
$\ST{A_7}$ & $\ST{{A_7}^2}$ \\ \ \\
\includegraphics[width=0.5\textwidth]{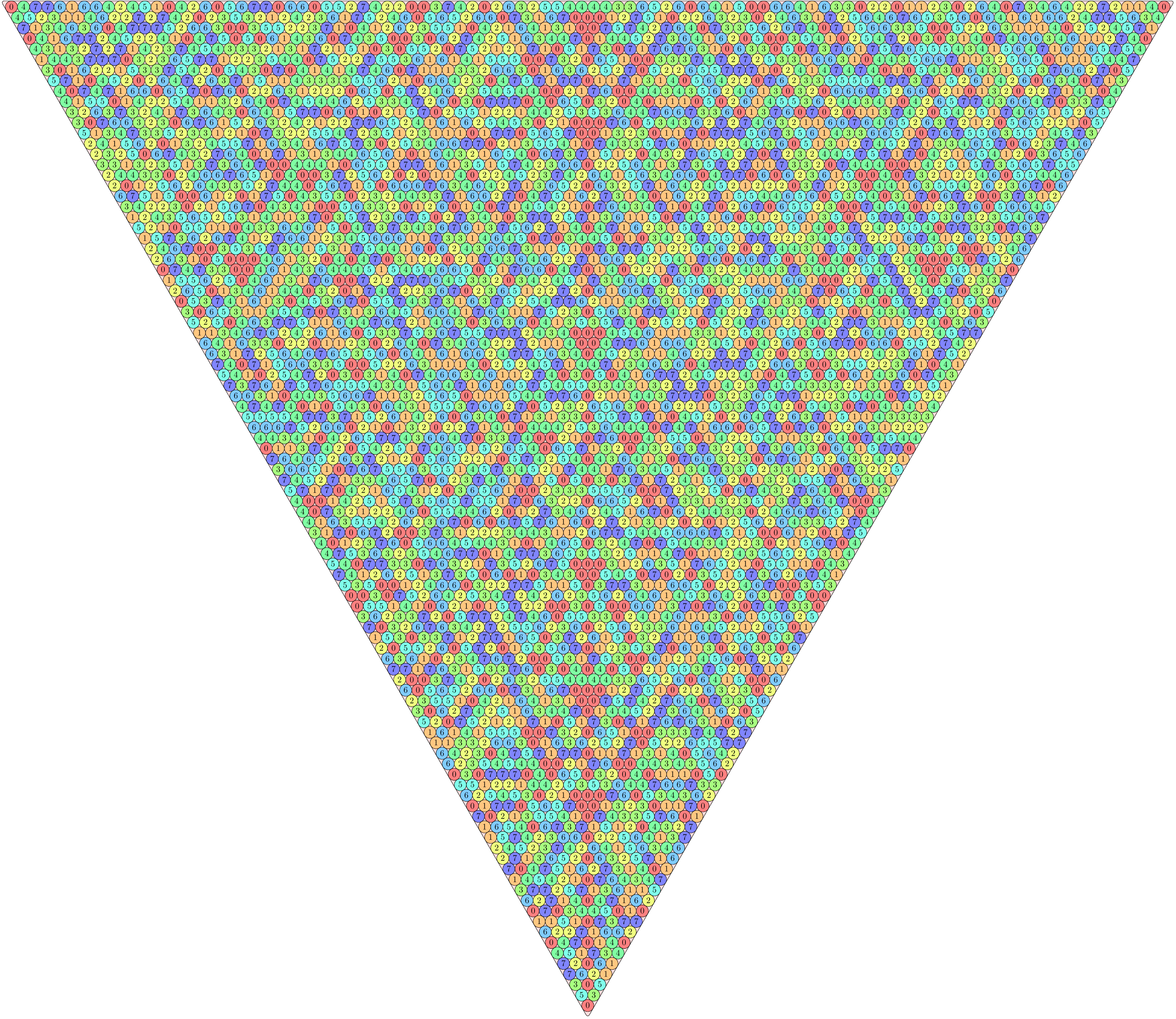}
&
\includegraphics[width=0.5\textwidth]{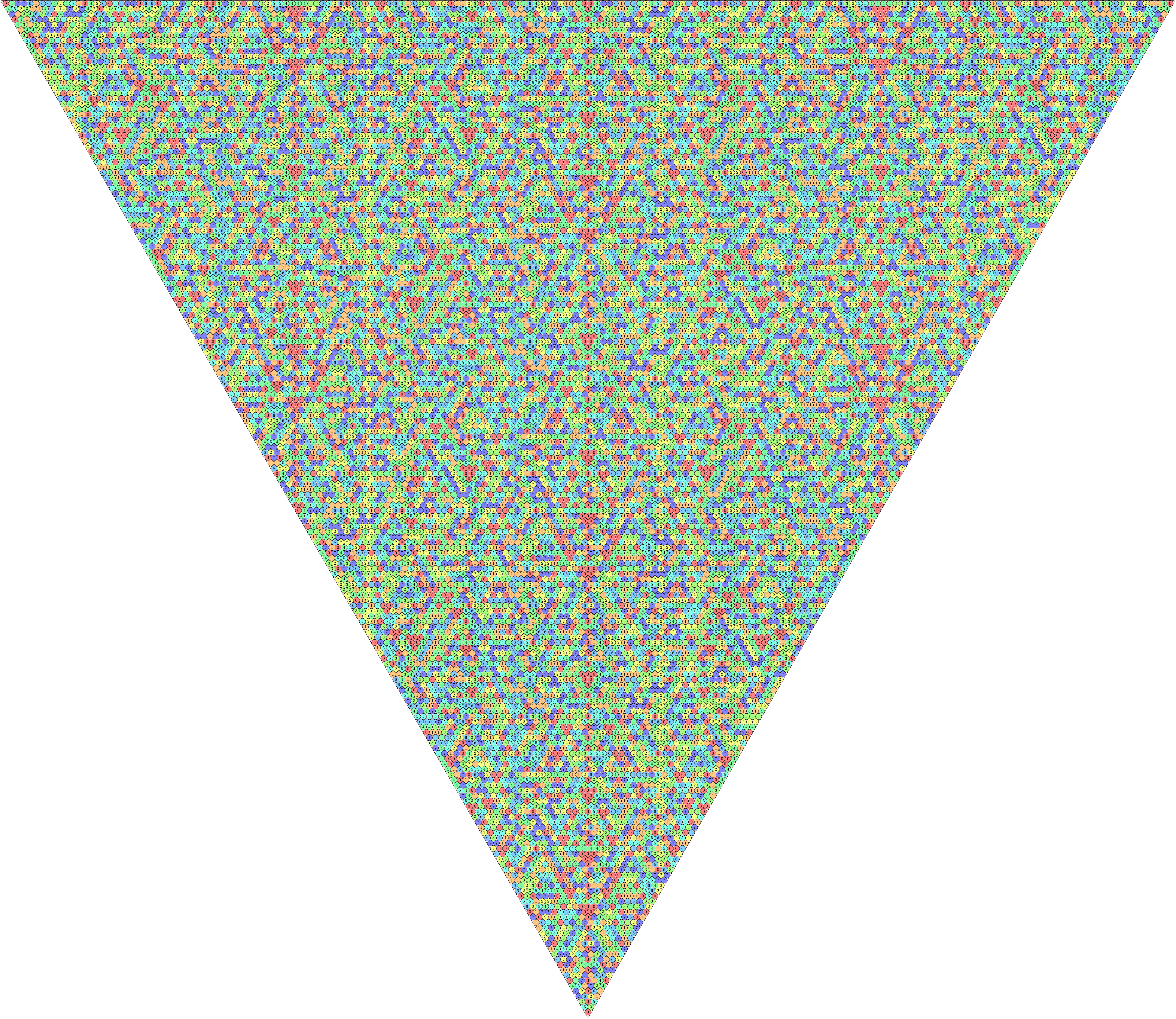} \\
\includegraphics{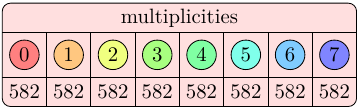}
&
\includegraphics{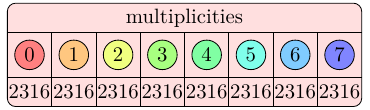} \\
$\ST{A_8}$ & $\ST{{A_8}^2}$ \\
\end{tabular}
}}
\caption{The balanced triangles $\ST{{A_m}^{\lambda}}$ in $\Zn{m}$ for $\lambda\in\{1,2\}$ and $m\in\{7,8\}$}\label{fig*12}
\end{figure}

\begin{figure}[htbp]
\centering{
\resizebox{\textwidth}{!}{
\begin{tabular}{cc}
\includegraphics[width=0.5\textwidth]{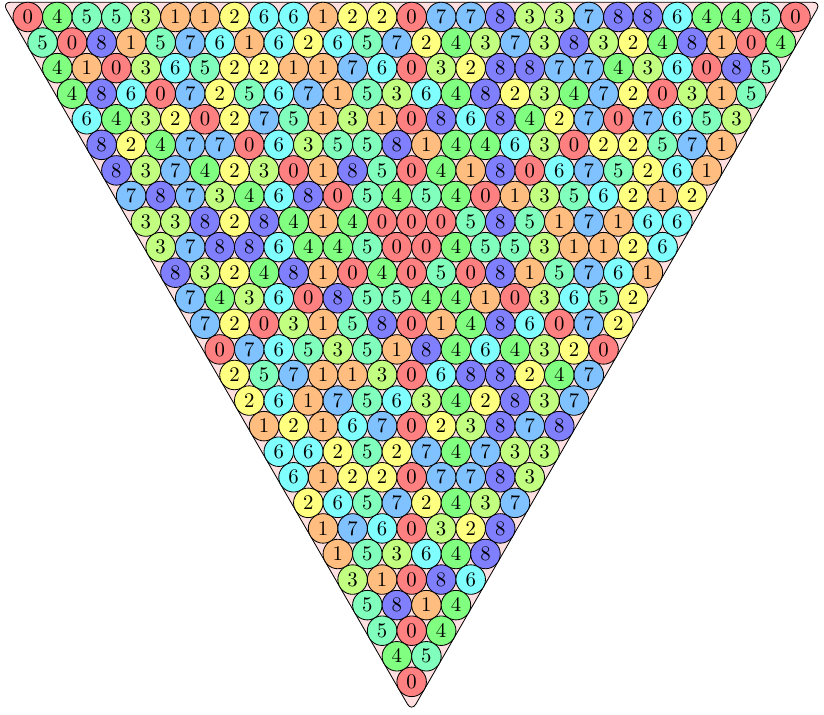}
&
\includegraphics[width=0.5\textwidth]{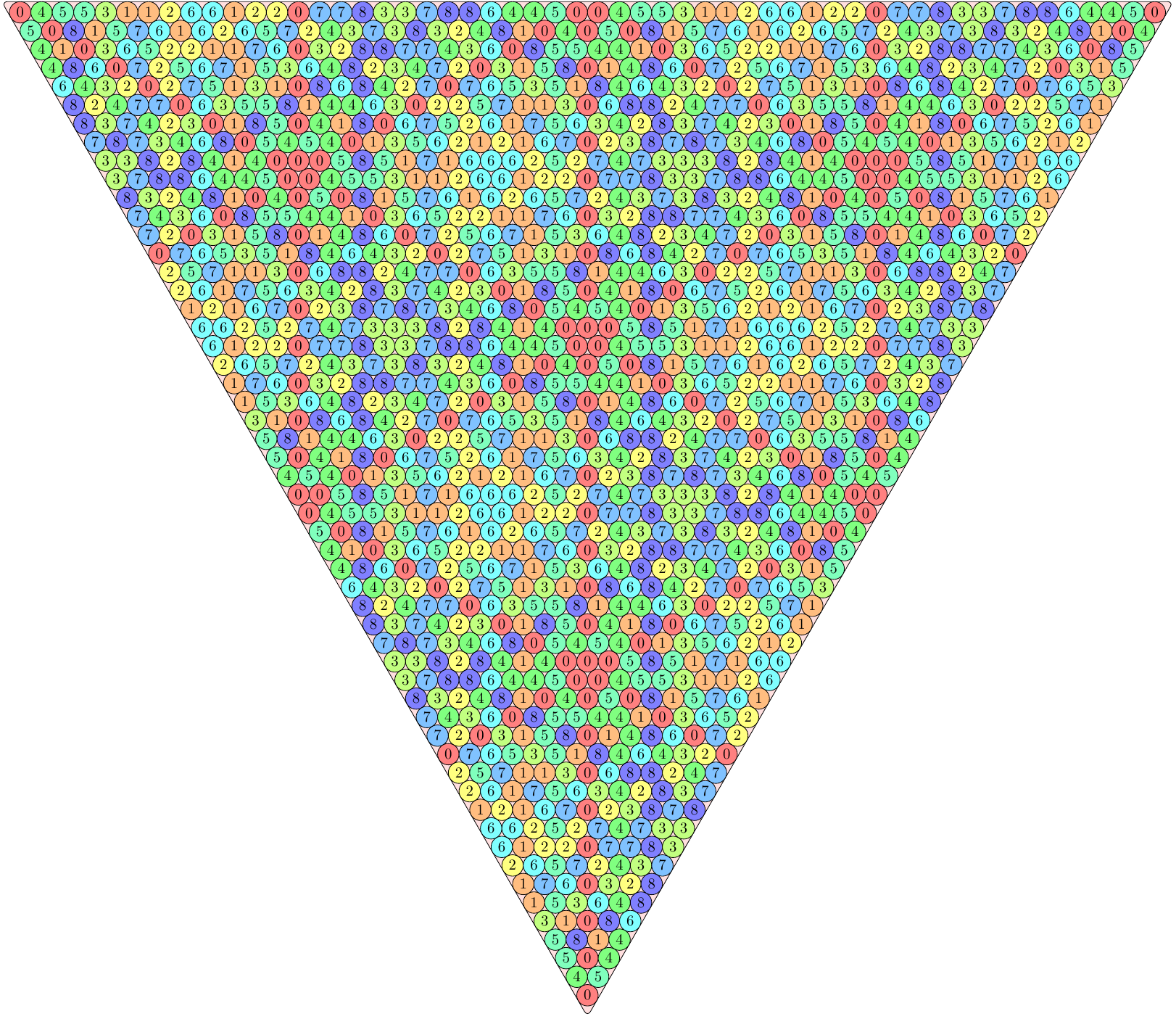} \\
\includegraphics{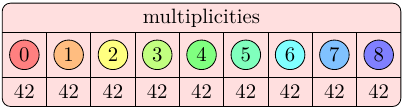}
&
\includegraphics{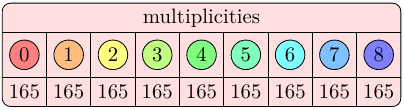} \\
$\ST{A_9}$ & $\ST{{A_9}^2}$ \\ \ \\
\includegraphics[width=0.5\textwidth]{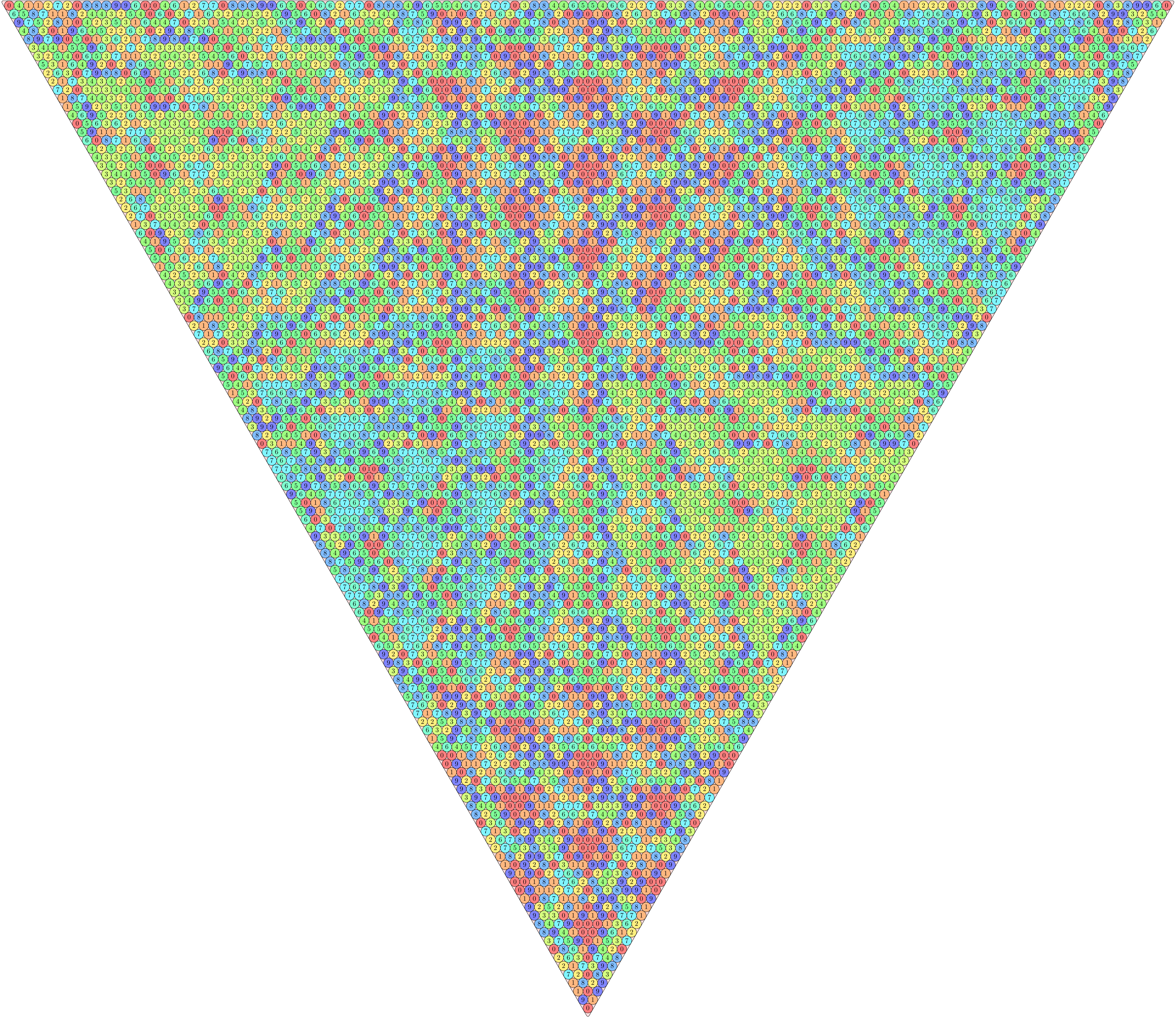}
&
\includegraphics[width=0.5\textwidth]{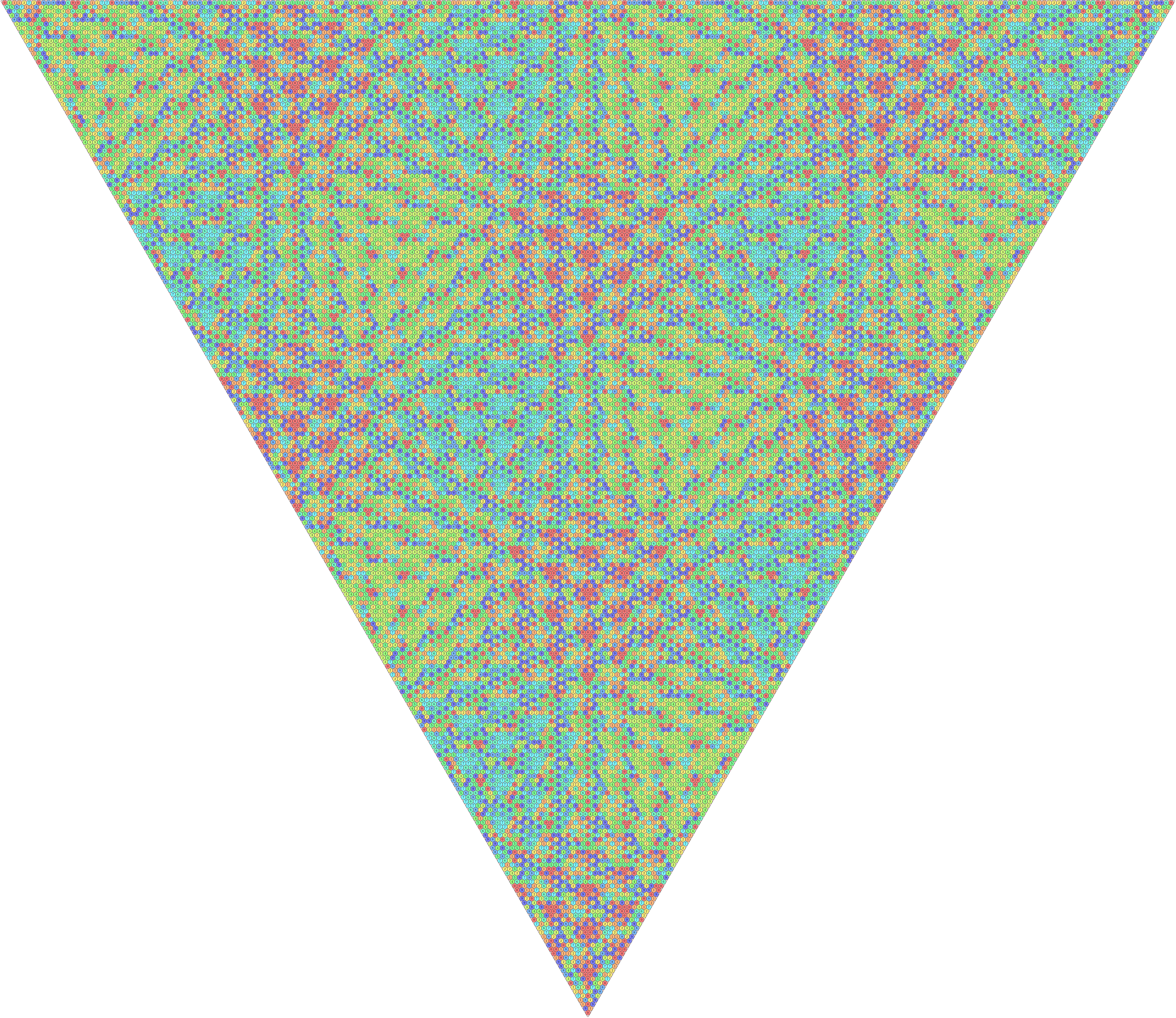} \\
\includegraphics{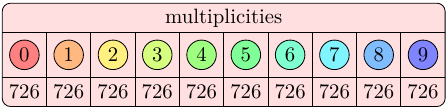}
&
\includegraphics{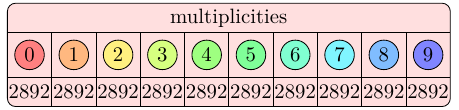} \\
$\ST{A_{10}}$ & $\ST{{A_{10}}^2}$ \\
\end{tabular}
}}
\caption{The balanced triangles $\ST{{A_m}^{\lambda}}$ in $\Zn{m}$ for $\lambda\in\{1,2\}$ and $m\in\{9,10\}$}\label{fig*13}
\end{figure}



\section*{Acknowledgements}

This work has been realized with the support of the ISDM-MESO Platform at the University of Montpellier funded under the CPER by the French Government, the Occitanie Region, the Metropole of Montpellier and the University of Montpellier.

\addcontentsline{toc}{section}{References}
\bibliographystyle{plain}


\end{document}